\documentclass[a4paper,12pt]{article}
\usepackage{amsmath,amsthm,amssymb,amsfonts,mathrsfs}
\usepackage{bbm}
\usepackage{enumerate}
\usepackage{color}
\usepackage[usenames,dvipsnames]{xcolor}

\usepackage[]{amsmath}
\usepackage{amsthm}
\usepackage{amssymb}
\usepackage{bbm}
\usepackage{stmaryrd}
\usepackage{microtype}

\usepackage{xr}
\PassOptionsToPackage{colorlinks}{hyperref}
\usepackage[textsize=small]{todonotes}

\usepackage[colorlinks]{hyperref}

\newcommand{\E}{\mathbb{E}}
\newcommand{\R}{\mathbb{R}}
\newcommand{\N}{\mathbb{N}}
\newcommand{\Z}{\mathbb{Z}}

\renewcommand{\P}{\mathbb{P}}

\newcommand{\F}{\mathscr{F}}

\newcommand{\cH}{\mathcal{H}}
\newcommand{\cL}{\mathcal{L}}
\newcommand{\cM}{\mathcal{M}}
\newcommand{\cP}{\mathcal{P}}

\newcommand{\cS}{\mathcal{S}}
\newcommand{\cX}{\mathcal{X}}
\newcommand{\cU}{\mathcal{U}}

\newcommand{\whp}{whp}
\newcommand{\whpdot}{whp.} 

\newcommand{\condparentheses}[2]{\left(\left.#1\,\right\vert#2\right)}
\newcommand{\condparenthesesreversed}[2]{\left(#1\left\vert\,#2\right.\right)}
\newcommand{\condP}[2]{\mathbb{P}\condparentheses{#1}{#2}}
\newcommand{\condE}[2]{\mathbb{E}\condparentheses{#1}{#2}}

\DeclareMathOperator{\Var}{Var}
\DeclareMathOperator{\Cov}{Cov}

\DeclareMathOperator{\supp}{supp}

\newcommand{\boundary}{\partial}

\newcommand{\set}[1]{\left\{#1\right\}}
\newcommand{\shortset}[1]{\{#1\}}
\newcommand{\bigset}[1]{\big\{#1\big\}}
\newcommand{\abs}[1]{\left\vert#1\right\vert}
\newcommand{\shortabs}[1]{\vert#1\vert}
\newcommand{\bigabs}[1]{\big\vert#1\big\vert}
\newcommand{\norm}[1]{\left\Vert#1\right\Vert}
\newcommand{\shortnorm}[1]{\Vert#1\Vert}
\newcommand{\floor}[1]{\lfloor#1\rfloor}

\newcommand{\ocinterval}[1]{\left(#1\right]} 
\newcommand{\cointerval}[1]{\left[#1\right)}

\newcommand{\indicatorofset}[1]{{\mathbbm{1}}_{#1}}
\newcommand{\indicator}[1]{\indicatorofset{\set{#1}}}

\newcommand{\union}{\cup}
\newcommand{\bigunion}{\bigcup}
\newcommand{\intersect}{\cap}

\newcommand{\equalsd}{\overset{d}{=}}

\newcommand{\decreasesto}{\downarrow}

\newcommand{\st}{\text{st}}

\renewcommand{\th}{\text{th}}

\newcommand{\graphG}{\mathcal{G}}

\renewcommand{\emptyset}{\varnothing}

\newcommand{\basicparent}{p}
\newcommand{\parent}[1]{\basicparent\left(#1\right)}
\newcommand{\ancestor}[2]{\basicparent^{#1}\!\left(#2\right)}

\newcommand{\tree}{\mathcal{T}}
\newcommand{\twoPWITs}{\tree}

\newcommand{\SWT}{{\sf SWT}}

\newcommand{\BP}{{\sf BP}}
\newcommand{\fr}{{\rm fr}}
\newcommand{\unfr}{{\rm unfr}}
\newcommand{\coll}{{\rm coll}}
\newcommand{\first}{{\rm first}}
\newcommand{\lucky}{{\rm lucky}}
\newcommand{\rmnext}{{\rm next}}

\newcommand{\cemetery}{\dagger}

\newcommand{\cluster}{\mathcal{B}}
\newcommand{\thinnedcluster}{\widetilde{\cluster}}
\newcommand{\thinnedBP}{\widetilde{\BP}}

\newcommand{\explore}{\mathcal{E}}
\newcommand{\thinnedExplore}{\widetilde{\explore}}

\newcommand{\epsilonCondition}{\epsilon_0} 
\newcommand{\deltaCondition}{\delta_0} 

\newtheorem{theorem}{Theorem}[section]
\newtheorem{prop}[theorem]{Proposition}
\newtheorem{lemma}[theorem]{Lemma}
\newtheorem{coro}[theorem]{Corollary}
\newtheorem{cond}[theorem]{Condition}

\theoremstyle{definition}
\newtheorem{defn}[theorem]{Definition}
\newtheorem{remark}[theorem]{Remark}

\newtheorem{example}[theorem]{Example}

\newcommand{\ch}[1]{\color{red}#1\color{black}}

\newcommand{\textandreference}[2]{\texorpdfstring{\hyperref[PartII:#2]{#1\refPartII{#2}}}{#1\refstarPartII{#2}}}

\newcommand{\labelPartII}[1]{\label{PartII:#1}}
\newcommand{\refPartII}[1]{\ref{PartII:#1}}
\newcommand{\refstarPartII}[1]{\ref*{PartII:#1}}
\newcommand{\eqrefPartII}[1]{\eqref{PartII:#1}}

\newcommand{\lbsect}[1]{\labelPartII{s:#1}}
\newcommand{\refsect}[1]{\textandreference{Section~}{s:#1}}

\newcommand{\lbsubsect}[1]{\labelPartII{ss:#1}}
\newcommand{\refsubsect}[1]{\textandreference{Section~}{ss:#1}}

\newcommand{\lbthm}[1]{\labelPartII{t:#1}}
\newcommand{\refthm}[1]{\textandreference{Theorem~}{t:#1}}

\newcommand{\lbprop}[1]{\labelPartII{p:#1}}
\newcommand{\refprop}[1]{\textandreference{Proposition~}{p:#1}}

\newcommand{\lblemma}[1]{\labelPartII{l:#1}}
\newcommand{\reflemma}[1]{\textandreference{Lemma~}{l:#1}}

\newcommand{\lbcoro}[1]{\labelPartII{c:#1}}
\newcommand{\refcoro}[1]{\textandreference{Corollary~}{c:#1}}

\newcommand{\lbcond}[1]{\labelPartII{cond:#1}}
\newcommand{\refcond}[1]{\textandreference{Condition~}{cond:#1}}

\newcommand{\lbdefn}[1]{\labelPartII{d:#1}}
\newcommand{\refdefn}[1]{\textandreference{Definition~}{d:#1}}

\newcommand{\lbexample}[1]{\labelPartII{ex:#1}}
\newcommand{\refexample}[1]{\textandreference{Example~}{ex:#1}}

\newcommand{\lbremark}[1]{\labelPartII{rem:#1}}
\newcommand{\refremark}[1]{\textandreference{Remark~}{rem:#1}}

\newcommand{\lbitem}[1]{\labelPartII{item:#1}}
\newcommand{\refitem}[1]{\refPartII{item:#1}}

\newcommand{\refother}[1]{[Part I, #1]}

\newcommand{\textandreferencePartI}[2]{#1\ref*{PartI:#2}} 

\newcommand{\refstarPartI}[1]{\ref*{PartI:#1}}
\newcommand{\eqrefPartI}[1]{\eqref{PartI:#1}}

\newcommand{\refsectPartI}[1]{\textandreferencePartI{Section~}{s:#1}}
\newcommand{\refsubsectPartI}[1]{\textandreferencePartI{Section~}{ss:#1}}
\newcommand{\refthmPartI}[1]{\textandreferencePartI{Theorem~}{t:#1}}
\newcommand{\refpropPartI}[1]{\textandreferencePartI{Proposition~}{p:#1}}
\newcommand{\reflemmaPartI}[1]{\textandreferencePartI{Lemma~}{l:#1}}
\newcommand{\refcoroPartI}[1]{\textandreferencePartI{Corollary~}{c:#1}}

\newcommand{\refdefnPartI}[1]{\textandreferencePartI{Definition~}{d:#1}}

\definecolor{MyDarkBlue}{rgb}{0,0.08,0.50}
\definecolor{BrickRed}{rgb}{0.65,0.08,0}
\hypersetup{
colorlinks=true,       
    linkcolor=MyDarkBlue,          
    citecolor=BrickRed,        
    filecolor=red,      
    urlcolor=cyan           
}

\newcommand{\blank}[1]{}

\numberwithin{equation}{section}

\renewcommand{\epsilon}{\varepsilon}

\newcommand{\e}{{\mathrm e}}
\newcommand{\sss}{\scriptscriptstyle}
\newcommand {\convd}{\stackrel{d}{\longrightarrow}}
\newcommand {\convp}{\stackrel{\sss {\mathbb P}}{\longrightarrow}}
\newcommand {\convas}{\stackrel{a.s.}{\longrightarrow}}
\newcommand{\nn}{\nonumber}
\newcommand{\eqn}[1]{\begin{equation} #1 \end{equation}}

\newcommand{\prob}{{\mathbb P}}
\newcommand{\expec}{{\mathbb E}}

\newcommand{\IP}{{\rm IP}}

\newcommand{\Op}{O_{\sss \prob}}
\newcommand{\sop}{o_{\sss \prob}}
\newcommand{\FY}{F_{\sss Y}}

\newcommand{\CovSumsGen}{\Xi}
\newcommand{\eps}{\epsilon}
\newcommand{\increm}{D}

\title{Long paths in first passage percolation\\
on the complete graph II. Global branching dynamics}
\author{
Maren Eckhoff\thanks{Department of Mathematical Sciences, University of Bath, Bath, BA2 7AY, United Kingdom. Email: {\tt eckhoff.maren@gmail.com}}
\and
Jesse Goodman\thanks{Department of Statistics, University of Auckland, Private Bag 92019, Auckland 1142, New Zealand. Email: {\tt jesse.goodman@auckland.ac.nz}}
\and
Remco van der Hofstad\thanks{Department of Mathematics and Computer Science,
Eindhoven University of Technology, P.O.\ Box 513,
5600~MB Eindhoven, The Netherlands. Email:
{\tt rhofstad@win.tue.nl, f.r.nardi@tue.nl}}
\and
Francesca R.\ Nardi\footnotemark[3]
}

\begin{document}

\maketitle

\begin{abstract}
We study the random geometry of first passage percolation on the complete graph equipped with independent and identically distributed edge weights, continuing the program initiated by Bhamidi and van der Hofstad~\cite{BhaHof12}.
We describe our results in terms of a sequence of parameters $(s_n)_{n\geq 1}$ that quantifies the extreme-value behavior of small weights, and that describes different universality classes for first passage percolation on the complete graph.
We consider both $n$-independent as well as $n$-dependent edge weights.
The simplest example consists of edge weights of the form $E^{s_n}$, where $E$ is an exponential random variable with mean 1.

In this paper, we focus on the case where $s_n\rightarrow \infty$ with $s_n=o(n^{1/3})$.
Under mild regularity conditions, we identify the scaling limit of the weight of the smallest-weight path between two uniform vertices, and we prove that the number of edges in this path obeys a central limit theorem with asymptotic mean $s_n\log{(n/s_n^3)}$ and variance $s_n^2\log{(n/s_n^3)}$.
This settles a conjecture of Bhamidi and van der Hofstad~\cite{BhaHof12}.
The proof relies on a decomposition of the smallest-weight tree into an initial part following invasion percolation dynamics, and a main part following branching process dynamics.
The initial part has been studied in \cite{EckGooHofNar14a}; the current article focuses on the global branching dynamics.
\end{abstract}

\section{Model and results}
\lbsect{IntroFPP}
In this paper, we continue the program of studying first passage percolation on the complete graph initiated in \cite{BhaHof12}.
We often need to refer to results presented in \cite{EckGooHofNar14a}, so we choose to cite specific results with the following expression, e.g., \refother{\reflemmaPartI{RluckyBornSoon}},
to say Lemma 2.13 in \cite{EckGooHofNar14a}.

We start by introducing first passage percolation (FPP).
Given a graph $\graphG=(V(\graphG),E(\graphG))$, let $(Y_e^{\sss (\graphG)})_{e\in E(\graphG)}$ denote a collection of positive edge weights.
Thinking of $Y_e^{\sss (\graphG)}$ as the cost of crossing an edge $e$, we can define a metric on $V(\graphG)$ by setting
	\begin{equation}\labelPartII{FPPdistance}
	d_{\graphG,Y^{(\graphG)}}(i,j)=\inf_{\pi \colon i\to j} \sum_{e\in\pi} Y_e^{\sss (\graphG)},
	\end{equation}
where the infimum is over all paths $\pi$ in $\graphG$ that join $i$ to $j$, and $Y^{\sss(\graphG)}$ represents the edge weights $(Y_e^{\sss (\graphG)})_{e\in E(\graphG)}$.
We will always assume that the infimum in \eqrefPartII{FPPdistance} is attained uniquely, by some (finite) path $\pi_{i,j}$.
We are interested in the situation where the edge weights $Y_e^{\sss (\graphG)}$ are \emph{random}, so that $d_{\graphG,Y^{(\graphG)}}$ is a random metric.
In particular, when the graph $\graphG$ is very large, with $\abs{V(\graphG)}=n$ say, we wish to understand for fixed $i,j \in V(\graphG)$ the scaling behavior of the following quantities:
\begin{enumerate}
\item
The \emph{distance} $W_n=d_{\graphG,Y^{(\graphG)}}(i,j)$ -- the total edge cost of the optimal path $\pi_{i,j}$;
\item
The \emph{hopcount} $H_n$ -- the number of edges in the optimal path $\pi_{i,j}$;
\item
The \emph{topological structure} -- the shape of the random neighborhood of a point.
\end{enumerate}

In this paper, we consider FPP on the complete graph, which acts as a mean-field model for FPP on finite graphs. We focus on problem (a) and (b) using results of problem (c) that are given in  \refother{\refsectPartI{IntroFPP}} and \refother{\refsectPartI{IPpartFPP}}.
In \cite{BhaHof12}, the question was raised what the {\it universality classes} are for this model.
We bring the discussion substantially further by describing a way to distinguish several universality classes and by identifying the limiting behavior of first passage percolation in one of these classes.
The cost regime introduced in \eqrefPartII{FPPdistance} uses the information from all edges along the path and is known as the {\it weak disorder} regime.
By contrast, in the {\it strong disorder} regime the cost of a path $\pi$ is given by $\max_{e \in \pi} Y_e^{\sss (\graphG)}$.
We establish a firm connection between the weak and strong disorder regimes in first passage percolation.
Interestingly, this connection also establishes a strong relation to invasion percolation (IP) on the Poisson-weighted infinite tree (PWIT), which is the scaling limit of IP on the complete graph.
This process also arises as the scaling limit of the minimal spanning tree on $K_n$.

Our main interest is in the case $\graphG=K_n$, the complete graph on $n$ vertices $V(K_n)=[n]:=\set{1,\ldots,n}$, equipped with independent and identically distributed (i.i.d.) edge weights $(Y_e^{\sss(K_n)})_{e \in E(K_n)}$.
We write $Y$ for a random variable with $Y\overset{d}{=}Y_e^{\sss(\graphG)}$, and
assume that the distribution function $\FY$ of $Y$ is continuous.
For definiteness, we study the optimal path $\pi_{1,2}$ between vertices $1$ and $2$.
In \cite{BhaHof12} and \cite{EckGooHofNar13} this setup was studied for the case that $Y_e^{\sss(K_n)}\equalsd E^{s}$ where $E$ is an exponential mean $1$ random variable, and $s>0$ constant and $s=s_n>0$ a null-sequence, respectively.
We start by stating our main theorem for this situation where $s=s_n$ tends to infinity.
First, we introduce some notation:

\paragraph{Notation.}
All limits in this paper are taken as $n$ tends to infinity unless stated otherwise.
A sequence of events $(\mathcal{A}_n)_n$ happens \emph{with high probability (\whp)} if $\P(\mathcal{A}_n) \to 1$.
For random variables $(X_n)_n, X$, we write $X_n \convd X$, $X_n \convp X$ and $X_n \convas X$ to denote convergence in distribution, in probability and almost surely, respectively.
For real-valued sequences $(a_n)_n$, $(b_n)_n$, we write $a_n=O(b_n)$ if the sequence $(a_n/b_n)_n$ is bounded; $a_n=o(b_n)$ if $a_n/b_n \to 0$; $a_n =\Theta(b_n)$ if the sequences $(a_n/b_n)_n$ and $(b_n/a_n)_n$ are both bounded; and $a_n \sim b_n$ if $a_n/b_n \to 1$.
Similarly, for sequences $(X_n)_n$, $(Y_n)_n$ of random variables, we write $X_n=\Op(Y_n)$ if the sequence $(X_n/Y_n)_n$ is tight; $X_n=\sop(Y_n)$ if $X_n/ Y_n \convp 0$; and $X_n =\Theta_{\P}(Y_n)$ if the sequences $(X_n/Y_n)_n$ and $(Y_n/X_n)_n$ are both tight.
Moreover, $E$ always denotes an exponentially distributed random variable with mean $1$.

\subsection{\texorpdfstring{First passage percolation with $n$-dependent edge weights}{First passage percolation with n-dependent edge weights}}
\lbsubsect{FPPn-dependent}

We start by investigating the case where $Y=E^{s_n}$ where $s_n\rightarrow \infty$:

\begin{theorem}[Weight and hopcount -- $n$-dependent edge weights]
\lbthm{WH-exp}
Let $Y_e^{\sss(K_n)}\equalsd E^{s_n}$, where $(s_n)_n$ is a positive sequence with $s_n \to \infty$, $s_n=o(n^{1/3})$.
Then
	\begin{equation}
	\labelPartII{weight-res_exp}
	n\Big(W_n-\frac{1}{n^{s_n} \Gamma(1+1/s_n)^{s_n}}\log{(n/s_n^3)}\Big)^{1/s_n} \convd M^{\sss(1)}\vee M^{\sss(2)},
	\end{equation}
and
	\eqn{
	\labelPartII{hopcount-CLT_exp}
	\frac{H_n-s_n\log{(n/s_n^3)}}{\sqrt{s_n^2\log{(n/s_n^3)}}}\convd Z,
	}
where $Z$ is standard normal and $M^{\sss(1)}, M^{\sss(2)}$ are i.i.d.\ random variables
for which $\prob(M^{\sss(j)}\leq x)$ is the survival probability of a Poisson Galton--Watson branching process with
mean $x$.
\end{theorem}

The convergence in \eqrefPartII{weight-res_exp} was proved in \refother{\refthmPartI{IPWeight_Exp}} without the subtraction of the term $\frac{1}{n^{s_n} \Gamma(1+1/s_n)^{s_n}}\log{(n/s_n^3)}$ in the argument, and under the stronger assumption that $s_n/\log\log{n}\rightarrow \infty$. The convergence \eqrefPartII{hopcount-CLT_exp} states the central limit theorem for the hopcount $H_n$ and verifies the heuristics for the strong disorder regime in \cite[Section 1.4]{BhaHof12}. See \refremark{consistent} for a discussion of the relation between these two results.
Here strong disorder refers to the fact that when $s_n\rightarrow \infty$ the values of the random weights $E_e^{s_n}$ depend strongly on the disorder $(E_e)_{e\in E(\graphG)}$, making small values increasingly more favorable and large values increasingly less favorable.
Mathematically, the elementary limit
	\begin{equation}\labelPartII{PowerToMax}
	\lim_{s\to\infty}(x_1^s+x_2^s)^{1/s}=x_1\vee x_2
	\end{equation}
expresses the convergence of the $\ell^s$ norm towards the $\ell^\infty$ norm and establishes a relationship between the weak disorder regime and the strong disorder regime of FPP.

We continue by discussing the result in \refthm{WH-exp} in more detail.
Any sequence $(u_n(x))_n$ for which $n\FY(u_n(x))\rightarrow x$ is such that, for i.i.d.\ random variables $(Y_i)_{i \in \N}$ with distribution function $\FY$,
	\eqn{
	\labelPartII{unx-res}
	\prob\Bigl( \min_{i\in [n]} Y_i\leq u_n(x) \Bigr) \rightarrow 1-\e^{-x}.
	}
As the distribution function $\FY$ is continuous, we will choose $u_n(x)=\FY^{-1}(x/n)$.
The value $u_n(1)$ is denoted by $u_n$.
In view of \eqrefPartII{unx-res}, the family $(u_n(x))_{x\in(0,\infty)}$ are the \emph{characteristic values} for $\min_{i\in [n]} Y_i$.
See \cite{EmbKluMik97} for a detailed discussion of extreme value theory.
(In the strong disorder regime, $u_n(x)$ varies heavily in $x$ such that the phrase characteristic values can be misleading.)
In the setting of \refthm{WH-exp}, $\FY(y)=1-\e^{-y^{1/s_n}} \approx y^{1/s_n}$ when $y=y_n$ tends to zero fast enough, so that $u_n(x)$ can be taken as $u_n(x)\approx(x/n)^{s_n}$ (here $\approx$ indicates approximation with uncontrolled error).
Then we see in \eqrefPartII{weight-res_exp} that $W_n-\frac{1}{\lambda_n}\log{(n/s_n^3)}\approx u_n(M^{\sss(1)}\vee M^{\sss(2)})$ for $\lambda_n=n^{s_n} \Gamma(1+1/s_n)^{s_n}$, which means that the weight of the smallest-weight path has a deterministic part $\frac{1}{\lambda_n}\log{(n/s_n^3)}$, while its random fluctuations are of the same order of magnitude as some of the typical values for the minimal edge weight adjacent to vertices $1$ and $2$.
For $j \in \set{1,2}$, one can think of $M^{\sss(j)}$ as the time needed to escape from vertex $j$.

\subsection{\texorpdfstring{First passage percolation with $n$-independent edge weights}{First passage percolation with n-independent edge weights}}
\lbsubsect{FPPn-independent}

Our results are applicable not only to the edge-weight distribution $Y_e^{\sss(K_n)}\overset{d}{=}E^{s_n}$,
but also to settings where the edge-weights have an $n$-indepedent distribution.
Suppose that the edge weights $Y_e^{\sss(K_n)}$ follow an $n$-independent distribution $\FY$ with no atoms and $u\mapsto u(\FY^{-1})'(u)/\FY^{-1}(u)$ is regularly varying with index $-\alpha$ as $u\decreasesto 0$, i.e.,
\begin{equation}
\labelPartII{SlowVarCondFY}
u \frac{d}{du} \log \FY^{-1}(u)=u^{-\alpha}L(1/u) \qquad \text{for all } u \in (0,1),
\end{equation}
where $t \mapsto L(t)$ is slowly varying as $t\to \infty$. \refexample{AllExamples} contains 4 classes of examples that satisfy this condition.
We have the following result:

\begin{theorem}[Weight and hopcount -- $n$-independent edge weights]
\lbthm{WH-example}
Suppose that the edge weights $(Y_e^{\sss (K_n)})_{e \in E(K_n)}$ follow an $n$-independent distribution $\FY$ that satisfies \eqrefPartII{SlowVarCondFY}.
Define
\begin{equation}\labelPartII{snSlowVarFormula}
	s_n=n^\alpha L(n)
\end{equation}
and assume that $s_n/\log\log n \to \infty$ and $s_n=o(n^{1/3})$.
Then there exist sequences $(\lambda_n)_n$ and $(\phi_n)_n$ such that $\phi_n/s_n \to 1$, $\lambda_n u_n \to \e^{-\gamma}$, where $\gamma$ is Euler's constant, and
\begin{align}
	\labelPartII{weight-res-SlowVar}
	n\FY\Big(W_n-\frac{1}{\lambda_n}\log{(n/s_n^3)}\Big)
	&\convd M^{\sss(1)}\vee M^{\sss(2)},
	\\
	\labelPartII{hopcount-CLT-SlowVar}
	\frac{H_n-\phi_n\log{(n/s_n^3)}}{\sqrt{s_n^2\log{(n/s_n^3)}}}
	&\convd Z,
	\end{align}
where $Z$ is standard normal, and $M^{\sss(1)}, M^{\sss(2)}$ are i.i.d.\ random variables for which $\prob(M^{\sss(j)}\leq x)$ is the survival probability of a Poisson Galton--Watson branching process with mean $x$.
\end{theorem}

The sequences $(\lambda_n)_n$ and $(\phi_n)_n$ will be identified in \eqrefPartII{lambdanDefinition}--\eqrefPartII{phinDefinitionShort} below, subject to slightly stronger assumptions.
In the setting of \refthm{WH-exp}, $\phi_n=s_n$ and $\lambda_n=n^{s_n}\Gamma(1+1/s_n)^{s_n}$.
Recalling that $u_n\sim n^{-s_n}$, we see that $\lambda_n u_n\sim \Gamma(1+1/s_n)^{s_n}\rightarrow \e^{-\gamma}$.
Thus, the conditions on $(\lambda_n)_n$ and $(\phi_n)_n$ in \refthm{WH-example} are strongly inspired by \refthm{WH-exp}.

Notice that every sequence $(s_n)_n$ of the form $s_n=n^{\alpha} L(n)$, for $\alpha \ge 0$ and $L$ slowly varying at infinity, can be obtained from a $n$-independent distribution by taking $\log \FY^{-1}(u)=\int u^{-1-\alpha}L(1/u)du$, i.e., the indefinite integral of the function $u\mapsto u^{-1-\alpha}L(1/u)$. In \refsubsect{univ-class} we will weaken the requirement $s_n /\log\log n\to \infty$ to the requirement $s_n\to\infty$ subject to an additional regularity assumption.

The optimal paths in Theorems~\refPartII{t:WH-exp} and \refPartII{t:WH-example} are \emph{long paths} because the asymptotic mean of the path length $H_n$ in \eqrefPartII{hopcount-CLT_exp} and \eqrefPartII{hopcount-CLT-SlowVar} is of larger order than $\log n$, that is the path length that arises in many random graph contexts.
See \refsect{DiscExt} for a comprehensive literature overview.
The following example collects some edge-weight distributions that are covered by Theorems~\refPartII{t:WH-exp} and \refPartII{t:WH-example}. See \refother{\refsubsectPartI{univ-class}}
for detailed discussion of this example.

\begin{example}[Examples of weight distributions]
\lbexample{AllExamples}\hfill
\begin{enumerate}
\item\lbitem{EsnExample}
Let $(s_n)_{n \in \N} \in (0,\infty)^{\N}$ with $s_n \to \infty$, $s_n=o(n^{1/3})$.
Take $Y_e^{\sss(K_n)} \equalsd E^{s_n}$, i.e., $\FY(y)=1-\e^{-y^{1/s_n}}$.

\item\lbitem{PowerOfLogExampleFY}
Let $\rho>0, \kappa \in (0,1)$.
Take $Y_e^{\sss(K_n)} \equalsd \exp(-(E/\rho)^{1/\kappa})$, i.e., $\FY(y)=\exp(-\rho(\log(1/y))^\kappa)$, and define $s_n=\frac{(\log n)^{1/\kappa-1}}{\kappa \rho^{1/\kappa}}$.


\item \lbitem{PowerOfnExampleFY}
Let $\rho>0, \alpha \in (0,1)$.
Take $Y_e^{\sss(K_n)} \equalsd \exp\left( - \rho \e^{\alpha E}/\alpha \right)$, i.e., $\FY(y)=\exp(-\frac{1}{\alpha} \log(\frac{\alpha}{\rho} \log(1/y)))$, and define $s_n=\rho n^{\alpha}$.

\item \lbitem{PowerOfnExampleg}
Let $\rho>0,\alpha \in (0,1)$.
Take $Y_e^{\sss(K_n)} \equalsd \exp(-\rho E^{-\alpha}/\alpha)$, i.e., $\FY(y)=1-\exp\bigl( -(\frac{\alpha}{\rho} \log(1/y))^{-1/\alpha} \bigr)$, and define $s_n=\frac{\rho}{n-1} (-\log(1-\frac{1}{n}))^{-(\alpha+1)}$.
Note that $s_n \sim \rho n^{\alpha}$.
\end{enumerate}
\end{example}

\refthm{WH-example} shows that $n$-independent distributions can be understood in the same framework as the $n$-dependent distribution in \refexample{AllExamples}~\refitem{EsnExample}.
We next explain this comparison and generalize our result further.

\subsection{The universal picture}
\lbsubsect{univ-class}
For fixed $n$, the edge weights $(Y_e^{\sss(K_n)})_{e\in E(K_n)}$ are independent for different $e$.
However, there is no requirement that they are independent over $n$, and in fact in \refsect{Coupling}, we will produce $Y_e^{\sss(K_n)}$ using a fixed source of randomness not depending on $n$.
Therefore, it will be useful to describe the randomness on the edge weights $((Y_e^{\sss(K_n)})_{e\in E(K_n)}\colon n \in \N)$ uniformly across the sequence.
It will be most useful to give this description in terms of exponential random variables.
\begin{description}
\item[Distribution function:]
Choose a distribution function $\FY(y)$ having no atoms, and draw the edge weights $Y_e^{\sss(K_n)}$ independently according to $\FY(y)$:
	\begin{equation}\labelPartII{EdgesByDistrFunct}
	\P(Y_e^{\sss(K_n)} \leq y)= \FY(y).
	\end{equation}

\item[Parametrization by Exponential variables:]
Fix independent exponential mean $1$ variables $(E_e^{\sss (K_n)})_{e\in E(K_n)}$, choose a strictly increasing function $g\colon(0,\infty)\to(0,\infty)$, and define
	\begin{equation}\labelPartII{EdgesByIncrFunct}
	Y_e^{\sss(K_n)}=g(E_e^{\sss (K_n)}).
	\end{equation}
\end{description}
The relations between these parametrizations are given by
\begin{equation}\labelPartII{FYandg}
	\FY(y)=1-\e^{-g^{-1}(y)}\quad \text{and}
	\quad
	g(x)=\FY^{-1} \left( 1-\e^{-x} \right)
	.
\end{equation}
We define
	\eqn{
	\labelPartII{fnFromParamEdges}
	f_n(x)=g(x/n)=\FY^{-1} \left( 1-\e^{-x/n} \right).
	}
Let $Y_1,\dotsc,Y_n$ be i.i.d.\ with $Y_i=g(E_i)$ as in \eqrefPartII{EdgesByIncrFunct}.
Since $g$ is increasing,
	\begin{equation}\labelPartII{minYifn}
	\min_{i\in[n]}Y_i=g\bigl( \min_{i\in[n]} E_i \bigr) \equalsd g(E/n)=f_n(E).
	\end{equation}
Because of this convenient relation between the edge weights $Y_e^{\sss(K_n)}$ and exponential random variables, we will express our hypotheses about the distribution of the edge weights in terms of conditions on the functions $f_n(x)$ as $n\to\infty$.

Consider first the case $Y_e^{\sss(K_n)} \equalsd E^{s_n}$ from \refthm{WH-exp} and \refexample{AllExamples}~\refitem{EsnExample}.
From \eqrefPartII{EdgesByIncrFunct}, we have $g(x)=g_n(x)=x^{s_n}$, so that \eqrefPartII{fnFromParamEdges} yields
	\begin{equation}\labelPartII{fnEsnCase}
	\text{for }\quad Y_e^{\sss(K_n)} \equalsd E^{s_n}, \qquad f_n(x)=(x/n)^{s_n}=f_n(1)x^{s_n}.
	\end{equation}
Thus, \eqrefPartII{minYifn}--\eqrefPartII{fnEsnCase} show that the parameter $s_n$ measures the relative \emph{sensitivity} of $\min_{i\in[n]}Y_i$ to fluctuations in the variable $E$.
In general, we will have $f_n(x)\approx f_n(1)x^{s_n}$ if $x$ is appropriately close to 1 and $s_n\approx f_n'(1)/f_n(1)$.
These observations motivate the following conditions on the functions $(f_n)_n$, which we will use to relate the distributions of the edge weights $Y_e^{\sss(K_n)}$, $n\in\N$, to a sequence $(s_n)_n$:

\begin{cond}[Scaling of $f_n$]
\lbcond{scalingfn}
For every $x\geq 0$,
	\eqn{\labelPartII{fnx1oversn}
	\frac{f_n(x^{1/s_n})}{f_n(1)} \to x.
	}
\end{cond}

\begin{cond}[Density bound for small weights]\lbcond{LowerBoundfn}
There exist $\epsilonCondition>0$, $\deltaCondition \in \ocinterval{0,1}$ and $n_0 \in \N$ such that
	\begin{equation}\labelPartII{BoundfnSmall}
	\epsilonCondition s_n \leq \frac{xf_n'(x)}{f_n(x)}\leq s_n/\epsilonCondition,
	\qquad \text{whenever } \quad 1-\deltaCondition\leq x\leq 1, \text{ and }n\geq n_0.
	\end{equation}
\end{cond}

\begin{cond}[Density bound for large weights]
\lbcond{boundfn}\hfill
\begin{enumerate}
\item \lbitem{boundfnR}
For all $R>1$, there exist $\epsilon>0$ and $n_0\in \N$ such that for every $1\leq x\leq R$, and $n \ge n_0$,
	\eqn{
	\labelPartII{fn-lb}
		x\frac{d}{dx} \log f_n(x)\geq \epsilon s_n.
	}
\item \lbitem{boundfnlog}
For all $C>1$, there exist $\epsilon>0$ and $n_0\in\N$ such that \eqrefPartII{fn-lb} holds for every $n\geq n_0$ and every $x\geq 1$ satisfying $f_n(x)\leq C f_n(1) \log n$.
\end{enumerate}
\end{cond}

Notice that \refcond{scalingfn} implies that $f_n(1) \sim u_n$ whenever $s_n=o(n)$.
Indeed, by \eqrefPartII{fnFromParamEdges} we can write $u_n=f_n(x_n^{1/s_n})$ for $x_n = (-n \log(1-1/n))^{s_n}$.
Since $s_n=o(n)$, we have $x_n=1-o(1)$ and the monotonicity of $f_n$ implies that $f_n(x_n^{1/s_n})/f_n(1)\to 1$.
We remark also that \eqrefPartII{unx-res} remains valid if $u_n(x)$ is replaced by $f_n(x)$. In  \refother{\refcoroPartI{examplesconditions}} it has been proved that the edge weights in \refexample{AllExamples} satisfy conditions Conditions~\refPartII{cond:scalingfn}--\refPartII{cond:boundfn}.

We are now in a position to state our main theorem:

\begin{theorem}[Weight and hopcount -- general edge weights]
\lbthm{WH-gen}
Assume that Conditions~\refPartII{cond:scalingfn}--\refPartII{cond:boundfn} hold for a positive sequence $(s_n)_n$ with $s_n\rightarrow \infty$ and $s_n=o(n^{1/3})$.
Then there exist sequences $(\lambda_n)_n$ and $(\phi_n)_n$ such that $\phi_n/s_n \to 1$, $\lambda_n f_n(1) \to \e^{-\gamma}$, where $\gamma$ is Euler's constant, and
\begin{align}
	\labelPartII{weight-res}
	f_n^{-1}\Big(W_n-\frac{1}{\lambda_n}\log{(n/s_n^3)}\Big)
	&\convd M^{\sss(1)}\vee M^{\sss(2)},
	\\
	\labelPartII{hopcount-CLT}
	\frac{H_n-\phi_n\log{(n/s_n^3)}}{\sqrt{s_n^2\log{(n/s_n^3)}}}
	&\convd Z,
	\end{align}
where $Z$ is standard normal, and $M^{\sss(1)}, M^{\sss(2)}$ are i.i.d.\ random variables
for which $\prob(M^{\sss(j)}\leq x)$ is the survival probability of a Poisson Galton--Watson branching process with
mean $x$.
The convergences in \eqrefPartII{weight-res}--\eqrefPartII{hopcount-CLT} hold jointly and the limiting random variables are independent.
\end{theorem}

The sequences $(\lambda_n)_n$ and $(\phi_n)_n$ are identified in \eqrefPartII{lambdanDefinition}--\eqrefPartII{phinDefinitionShort} below, subject to the additional \refcond{boundfnExtended}.
The proof of \refthm{WH-gen} is given in \refsubsect{ProofCompletion}.

\paragraph{Relation between \refthm{WH-gen} and Theorems~\refPartII{t:WH-exp} and \refPartII{t:WH-example}.} Theorems~\refPartII{t:WH-exp} and \refPartII{t:WH-example} follow from \refthm{WH-gen}:
in the case $Y_e^{\sss (K_n)}\equalsd E^{s_n}$ from \refthm{WH-exp}, \eqrefPartII{fnx1oversn}--\eqrefPartII{fn-lb} hold identically with $\epsilonCondition=\eps=1$ and we explicitly compute $\lambda_n=n^{s_n} \Gamma(1+1/s_n)^{s_n}$ and $\phi_n=s_n$ in \refexample{ex:Esn} below.
It has been proved in \refother{\refpropPartI{RegVarImplies}} that the distributions in \refthm{WH-example} satisfy the assumptions of \refthm{WH-gen}.
The convergence \eqrefPartII{weight-res-SlowVar} in \refthm{WH-example} is
equivalent to \eqrefPartII{weight-res} in \refthm{WH-gen} by the observation that, for any non-negative random variables $(T_n)_n$ and $\cM$,
	\begin{equation}\labelPartII{nFYequivfn-1}
	n\FY(T_n)\rightarrow \cM \qquad\iff\qquad f_n^{-1}(T_n)\rightarrow \cM,
	\end{equation}
where the convergence is in distribution, in probability or almost surely.

The following example describes a generalization of \refexample{AllExamples}~\refitem{EsnExample}:

\begin{example}
\lbexample{ExZtosn}
Let $(s_n)_n$ be a positive sequence with $s_n \to \infty$, $s_n=o(n^{1/3})$.
Let $U$ be a positive, continuous random variable with distribution function $G$ and
	\begin{equation}\labelPartII{asympExd}
	\lim_{u \downarrow 0} \frac{u (G^{-1})'(u)}{G^{-1}(u)} =1.
	\end{equation}
Take $Y_e^{\sss(K_n)} \equalsd U^{s_n}$, i.e., $\FY(y)=G(y^{1/s_n})$.
Examples for $G$ are the uniform distribution on an interval $(0,b)$, for any $b>0$, or the exponential distribution with any parameter.
\end{example}

It has been proved in \refother{\reflemmaPartI{ExZtosnCond}} that the edge weights in \refexample{ExZtosn} satisfy Conditions~\refPartII{cond:scalingfn}--\refPartII{cond:boundfn}.

\refcond{boundfn} can be strengthened to the following condition that will be equivalent for our purposes:

\begin{cond}[Extended density bound]
\lbcond{boundfnExtended}
There exist $\epsilonCondition>0$ and $n_0 \in \N$ such that
	\eqn{
	\labelPartII{fn-bound}
	\frac{xf_n'(x)}{f_n(x)}\geq \epsilonCondition s_n \qquad\text{for every }x\geq 1, n \ge n_0.
	}
\end{cond}

\begin{lemma}\lblemma{AssumeStrongCond}
It suffices to prove \refthm{WH-gen} assuming Conditions~\refPartII{cond:scalingfn}, \refPartII{cond:LowerBoundfn} and \refPartII{cond:boundfnExtended}.
\end{lemma}

\reflemma{AssumeStrongCond}, which is proved in \refsubsect{StrongAssumption}, reflects the fact that the upper tail of the edge weight distribution does not substantially influence the first passage percolation (FPP) problem.

Henceforth, except where otherwise noted, we will assume Conditions~\refPartII{cond:scalingfn}, \refPartII{cond:LowerBoundfn} and \refPartII{cond:boundfnExtended}.
We will reserve the notation $\epsilonCondition, \deltaCondition$ for some fixed choice of the constants in Conditions~\refPartII{cond:LowerBoundfn} and \refPartII{cond:boundfnExtended}, with $\epsilonCondition$ chosen small enough to satisfy both conditions.

\subsection{Discussion of our results}
\lbsect{DiscExt}

In this section we discuss our results and state open problems.

\subsubsection{\texorpdfstring{The universality class in terms of $s_n$}{The universality class in terms of sn}}
\lbsect{Universalityclass}

As we have seen in \refsubsect{univ-class}, there are many examples for which Conditions~\refPartII{cond:scalingfn}--\refPartII{cond:boundfn} hold.
In this paper, we have investigated the case where $s_n\to \infty$ with $s_n=o(n^{1/3})$.
We now discuss other choices of $s_n$.

\paragraph{The regime $s_n \to 0$.}
In view of \refexample{ExZtosn}, for $s_n \to 0$ the first passage percolation problem approximates the graph metric, where the approximation is stronger the faster $s_n$ tends to zero.
We distinguish three different scaling regimes:

(i) Firstly, $s_n\log{n}\to \gamma\in \cointerval{0,\infty}$:
The case that $Y\equalsd E^{-\gamma}$ for $\gamma \in (0,\infty)$ falls into this class with $s_n=\gamma /\log n$ (see \cite[Section 1.2]{EckGooHofNar13}) and was investigated in \cite{BhaHofHoo13}, where it is proved that $H_n$ is concentrated on at most two values.
It was observed in \cite{EckGooHofNar13} that when $s_n \log n$ in \refexample{ExZtosn} converges to $\gamma$ fast enough, then the methods of \cite{BhaHofHoo13} can be imitated and the concentration result for the hopcount continues to hold.

(ii) When $s_n\log{n}\to \infty$ but $s_n^2\log{n}\to 0$, the variance factor $s_n^2 \log(n/s_n^3)$ in the CLT in \eqrefPartII{hopcount-CLT} tends to zero.
Since $H_n$ is integer-valued, it follows that \eqrefPartII{hopcount-CLT} must fail in this case.
First order asymptotics are investigated in \cite{EckGooHofNar13}, and it is shown that $H_n/(s_n\log{n})\convp 1$, $W_n/(u_n s_n\log{n})\convp \e$.
It is tempting to conjecture that there exists an integer $k=k_n\approx s_n\log{n}$ such that $H_n\in\set{k_n,k_n+1}$ {\whpdot}

(iii) When $s_n \to 0, s_n^2\log{n}\to \infty$, we conjecture that the CLT for the hopcount in \refthm{WH-gen} remains true, and that $u_n^{-1}W_n-\frac{1}{\lambda_n}\log n$ converges to a Gumbel distribution for a suitable sequence $\lambda_n$.
Unlike in the fixed $s$ case, we expect no martingale limit terms to appear in the limiting distribution.

\paragraph{The fixed $s$ regime.}
The fixed $s$ regime was investigated in \cite{BhaHof12} in the case where $Y\equalsd E^s$.
We conjecture that for other sequences of random variables for which \refcond{scalingfn} is valid for some $s\in (0,\infty)$ the CLT for the hopcount remains valid, while there exist $V$ and $\lambda(s)$ depending on the distribution $Y$ such that $u_n^{-1}W_n-\frac{1}{\lambda(s)}\log{(n/s^3)}\convd V$.
In general, $V$ will not be equal to $(M^{\sss (1)}\vee M^{\sss (2)})^s$, see for example \cite{BhaHof12}.
Instead, it is described by a sum of different terms involving Gumbel distributions and the martingale limit of a certain continuous-time branching process depending on the distribution.
Our proof is inspired by the methods developed in \cite{BhaHof12}.
The CLT for $H_n$ in the fixed $s$ regime can be recovered from our proofs; in fact, the reasoning in that case simplifies considerably compared to our more general setup.

The results in \cite{BhaHof12} match up nicely with ours.
Indeed, in \cite{BhaHof12}, it was shown that
\begin{equation}	\labelPartII{res-Bham-Hof}
	f_n^{-1}
	\big(W_n-\log{n}/\lambda(s)\big) \convd \lambda(s)^{-1/s}
	\big(\Lambda_{1,2}-\log L_s^{\sss(1)}-\log L_s^{\sss(2)}-\log(1/s)\big)^{1/s},
\end{equation}
where $\lambda(s)=\Gamma(1+1/s)^s$, $\Lambda_{1,2}$ is a Gumbel variable so that $\prob(\Lambda_{1,2}\leq x)=\e^{-\e^{-x}}$ and $L_s^{\sss(1)}, L_s^{\sss(2)}$ are two independent copies of the random variable $L_s$ with $\expec(L_s)=1$ solving the distributional equation
\begin{equation}\labelPartII{eq:WfixedPoint}
	L_s\stackrel{d}{=}\sum_{i\geq 1} \e^{-\lambda(s) (E_1+\cdots +E_i)^{s}} L_{s,i},
\end{equation}
where $(L_{s,i})_{i\geq 1}$ are i.i.d.\ copies of $L_s$ and $(E_i)_{i\geq 1}$ are i.i.d.\ exponentials with mean $1$.
We claim that the right-hand side of \eqrefPartII{res-Bham-Hof} converges to $M^{\sss(1)}\vee M^{\sss(2)}$ as $s\rightarrow \infty$, where $M^{\sss(1)}, M^{\sss(2)}$ are as in \refthm{WH-gen}.
This is equivalent to the statement that $(-\log L_s^{\sss(j)})^{1/s} \convd M^{\sss(j)}$ as $s\to \infty$.
Assume that $(-\log L_s^{\sss(j)})^{1/s}$ converges in distribution to a random variable $\cM$.
Then
\begin{align}
\lim_{s \to \infty}\Big(-\log\Big(\sum_{i\geq 1} \e^{-\lambda(s) (E_1+\cdots +E_i)^{s}} L_{s,i}\Big)\Big)^{1/s}
	&= \min_{i\geq 1} \lim_{s\rightarrow \infty} \Big(\lambda(s) (E_1+\cdots +E_i)^{s} -\log L_{s,i}\Big)^{1/s}\nn\\
	&=\min_{i\geq 1}\Big((E_1+\cdots +E_i)\vee \big(\lim_{s\rightarrow \infty} (-\log L_{s,i})^{1/s}\big)\Big),
	\end{align}
	and using
\eqrefPartII{eq:WfixedPoint} we deduce that $\cM$ is the solution of the equation
	\eqn{
	\labelPartII{stoch-rec-M}
	\cM\equalsd \min_{i\geq 1} (E_1+\cdots +E_i)\vee \cM_i,
	}
where $(\cM_i)_{i\geq 1}$ are i.i.d.\ copies of $\cM$
independent of $(E_i)_{i\geq 1}$.
The unique solution to \eqrefPartII{stoch-rec-M} is the random variable with $\prob(\cM\leq x)$ being the survival probability of a Poisson Galton--Watson branching process with mean $x$.

\paragraph{The regime $s_n/n^{1/3} \to \infty$.}
Several of our methods do not extend to the case where $s_n/n^{1/3} \to \infty$; indeed, we conjecture that the CLT in \refthm{WH-gen} ceases to hold in this regime.
First passage percolation (FPP) on the complete graph is closely approximated by invasion percolation (IP) on the Poisson-weighted infinite tree (PWIT), studied in \cite{AddGriKan12}, whenever $s_n\to\infty$, see \cite{EckGooHofNar14a}.

\subsubsection{First passage percolation on random graphs}
\lbsect{FPPonRG}

FPP on random graphs has attracted considerable attention in the past years, and our research was
strongly inspired by its studies.
In \cite{BHH12bPre}, the authors show that for the configuration model
with finite-variance degrees (and related graphs) and edge weights with a continuous distribution not depending on $n$, there exists only a \emph{single} universality class.
Indeed, if we define $W_n$ and $H_n$ to be the weight of and the number of edges in the smallest-weight path between two uniform vertices in the graph, then there exist positive, finite constants $\alpha, \beta, \lambda$ and sequences $(\alpha_n)_n,(\lambda_n)_n$, with $\alpha_n\rightarrow \alpha$, $\lambda_n\rightarrow \lambda$, such that $W_n-(1/\lambda_n)\log{n}$ converges in distribution, while $H_n$ satisfies a CLT with asymptotic mean  $\alpha_n \log{n}$ and asymptotic variance $\beta \log{n}$.

Related results for exponential edge weights appear for the Erd\H{o}s-R\'enyi random graph in \cite{BhaHofHoo11} and to certain inhomogeneous random graphs in \cite{KolKom13pre}.
The diameter of the weighted graph is studied in \cite{AmiDraLel11}, and relations to competition on $r$-regular graphs are examined in \cite{AntDekMosPer11pre}.
Finally, the smallest-weight paths with most edges from a single source or between any pair in the graph are investigated in \cite{AmiPer12pre}.

We conjecture that our results are closely related to FPP on random graphs with \emph{infinite-variance degrees}.
Such graphs, sometimes called \emph{scale-free} random graphs, have been suggested in the networking community as appropriate models for various real-world networks.
See \cite{BarAlb99, Newm03a} for extended surveys of real-world networks, and \cite{Durr07, Hof14pre, Newm10a} for more details on random graph models of such real-world networks.
FPP on infinite-variance random graphs with exponential weights was first studied in \cite{BhaHofHoo10a, BhaHofHoo10b}, of which the case of finite-mean degrees studied in \cite{BhaHofHoo10b} is most relevant for our discussion here.
There, it was shown that a linear transformation of $W_n$ converges in distribution, while $H_n$ satisfies a CLT with asymptotic mean and variance $\alpha \log{n}$, where $\alpha$ is a simple function of the power-law exponent of the degree distribution of the configuration model.
Since the configuration model with infinite-variance degrees \whp\ contains a complete graph of size a positive power of $n$, it can be expected that the universality classes on these random graphs are closely related to those on the complete graph $K_n$.
In particular, the strong universality result for finite-variance random graphs is false, which can be easily seen by observing that for the weight distribution $1+E$, where $E$ is an exponential random variable, the hopcount $H_n$ is of order $\log\log{n}$ (as for the graph distance \cite{HofHooZna07}), rather than $\log{n}$ as it is for exponential weights. See \cite{BarHofKom15} for two examples proving that strong universality indeed fails in the infinite-variance setting.

\subsubsection{Extremal functionals for FPP on the complete graph}
\lbsect{ExtremalFPPonCG}
Many more fine results are known for FPP on the complete graph with exponential edge weights.
In \cite{JAN99}, the weak limits of the rescaled path weight and flooding are determined, where the flooding is the maximal smallest weight between a source and all vertices in the graph.
In \cite{BhaHof14} the same is performed for the diameter of the graph.
It would be of interest to investigate the weak limits of the flooding and diameter in our setting.

\section{Overview of the proof}
\lbsect{OverPf}
In this section, we provide an overview of the proof of our main results.
For a short guide to notation, see page~\pageref{PartII:notation}.

\subsection{Exploration process and FPP on the Poisson-weighted infinite tree}\lbsubsect{PWITFPPIP}

To understand smallest-weight paths in the complete graph, we study the first passage \emph{exploration process}.
Recall from \eqrefPartII{FPPdistance} that $d_{K_n,Y^{\sss(K_n)}}(i,j)$ denotes the total cost of the optimal path $\pi_{i,j}$ between vertices $i$ and $j$.
For a vertex $j\in V(K_n)$, let the \emph{smallest-weight tree} $\SWT_t^{\sss(j)}$ be the connected subgraph of $K_n$ defined by
	\begin{align}
	\begin{split}
	\labelPartII{OneSourceSWT}
	V(\SWT_t^{\sss(j)})
	&=
	\set{i\in V(K_n)\colon d_{K_n,Y^{\sss(K_n)}}(i,j)\leq t} \! ,
	\\
	E(\SWT_t^{\sss(j)})
	&=
	\set{e\in E(K_n)\colon e\in\pi_{j,i}\text{ for some }i\in V(\SWT_t^{\sss(j)})}.
	\end{split}
	\end{align}
Note that $\SWT_t^{\sss(j)}$ is indeed a tree: if two optimal paths $\pi_{j,k},\pi_{j,k'}$ pass through a common vertex $i$, both paths must contain $\pi_{j,i}$ since the minimizers of \eqrefPartII{FPPdistance} are unique.

To visualize the process $(\SWT_t^{\sss(j)})_{t\geq 0}$, think of the edge weight $Y_e^{\sss (K_n)}$ as the time required for fluid to flow across the edge $e$.
Place a source of fluid at $j$ and allow it to spread through the graph.
Then $V(\SWT_t^{\sss(j)})$ is precisely the set of vertices that have been wetted by time $t$, while $E(\SWT_t^{\sss(j)})$ is the set of edges along which, at any time up to $t$, fluid has flowed from a wet vertex to a previously dry vertex.
Equivalently, an edge is added to $\SWT_t^{\sss(j)}$ whenever it becomes completely wet, with the additional rule that an edge is not added if it would create a cycle.

Because fluid begins to flow across an edge only after one of its endpoints has been wetted, the \emph{age} of a vertex -- the length of time that a vertex has been wet -- determines how far fluid has traveled along the adjoining edges.
Given $\SWT_t^{\sss(j)}$, the future of the exploration process will therefore be influenced by the current ages of vertices in $\SWT_t^{\sss(j)}$, and the nature of this effect depends on the probability law of the edge weights $(Y_e^{\sss (K_n)})_e$.
In the sequel, for a subgraph $\graphG=(V(\graphG),E(\graphG))$ of $K_n$, we write $\graphG$ instead of $V(\graphG)$ for the vertex set when there is no risk of ambiguity.

To study the smallest-weight tree from a vertex, say vertex $1$, let us consider the time until the first vertex is added.
By construction, $\min_{i\in [n]\setminus \set{1}} Y_{\set{1,i}}^{\sss(K_n)} \equalsd f_n(\frac{n}{n-1}E)$ (cf.\ \eqrefPartII{minYifn}), where $E$ is an exponential random variable of mean $1$.
We next extend this to describe the distribution of the order statistics of the weights of edges from vertex 1 to \emph{all} other vertices.

Denote by $Y_{(k)}^{\sss(K_n)}$ the $k^\th$ smallest weight from $(Y_{\set{1,i}}^{\sss(K_n)})_{i\in [n]\setminus \set{1}}$.
Then $(Y_{(k)}^{\sss(K_n)})_{k\in [n-1]}\equalsd (f_n(S_{k,n}))_{k\in [n-1]}$, where $S_{k,n}=\sum_{j=1}^k \frac{n}{n-j} E_j$ and $(E_j)_{j\in [n-1]}$ are i.i.d.\ exponential random variables with mean $1$.
The fact that the distribution of $S_{k,n}$ depends on $n$ is awkward, and can be changed by using a \emph{thinned} Poisson point process.
Let $X_1<X_2<\dotsb$ be the points of a Poisson point process with intensity 1, so that $X_k\equalsd\sum_{j=1}^kE_j=\lim_{n\to \infty} S_{k,n}$.
To each $k\in\N$, we associate a \emph{mark} $M_k$ which is chosen uniformly at random from $[n]$, different marks being independent.
We \emph{thin} a point $X_k$ when $M_k=1$ (since 1 is the initial vertex) or when $M_k=M_{k'}$ for some $k'<k$.
Then
\begin{equation}\labelPartII{SingleVertexCoupling}
(Y_{(k)}^{\sss(K_n)})_{k\in [n-1]}\equalsd (f_n(X_k))_{k\in\N, \, X_k\text{~unthinned}}.
\end{equation}
In the next step, we extend this result to the smallest-weight tree $\SWT^{\sss(1)}$ using a relation to FPP on the Poisson-weighted infinite tree.
Before giving the definitions, we recall the Ulam--Harris notation for describing trees.

Define the tree $\tree^{\sss(1)}$ as follows.
The vertices of $\tree^{\sss(1)}$ are given by finite sequences of natural numbers headed by the symbol $\emptyset_1$, which we write as $\emptyset_1 j_1 j_2\dotsb j_k$.
The sequence $\emptyset_1$ denotes the root vertex of $\tree^{\sss(1)}$.
We concatenate sequences $v=\emptyset_1 i_1\dotsb i_k$ and $w=\emptyset_1 j_1\dotsb j_m$ to form the sequence $vw=\emptyset_1 i_1\dotsb i_k j_1\dotsb j_m$ of length $\abs{vw}=\abs{v}+\abs{w}=k+m$.
Identifying a natural number $j$ with the corresponding sequence of length 1, the $j^\th$ child of a vertex $v$ is $vj$, and we say that $v$ is the parent of $vj$.
Write $\parent{v}$ for the (unique) parent of $v\neq\emptyset_1$, and $\ancestor{k}{v}$ for the ancestor $k$ generations before, $k\leq \abs{v}$.

We can place an edge (which we could consider to be directed) between every $v\neq\emptyset_1$ and its parent; this turns $\tree^{\sss(1)}$ into a tree with root $\emptyset_1$.
With a slight abuse of notation, we will use $\tree^{\sss(1)}$ to mean both the set of vertices and the associated graph, with the edges given implicitly according to the above discussion, and we will extend this convention to any subset $\tau\subset\tree^{\sss(1)}$.
We also write $\boundary \tau=\set{v\notin\tau\colon \parent{v}\in\tau}$ for the set of children one generation away from $\tau$.

The Poisson-weighted infinite tree is an infinite edge-weighted tree in which every vertex has infinitely many (ordered) children.
To describe it formally, we associate weights to the edges of $\tree^{\sss(1)}$.
By construction, we can index these edge weights by non-root vertices, writing the weights as $X=(X_v)_{v\neq\emptyset_1}$, where the weight $X_v$ is associated to the edge between $v$ and its parent $p(v)$.
We make the convention that $X_{v0}=0$.

\begin{defn}[Poisson-weighted infinite tree]\lbdefn{PWITdef}
The \emph{Poisson-weighted infinite tree} (PWIT) is the random tree $(\tree^{\sss(1)},X)$ for which $X_{vk}-X_{v(k-1)}$ is exponentially distributed with mean 1, independently for each
$v\in \tree^{\sss(1)}$ and each $k\in\N$.
Equivalently, the weights $(X_{v1},X_{v2},\dotsc)$ are the (ordered) points of a Poisson point process of intensity 1 on $(0,\infty)$, independently for each $v$.
\end{defn}

Motivated by \eqrefPartII{SingleVertexCoupling}, we study FPP on $\tree^{\sss(1)}$ with edge weights $(f_n(X_v))_v$:

\begin{defn}[First passage percolation on the Poisson-weighted infinite tree]\lbdefn{FPPonPWITdef}
For FPP on $\tree^{\sss(1)}$ with edge weights $(f_n(X_v))_v$, let
the FPP edge weight between $v\in\tree^{\sss(1)}\setminus\set{\emptyset_1}$ and $\parent{v}$ be $f_n(X_v)$. The FPP distance from $\emptyset_1$ to $v\in\tree^{\sss(1)}$ is
\begin{equation}\labelPartII{TvDefinition}
T_v = \sum_{k=0}^{\abs{v}-1} f_n(X_{\ancestor{k}{v}})
\end{equation}
and the FPP exploration process $\BP^{\sss(1)}=(\BP^{\sss(1)}_t)_{t\geq 0}$ on $\tree^{\sss(1)}$ is defined by $\BP^{\sss(1)}_t=\set{v\in\tree^{\sss(1)}\colon T_v\leq t}$.
\end{defn}

Note that the FPP edge weights $(f_n(X_{vk}))_{k\in\N}$ are themselves the points of a Poisson point process on $(0,\infty)$, independently for each $v\in\tree^{\sss(1)}$.
The intensity measure of this Poisson point process, which we denote by $\mu_n$, is the image of Lebesgue measure on $(0,\infty)$ under $f_n$.
Since $f_n$ is strictly increasing by assumption, $\mu_n$ has no atoms and we may abbreviate $\mu_n(\ocinterval{a,b})$ as $\mu_n(a,b)$ for simplicity.
Thus $\mu_n$ is characterized by
\begin{equation}\labelPartII{munCharacterization}
\mu_n(a,b) = f_n^{-1}(b) - f_n^{-1}(a),
\qquad
\int_0^\infty h(y) d\mu_n(y) = \int_0^\infty h(f_n(x)) dx,
\end{equation}
for any measurable function $h\colon \cointerval{0,\infty}\to\cointerval{0,\infty}$.

Clearly, and as suggested by the notation, the FPP exploration process $\BP$ is a continuous-time branching process:

\begin{prop}\lbprop{FPPisCTBP}
The process $\BP^{\sss(1)}$ is a continuous-time branching process (CTBP), started from a single individual $\emptyset_1$, where the ages at childbearing of an individual form a Poisson point process with intensity $\mu_n$, independently for each individual.
The time $T_v$ is the birth time $T_v=\inf\set{t\geq 0\colon v\in\BP^{\sss(1)}_t}$ of the individual $v\in\tree^{\sss(1)}$.
\end{prop}

Similar to the analysis of the weights of the edges containing vertex $1$, we now introduce a thinning procedure.
Define $M_{\emptyset_1}=1$ and to each $v\in \tree^{\sss(1)}\setminus\set{\emptyset_1}$ associate a mark $M_v$ chosen independently and uniformly from $[n]$.

\begin{defn}\lbdefn{ThinningBP}
The root $\emptyset_1 \in \tree^{\sss(1)}$ is not thinned, i.e.\ \emph{unthinned}.
The vertex $v\in\tree^{\sss(1)}\setminus\set{\emptyset_1}$ is \emph{thinned} if it has an ancestor $v_0=\ancestor{k}{v}$ (possibly $v$ itself) such that $M_{v_0}=M_w$ for some unthinned vertex $w$ with $T_w<T_{v_0}$.
\end{defn}
This definition also appears as \refother{\refdefnPartI{ThinningBP}}.
As explained there, this definition is not circular since whether or not a vertex $v$ is thinned can be assessed recursively in terms of earlier-born vertices.
Write $\thinnedBP_t^{\sss(1)}$ for the subgraph of $\BP_t^{\sss(1)}$ consisting of unthinned vertices.

\begin{defn}\lbdefn{InducedGraph}
Given a subset $\tau\subset\tree^{\sss(1)}$ and marks $M=(M_v \colon v \in \tau)$ with $M_v\in[n]$, define $\pi_M(\tau)$ to be the subgraph of $K_n$ induced by the mapping $\tau\to [n]$, $v\mapsto M_v$.
That is, $\pi_M(\tau)$ has vertex set $\set{M_v\colon v\in\tau}$, with an edge between $M_v$ and $M_{\parent{v}}$ whenever $v,\parent{v}\in\tau$.
\end{defn}
Note that if the marks $(M_v)_{v\in\tau}$ are distinct then $\pi_M(\tau)$ and $\tau$ are isomorphic graphs.

The following theorem, taken from \refother{\refthmPartI{Coupling1Source}} and proved in \refother{\refsubsectPartI{ProofCoupling1Source}}, establishes a close connection between FPP on $K_n$ and FPP on the PWIT with edge weights $(f_n(X_v))_{v\in\tau}$:

\begin{theorem}[Coupling to FPP on PWIT]\lbthm{Coupling1Source}
The law of $(\SWT_t^{\sss(1)})_{t\ge 0}$ is the same as the law of $\Bigl( \pi_M\bigl( \thinnedBP_t^{\sss(1)} \bigr) \Bigr)_{t\ge 0}$.
\end{theorem}

\refthm{Coupling1Source} is based on an explicit coupling between the edge weights $(Y_e^{\sss(K_n)})_e$ on $K_n$ and $(X_v)_v$ on $\tree^{\sss(1)}$.
We will describe a related coupling in \refsubsect{CouplingBP}.
A general form of those couplings are given in \refsect{Coupling}.

\subsection{CTBPs and one-vertex characteristics}\lbsubsect{FPPVertexCharact}
In this section, we analyze the CTBP $\BP^{\sss(1)}$ introduced in \refsubsect{PWITFPPIP}.
Notice that $(\BP_t^{\sss(1)})_{t\ge 0}$ depends on $n$ through its offspring distribution.
We have to understand the coupled double asymptotics of $n$ and $t$ tending to infinity simultaneously.

Recall that we write $\abs{v}$ for the generation of $v$ (i.e., its graph distance from the root in the genealogical tree).
To count particles in $\BP_t^{\sss(1)}$, we use a non-random characteristic $\chi\colon \cointerval{0,\infty} \to \cointerval{0,\infty}$.
Following \cite{BhaHof12}, define the \emph{generation-weighted vertex characteristic} by
	\begin{align}\labelPartII{1VertexCharDef}
	z_t^\chi(a)=z_t^{\chi, \BP^{\sss(1)}}(a)
	=
	\sum_{v\in\BP_t^{\sss(1)}} a^{|v|} \chi(t-T_v) \qquad \text{for all }a,t\ge 0.
	\end{align}
We make the convention that $\chi(t)=z_t^\chi(a)=0$ for $t<0$.
For characteristics $\chi$, $\eta$ and for $a,b,t,u\ge 0$, write
	\begin{align}\labelPartII{NonRescaled1VertexChar}
	m_t^\chi(a)=\E(z_t^\chi(a)) \qquad \text{and} \qquad M_{t,u}^{\chi,\eta}(a,b)=\E(z_t^\chi(a) z_u^\eta(b)).
	\end{align}
Let $\hat{\mu}_n(\lambda)=\int\e^{-\lambda y}d\mu_n(y)$ denote the Laplace transform of $\mu_n$.
For $a>0$, define $\lambda_n(a)>0$ by
	\begin{equation}\labelPartII{lambdaaDefn}
	a\hat{\mu}_n(\lambda_n(a))=1
	\end{equation}
whenever \eqrefPartII{lambdaaDefn} has a unique solution.
The parameters $\lambda_n$ and $\phi_n$ in \refthm{WH-gen} are given by
\begin{align}
\lambda_n&=\lambda_n(1),\labelPartII{lambdanDefinition}\\
\phi_n&= \lambda_n'(1)/\lambda_n(1)
.
\labelPartII{phinDefinitionShort}
\end{align}
The asymptotics of $\lambda_n$ and $\phi_n$ stated in \refthm{WH-gen} is the content of the following lemma:
\begin{lemma}[Asymptotics of BP-parameters]\lblemma{lambdanAsymp}
$\phi_n/s_n \to 1$ and $\lambda_n f_n(1) \to \e^{-\gamma}$, where $\gamma$ is Euler's constant.
\end{lemma}
\reflemma{lambdanAsymp} is proved in \refsubsect{ConvRWPf}.

Typically, $z_t^\chi(a)$ grows exponentially in $t$ at rate $\lambda_n(a)$.
Therefore, we write
	\begin{align}
\bar{z}_t^\chi(a)&= \e^{-\lambda_n(a) t} z_t^\chi(a),
\notag \\
\bar{m}_t^\chi(a)&= \E(\bar{z}_t^\chi(a)) = \e^{-\lambda_n(a) t} m_t^\chi(a),
	\labelPartII{barredOneVertexChars}
	\\
\bar{M}_{t,u}^{\chi,\eta}(a,b)&=
	\E(\bar{z}_t^\chi(a) \bar{z}_u^\eta(b)) = \e^{-\lambda_n(a)t}\e^{-\lambda_n(b)u} M_{t,u}^{\chi,\eta}(a,b).
	\notag
	\end{align}
In the following theorem, we investigate the asymptotics of such generation-weighted one-vertex characteristics:

\begin{theorem}[Asymptotics one-vertex characteristics]\lbthm{OneCharConv}
Given $\epsilon>0$ and a compact subset $A\subset(0,2)$, there is a constant $K<\infty$ such that for $n$ sufficiently large, uniformly for $a,b\in A$ and for $\chi$ and $\eta$ bounded, non-negative, non-decreasing functions, $\lambda_n(1)[t\wedge u] \geq K$,
\begin{align}
&\abs{s_n^{-1} \bar{m}_t^\chi(a^{1/s_n}) - \textstyle{\int_0^\infty} \e^{-z}\chi\bigl(z/\lambda_n(a^{1/s_n})\bigr) dz}\leq \epsilon\norm{\chi}_\infty,
\labelPartII{OneCharConvFormula}\\
&\abs{s_n^{-3} \bar{M}_{t,u}^{\chi,\eta}(a^{1/s_n},b^{1/s_n}) - \frac{\textstyle{\int_0^\infty} \e^{-z} \chi\bigl( z/\lambda_n(a^{1/s_n}) \bigr)dz \, \textstyle{\int_0^\infty} \e^{-w}\eta\bigl( w/\lambda_n(b^{1/s_n}) \bigr)dw}{\log(1/a+1/b)}} \leq \epsilon \norm{\chi}_\infty\norm{\eta}_\infty.\labelPartII{OneCharCovConv}
\end{align}

Moreover, there is a constant $K'<\infty$ independent of $\epsilon$ such that $\bar{m}_t^\chi(a^{1/s_n}) \leq K' \norm{\chi}_\infty s_n$ and $\bar{M}_{t,u}^{\chi,\eta}(a^{1/s_n},b^{1/s_n}) \leq K' \norm{\chi}_\infty\norm{\eta}_\infty s_n^3$ for all $n$ sufficiently large, uniformly over $u,t\geq 0$ and $a,b\in A$.
\end{theorem}
\begin{coro}[Asymptotics of means and variance of population size]\lbcoro{BPSizeAsymptotics}
The population size $\abs{\BP^{\sss(1)}_t}$ satisfies $\E(\abs{\BP^{\sss(1)}_t})\sim s_n \e^{\lambda_n(1)t}$ and $\Var(\abs{\BP^{\sss(1)}_t})\sim s_n^3 \e^{2\lambda_n(1)t}/\log 2$ in the limit as $\lambda_n(1)t\to\infty$, $n\to\infty$.
\end{coro}

\refthm{OneCharConv} is proved in \refsubsect{LimitingBounds}.
Generally, we will be interested in characteristics $\chi=\chi_n$ for which $\chi_n\left(\lambda_n(1)^{-1}\cdot\right)$ converges as $n\to\infty$, so that the integral in \eqrefPartII{OneCharConvFormula} acts as a limiting value.
In particular, \refcoro{BPSizeAsymptotics} is the special case $\chi=\indicatorofset{\cointerval{0,\infty}}$, $a=1$.

Since $s_n\to\infty$, \refthm{OneCharConv} and \refcoro{BPSizeAsymptotics} show that the variance of $\bar{z}^\chi_t(a^{1/s_n})$ is larger compared to the square of the mean, by a factor of order $s_n$.
This suggests that $\BP^{\sss(1)}_t$ is typically of order 1 when $\lambda_n(1)t$ is of order 1 (i.e., when $t$ is of order $f_n(1)$, see \reflemma{lambdanAsymp}) but has probability of order $1/s_n$ of being of size of order $s_n^2$.
See also \refother{\refpropPartI{LuckyProb}} which confirms this behavior.

\subsection{FPP exploration process from two sources}\lbsubsect{FPP2Sources}

So far we have studied the smallest-weight tree $\SWT^{\sss(1)}$ from one vertex and its coupling to a suitably thinned version of a CTBP starting from one root.
Since we are looking for the optimal path between the vertices $1$ and $2$, we now place sources of fluid on both vertices and wait for the two smallest-weight trees to merge.
Thus, we need to describe exploration processes from two sources.
The analysis of processes from one source plays a key role in the study of processes from two sources and is crucial to understand the much more involved version from two sources.
However, the coupling presented in \refsubsect{PWITFPPIP}, i.e., \refthm{Coupling1Source}, will not be used in the proofs of our main theorems.

As we will see later, because the edge-weight distributions have a heavy tail at zero, the first passage distance between the two vertices comes primarily from a single edge of large weight.
For this reason, we will not grow the two clusters at the same speed.
When one of them becomes large enough (what this means precisely will be explained later), it has to wait for the other one to catch up.
Similarly, we will consider two CTBPs evolving at different speeds.
To formalize this, let $\twoPWITs$ be the disjoint union of two independent copies $(\tree^{\sss(j)},X^{\sss(j)})$, $j \in \set{1,2}$, of the PWIT.
We shall assume that the copies $\tree^{\sss(j)}$ are vertex-disjoint, with roots $\emptyset_j$, so that we can unambiguously write $X_v$ instead of $X^{\sss(j)}_v$ for $v\in \tree^{\sss(j)}$, $v\neq\emptyset_j$.
The notation introduced for $\tree^{\sss(1)}$ is used on $\tree$, verbatim.
For example, for any subset $\tau \subseteq \tree$, we write $\boundary \tau=\set{v \not\in \tau\colon \parent{v} \in \tau}$ for the boundary vertices of $\tau$.

The FPP process on $\twoPWITs$ with edge weights $(f_n(X_v))_v$ starting from $\varnothing_1$ and $\varnothing_2$ is equivalent to the union $\BP=\BP^{\sss(1)} \cup \BP^{\sss(2)}$ of two CTBPs.
Let $T_\fr^{\sss(j)}$ be a stopping time with respect to the filtration induced by $\BP^{\sss(j)}$, $j \in\set{1,2}$.
We call $T_\fr^{\sss(1)}$ and $T_\fr^{\sss(2)}$ \emph{freezing times} and run $\BP$ until $T_\fr^{\sss(1)}\wedge T_\fr^{\sss(2)}$, the time when one of the two CTBPs is large enough.
Then we freeze the larger CTBP and allow the smaller one to evolve normally until it is large enough, at time $T_\fr^{\sss(1)}\vee T_\fr^{\sss(2)}$.
At this time, which we call the \emph{unfreezing} time $T_\unfr=T_\fr^{\sss(1)}\vee T_\fr^{\sss(2)}$, both CTBPs resume their usual evolution.
We denote by $R_j(t)$ the on-off processes describing this behavior: that is, for $j \in \set{1,2}$, 
	\begin{equation}
	\labelPartII{OnOff}
  R_j(t)=(t\wedge T_\fr^{\sss(j)}) + ((t-T_\unfr)\vee 0)
  .
	\end{equation}
The version of $\BP=(\BP_t)_{t\ge 0}$ including freezing is then given by
\begin{equation}\labelPartII{clusterDefinition}
\cluster_t=\bigunion_{j=1}^2 \cluster_t^{\sss(j)},  \quad  \cluster_t^{\sss(j)}=\set{v\in\tree^{\sss(j)} \colon T_v \leq R_j(t)} = \BP_{R_j(t)}^{\sss(j)} \quad \text{for all }t\ge 0.
\end{equation}
As with $\tree$, we can consider $\cluster_t$ to be the union of two trees by placing an edge between each non-root vertex $v\notin\set{\emptyset_1,\emptyset_2}$ and its parent.
We denote by $T_v^\cluster=\inf\set{t\ge 0 \colon  v\in\cluster_t}$ the arrival time of the individual $v\in\tree$ in $\cluster=(\cluster_t)_{t\ge 0}$.
Using the left-continuous inverse of $R_j(t)$, defined by
\begin{equation}\labelPartII{OnOffInverse}
R_j^{-1}(y) = \inf\set{t\geq 0\colon R_j(t)\geq y}
=
\begin{cases}
t & \text{if }t\leq T_\fr^{\sss(j)},\\
T_\unfr-T_\fr^{\sss(j)}+t & \text{if }t>T_\fr^{\sss(j)},
\end{cases}
\end{equation}
we obtain $T_v^\cluster=R_j^{-1}(T_v)$ for $v\in\tree^{\sss(j)}$.

We next define what we mean by two smallest-weight trees on $K_n$ with freezing.
Here we have to be more careful since, in contrast to the definition of $\SWT$, there is competition between the two clusters and we have to actively forbid their merger.
Consequently, we will define the smallest-weight tree $\cS=(\cS_t)_{t\ge 0}$ inductively using $R_j^{-1}$ such that at every time $t\ge 0$, $\cS_t=\cS_t^{\sss(1)}\cup \cS_t^{\sss(2)}$ is the disjoint union of two trees $\cS_t^{\sss(1)}$ and $\cS_t^{\sss(2)}$ with root $1$ and $2$, respectively.

At time $t=0$, let $\cS_0$ be the subgraph of $K_n$ with vertex set $\set{1,2}$ and no edges.
Suppose inductively that we have constructed $(\cS_t)_{0\leq t\leq \tau_{k-1}}$ up to the time $\tau_{k-1}$ where the $(k-1)^\st$ vertex (not including the vertices 1 and 2) was added, $1\leq k\leq n-2$, with the convention that $\tau_0=0$.
Denote by $T^{\sss \cS}(i)=\inf\set{t\geq 0\colon i\in \cS_t}$ the arrival time of a vertex $i\in[n]$.

Consider the set $\boundary \cS_{\tau_{k-1}}$ of edges $e$ joining a previously explored vertex $\underline{e}\in\cS_{\tau_{k-1}}$ to a new vertex $\overline{e}\notin\cS_{\tau_{k-1}}$.
For such an edge, write $j(e)\in\set{1,2}$ for the index defined by $i\in\cS_{\tau_{k-1}}^{\sss(j(e))}$.
At time
\begin{equation}\labelPartII{cSConstruction}
\tau_k=\min_{e\in\boundary \cS_{\tau_{k-1}}} R_{j(e)}^{-1}\left(\Big. R_{j(e)}\left(\big. T^{\sss\cS}(\underline{e}) \right)+Y_e^{\sss(K_n)} \right),
\end{equation}
we add the edge $e_k$ that attains the minimum in \eqrefPartII{cSConstruction}.
Our standing assumption on the edge weights $Y_e^{\sss(K_n)}$ and the processes $R_1,R_2$ will be that this minimum, and in addition the minimum
\begin{equation}\labelPartII{EnlargedMinimum}
\min_{j\in\set{1,2}}\min_{e\colon e\in\boundary\cS^{(j)}_{\tau_{k-1}}} R_j^{-1}\left( R_j(T^{\sss\cS}(\underline{e}))+Y_e^{\sss(K_n)} \right)
\end{equation}
(with edges between $\cS^{\sss(1)}$ and $\cS^{\sss(2)}$ included) are uniquely attained a.s.
We set $\cS_t=\cS_{\tau_{k-1}}$ for $\tau_{k-1}\leq t<\tau_k$, and we define $\cS_{\tau_k}$ to be the graph obtained by adjoining $e_k$ to $\cS_{\tau_{k-1}}$.

The process $\cS=(\cS_t)_{t\ge 0}$ is the smallest-weight tree from vertices $1$ and $2$ with freezing.
We remark that $\cS$ depends implicitly on the choice of the processes $R_1,R_2$, which we have defined in terms of $\cluster$.
In particular, the law of $\cS$ will depend on the relationship between $\cluster$ and the edge weights $(Y_e^{\sss(K_n)})$, which we will specify in \refsubsect{CouplingBP}.

Because the processes $R_1,R_2$ increase at variable speeds, the relationship between $\cS$, $T^{\sss \cS}(i)$ and the FPP distances $d_{K_n,Y^{(K_n)}}(i,j)$ is subtle.
For instance, it need not hold that $d_{K_n,Y^{(K_n)}}(1,i)=R_1(T^{\sss\cS}(i))$ for $i\in\union_{t\geq 0}\cS^{\sss(1)}_t$.

\begin{lemma}\lblemma{WnAsMinimum}
The weight of the optimal path $\pi_{1,2}$ from vertex 1 to vertex 2 is given by
\begin{equation}\labelPartII{WnMin}
W_n = \min_{i_1\in\cS^{\sss(1)},i_2\in\cS^{\sss(2)}} \left( R_1(T^{\sss\cS}(i_1)) + Y^{\sss(K_n)}_{\set{i_1,i_2}} + R_2(T^{\sss\cS}(i_2)) \right),
\end{equation}
and the minimum is attained uniquely a.s.
\end{lemma}

The conclusion of \reflemma{WnAsMinimum} is easily seen when $R_1(t)=t, R_2(t)=0$ (in which case $\cS^{\sss(1)}$ is the same as $\SWT^{\sss(1)}$ with vertex 2 removed) or when $R_1(t)=R_2(t)=t$ (in which case, in the quotient graph where vertices 1 and 2 are identified, and thus $\cS$ is the same as the shortest-weight tree from a single source).
The proof of \reflemma{WnAsMinimum} in general requires some care, and is given in \refsubsect{ProofWnHnCollision}.

\begin{defn}\lbdefn{CollisionTimeDef}
The \emph{collision time} is
\begin{equation}
T_\coll = \inf\set{t\geq 0\colon R_1(t)+R_2(t)\geq W_n}.
\end{equation}
The \emph{collision edge} is the edge between the vertices $I_1\in\cS^{\sss(1)}$ and $I_2\in\cS^{\sss(2)}$ that attain the minimum in \eqrefPartII{WnMin}.
We denote by $H(I_1)$, $H(I_2)$ the graph distance between $1$ and $I_1$ in $\cS^{\sss(1)}$ and between $2$ and $I_2$ in $\cS^{\sss(2)}$, respectively.
\end{defn}

\begin{theorem}[Exploration process at the collision time]\lbthm{WnHnFromCollision}
The following statements hold almost surely:
The endpoints $I_1,I_2$ of the collision edge are explored before time $T_\coll$.
The optimal path $\pi_{1,2}$ from vertex 1 to vertex 2 is the union of the unique path in $\cS^{\sss(1)}_{T_\coll}$ from 1 to $I_1$; the collision edge; and the unique path in $\cS^{\sss(2)}_{T_\coll}$ from $I_2$ to 2.
The weight and hopcount satisfy
\begin{equation}\labelPartII{WnHnFromTcollI1I2}
W_n = R_1(T_\coll)+R_2(T_\coll), \qquad H_n = H(I_1)+H(I_2)+1.
\end{equation}
\end{theorem}

\refthm{WnHnFromCollision} is proved in \refsubsect{ProofWnHnCollision}. The first equality in
\eqrefPartII{WnHnFromTcollI1I2} is a simple consequence of continuity of $t\mapsto R_j(t)$ and the definition of $T_\coll$ in \refdefn{CollisionTimeDef}.
The equality in \refPartII{WnMin} will be the basis of our analysis of $W_n$.

We note that the law of $(T_{\coll},I_1,I_2)$ depends on the choice of the on-off processes $R_1(t),R_2(t)$; however, the laws of $W_n$ and $H_n$ do not.

\subsection{\texorpdfstring{Coupling FPP on $K_n$ from two sources to a CTBP}{Coupling FPP on Kn to a CTBP}}\lbsubsect{CouplingBP}

Similarly to \refthm{Coupling1Source}, we next couple the FPP process $\cS$ on $K_n$ and the FPP process $\cluster$ on $\tree$.
To this end, we introduce a thinning procedure for $\cluster=(\cluster_t)_{t\ge 0}$:
Define $M_{\emptyset_j}=j$, for $j=1,2$.
To each $v\in \twoPWITs\setminus\set{\emptyset_1,\emptyset_2}$, we associate a mark $M_v$ chosen uniformly and independently from $[n]$.

\begin{defn}\lbdefn{Thinningcluster}
The roots $\emptyset_1, \emptyset_2 \in \tree$ are unthinned.
The vertex $v\in\tree\setminus\set{\emptyset_1,\emptyset_2}$ is \emph{thinned} if it has an ancestor $v_0=\ancestor{k}{v}$ (possibly $v$ itself) such that $M_{v_0}=M_w$ for some unthinned vertex $w$ with $T^\cluster_w<T^\cluster_{v_0}$.
\end{defn}
As with \refdefn{ThinningBP}, this definition is not circular, as vertices are investigated in their order of appearance.
Write $\thinnedcluster_t$ for the subgraph of $\cluster_t$ consisting of unthinned vertices.

From here onwards, we will work on a probability space that contains
\begin{itemize}
\item the two independent PWITs $(\tree^{\sss(j)},X^{\sss(j)})$, $j\in \set{1,2}$,
\item the marks $M_v$, $v \in \tree$,
\item and a family of independent exponential random variables $E_e$, $e\in E(K_{\infty})$, with mean $1$, independent of the PWITs and the marks,
\end{itemize}
where $E(K_\infty)=\set{\set{i,j}\colon i,j \in \N, i< j}$.

On this probability space, we can construct the FPP edge weights on $K_n$ as follows.
Let
\begin{equation}\labelPartII{Tthinnedcluster}
T^{\sss \thinnedcluster}(i)=\inf\set{t \ge 0\colon M_v=i \text{ for some } v\in \thinnedcluster_t}
\end{equation}
be the first time that a vertex with mark $i$ appears in $\thinnedcluster$ and denote the corresponding vertex by $V(i)\in \tree$. Note that $T^{\sss \thinnedcluster}(i)$ is finite for all $i$ almost surely since the FPP exploration process eventually explores every edge.
For every edge $\set{i,i'} \in E(K_n)$, we define
\begin{equation}
X(i,i')=\min\set{X_v \colon M_v=i', \parent{v}=V(i)},
\end{equation}
and
\begin{equation}\labelPartII{EdgeWeightCoupling-FPP2Source}
X_{\set{i,i'}}^{\sss(K_n)}=
\begin{cases}
\tfrac{1}{n} X(i,i') & \text{if } T^{\sss \thinnedcluster}(i)<T^{\sss \thinnedcluster}(i'),\\
\tfrac{1}{n} X(i',i) & \text{if } T^{\sss \thinnedcluster}(i')<T^{\sss \thinnedcluster}(i),\\
E_{\set{i,i'}} & \text{if } T^{\sss \thinnedcluster}(i)=T^{\sss \thinnedcluster}(i')=0.
\end{cases}
\end{equation}

The following proposition states that the random variables in \eqrefPartII{EdgeWeightCoupling-FPP2Source} can be used to produce the correct edge weights on $K_n$ for our FPP problem:

\begin{prop}\lbprop{EdgeWeightCouplingClusterFPP2Source}
Let $(X_e^{\sss(K_n)})_{e \in E(K_n)}$ be defined in \eqrefPartII{EdgeWeightCoupling-FPP2Source}, and write $Y_e^{\sss(K_n)}=g(X_e^{\sss(K_n)})$, where $g$ is the strictly increasing function from \eqrefPartII{EdgesByIncrFunct}--\eqrefPartII{FYandg}.
Then the edge weights $(Y_e^{\sss(K_n)})_{e\in E(K_n)}$ are i.i.d.\ with distribution function $\FY$.
\end{prop}

\refprop{EdgeWeightCouplingClusterFPP2Source} will be generalized in \refthm{CouplingExpl} and proved in \refsubsect{ProofCouplingFPP}.

In \refthm{WnHnFromCollision}, we explained the (deterministic) relationship between $\cS$ and the FPP problem with edge weights $Y_e^{\sss(K_n)}$.
\refprop{EdgeWeightCouplingClusterFPP2Source} shows that, subject to \eqrefPartII{EdgeWeightCoupling-FPP2Source}, these edge weights have the desired distribution.
We next explain the relationship between $\cluster$ and $\cS$.
Recall the subgraph $\pi_M(\tau)$ of $K_n$ introduced in \refdefn{InducedGraph}, which we extend to the case where $\tau\subset\tree$.

\begin{theorem}[The coupling]
\lbthm{CouplingFPP}
Under the edge-weight coupling \eqrefPartII{EdgeWeightCoupling-FPP2Source}, $\cS_t=\pi_M(\thinnedcluster_t)$ for all $t\ge 0$ almost surely.
Moreover, when $\thinnedcluster$ and $\cS$ are equipped with the FPP edge weights $Y_v=f_n(X_v)$ and $Y^{\sss(K_n)}_e=g(X^{\sss(K_n)}_e)$, respectively, the mapping $\pi_M\colon \thinnedcluster\to\cS$ is an isomorphism of edge-weighted graphs.
\end{theorem}

The proof is given in \refsubsect{ProofCouplingFPP}.

\refthm{CouplingFPP} achieves two goals.
First, it relates the exploration process $\cluster$, defined in terms of two infinite underlying trees, to the smallest-weight tree process $\cS$, defined in terms of a single finite graph.
Because thinning gives an explicit coupling between these two objects, we will be able to control its effect, even when the total number of thinned vertices is relatively large.
Consequently we will be able to study the FPP problem by analyzing a corresponding problem expressed in terms of $\cluster$ (see Theorems~\refPartII{t:CouplingFPPCollision} and \refPartII{t:FirstPointCox}) and showing that \whp\ thinning does not affect our conclusions (see \refthm{FirstPointCoxUnthinned}).

Second, \refthm{CouplingFPP} allows us to relate FPP on the complete graph ($n$-independent dynamics run on an $n$-dependent weighted graph) with an exploration defined in terms of a pair of Poisson-weighted infinite trees ($n$-dependent dynamics run on an $n$-independent weighted graph).
By analyzing the dynamics of $\cluster$ when $n$ and $s_n$ are large, we obtain a fruitful \emph{dual picture}:
When the number of explored vertices is large, we find a \emph{dynamic} rescaled branching process approximation that is essentially independent of $n$.
When the number of explored vertices is small, we make use of a \emph{static} approximation by invasion percolation found in \cite{EckGooHofNar14a}.
In fact,
under our scaling assumptions, FPP on the PWIT is closely related to \emph{invasion percolation} (IP) on the PWIT which is defined as follows.
Set $\IP^{\sss(1)}(0)$ to be the subgraph consisting of $\emptyset_1$ only.
For $k\in\N$, form $\IP^{\sss(1)}(k)$ inductively by adjoining to $\IP^{\sss(1)}(k-1)$ the boundary vertex $v\in\boundary\IP^{\sss(1)}(k-1)$ of minimal weight.
We note that, since we consider only the relative ordering of the various edge weights, we can use either the PWIT edge weights $(X_v)_v$ or the FPP edge weights $(f_n(X_v))_v$.

Write $\IP^{\sss(1)}(\infty)=\bigunion_{k=1}^\infty\IP^{\sss(1)}(k)$ for the limiting subgraph.
We remark that $\IP^{\sss(1)}(\infty)$ is a strict subgraph of $\tree^{\sss(1)}$ a.s.\ (in contrast to FPP, which eventually explores every edge).
Indeed, define
\begin{equation}\labelPartII{MDefinition}
M^{\sss(1)}=\sup\set{X_{v}\colon v\in\IP^{\sss(1)}(\infty)\setminus \set{\emptyset_1}},
\end{equation}
the largest weight of an invaded edge.
Then $\P(M^{\sss(1)}<x)$ is the survival probability of a Poisson Galton--Watson branching process with mean $x$, as in Theorems~\refPartII{t:WH-exp}, \refPartII{t:WH-example} and \refPartII{t:WH-gen}.

Consequently, \eqrefPartII{weight-res} in \refthm{WH-gen} can be read as a decomposition of the weight $W_n$ of the smallest-weight path into a deterministic part $\frac{1}{\lambda_n} \log(n/s_n^3)$ coming from the branching process dynamics and the weight of the largest edge explored by invasion percolation starting from two sources $f_n(M^{\sss(1)}\vee M^{\sss(2)})$.

\subsection{The conditional law of the collision time and collision edge}\lbsubsect{FPPExpl}

By definition, there is no path between vertices $1$ and $2$ in $\cS_t$ for any $t$. Hence,
knowing $\cS_t$ for all $t$ (and therefore also the edge weights $Y_e^{\sss(K_n)}$, $e\in\bigunion_{t\geq 0}E(\cS_t)$) does not determine $W_n$ and $T_\coll$.
However, by examining the conditional distribution of the edges joining $\cS^{\sss(1)}$ and $\cS^{\sss(2)}$, we can express the conditional law of $T_\coll$.

\begin{theorem}[Conditional law of the collision time]\lbthm{CollCondProb}
The processes $(\thinnedcluster_u,\cS_u,R_1(u),R_2(u))_{u\geq 0}$ and the edge weights $Y_e^{\sss(K_n)}$ may be coupled so that, with probability 1, $\cS_u=\pi_M(\thinnedcluster_u)$ for all $u\ge 0$ and
	\begin{align}
	\labelPartII{CondLawOfCollision}
	&\P \condparenthesesreversed{T_\coll>t \, }{(\thinnedcluster_u, R_1(u), R_2(u))_{u\geq 0}, (M_v)_{v\in\union_{u\geq 0}\thinnedcluster_u}}
	\\
	&\quad=
	\exp\biggl(
	-\frac{1}{n}\sum_{i_1\in \cS_t^{\sss(1)}} \sum_{i_2\in \cS_t^{\sss(2)}} f_n^{-1}\bigl(R_1(t)-R_1(T^{\sss \cS}(i_1))+R_2(t)-R_2(T^{\sss \cS}(i_2))\bigr)
	-f_n^{-1}\left( \Delta R(i_1,i_2) \right)
	\biggr),
	\notag
	\end{align}
where
	\begin{equation}
	\Delta R(i_1,i_2)=
	\begin{cases}
	R_1(T^{\sss \cS}(i_2))-R_1(T^{\sss \cS}(i_1)), & \text{if } T^{\sss \cS}(i_1)\leq T^{\sss \cS }(i_2),\\
	R_2(T^{\sss \cS }(i_1))-R_2(T^{\sss \cS }(i_2)), &\text{if } T^{\sss \cS}(i_2)\leq T^{\sss \cS}(i_1).
	\end{cases}
	\end{equation}
\end{theorem}

For technical reasons related to the freezing procedure, the coupling in \refthm{CollCondProb} is not quite the coupling from \eqrefPartII{EdgeWeightCoupling-FPP2Source}.
We describe the precise coupling and give the proof of \refthm{CollCondProb} in \refsubsect{CouplingCox}.

\begin{proof}[Sketch of the proof]
By \reflemma{WnAsMinimum}, \refthm{WnHnFromCollision} and the fact that $R_1+R_2$ is strictly increasing, $T_\coll>t$ is equivalent to $W_n>R_1(t)+R_2(t)$, which is in turn equivalent to
\begin{equation}\labelPartII{ConnectingEdgeCondition}
Y_{\set{i_1,i_2}}^{\sss(K_n)} > R_1(t)+R_2(t)-R_1(T^{\sss\cS}(i_1))-R_2(T^{\sss\cS}(i_2))
\quad
\text{for all }i_1\in \cS^{\sss(1)}_t, i_2\in \cS^{\sss(2)}_t
.
\end{equation}
On the other hand, the fact that $i_1\in \cS^{\sss(1)}_t$ and $i_2\in \cS^{\sss(2)}_t$, implies that between the times $T^{\sss \cS}(i_1)\wedge T^{\sss \cS}(i_2)$ and $T^{\sss \cS}(i_1)\vee T^{\sss \cS}(i_2)$ when the first vertex and the second vertex of $\set{i_1,i_2}$ were explored, respectively, the flow from the first explored vertex did not reach the other vertex.
This translates to precisely the information that $Y_{\set{i_1,i_2}}^{\sss(K_n)}>\Delta R(i_1,i_2)$.

Because of the representation of the edge weights $Y_{\set{i,j}}^{\sss(K_n)}$ in terms of exponential variables, as in \refsubsect{univ-class},
\begin{equation}\labelPartII{EdgeConditionalProbability}
\condP{Y_e^{\sss(K_n)}>y_2}{Y_e^{\sss(K_n)}>y_1} = \exp\left( -\tfrac{1}{n}\left( f_n^{-1}(y_2) - f_n^{-1}(y_1) \right) \right)
\end{equation}
for any fixed $y_2>y_1$.
The right-hand side of \eqrefPartII{CondLawOfCollision} is therefore a product of conditional probabilities.
To complete the proof, it remains to ensure that knowledge of $\thinnedcluster$, $\cS$ and $(M_v)_{v\in\thinnedcluster}$ does not reveal any other information about the edge weights $Y_e^{\sss(K_n)}$ other than $Y_{\set{i_1,i_2}}^{\sss(K_n)}>\Delta R(i_1,i_2)$.
See \refsubsect{CouplingCox} for more details.
\end{proof}

Because of the form of \eqrefPartII{CondLawOfCollision}, \refthm{CollCondProb} shows that $T_\coll$ is equal in law to the first point of a Cox process (see also \cite[Proposition 2.3]{BhaHof12}).
We next reformulate \refthm{CollCondProb} to make this connection explicit.

By a \emph{Cox process} with random intensity measure $Z$ (with respect to a $\sigma$-algebra $\F$) we mean a random point measure $\cP$ such that $Z$ is $\F$-measurable and, conditional on $\F$, $\cP$ has the distribution of a Poisson point process with intensity measure $Z$.
For notational convenience, given a sequence of intensity measures $Z_n$ on $\R \times \cX$, for some measurable space $\cX$, we write $Z_{n,t}$ for the measures on $\cX$ defined by $Z_{n,t}(\cdot)=Z_n(\ocinterval{-\infty, t}\times\cdot)$.

Let $\cP_n^{\sss (K_n)}$ be a Cox process on $\cointerval{0,\infty} \times [n] \times [n]$ (with respect to the $\sigma$-algebra generated by $\thinnedcluster$, $\cS$, $R_1$, $R_2$ and $(M_v)_{v\in\thinnedcluster}$) whose intensity measure $Z_n^{\sss(K_n)}$ is characterized by
\begin{align}\labelPartII{CollisionIntensityKn}
	&Z_{n,t}^{\sss (K_n)}(\set{i_1}\times\set{i_2})\\
	&\qquad= \indicator{i_1\in\cS_t^{\sss(1)},i_2\in\cS_t^{\sss(2)}}
	\tfrac{1}{n} \mu_n\bigl(\Delta R(i_1,i_2), R_1(t)-R_1(T^{\sss \cS}(i_1))+R_2(t)-R_2(T^{\sss \cS}(i_2))\bigr) .\nn
\end{align}

\begin{theorem}
\lbthm{CollisionThm}
The processes $\thinnedcluster,\cS,R_1,R_2$, the edge weights $Y_e^{\sss(K_n)}$ and the Cox process $\cP_n^{\sss(K_n)}$ may be coupled so that, with probability 1, $\cS_t=\pi_M(\thinnedcluster_t)$ for all $t\geq 0$ and the first point of $\cP_n^{\sss (K_n)}$ is equal to $(T_{\coll},I_1,I_2)$.
\end{theorem}

The proof of \refthm{CollisionThm} is postponed to \refsubsect{CouplingCox}.

Finally, under the coupling from \refthm{CouplingFPP}, we may lift the Cox process $\cP^{\sss(K_n)}_n$ to a Cox process $\cP_n$ on $\cointerval{0,\infty}\times\tree^{\sss(1)}\times\tree^{\sss(2)}$.

\begin{theorem}[A Cox process for the collision edges]
\lbthm{CouplingFPPCollision}
Let $\cP_n$ be a Cox process on $\cointerval{0,\infty}\times \tree^{\sss (1)} \times \tree^{\sss (2)}$ (with respect to the $\sigma$-algebra generated by $\cluster$ and $(M_v)_{v\in\tree}$) with random intensity measure $Z_n=(Z_{n,t})_{t\ge 0}$ defined by
	\begin{align}\labelPartII{CollisionIntensityTree}
	&Z_{n,t}(\set{v_1}\times\set{v_2}) =
	\indicator{v_1\in\cluster_t^{\sss(1)},v_2\in\cluster_t^{\sss(2)}} \tfrac{1}{n}
  \mu_n\bigl(\Delta R_{v_1,v_2}, R_1(t)-R_1(T_{v_1}^\cluster)+R_2(t)-R_2(T_{v_2}^\cluster)\bigr)
	\end{align}
for all $t\ge 0$, where
\begin{align}\labelPartII{DeltaRcluster}
\Delta R_{v_1,v_2}=
\begin{cases}
R_1(T_{v_2}^\cluster)-R_1(T_{v_1}^\cluster), & \text{if }T_{v_1}^\cluster\leq T_{v_2}^\cluster,\\
R_2(T_{v_1}^\cluster)-R_2(T_{v_2}^\cluster), & \text{if }T_{v_2}^\cluster\leq T_{v_1}^\cluster.
\end{cases}
\end{align}
Let $(T_\coll^{\sss(\cP_n)},V_\coll^{\sss(1)},V_\coll^{\sss(2)})$ denote the first point of $\cP_n$ for which $V_\coll^{\sss(1)}$ and $V_\coll^{\sss(2)}$ are unthinned.
Then the law of $(T_\coll^{\sss(\cP_n)},R_1(T_\coll^{\sss(\cP_n)})+R_2(T_\coll^{\sss(\cP_n)}),\pi_M(\thinnedcluster_{T_\coll^{\sss(\cP_n)}}),M_{V_\coll^{\sss(1)}},M_{V_\coll^{\sss(2)}})$ is the same as the joint law of the collision time $T_\coll$; the optimal weight $W_n$; the smallest-weight tree $\cS_{T_\coll}$ at time $T_\coll$; and the endpoints $I_1,I_2$ of the collision edge.
In particular, the hopcount $H_n$ has the same distribution as $\abs{V_\coll^{\sss(1)}}+\abs{V_\coll^{\sss(2)}}+1$.
\end{theorem}

\refthm{CouplingFPPCollision} is proved in \refsubsect{CouplingCox}.

\refthm{CouplingFPPCollision} means that we can study first passage percolation on $K_n$ by studying a CTBP problem and then controlling the effect of thinning. In fact, with \refthm{CouplingFPPCollision} in hand, we will no longer need to refer to $K_n$ at all.

\subsection{Freezing a CTBP}\lbsubsect{FreezingDisc}

It has been argued in \refother{\refsubsectPartI{FreezingDisc}}, that the freezing times $T_\fr^{\sss(j)}$ have to achieve two goals.
To align the growth of the two CTBPs $\BP^{\sss(1)}$ and $\BP^{\sss(2)}$, the freezing times $T_\fr^{\sss(j)}$ have to
come after exploring the edge of weight $f_n(M^{\sss(j)})$ (and they are instantaneously unfrozen after the last of these times),
and after the population size reached order $s_n$ at the unfreezing time.
The latter requirement guarantees that $\BP^{\sss(j)}$ exhibits typical branching dynamics with exponential growth, and is no longer subject to strong fluctuations caused by a few individuals.

These considerations lead to the following definition of the \emph{freezing times}:

\begin{defn}[Freezing]
\lbdefn{Freezing}
Define, for $j=1,2$, the \emph{freezing times}
	\begin{equation}\labelPartII{TfrDefn}
	T_\fr^{\sss(j)} = \inf\bigg\{ t\geq 0\colon \sum_{v\in\BP_t^{\sss(j)}} \int_{t-T_v}^\infty \e^{-\lambda_n(1) \left(y-(t-T_v)\right)}
	d\mu_n(y) \geq s_n \bigg\},
	\end{equation}
and the \emph{unfreezing time} $T_\unfr=T_\fr^{\sss(1)} \vee T_\fr^{\sss(2)}$.
The \emph{frozen cluster} is given by
\begin{equation}
\cluster_\fr=\cluster_{T_\unfr}=\cluster_\fr^{\sss(1)}\union\cluster_\fr^{\sss(2)} \quad \text{where} \quad \cluster_\fr^{\sss(j)}=\cluster_{T_\fr^{(j)}}^{\sss(j)}.
\end{equation}
\end{defn}
The variable $\int_{t-T_v}^\infty \e^{-\lambda_n(1) \left(y-(t-T_v)\right)} d\mu_n(y)$ represents the expected number of future offspring of vertex $v\in\BP_t^{\sss(j)}$, exponentially time-discounted at rate $\lambda_n(1)$.
Recall from \eqrefPartII{OnOff} that $R_j(t)=(t\wedge T_\fr^{\sss(j)}) + ((t-T_\unfr)\vee 0)$.
Thus, each CTBP evolves at rate $1$ until its expected number of future offspring, exponentially time-discounted at rate $\lambda_n(1)$, first exceeds $s_n$.
At that time the configuration is ``frozen'' and ceases to evolve until both sides have been frozen.
The two sides, which are now of a comparably large size, are then simultaneously unfrozen and thereafter evolve at rate 1.
Henceforth we will always assume this choice of $T_\fr^{\sss(1)},T_\fr^{\sss(2)}$.

The asymptotics of the freezing times have been investigated in \refother{\refthmPartI{TfrScaling}}:
\begin{theorem}[Scaling of the freezing times]
\lbthm{TfrScaling}
The freezing times satisfy $f_n^{-1}(T_\fr^{\sss(j)}) \convp M^{\sss (j)}$ for $j=1,2$.
\end{theorem}

Since $M^{\sss(1)}\neq M^{\sss(2)}$ a.s., \refthm{TfrScaling} confirms that the two CTBPs $\BP^{\sss(1)}$ and $\BP^{\sss(2)}$ require substantially different times to grow large enough.

The following lemma allows us to split $W_n$ into the two freezing times and the remainder for the proof of \refthm{WH-gen}:

\begin{lemma}[Sums behave like maxima]\lblemma{dfWeakMax}
Let $(T_n^{\sss(1)})_n$, $(T_n^{\sss(2)})_n$, $\mathcal{M}^{\sss(1)}$, $\mathcal{M}^{\sss(2)}$ be random variables such that $\mathcal{M}^{\sss (1)} \vee \mathcal{M}^{\sss (2)} \geq 1$ a.s.\ and $(f_n^{-1}(T_n^{\sss(1)}),f_n^{-1}(T_n^{\sss(2)})) \convp (\mathcal{M}^{\sss(1)},\mathcal{M}^{\sss(2)})$.
Then $f_n^{-1}(T_n^{\sss(1)}+T_n^{\sss(2)}) \convp \mathcal{M}^{\sss(1)} \vee \mathcal{M}^{\sss(2)}$.
\end{lemma}

\reflemma{dfWeakMax} is proved in \refsect{cond}.
\refthm{TfrScaling} and \reflemma{dfWeakMax} yield that $f_n^{-1}(T_{\fr}^{\sss(1)}+T_{\fr}^{\sss(2)})\convp M^{\sss (1)}\vee M^{\sss (2)}$.
If we can show that
	\begin{equation}\labelPartII{W-Tfr-log}
	f_n^{-1}\Big(W_n-T_{\fr}^{\sss(1)}-T_{\fr}^{\sss(2)}-\frac{1}{\lambda_n}\log{(n/s_n^3)}\Big)\convd 1,
	\end{equation}
then the scaling of $W_n$ in \refthm{WH-gen} will follow by another application of \reflemma{dfWeakMax} and the fact that $M^{\sss (1)}, M^{\sss (2)}>1$ a.s.
The presence of a logarithm in \eqrefPartII{W-Tfr-log} reflects the fact that, after $T_\unfr$, the branching processes grow exponentially.

For later use, we state bounds on the geometry of the frozen cluster that have been proved in \refother{\refthmPartI{FrozenCluster}}:
\begin{theorem}[Properties of the frozen cluster]\lbthm{FrozenCluster}{\hfill}
\begin{enumerate}
\item \lbitem{FrozenVolume}
The volume $\abs{\cluster_\fr}$ of the frozen cluster is $O_\P(s_n^2)$.
\item \lbitem{FrozenDiameter}
The diameter $\max\set{\abs{v}\colon v\in\cluster_\fr}$ of the frozen cluster is $O_\P(s_n)$.
\end{enumerate}
\end{theorem}
\refthm{FrozenCluster} will allow us to ignore the elements coming from the frozen cluster in the proof of \refthm{WH-gen}.
For instance, part \refitem{FrozenDiameter} shows that heights within the frozen cluster are negligible in the central limit theorem scaling of \eqrefPartII{hopcount-CLT}.

\subsection{Two-vertex characteristics for CTBP}\lbsubsect{TwoVerChar}

To analyze expressions like the one appearing in the exponent in \refthm{CollCondProb}, we introduce \emph{generation-weighted two-vertex characteristics}.
Let $\chi$ be a non-random, non-negative function on $\cointerval{0,\infty}^2$ and recall that $T_v=\inf\set{t \ge 0\colon v \in \BP_t}$ denotes the birth time of vertex $v$ and $\abs{v}$ the generation of $v$.
The generation-weighted two-vertex characteristic to count $\BP$ is given by
	\begin{equation}\labelPartII{2VertexCharDef}
	z_{\vec{t}}^\chi(\vec{a}) = \sum_{v_1\in\BP_{t_1}^{\sss(1)}}
	\sum_{v_2\in\BP_{t_2}^{\sss(2)}} a_1^{\abs{v_1}} a_2^{\abs{v_2}} \chi(t_1-T_{v_1},t_2-T_{v_2}),
	\end{equation}
for all $t_1,t_2,a_1,a_2 \ge 0$, where we use vector notation $\vec{a}=(a_1,a_2)$, $\vec{t}=(t_1,t_2)$, and so on.
We make the convention that $\chi(t_1,t_2)=z_{t_1,t_2}^\chi(\vec{a})=0$ for $t_1\wedge t_2<0$.

To apply first and second moment methods, we estimate the moments of $z_{\vec{t}}^\chi(\vec{a})$.
As in \eqrefPartII{barredOneVertexChars}, we rescale and write
	\begin{align}\labelPartII{barredTwoVertexChars}
	\begin{split}
	\bar{z}_{t_1,t_2}^\chi(\vec{a})
	&=
	\e^{-\lambda_n(a_1) t_1} \e^{-\lambda_n(a_2) t_2} z_{t_1,t_2}^\chi(\vec{a}),
	\\
	\bar{m}_{\vec{t}}^\chi(\vec{a})
	&=
	\E(\bar{z}_{\vec{t}}^\chi(\vec{a})),
	\\
	\bar{M}_{\vec{t},\vec{u}}^{\chi,\eta}(\vec{a},\vec{b})
	&=
	\E(\bar{z}_{\vec{t}}^\chi(\vec{a}) \bar{z}_{\vec{u}}^\eta(\vec{b})).
	\end{split}
	\end{align}

Looking again at the two-vertex characteristic appearing in \refthm{CollCondProb}, we notice that the summands in \eqrefPartII{CondLawOfCollision} with $T^{\sss \cS}(i),T^{\sss \cS}(j)>T_\unfr$ simplify to $\tfrac{1}{n}f_n^{-1}(t-T^{\sss \cS}(i)+t-T^{\sss \cS}(j))-\tfrac{1}{n}f_n^{-1}(\abs{T^{\sss \cS}(i)-T^{\sss \cS}(j)})$ since our chosen functions $R_1(t),R_2(t)$ are linear with slope 1 for $t>T_\unfr=T_{\fr}^{\sss (1)}\vee T_{\fr}^{\sss(2)}$.
Hence, we are led to study the characteristic
	\begin{equation}\labelPartII{TwoVertexChar}
	\chi_n(t_1,t_2)=\mu_n(\abs{t_1-t_2},t_1+t_2).
	\end{equation}
	
The characteristic $\chi_n$ will prove difficult to control directly, because its values fluctuate significantly in size: for instance, $\chi_n(\tfrac{1}{2}f_n(1),\tfrac{1}{2}f_n(1))=1$ whereas $\chi_n(\tfrac{1}{2}f_n(1),f_n(1))=O(1/s_n)$.
Therefore, for $K\in(0,\infty)$, we define the truncated measure
	\begin{equation}\labelPartII{muRestricted}
	\mu_n^{\sss(K)}=\mu_n\big\vert_{\ocinterval{f_n(1-K/s_n),f_n(1+K/s_n)}},
	\end{equation}
and again write $\mu_n^{\sss(K)}(\ocinterval{a,b})=\mu_n^{\sss(K)}(a,b)$ to shorten notation.
For convenience, we will always assume that $n$ is large enough that $s_n\geq K$.
By analogy with \eqrefPartII{TwoVertexChar}, define
	\begin{equation}\labelPartII{TwoVertexCharRestricted}
	\chi_n^{\sss(K)}(t_1,t_2)=\mu_n^{\sss(K)}(\abs{t_1-t_2},t_1+t_2).
	\end{equation}
By construction, the total mass of $\mu_n^{\sss(K)}$ is $2K/s_n$, so that $s_n \chi_n^{\sss(K)}$ is uniformly bounded.

The main aim in this section is to state results identifying the asymptotic behavior of $\bar{z}_{t_1,t_2}^{\chi^{\sss(K)}_n}(\vec{a})$ and showing that, for $K\rightarrow \infty$,
the contribution due to $\chi_n-\chi^{\sss(K)}_n$ becomes negligible.
These results are formulated in \refthm{TwoVertexConvRestrictedSum}, which investigates the truncated two-vertex characteristic, and \refthm{TwoVertexRemainder}, which studies the effect of truncation:

\begin{theorem}[Convergence of truncated two-vertex characteristic]\lbthm{TwoVertexConvRestrictedSum}
For every $\epsilon>0$ and every compact subset $A\subset(0,2)$, there exists a constant $K_0<\infty$ such that for every $K \ge K_0$ there are constants $K'<\infty$ and $n_0 \in \N$ such that for all $n \ge n_0$, $a_1,a_2,b_1,b_2\in A$ and $\lambda_n(1)[t_1\wedge t_2\wedge u_1\wedge u_2] \geq K'$,
	\begin{equation}\labelPartII{CharMain1stMomConv}
	\abs{s_n^{-1} \bar{m}_{\vec{t}}^{\chi_n^{\sss(K)}}(\vec{a}^{1/s_n}) - \zeta(a_2/a_1)}
	\leq
	\epsilon,
	\end{equation}
and
	\begin{equation}
	\abs{s_n^{-4} \bar{M}_{\vec{t},\vec{u}}^{\chi_n^{\sss(K)},\chi_n^{\sss(K)}}
	(\vec{a}^{1/s_n},\vec{b}^{1/s_n})
	- \frac{\zeta(a_2/a_1)\zeta(b_2/b_1)}{\log(1/a_1+1/b_1)\log(1/a_2+1/b_2)}}
	\leq
	\epsilon,
	\end{equation}
where $\zeta\colon (0,\infty)\to\R$ is the continuous function defined by
	\begin{equation}\labelPartII{zetaDef}
	\zeta(a_1/a_2)=
	\begin{cases}
	\frac{2a_1 a_2}{a_1+a_2}\frac{\log(a_2/a_1)}{a_2-a_1} , & \text{if } a_1\neq a_2,\\
	1, & \text{if }a_1=a_2.
	\end{cases}
	\end{equation}
Moreover, for every $K<\infty$ there are constants $K''<\infty$ and $n_0'\in \N$ such that for all $n \ge n_0'$, $t_1,t_2,u_1,u_2\geq 0$ and $a_1,a_2,b_1,b_2\in A$, $\bar{m}_{\vec{t}}^{\chi_n^{\sss(K)}}(\vec{a}^{1/s_n}) \leq K'' s_n$ and $\bar{M}_{\vec{t},\vec{u}}^{\chi_n^{\sss(K)},\chi_n^{\sss(K)}}(\vec{a}^{1/s_n},\vec{b}^{1/s_n})\leq K'' s_n^4$.
\end{theorem}

The exponents in \refthm{TwoVertexConvRestrictedSum} can be understood as follows.
By \refthm{OneCharConv}, the first and second moments of a bounded one-vertex characteristic are of order $s_n$ and $s_n^3$, respectively.
Therefore, for two-vertex characteristics, one can expect $s_n^2$ and $s_n^6$.
Since $\chi_n^{\sss(K)}=\frac{1}{s_n}s_n\chi_n^{\sss(K)}$ appears once in the first and twice in the second moment, we arrive at $s_n$ and $s_n^4$, respectively.

\begin{theorem}[The effect of truncation]\lbthm{TwoVertexRemainder}
For every $K>0$, $\bar{m}_{\vec{t}}^{\chi_n-\chi_n^{(K)}}(\vec{1})=O(s_n)$, uniformly over $t_1,t_2$.
Furthermore, given $\epsilon>0$, there exists $K<\infty$ such that, for all $n$ sufficiently large, $\bar{m}_{\vec{t}}^{\chi_n-\chi_n^{\sss(K)}}(\vec{1})\leq \epsilon s_n$ whenever $\lambda_n(1)[t_1\wedge t_2]\geq K$.
\end{theorem}

Theorems \refPartII{t:TwoVertexConvRestrictedSum} and \refPartII{t:TwoVertexRemainder} are proved in \refsect{2VertexCharSec}.

\subsection{The collision edge and its properties}\lbsubsect{CollisionProperties}

\refthm{CouplingFPPCollision} expresses the collision edge in terms of the first unthinned point of $\cP_n$.
We begin by stating the asymptotic behavior of the first point (whether thinned or not) of $\cP_n$:

\begin{theorem}[The first point of the Cox process]
\lbthm{FirstPointCox}
Let $\cP_n$ be the Cox process in \refthm{CouplingFPPCollision}, and let $(T_\first,V_\first^{\sss(1)},V_\first^{\sss(2)})$ denote its first point.
Then
	\begin{equation}
	\labelPartII{Tfirst}
	T_\first=T_\unfr+\frac{\log(n/s_n^3)}{2\lambda_n}+\Op(1/\lambda_n).
	\end{equation}
Furthermore, recalling the sequence $(\phi_n)_n$ from \eqrefPartII{phinDefinitionShort}, the pair
	\eqn{
	\labelPartII{RescaledPair}
	\left( \frac{\abs{V_\first^{\sss(1)}}-\tfrac{1}{2}\phi_n\log(n/s_n^3)}{\sqrt{s_n^2\log(n/s_n^3)}} , \frac{\abs{V_\first^{\sss(2)}}-\tfrac{1}{2}\phi_n\log(n/s_n^3)}{\sqrt{s_n^2\log(n/s_n^3)}} \right)
	}
converges in distribution to a pair of independent normal random variables of mean 0 and variance $\tfrac{1}{2}$, and is asymptotically independent of $T_\first$ and of $\cluster$.
\end{theorem}

The proof of \refthm{FirstPointCox}, presented at the end of the current section, is based on a general convergence result for Cox processes, which we now describe.
Consider a sequence of Cox processes $(\cP_n^*)_n$ on $\R \times \R^2$ with random intensity measures $(Z_n^*)_n$, with respect to $\sigma$-fields $(\F_n)_n$.
We will write $\cP_{n,t}^*$ for the measure defined by $\cP_{n,t}^*(\cdot)=\cP_n^*(\cointerval{-\infty,t}\times \cdot)$.
Define
\begin{equation}\labelPartII{FirstPointOfPPPTnk}
T_{n,k}^*=\inf\set{t\colon \abs{\cP_{n,t}^*}\geq k}
\end{equation}
and let $A_{n,k}$ be the event that $T_{n,j}^*\notin\set{\pm\infty}$ and $\shortabs{\cP_{n,T_{n,j}}^*}=j$, for $j=1,\dotsc,k$.
That is, $A_{n,k}$ is the event that the points of $\cP_n^*$ with the $k$ smallest $t$-values are uniquely defined.
On $A_{n,k}$, let $X_{n,k}$ denote the unique point for which $\cP_n^*(\set{T_{n,k}^*}\times\set{X_{n,k}})=1$, and otherwise set $X_{n,k}=\cemetery$, an isolated cemetery point.

The following theorem gives a sufficient condition for the first points of such a Cox process to converge towards independent realizations of a probability measure $Q$.
To state it, we write
\begin{equation}\labelPartII{MomentGeneratingNotation}
\hat{R}(\vec{\xi})=\int_{\R^d} \! \e^{\vec{\xi}\cdot \vec{x}} dR(\vec{x})
\end{equation}
for the moment generating function of a measure $R$ on $\R^d$ and $\vec{\xi} \in \R^d$.

\begin{theorem}\lbthm{FirstPointCoxAsymp}
Fix a probability measure $Q$ on $\R^2$ with $\hat{Q}(\vec{\xi})<\infty$ for all $\vec{\xi} \in \R^2$, a non-decreasing continuous function $q\colon\R\to(0,\infty)$ satisfying $\lim_{t\to-\infty}q(t)=0$.
Suppose that we can find a decomposition $Z_n^*=Z^{\prime\sss(K)}_n+Z^{\prime\prime\sss(K)}_n$ for each $K>0$, and sub-$\sigma$-fields $\F'_n\subset\F_n$, such that
\begin{enumerate}
\item \lbitem{MomentAssumptionsCoxTheorem}
for each fixed $\epsilon>0, t,u\in\R, \vec{\xi}\in\R^2$, there exists $K_0<\infty$ such that, for all $K\geq K_0$,
\begin{equation}\labelPartII{FirstMomentAgainstxi}
(1-\epsilon)q(t)\hat{Q}(\vec{\xi})\leq \condE{ \hat{Z}'_{n,t}(\vec{\xi})}{\F'_n} \leq (1+\epsilon)q(t)\hat{Q}(\vec{\xi}),
\end{equation}
\begin{align}\labelPartII{SecondMomentAgainstxi}
&&\condE{\Big( \frac{\hat{Z}^{\prime\sss(K)}_{n,t}(\vec{\xi})}{q(t)\hat{Q}(\vec{\xi})}-\frac{\abs{Z^{\prime\sss(K)}_{n,u}}}{q(u)} \Big)^2}{\F'_n} \leq \epsilon
,&&\text{and}
\end{align}
\begin{equation}\labelPartII{SplittingAssumptions}
\condE{\shortabs{Z''_{n,t}}}{\F'_n} < \epsilon q(t)
\end{equation}
with probability at least $1-\epsilon$ for $n$ sufficiently large; and

\item\lbitem{PPPTightIntensityRightCoxTheorem}
for each $\epsilon>0$, there exists $\overline{t}$ such that
\begin{equation}\labelPartII{ManyParticles}
\liminf_{n\to\infty}\P\left( \bigabs{Z_{n,\overline{t}}^*}>1/\epsilon \right) \geq 1-\epsilon.
\end{equation}

\end{enumerate}
Then the random sequence $(X_{n,j})_{j=1}^\infty$ converges in distribution to an i.i.d.\ random sequence $(X_j)_{j=1}^\infty$ where $X_j$ has law $Q$.
Moreover $\set{(T_{n,j})_{j=1}^k\colon n\in\N}$ is tight, $(X_{n,j})_{j=1}^\infty$ is asymptotically independent of $\F_n$ and, if $(T_j,X_j)_{j=1}^\infty$ is any subsequential limit of $(T_{n,j},X_{n,j})_{j=1}^\infty$, then $(T_j)_{j=1}^\infty$ and $(X_j)_{j=1}^\infty$ are independent.
\end{theorem}
\refthm{FirstPointCoxAsymp} is proved in \refsect{FirstPointsCoxSec}.

To apply \refthm{FirstPointCoxAsymp}, we will rescale and recentre both time and the heights of vertices.
Furthermore, we will remove the effect of the frozen cluster $\cluster_\fr$.

\begin{defn}\lbdefn{PnStar}
For $v\in\tree\setminus\cluster_\fr$, define $p^\unfr(v)$ to be the unique ancestor $v'$ of $v$ for which $v'\in\boundary\cluster_\fr$ (with $p^\unfr(v)=v$ if $\parent{v}\in\cluster_\fr$).
Write
\begin{align}
\abs{v}^*
&=
\frac{\abs{v} -\abs{p^\unfr(v)} - \tfrac{1}{2}\phi_n \log(n/s_n^3)}{ \sqrt{s_n^2\log(n/s_n^3)}}
,
\labelPartII{v*General}
\\
t^*
&=
\lambda_n(1)(t -T_\unfr) - \tfrac{1}{2}\log(n/s_n^3)
.
\labelPartII{tt*}
\end{align}
Define $\cP_n^*$ to be the image under the mapping $(t,v_1,v_2)\mapsto (t^*,\abs{v_1}^*,\abs{v_2}^*)$ of the restriction of $\cP_n$ to $\cointerval{0,\infty}\times(\tree^{\sss(1)}\setminus\cluster_\fr^{\sss(1)})\times(\tree^{\sss(2)}\setminus\cluster_\fr^{\sss(2)})$.
\end{defn}
\begin{theorem}\lbthm{PnStarSatisfiesConditions}
The point measures $(\cP_n^*)_n$ are Cox processes and satisfy the hypotheses of \refthm{FirstPointCoxAsymp} when $Q$ is the law of a pair of independent $N(0,\tfrac{1}{2})$ random variables, $q(t^*)=\e^{2t^*}$, and $\F'_n$ is the $\sigma$-field generated by the frozen cluster $\cluster_\fr$.
\end{theorem}

We prove \refthm{PnStarSatisfiesConditions} in \refsubsect{SecondMomentEstimates}.
All the vertices relevant to $\cP_n^*$ are born after the unfreezing time $T_\unfr$, and therefore appear according to certain CTBPs.
\refthm{PnStarSatisfiesConditions} will therefore be proved by a first and second moment analysis of the two-vertex characteristics from \refsubsect{TwoVerChar}.

To use \refthm{PnStarSatisfiesConditions} in the proof of \refthm{FirstPointCox}, we will show that the first point $(T_\first^*,H_1^*,H_2^*)$ of $\cP_n^*$ and the first point $(T_\first,V_\first^{\sss(1)},V_\first^{\sss(2)})$ of $\cP_n$ are \whp\ related as in \eqrefPartII{v*General}--\eqrefPartII{tt*}.
This will follow from part~\refitem{NoFrozenCollision} of the following lemma, which we will prove in Sections~\refPartII{ss:VolumeCollisionProof} and \refPartII{ss:FrozenCollisions}.

\begin{lemma}\lblemma{AtCollision}
Let $K<\infty$ and $\overline{t}=T_\unfr+\lambda_n(1)^{-1}(\frac{1}{2} \log(n/s_n^3) +K)$.
Then
\begin{enumerate}
\item\lbitem{VolumeAtCollision}
$\abs{\cluster_{\overline{t}}}=O_{\P}(\sqrt{ns_n})$; and
\item\lbitem{NoFrozenCollision}
$\cP_n\left( [0,\overline{t}]\times \cluster_\fr^{\sss(1)}\times \tree^{\sss(2)} \right) = \cP_n\left( [0,\overline{t}]\times \tree^{\sss(1)}\times \cluster_\fr^{\sss(2)} \right) = 0$ {\whpdot}
\end{enumerate}
\end{lemma}

Assuming \reflemma{AtCollision}~\refitem{NoFrozenCollision} and \refthm{PnStarSatisfiesConditions}, we can now prove \refthm{FirstPointCox}:

\begin{proof}[Proof of \refthm{FirstPointCox}]
By construction, the first point $(T_\first^*,H_1^*,H_2^*)$ of $\cP_n^*$ is the image of some point $(T,V_1,V_2)$ of $\cP_n$ under the mapping $(t,v_1,v_2)\mapsto (t^*,\abs{v_1}^*,\abs{v_2}^*)$.
Theorems~\refPartII{t:FirstPointCoxAsymp} and \refPartII{t:PnStarSatisfiesConditions} imply that $T_\first^*=O_\P(1)$, so that $T=T_\unfr+\lambda_n(1)^{-1}(\tfrac{1}{2}\log(n/s_n^3)+O_\P(1))$ by \eqrefPartII{tt*}.
We may therefore apply \reflemma{AtCollision}~\refitem{NoFrozenCollision} to conclude that $\cP_n\left( [0,T]\times\cluster_\fr^{\sss(1)}\times \tree^{\sss(2)} \right)= \cP_n\left( [0,T]\times \tree^{\sss(1)}\times \cluster_\fr^{\sss(2)} \right) = 0$ {\whpdot}

In particular, \whp, $(T,V_1,V_2)$ equals the first point $(T_\first,V_\first^{\sss(1)},V_\first^{\sss(2)})$ of $\cP_n$, and therefore $H_j^*=\abs{V_\first^{\sss(j)}}^*$.
In \refthm{FirstPointCox}, the heights are to be rescaled as in \eqrefPartII{RescaledPair} rather than \eqrefPartII{v*General}.
However, these differ only by the term $\abs{p^\unfr(V_\first^{\sss(j)})}/s_n\sqrt{\log(n/s_n^3)}$.
By \refthm{FrozenCluster}~\refitem{FrozenDiameter}, we have $\abs{p^\unfr(V_\first^{\sss(j)})}=1+O_\P(s_n)$, since $p(p^\unfr(V_\first^{\sss(j)}))\in\cluster^{\sss(j)}_\fr$ by construction.
Hence the term $\abs{p^\unfr(V_\first^{\sss(j)})}/s_n\sqrt{\log(n/s_n^3)}$ is $o_\P(1)$.
Finally, the asymptotic independence statements follow from those in \refthm{FirstPointCoxAsymp} and \eqrefPartII{Tfirst} follows from the tightness of $T_\first^*$.
\end{proof}

\subsection{Thinning and completion of the proof}\lbsubsect{ProofCompletion}
In this section, we explain that the first point of the Cox process is \whp\ unthinned and conclude our main results:
\begin{theorem}[First point of Cox process is \whp\ unthinned]
\lbthm{FirstPointCoxUnthinned}
Let $\cP_n$ be the Cox process in \refthm{CouplingFPPCollision}, and let $(T_\first,V_\first^{\sss(1)},V_\first^{\sss(2)})$ denote its first point.
Then $V_\first^{\sss(1)}$ and $V_\first^{\sss(2)}$ are \whp\ unthinned. Consequently, whp $T_\coll^{\sss(\cP_n)}=T_\first,V_\coll^{\sss(1)}=V_\first^{\sss(1)},
V_\coll^{\sss(2)}=V_\first^{\sss(2)}$.
\end{theorem}
\begin{proof}
According to \refdefn{Thinningcluster}, the vertex $V_\first^{\sss(j)}$, $j\in\set{1,2}$, will be thinned if and only if some non-root ancestor $v_0$ of $V_\first^{\sss(j)}$ has $M_{v_0}=M_w$, where $w\in\cluster_{T_\first}$ is unthinned and $T_w^\cluster<T_{v_0}^\cluster$.
We obtain an upper bound by dropping the requirement that $w$ should be unthinned and relaxing the condition $T_w^\cluster<T_{v_0}^\cluster$ to $T_w^\cluster\leq T_\first$ and $w\neq v_0$.
Each such pair of vertices $(v_0,w)$ has conditional probability $1/n$ of having the same mark, so
\begin{equation}
\condP{\big.\text{$V_\first^{\sss(1)}$ or $V_\first^{\sss(2)}$ is thinned}}{V_\first^{\sss(1)},V_\first^{\sss(2)},\abs{\cluster_{T_{\first}}}} \leq \tfrac{1}{n}(\abs{V_\first^{\sss(1)}}+\abs{V_\first^{\sss(2)}})\abs{\cluster_{T_{\first}}}.
\end{equation}
By \refthm{FirstPointCox}, $\abs{V_\first^{\sss(j)}}=O_\P(s_n\log(n/s_n^3))$.
Moreover $T_\first=T_\unfr+\lambda_n(1)^{-1}(\tfrac{1}{2}\log(n/s_n^3)+O_\P(1))$, so that $\abs{\cluster_{T_\first}}=O_\P(\sqrt{ns_n})$ by \reflemma{AtCollision}~\refitem{VolumeAtCollision}.
Hence
\begin{equation}
\condP{\big.\text{$V_\first^{\sss(1)}$ or $V_\first^{\sss(2)}$ is thinned}}{V_\first^{\sss(1)},V_\first^{\sss(2)},\abs{\cluster_{T_{\first}}}} \leq O_\P\left( \frac{\log(n/s_n^3)}{\sqrt{n/s_n^3}} \right),
\end{equation}
and this upper bound is $o_\P(1)$ since $n/s_n^3\to\infty$.
\end{proof}

Note that other choices of $R_1(t),R_2(t)$ would make \refthm{FirstPointCoxUnthinned} false.
For the first point of $\cP_n$ to appear, the intensity measure $Z_{n,t}$, which is given by $1/n$ times a sum over $\cluster_t^{\sss(1)}\times\cluster_t^{\sss(2)}$, must be of order $1$.
If $R_1(t)=t$, $R_2(t)=0$, for instance, then $\cluster_t^{\sss(2)}$ is small and it follows that $\cluster_t^{\sss(1)}$ must be large (typically at least of order $n$) at time $t=T_\first$.
In this case thinning would have a very strong effect.
We note that this argument applies even to relatively well-behaved edge distributions such as the $E^s$ edge weights considered in \cite{BhaHof12}, where the exploration must proceed simultaneously from both endpoints with $R_1(t)=R_2(t)=t$.

In the heavy-tailed case that we consider, even the choice $R_1(t)=R_2(t)=t$ is insufficiently balanced: see the discussion in \refother{Sections~\refstarPartI{ss:explorationfromtwosources}--\refstarPartI{ss:FreezingDisc}}.
This is a crucial reason for introducing the freezing procedure of \refsubsect{FreezingDisc}.

We are now ready to complete the proof of \refthm{WH-gen}:

\begin{proof}[Proof of \refthm{WH-gen}]
According to \reflemma{AssumeStrongCond}, we can assume  Conditions~\refPartII{cond:scalingfn}, \refPartII{cond:LowerBoundfn} and \refPartII{cond:boundfnExtended}.
We begin with the hopcount result \eqrefPartII{hopcount-CLT}.
By \refthm{CouplingFPPCollision}, $H_n \equalsd \abs{V_\coll^{\sss(1)}}+\abs{V_\coll^{\sss(2)}}+1$, where $(T_\coll^{\sss(\cP_n)},V_\coll^{\sss(1)},V_\coll^{\sss(2)})$ is the first point of the Cox process $\cP_n$ for which $V_\coll^{\sss(1)}$ and $V_\coll^{\sss(1)}$ are unthinned.
By \refthm{FirstPointCoxUnthinned}, $T_\coll^{\sss(\cP_n)}=T_\first$ \whp, so that the pairs $(V_\coll^{\sss(1)},V_\coll^{\sss(2)})$ and $(V_\first^{\sss(1)},V_\first^{\sss(2)})$ in Theorems~\refPartII{t:CouplingFPPCollision} and \refPartII{t:FirstPointCox} are the same {\whpdot}
Hence, \whp,
\begin{align}
\frac{H_n-\phi_n\log(n/s_n^3)}{\sqrt{s_n^2 \log(n/s_n^3)}} &\equalsd \frac{\abs{V_\first^{\sss(1)}}-\tfrac{1}{2}\phi_n\log(n/s_n^3)}{\sqrt{s_n^2 \log(n/s_n^3)}} + \frac{\abs{V_\first^{\sss(2)}}-\tfrac{1}{2}\phi_n\log(n/s_n^3)}{\sqrt{s_n^2 \log(n/s_n^3)}} + \sop(1)
,
\end{align}
so that \refthm{FirstPointCox} implies the CLT for $H_n$ in \eqrefPartII{hopcount-CLT}.

For the weight result \eqrefPartII{weight-res}, \refthm{CouplingFPPCollision} states that $W_n \equalsd R_1(T_\coll^{\sss(\cP_n)})+R_2(T_\coll^{\sss(\cP_n)})$.
On the event $\set{T_\first\geq T_\unfr}\intersect\set{T_\first=T_\coll^{\sss(\cP_n)}}$ (which, by \refthm{FirstPointCox} and the argument above, occurs \whp), the definition \eqrefPartII{OnOff} of $R_1(t),R_2(t)$ leads to $R_1(T_\coll^{\sss(\cP_n)})+R_2(T_\coll^{\sss(\cP_n)})=T_\fr^{\sss(1)}+T_\fr^{\sss(2)}+2(T_\first-T_\unfr)$.
Using again \refthm{FirstPointCox}, we obtain
\begin{equation}
W_n \equalsd T_\fr^{\sss(1)}+T_\fr^{\sss(2)}+\frac{\log(n/s_n^3)}{\lambda_n}+\Op(1/\lambda_n).
\end{equation}
Therefore, \eqrefPartII{weight-res} follows from \refthm{TfrScaling} and Lemmas~\refPartII{l:lambdanAsymp} and \refPartII{l:dfWeakMax}.
Finally, the independence of the limiting variables follows from the asymptotic independence in \refthm{FirstPointCox}.
\end{proof}

\section{\texorpdfstring{Scaling properties of $f_n$ and $\mu_n$}{Scaling properties of fn and mun}}\lbsect{cond}

In this section, we prove \reflemma{dfWeakMax}, and we restate \refcond{boundfnExtended} in terms of the density of $\mu_n$.

\subsection{\texorpdfstring{Growth and density bounds for $f_n$ and $\mu_n$}{Growth and density bounds for fn and mun}}\lbsubsect{densityBoundsfn}

In this section, we explore the key implications of Conditions~\refPartII{cond:scalingfn}--\refPartII{cond:LowerBoundfn} and \refPartII{cond:boundfnExtended} on $f_n$ and on the intensity measure $\mu_n$. This section also appears in \refother{\refsectPartI{densityBoundsfn}},
but since the proofs are just a few lines, we have decided to repeat them here.

\begin{lemma}\lblemma{ExtendedImpliesWeak}
There exists $n_0 \in \N$ such that
\begin{equation}
f_n(x)\leq \left( \frac{x}{x'} \right)^{\epsilonCondition s_n} f_n(x') \qquad\text{whenever }1-\deltaCondition\leq x\leq x', n \ge n_0
.
\end{equation}
\end{lemma}

\begin{proof}
Divide \eqrefPartII{BoundfnSmall} or \eqrefPartII{fn-bound} by $x$ and integrate between $x$ and $x'$ to obtain $\log{f_n(x')}-\log{f_n(x)}\geq \epsilonCondition s_n \left( \log x' - \log x \right)$ whenever $1-\deltaCondition\leq x\leq x'$, $n \ge n_0$, as claimed.
\end{proof}

We call \refcond{boundfnExtended} a density bound because it implies the following lemma, which will also be useful in the study of two-vertex characteristics in \refsect{2VertexCharSec}:
\begin{lemma}\lblemma{munDensityBound}
For $n$ sufficiently large, on the interval $(f_n(1-\deltaCondition),\infty)$, the measure $\mu_n$ is absolutely continuous with respect to Lebesgue measure and
\begin{equation}
\indicator{y>f_n(1-\deltaCondition)} d\mu_n(y) \leq \frac{1}{\epsilonCondition s_n} \frac{f_n^{-1}(y)}{y} dy.
\end{equation}
\end{lemma}
\begin{proof}
By Conditions~\refPartII{cond:LowerBoundfn} and \refPartII{cond:boundfnExtended}, $f_n$ is strictly increasing on $(1-\deltaCondition,\infty)$, so $y=f_n(\mu_n(0,y))$ for $y>f_n(1-\deltaCondition)$.
Differentiating and again applying Conditions~\refPartII{cond:LowerBoundfn} and \refPartII{cond:boundfnExtended}, we get
\begin{equation*}
1=f'_n(\mu_n(0,y)) \frac{d}{dy}\mu_n(0,y) \geq \epsilonCondition s_n \frac{f_n(\mu_n(0,y))}{\mu_n(0,y)} \frac{d}{dy}\mu_n(0,y) , \qquad y>f_n(1-\deltaCondition).
\qedhere
\end{equation*}
\end{proof}

\begin{lemma}\lblemma{munDensityBounded}
For $n$ sufficiently large, the density of $\mu_n$ with respect to Lebesgue measure is at most $1/(\epsilonCondition s_n f_n(1))$ on the interval $(f_n(1),\infty)$.
\end{lemma}
\begin{proof}
From \reflemma{ExtendedImpliesWeak} it follows immediately that $f_n^{-1}(y)\leq (y/f_n(1))^{1/\epsilonCondition s_n}\leq y/f_n(1)$  for all $y> f_n(1)$ and sufficiently large $n$.
The result now follows from \reflemma{munDensityBound}.
\end{proof}
\begin{lemma}\lblemma{BoundOnContribution}
Given $\epsilon, \bar{\epsilon}>0$, there exist $n_0 \in \N$ and $K<\infty$ such that, for all $n \ge n_0$ and $t\geq 0$,
\begin{equation}
\int \e^{-\epsilon y/f_n(1)}\indicator{y\geq Kf_n(1)} \mu_n(t+dy)\leq \bar{\epsilon}/s_n.
\end{equation}
\end{lemma}
\begin{proof}
By \reflemma{munDensityBounded}, for large $n$, the density of $\mu_n$ with respect to Lebesgue measure is bounded from above by $1/(\epsilonCondition s_n f_n(1))$ on $(f_n(1),\infty)$.
Hence, for $K>1$,
\begin{equation*}
\int \e^{-\epsilon y/f_n(1)}\indicator{y\geq Kf_n(1)} \mu_n(t+dy)\le  \int_t^{\infty}  \e^{-\epsilon (y-t) /f_n(1)}\indicator{y-t\geq Kf_n(1)}  \frac{dy}{\epsilonCondition s_n f_n(1)}
= \frac{\e^{-\epsilon K}}{\epsilonCondition s_n \epsilon}.
\qedhere
\end{equation*}
\end{proof}

\begin{lemma}\lblemma{ModerateAgeContribution}
Given $K<\infty$, there exist $\epsilon_K>0$ and $n_0\in\N$ such that, for $0\leq t\leq K f_n(1)$ and $n\geq n_0$,
\begin{equation}
\int \e^{-\lambda_n(1) y} \mu_n(t+dy)\geq \epsilon_K/s_n.
\end{equation}
\end{lemma}
\begin{proof}
For any $0\leq t\leq Kf_n(1)$,
\begin{equation}
\int \e^{-\lambda_n(1)y} \mu_n(t+dy) = \int \e^{-\lambda_n(1)(y-t)} \indicator{y\geq t} d\mu_n(y) \geq \e^{-2\lambda_n(1)Kf_n(1)} \mu_n(Kf_n(1), 2Kf_n(1)).
\end{equation}
By \reflemma{lambdanAsymp}, $\lambda_n f_n(1)$ converges to a finite constant.
By \refcond{scalingfn}, $f_n(1+x/s_n)/f_n(1)\to \e^x$, and it follows that $\mu_n(Kf_n(1),2Kf_n(1))=f_n^{-1}(2Kf_n(1))-f_n^{-1}(Kf_n(1)) \sim (\log 2)/s_n$.
\end{proof}

We are now in the position to prove \reflemma{dfWeakMax}:

\begin{proof}[Proof of \reflemma{dfWeakMax}]
By monotonicity, $f_n^{-1}(T_n^{\sss (1)}+T_n^{\sss (2)}) \ge \max\set{f_n^{-1}(T_n^{\sss (1)}),f_n^{-1}(T_n^{\sss (2)})} \convp \mathcal{M}^{\sss(1)} \vee \mathcal{M}^{\sss(2)}$.
For the matching upper bound, let $T_n=T_n^{\sss(1)} \vee T_n^{\sss(2)}$ and $\mathcal{M}=\mathcal{M}^{\sss(1)} \vee \mathcal{M}^{\sss(2)}$, so that $f_n^{-1}(T_n)\convp \mathcal{M}$.
Noting that $f_n^{-1}(T_n^{\sss (1)}+T_n^{\sss (2)}) \le f_n^{-1}(2 T_n)$, it suffices to show that $f_n^{-1}(2T_n)\convp \mathcal{M}$.
But for $\delta>0$, \reflemma{ExtendedImpliesWeak} implies that $f_n(x+\delta)/f_n(x)$ tends to infinity, and is in particular larger than $2$ for $n$ sufficiently large, uniformly over $x\in[1,R]$ for any $R<\infty$.
It follows that $f_n^{-1}(2T_n)\leq (1\vee f_n^{-1}(T_n))+\delta$ with high probability, for any $\delta>0$.
Since $\mathcal{M}\geq 1$ a.s.\ and $\delta>0$ was arbitrary, this completes the proof.
\end{proof}

We conclude the section with a remark on the connection between \refthm{WH-gen} above and \refother{\refthmPartI{IPWeightForFPP}} which states that, if $s_n/\log\log n \to \infty$, then
\begin{equation}
f_n^{-1}(W_n) \convd M^{\sss(1)}\vee M^{\sss(2)}.\labelPartII{WeightSnLargeCase}
\end{equation}

\begin{remark}\lbremark{consistent}
The statements for $W_n$ in \refthm{WH-gen} and \eqrefPartII{WeightSnLargeCase} are consistent.
Indeed, if $s_n/\log\log n\to \infty$, then Lemmas~\refPartII{l:ExtendedImpliesWeak} and \refPartII{l:lambdanAsymp} imply
\begin{equation}
f_n^{-1}\Big(\frac{1}{\lambda_n} \log\big(n/s_n^3\big)\Big) \le \Big(\frac{\log(n/s_n^3)}{\lambda_n f_n(1)}\Big)^{1/\eps_0 s_n}=(1+o(1)) \exp\Big(\frac{1}{\eps_0 s_n} \log \log(n/s_n^3)\Big)=1+o(1).
\end{equation}
Hence, \reflemma{dfWeakMax} gives that \eqrefPartII{weight-res} and \eqrefPartII{WeightSnLargeCase} agree if $s_n/\log\log n\to \infty$.
\end{remark}

\subsection{Equivalence of conditions: Proof of \reflemma{AssumeStrongCond}}
\lbsubsect{StrongAssumption}

The proof of \reflemma{AssumeStrongCond} is based on the observation that if the functions $f_n,\tilde{f}_n$ agree on the interval $[0,\overline{x}]$ then, for the FPP problems with edge weights $f_n(nX_e^{\sss(K_n)})$ and $\tilde{f}_n(nX_e^{\sss(K_n)})$, respectively, the optimal paths and their corresponding edge weights are identical whenever either optimal path has weight less than $f_n(\overline{x})=\tilde{f}_n(\overline{x})$.

\begin{proof}[Proof of \reflemma{AssumeStrongCond}]
Let $\deltaCondition$ be the constant from \refcond{LowerBoundfn}, let $R>1$, and define
\begin{equation}
x_{n,R}=R \vee \big(\inf\set{x\geq 1\colon f_n(x)\geq 4\e^\gamma f_n(1)\log n}\big),
\quad
\epsilonCondition^{\sss(R)}=\frac{1}{2} \liminf_{n\to\infty} \inf_{1-\deltaCondition\leq x \le x_{n,R}} \frac{x }{s_n} \frac{d}{dx} \log f_n(x).
\end{equation}
Conditions~\refPartII{cond:LowerBoundfn} and \refPartII{cond:boundfn} imply that $\epsilonCondition^{\sss(R)}>0$ for any $R>1$, and there exists $n_0^{\sss(R)}\in\N$ such that $x \frac{d}{dx} \log f_n(x)\geq \epsilonCondition^{\sss(R)}s_n$ for $x\in[1-\deltaCondition, x_{n,R}]$ whenever $n\geq n_0^{\sss(R)}$.
For definiteness, we take $n_0^{\sss(R)}$ minimal with this property.
We may uniquely define $f_{n,R}\colon \cointerval{0,\infty}\to\cointerval{0,\infty}$ by requiring that $f_{n,R}=f_n$ if $n<n_0^{\sss(R)}$ and if $n\ge n_0^{\sss(R)}$ then (a) $f_{n,R}(x)=f_n(x)$ for all $x\leq x_{n,R}$, and (b) $\frac{x}{s_n} \frac{d}{dx} \log f_{n,R}(x)$ is constant on $\cointerval{x_{n,R},\infty}$.
By construction, the sequence $(f_{n,R})_{n}$ satisfies \refcond{boundfnExtended} for any fixed $R>1$.
Furthermore, given any $x>0$, $R>1$ implies that $x^{1/s_n}\leq R\leq x_{n,R}$ for $n$ sufficiently large, and it follows that $f_{n,R}$ satisfies \refcond{scalingfn}.
Since $x_{n,R}\geq R>1$ it follows that \refcond{LowerBoundfn} holds for $(f_{n,R})_{n}$, too.

Let $\mu_{n,R}$ and $\lambda_{n,R}$ denote the analogues of $\mu_n$ and $\lambda_n$ when $f_n$ is replaced by $f_{n,R}$, and let $\lambda_{n,R}=\lambda_{n,R}(1)$ and $\phi_{n,R}=\lambda'_{n,R}(1)/\lambda_{n,R}(1)$ denote the corresponding parameters (see \eqrefPartII{munCharacterization} and \eqrefPartII{lambdaaDefn}--\eqrefPartII{phinDefinitionShort}).
Let $W_{n,R}, H_{n,R}$ denote the weight and hopcount, respectively, associated to the FPP problem on $K_n$ with edge weights $f_{n,R}(n X_e^{\sss(K_n)})$.
Abbreviate $w_{n,R}=W_{n,R}-\log(n/s_n^3)/\lambda_{n,R}$, $h_{n,R}=(H_{n,R}-\phi_{n,R}\log(n/s_n^3))/\sqrt{s_n^2\log(n/s_n^3)}$.

By assumption, \refthm{WH-gen} holds assuming Conditions~\refPartII{cond:scalingfn}, \refPartII{cond:LowerBoundfn} and \refPartII{cond:boundfnExtended}.
Therefore, it applies to $f_{n,R}$.
Using \refthm{WH-gen}, we conclude that for any $k\in\N$, we may find $n_0^{\sss(R,k)}\in\N$ such that $n_0^{\sss(R,k)}\geq n_0^{\sss(R)}$ and
\begin{subequations}
\begin{gather}
\labelPartII{n0RkCondition1}
\sup_{x,y\in\R} \abs{ \P\left( f_{n,R}^{-1}\left( w_{n,R} \right) \leq x, h_{n,R} \leq y \right) - \P(M^{\sss(1)}\vee M^{\sss(2)}\leq x)\P(Z\leq y) } \leq \frac{1}{k},
\\
\labelPartII{n0RkCondition2}
\abs{\lambda_{n,R}f_{n,R}(1)-\e^{-\gamma}}\leq \frac{1}{k},
\quad
\abs{\frac{\phi_{n,R}}{s_n}-1}\leq \frac{1}{k},
\\
\labelPartII{n0RkCondition3}
f_{n,R}(R-1) + \frac{\log(n/s_n^3)}{\lambda_{n,R}} \leq f_{n,R}(R) \vee 4\e^\gamma f_n(1)\log n
\end{gather}
\end{subequations}
whenever $n\geq n_0^{\sss(R,k)}$, and for definiteness we take $n_0^{\sss(R,k)}$ minimal with these properties.
Indeed, using the continuity of $M^{\sss(1)}\vee M^{\sss(2)}$ and $Z$, the uniform convergence in \eqrefPartII{n0RkCondition1} follows from the pointwise convergence at a finite grid $((x_i,y_i))_i$ depending on $k$ and monotonicity of the distribution functions.
For \eqrefPartII{n0RkCondition3}, use the inequality $a+b\leq 2(a\vee b)$, \reflemma{lambdanAsymp},  and note that $2f_{n,R}(R-1)\leq f_{n,R}(R)$ for $n$ sufficiently large by \reflemma{ExtendedImpliesWeak}.
Set
\begin{equation}
R_n=\big(2 \vee \max\set{k\in\N\colon n\geq n_0^{\sss(k,k)}}\big) \wedge n,
\quad\text{and}\quad
\lambda_n=\lambda_{n,R_n},
\quad
\phi_n=\phi_{n,R_n}
.
\end{equation}
Since each $n_0^{\sss(k,k)}$ is finite for each $k\in \N$, it follows that $R_n\to\infty$.
Moreover, as soon as $n\geq n_0^{\sss(2,2)}$, we have $n\geq n_0^{\sss(R_n,R_n)}$, so that \eqrefPartII{n0RkCondition1}--\eqrefPartII{n0RkCondition3} hold with $(R,k)=(R_n,R_n)$.
By construction, $f_{n,R}(1)=f_n(1)$, and we conclude in particular that $\phi_n/s_n\to 1$ and $\lambda_n f_n(1)\to \e^{-\gamma}$.

Given two functions $f_n$ and $\tilde{f}_n$, we can couple the corresponding FPP problems by choosing edge weights $f_n(nX_e^{\sss(K_n)})$ and $\tilde{f}_n(nX_e^{\sss(K_n)})$, respectively.
Let $\bar{x}>0$. On the event $\set{W_n \le f_n(\bar{x})}$, the optimal path $\pi_{1,2}$ uses only edges of weight at most $f_n(\bar{x})$.
If $f_n$ and $\tilde{f}_n$ agree on the interval $[0,\bar{x}]$, then the edges along that path have the same weights in the two FPP problems and we deduce that $W_n=\tilde{W}_n$ and $H_n=\tilde{H}_n$, where $\tilde{W}_n$ and $\tilde{H}_n$ are the weight and the hopcount of the optimal path in the problem corresponding to $\tilde{f}_n$.

Consequently, on the event $\set{W_{n,R_n}\le f_{n,R_n}(x_{n,R_n})}$,
$W_n=W_{n,R_n}$ and $H_n=H_{n,R_n}$. By \eqrefPartII{n0RkCondition1}, it remains to show that
this event occurs {\whpdot}
Since $R_n\to\infty$, we conclude from \eqrefPartII{n0RkCondition1} that $W_{n,R_n}\leq f_{n,R_n}(R_n-1)+\log(n/s_n^3)/\lambda_{n,R_n}$ {\whpdot}
But from the definition of $x_{n,R_n}$ and $f_{n,R_n}$ it follows that $f_{n,R_n}(x_{n,R_n})\geq f_{n,R_n}(R_n)\vee 4\e^\gamma f_{n,R_n}(1)\log n$, so \eqrefPartII{n0RkCondition3} completes the proof.
\end{proof}

\section{Coupling \texorpdfstring{$K_n$}{the complete graph} and the PWIT}\lbsect{Coupling}

In \refthm{CouplingFPP}, we indicated that two random processes, the first passage exploration processes $\cS$ and $\cluster$ on $K_n$ and $\tree$, respectively, could be coupled.
In this section we explain how this coupling arises as a special case of a general family of couplings between $K_n$, understood as a random edge-weighted graph with i.i.d.\ exponential edge weights, and the PWIT.

\subsection{Exploration processes and the definition of the coupling}
\lbsubsect{DefExploration}

As in \refsubsect{CouplingBP}, we define $M_{\emptyset_j}=j$, for $j=1,2$, and to each $v\in \twoPWITs\setminus\set{\emptyset_1,\emptyset_2}$, we associate a mark $M_v$ chosen uniformly and independently from $[n]$.
We next define what an exploration process is:

\begin{defn}[Exploration process]\lbdefn{Explore}
Let $\F_0$ be a $\sigma$-field containing all null sets, and let $(\tree,X)$ be independent of $\F_0$.
We call a sequence $\explore=(\explore_k)_{k\in \N_0}$ of subsets of $\tree$ an \emph{exploration process} if, with probability 1, $\explore_0=\set{\emptyset_1,\emptyset_2}$ and, for every $k\in\N$, either $\explore_k=\explore_{k-1}$ or else $\explore_k$ is formed by adjoining to $\explore_{k-1}$ a previously unexplored child $v_k\in\boundary\explore_{k-1}$, where the choice of $v_k$ depends only on the weights $X_w$ and marks $M_w$ for vertices $w \in \explore_{k-1}\union\boundary\explore_{k-1}$ and on events in $\F_0$.
\end{defn}
Examples for exploration processes are given by FPP and IP on $\tree$.
For FPP, as defined in \refdefn{FPPonPWITdef}, it is necessary to convert to discrete time by observing the branching process at those moments when a new vertex is added.
The standard IP on $\tree$ is defined as follows. Set $\IP(0)=\set{\emptyset_1,\emptyset_2}$. For $k \in \N$, form $\IP(k)$ inductively by adjoining to $\IP(k-1)$ the boundary vertex $v\in \boundary \IP(k-1)$ of minimal weight.
However, an exploration process is also obtained when we specify at each step (in any suitably measurable way) whether to perform an invasion step in $\tree^{\sss(1)}$ or $\tree^{\sss(2)}$.

For $k \in \N$, let $\F_k$ be the $\sigma$-field generated by $\F_0$ together with the weights $X_w$ and marks $M_w$ for vertices $w \in \explore_{k-1}\union\boundary\explore_{k-1}$.
Note that the requirement on the choice of $v_k$ in \refdefn{Explore} can be expressed as the requirement that $\explore$ is $(\F_k)_k$-adapted.

For $v\in \tree$, define the exploration time of $v$ by
	\begin{equation}
	N_v=\inf\set{k \in \N_0\colon v\in \explore_k} .
	\end{equation}

\begin{defn}[Thinning]\lbdefn{Thinning}
The roots $\emptyset_1, \emptyset_2 \in \tree$ are unthinned.
The vertex $v\in\tree\setminus\set{\emptyset_1,\emptyset_2}$ is \emph{thinned} if it has an ancestor $v_0=\ancestor{k}{v}$ (possibly $v$ itself) such that $M_{v_0}=M_w$ for some unthinned vertex $w$ with $N_w<N_{v_0}$.
Write $\thinnedExplore_k$ for the subgraph of $\explore_k$ consisting of unthinned vertices.
\end{defn}

We define the stopping times
	\begin{equation}
	N(i)=\inf\set{k \in \N_0\colon M_v=i\text{ for some }v\in\thinnedExplore_k}.
	\end{equation}
at which $i\in[n]$ first appears as a mark in the unthinned exploration process.
Note that, on the event $\set{N(i)<\infty}$, $\thinnedExplore_k$ contains a \emph{unique} vertex in $\tree$ whose mark is $i$, for any $k\geq N(i)$; call that vertex $V(i)$.
On this event, we define
	\begin{equation}\labelPartII{XijDefinition}
	X(i,i')=\min\set{X_w \colon M_w=i', \parent{w}=V(i)}.
	\end{equation}

We define, for an edge $\set{i,i'} \in E(K_n)$,
\begin{equation}\labelPartII{EdgeWeightCoupling}
X_{\set{i,i'}}^{\sss(K_n)}=
\begin{cases}
\tfrac{1}{n} X(i,i') & \text{if } N(i)<N(i'),\\
\tfrac{1}{n} X(i',i) & \text{if } N(i')<N(i),\\
E_{\set{i,i'}} & \text{if } N(i)=N(i')=\infty \text{ or } N(i)=N(i')=0,
\end{cases}
\end{equation}
where $(E_e)_{e\in E(K_n)}$ are exponential variables with mean 1, independent of each other and of $(X_v)_{v}$.
\begin{theorem}\lbthm{CouplingExpl}
If $\explore$ is an exploration process on the union $\tree$ of two PWITs, then the edge weights $X_e^{\sss(K_n)}$ defined in \eqrefPartII{EdgeWeightCoupling} are exponential with mean $1$, independently for each $e \in E(K_n)$.
\end{theorem}

The idea underlying \refthm{CouplingExpl} is that each variable $\tfrac{1}{n}X(i,i')$ is exponentially distributed conditional on the past up to the moment $N(i)$ when it may be used to set the value of $X_{\set{i,i'}}^{\sss(K_n)}$.
\refthm{CouplingExpl} restates \refother{\refthmPartI{CouplingExpl}} and is proved in that paper.

\subsection{Minimal-rule exploration processes}
\lbsubsect{MinimalRule}

An important class of exploration processes, which includes both FPP and IP, are those exploration processes determined by a minimal rule in the following sense:

\begin{defn}\lbdefn{MinimalRule}
A \emph{minimal rule} for an exploration process $\explore$ on $\tree$ is an $(\F_k)_k$-adapted sequence $(\makebox{$S_k,\prec_k$})_{k=1}^\infty$, where $S_k\subset\boundary\explore_{k-1}$ is a (possibly empty) subset of the boundary vertices of $\explore_{k-1}$ and $\prec_k$ is a strict total ordering of the elements of $S_k$ (if any) such that the implication
\begin{equation}\labelPartII{MinimalImplication}
w\in S_k, \parent{v}=\parent{w}, M_v=M_w, X_v<X_w \quad\implies\quad v\in S_k, v \prec_k w
\end{equation}
holds.
An exploration process is \emph{determined by the minimal rule} $(S_k,\prec_k)_{k=1}^\infty$ if $\explore_k=\explore_{k-1}$ whenever $S_k=\emptyset$ and otherwise $\explore_k$ is formed by adjoining to $\explore_{k-1}$ the unique vertex $v_k\in S_k$ that is minimal with respect to $\prec_k$.
\end{defn}
In words, in every step $k$ there is a set of boundary vertices $S_k$ from which we can select for the next exploration step.
The content of \eqrefPartII{MinimalImplication} is that, whenever a vertex $w\in S_k$ is available for selection, then all siblings of $w$ with the same mark but smaller weight are also available for selection and are preferred over $w$.

For FPP without freezing on $\tree$ with edge weights $f_n(X_v)$, we take $v \prec_k w$ if and only if $T_v < T_w$ (recall \eqrefPartII{TvDefinition}) and take $S_k=\boundary\explore_{k-1}$.
For IP on $\tree$, we have $v \prec_k w$ if and only if $X_v<X_w$; the choice of subset $S_k$ can be used to enforce, for instance, whether the $k^\th$ step is taken in $\tree^{\sss(1)}$ or $\tree^{\sss(2)}$.

Recall the subtree $\thinnedExplore_k$ of unthinned vertices from \refdefn{Thinning} and the subgraph $\pi_M(\thinnedExplore_k)$ from \refdefn{InducedGraph}.
That is, $\pi_M(\thinnedExplore_k)$ is the union of two trees with roots $1$ and $2$, respectively, and for $v\in \thinnedExplore_k \setminus\set{\emptyset_1,\emptyset_2}$, $\pi_M(\thinnedExplore_k)$ contains vertices $M_v$ and $M_{\parent{v}}$ and the edge $\set{M_v, M_{\parent{v}}}$.

For any $i\in [n]$ for which $N(i)<\infty$, recall that $V(i)$ is the unique vertex of $\thinnedExplore_k$ ($k\geq N(i)$) for which $M_{V(i)}=i$.
Define $V(i,i')$ to be the first child of $V(i)$ with mark $i'$.

Recalling \eqrefPartII{XijDefinition}, an equivalent characterization of $V(i,i')$ is
\begin{equation}\labelPartII{XijVij}
X(i,i')=X_{V(i,i')}.
\end{equation}
The following lemma shows that, for an exploration process determined by a minimal rule, unthinned vertices must have the form $V(i,i')$:

\begin{lemma}\lblemma{NextNotThinned}
Suppose $\explore$ is an exploration process determined by a minimal rule $(S_k,\prec_k)_{k=1}^\infty$ and $k\in \N$ is such that $\thinnedExplore_k\neq \thinnedExplore_{k-1}$. Let $i_k=M_{\parent{v_k}}$ and $i_k'=M_{v_k}$. Then $v_k=V(i_k,i_k')$.
\end{lemma}

See the proof of \refother{\reflemmaPartI{NextNotThinned}}.

If $\explore$ is an exploration process determined by a minimal rule, then we define
\begin{equation}\labelPartII{SelectableInKn}
S_k^{\sss (K_n)}=\set{\set{i,i'} \in E(K_n) \colon i \in \pi_M(\thinnedExplore_{k-1}), i' \notin \pi_M(\thinnedExplore_{k-1}), V(i,i') \in S_k}
\end{equation}
and
\begin{equation}\labelPartII{OrderInKn}
e_1 \;\widetilde{\prec}_k\; e_2 \quad\iff\quad V(i_1,i_1') \prec_k V(i_2,i_2'), \qquad e_1,e_2\in S_k^{\sss(K_n)},
\end{equation}
where $e_j=\set{i_j,i'_j}$ and $i_j \in \pi_M(\thinnedExplore_{k-1}), i'_j \notin \pi_M(\thinnedExplore_{k-1})$ as in \eqrefPartII{SelectableInKn}.
\begin{prop}[Thinned minimal rule]\lbprop{MinimalRuleThinning}
Suppose $\explore$ is an exploration process determined by a minimal rule $(S_k,\prec_k)_{k=1}^\infty$.
Then, under the edge-weight coupling \eqrefPartI{EdgeWeightCoupling}, the edge weights of $\pi_M(\thinnedExplore_k)$ are determined by
\begin{equation}\labelPartII{InternalEdgeWeights}
X_{\shortset{M_v,M_{\parent{v}}}}^{\sss(K_n)} = \tfrac{1}{n} X_v
\quad
\text{for any }v\in \union_{k=1}^\infty \thinnedExplore_k\setminus\set{\emptyset_1,\emptyset_2}
\end{equation}
and generally
\begin{equation}\labelPartII{BoundaryEdgeWeights}
X_{\set{i,i'}}^{\sss(K_n)} = \tfrac{1}{n} X_{V(i,i')}
\quad
\text{whenever}
\quad
i\in \pi_M(\thinnedExplore_{k-1}), i'\notin \pi_M(\thinnedExplore_{k-1})\text{ for some }k\in\N.
\end{equation}
Moreover, for any $k\in\N$ for which $\thinnedExplore_k\neq \thinnedExplore_{k-1}$, $\pi_M(\thinnedExplore_k)$ is formed by adjoining to $\pi_M(\thinnedExplore_{k-1})$ the unique edge $e_k\in S_k^{\sss (K_n)}$ that is minimal with respect to $\widetilde{\prec}_k$.
\end{prop}

\refprop{MinimalRuleThinning} asserts that the subgraph $\pi_M(\thinnedExplore_k)$ of $K_n$, equipped with the edge weights $(X_e^{\sss(K_n)})_{e\in E(\pi_M(\thinnedExplore_k))}$, is isomorphic as an edge-weighted graph to the subgraph $\thinnedExplore_k$ of $\tree$, equipped with the edge weights $(\tfrac{1}{n} X_v)_{v\in\thinnedExplore_k\setminus\set{\emptyset_1,\emptyset_2}}$.
Furthermore, the subgraphs $\pi_M(\thinnedExplore_k)$ can be grown by an inductive rule.
Thus the induced subgraphs $(\pi_M(\thinnedExplore_k))_{k=0}^\infty$ themselves form a minimal-rule exploration process on $K_n$, with a minimal rule derived from that of $\explore$, with the caveat that $\widetilde{\prec}_k$ may depend on edge weights from $\explore_{k-1}\setminus\thinnedExplore_{k-1}$ as well as from $\pi_M(\thinnedExplore_{k-1})$.

See the proof of \refother{\refpropPartI{MinimalRuleThinning}}.

\subsection{FPP and the two smallest-weight trees: Proof of \refthm{WnHnFromCollision}}
\lbsubsect{ProofWnHnCollision}

In this section, we discuss the relationship between $\cS$ and FPP distances, and we prove \reflemma{WnAsMinimum} and \refthm{WnHnFromCollision}.

\begin{proof}[Proof of \reflemma{WnAsMinimum}]
Define
\begin{equation}
\bar{\cS}_u=\bar{\cS}_u^{\sss(1)}\union \bar{\cS}_u^{\sss(2)}, \qquad \bar{\cS}_u^{\sss(j)}=\SWT^{\sss(j)}_{R_j(u)}.
\end{equation}
To describe the discrete time evolution of $\bar{\cS}=(\bar{\cS}_u)_{u \ge 0}$, denote by $\bar{\tau}_{k-1}$ the time where the $(k-1)^\st$ vertex (not including vertices $1$ and $2$) was added to $\bar{\cS}$.
At time $u=0$, $\bar{\cS}_0$ is equal to $\cS_0$, and contains vertex $1$ and $2$ and no edges.
Having constructed the process until time $\bar{\tau}_{k-1}$, at time
\begin{align}
\bar{\tau}_k
&=
\min_{j\in\set{1,2}}\min_{e\in\boundary\bar{\cS}_{\bar{\tau}_{k-1}}^{\sss(j)}} R_j^{-1}\left( d_{K_n,Y^{(K_n)}}(j,\underline{e})+Y_e^{\sss(K_n)} \right)
\labelPartII{TwoSWTConstructiond}
\end{align}
adjoin the boundary edge $e_k$ to $\bar{\cS}_{\bar{\tau}_{k-1}}^{\sss(j_k)}$, where $(j_k,e_k)$ is the minimizer in \eqrefPartII{TwoSWTConstructiond}.
(As in \eqrefPartII{cSConstruction}, $e$ is an edge from $\underline{e}\in \bar{\cS}_{\bar{\tau}_{k-1}}^{\sss(j)}$ to $\overline{e}\notin \bar{\cS}_{\bar{\tau}_{k-1}}^{\sss(j)}$.)
Note that the arrival times in $\bar{\cS}^{\sss(j)}$ are given by
\begin{equation}\labelPartII{TbarcS}
T^{\sss\bar{\cS}^{(j)}}(\underline{e})=R_j^{-1}(d_{K_n,Y^{(K_n)}}(j,\underline{e}))
\end{equation}
by construction, so that we may rewrite \eqrefPartII{TwoSWTConstructiond} as
\begin{equation}\labelPartII{TwoSWTConstructionRinverseR}
\bar{\tau}_k
=
\min_{j\in\set{1,2}}\min_{e\in\boundary\bar{\cS}_{\bar{\tau}_{k-1}}^{\sss(j)}} R_j^{-1}\left( R_j\bigl( T^{\sss\bar{\cS}^{(j)}}(\underline{e}) \bigr)+Y_e^{\sss(K_n)} \right).
\end{equation}
Comparing \eqrefPartII{TwoSWTConstructionRinverseR} with \eqrefPartII{cSConstruction}, $\cS$ and $\bar{\cS}$ will evolve in the same way until the time
\begin{equation}
\bar{\tau}=\min\set{t\colon \bar{\cS}^{\sss(1)}_t\intersect\bar{\cS}^{\sss(2)}_t\neq \emptyset}
\end{equation}
when $\bar{\cS}$ first accepts an edge between $\bar{\cS}^{\sss(1)}$ and $\bar{\cS}^{\sss(2)}$.
In particular, the minimization problem in \eqrefPartII{TwoSWTConstructionRinverseR} will be the same as in \eqrefPartII{EnlargedMinimum}, and the minimizer will be a.s.\ unique, as long as $\bar{\tau}_k\leq\bar{\tau}$.
Therefore we can choose $J,J'$ with $\set{J,J'}=\set{1,2}$ and $I\in\bar{\cS}^{\sss(J)}_{\bar{\tau}-}, I'\in\bar{\cS}^{\sss(J')}_{\bar{\tau}-}$ such that, at time $\bar{\tau}$, the edge between $I$ and $I'$ is adjoined to $\bar{\cS}^{\sss(J)}$.
(In other words, $j=J$ and $e=\set{I,I'}$ is the minimizer in \eqrefPartII{TwoSWTConstructionRinverseR}.)

Because the minimizer in \eqrefPartII{TwoSWTConstructionRinverseR} is unique, no vertex is added to $\cS$ at time $\bar{\tau}$.
In particular, $T^{\sss\cS}(i)<\bar{\tau}$ for every $i\in\cS_{\bar{\tau}}$.
Since $\cS^{\sss(j)}_t$ and $\bar{\cS}^{\sss(j)}_t$ agree for $t<\bar{\tau}$, the arrival times before $\bar{\tau}$ must coincide.
Recalling \eqrefPartII{TbarcS},
\begin{equation}\labelPartII{DistanceFromTcS}
\begin{aligned}
T^{\sss\cS}(i)&=T^{\sss\bar{\cS}^{(j)}}(i) = R_j^{-1}(d_{K_n,Y^{(K_n)}}(j,i)) \\
d_{K_n,Y^{(K_n)}}(j,i) &= R_j(T^{\sss\cS}(i))
\end{aligned}
\qquad\text{if }i\in\cS^{\sss(j)}_{\bar{\tau}}.
\end{equation}
In addition, $\cS^{\sss(J')}_{\bar{\tau}}$ and $\bar{\cS}^{\sss(J')}_{\bar{\tau}}$ have the same vertex set while $\cS^{\sss(J)}_{\bar{\tau}}$ and $\bar{\cS}^{\sss(J)}_{\bar{\tau}}$ differ only by the vertex $I'$.
It follows that
\begin{equation}\labelPartII{DistanceAftertau}
d_{K_n,Y^{(K_n)}}(j,i) \geq R_j(\tau) \qquad\text{if }i\notin\cS^{\sss(j)}_{\bar{\tau}}.
\end{equation}

Consider the optimal path from vertex $J$ to vertex $I'$.
Since $I'$ is adjoined to $\bar{\cS}^{\sss(J)}$ at time $\bar{\tau}$, it follows from \eqrefPartII{TbarcS} that $d_{K_n,Y^{(K_n)}}(J,I')=R_J(\bar{\tau})$.
Moreover, since $I$ is the parent of $I'$ in $\bar{\cS}^{\sss(J)}_{\bar{\tau}}=\SWT^{\sss(J)}_{R_J(\bar{\tau})}$, we have $d_{K_n,Y^{(K_n)}}(J,I')=d_{K_n,Y^{(K_n)}}(J,I)+Y^{\sss(K_n)}_{\set{I,I'}}$.
Applying \eqrefPartII{DistanceFromTcS} to the path from $J$ to $I$ to $I'$ to $J'$,
\begin{align}
W_n&\leq d_{K_n,Y^{(K_n)}}(J,I)+Y^{\sss(K_n)}_{\set{I,I'}} + d_{K_n,Y^{(K_n)}}(J',I')
\notag\\&
= R_J(T^{\sss\cS}(I))+Y^{\sss(K_n)}_{\set{I,I'}}+R_{J'}(T^{\sss\cS}(I')) =R_J(\tau)+R_{J'}(T^{\sss\cS}(I')) \leq R_J(\bar{\tau})+R_{J'}(\bar{\tau})
.
\labelPartII{WntauBound}
\end{align}
In particular, both sides of \eqrefPartII{WnMin} are bounded above by $R_1(\bar{\tau})+R_2(\bar{\tau})$.

The bound \eqrefPartII{WntauBound} will allow us to exclude vertices that arrive after time $\bar{\tau}$.
To this end, we will show that (a) if a path $\pi$ from vertex 1 to vertex 2 contains a vertex not belonging to $\cS_{\bar{\tau}}$, then the weight of $\pi$ is greater than $R_1(\bar{\tau})+R_2(\bar{\tau})$; and (b) if $i_1\in\cS^{\sss(1)}\setminus\cS^{\sss(1)}_{\bar{\tau}}$ or $i_2\in\cS^{\sss(2)}\setminus\cS^{\sss(2)}_{\bar{\tau}}$, then the term $R_1(T^{\sss\cS}(i_1)) + Y^{\sss(K_n)}_{\set{i_1,i_2}} + R_2(T^{\sss\cS}(i_2))$ in the minimum \eqrefPartII{WnMin} is greater than $R_1(\bar{\tau})+R_2(\bar{\tau})$.

For (a), suppose $\pi$ contains a vertex $i\notin\cS_{\bar{\tau}}$.
Since the vertex sets of $\cS_{\bar{\tau}}$ and $\bar{\cS}_{\bar{\tau}}$ coincide, it follows that $i\notin\bar{\cS}_{\bar{\tau}}$, and by right continuity $i\notin\bar{\cS}_t$ for some $t>\bar{\tau}$.
Since $R_1+R_2$ is strictly increasing, \eqrefPartII{TbarcS} shows that the weight of $\pi$ is at least
\begin{equation}
d_{K_n,Y^{(K_n)}}(1,i)+d_{K_n,Y^{(K_n)}}(2,i) \geq R_1(t)+R_2(t)>R_1(\bar{\tau})+R_2(\bar{\tau}).
\end{equation}

For (b), suppose for specificity that $i_1\in\cS^{\sss(1)}\setminus\cS^{\sss(1)}_{\bar{\tau}}$.
If in addition $i_2\in\cS^{\sss(2)}\setminus\cS^{\sss(2)}_{\bar{\tau}}$ then $T^{\sss\cS}(i_1),T^{\sss\cS}(i_2)>\bar{\tau}$ and the strict monotonicity of $R_1+R_2$ gives the desired result.
We may therefore suppose $i_2\in\cS^{\sss(2)}_{\bar{\tau}}$.
Since $\cS^{\sss(1)}$ and $\cS^{\sss(2)}$ are disjoint, we must have $i_1\notin\cS^{\sss(2)}_{\bar{\tau}}$, so that $d_{K_n,Y^{(K_n)}}(2,i_1)\geq R_2(\bar{\tau})$.
In particular, by considering the optimal path from 2 to $i_2$ together with the edge from $i_2$ to $i_1$, we conclude that $d_{K_n,Y^{(K_n)}}(2,i_2)+Y^{\sss(K_n)}_{\set{i_1,i_2}}\geq R_2(\bar{\tau})$.
By the assumption $i_2\in\cS^{\sss(2)}_{\bar{\tau}}$, we may rewrite this as $R_2(T^{\sss\cS}(i_2))+Y^{\sss(K_n)}_{\set{i_1,i_2}}\geq R_2(\bar{\tau})$.
Together with $R_1(T^{\sss\cS}(i_1))>R_1(\bar{\tau})$, this proves (b).

To complete the proof, consider a path $\pi$ from vertex 1 to vertex 2.
By statement (a) and \eqrefPartII{WntauBound}, $\pi$ must contain only vertices from $\cS_{\bar{\tau}}$ if it is to be optimal.
Since $1\in\cS^{\sss(1)}_{\bar{\tau}}$ but $2\in\cS^{\sss(2)}_{\bar{\tau}}$, it follows that $\pi$ must contain an edge between some pair of vertices $i_1\in\cS^{\sss(1)}_{\bar{\tau}}$ and $i_2\in\cS^{\sss(2)}_{\bar{\tau}}$.
The minimum possible weight of such a path is $d_{K_n,Y^{(K_n)}}(1,i_1)+Y^{\sss(K_n)}_{\set{i_1,i_2}}+d_{K_n,Y^{(K_n)}}(2,i_2)$, which agrees with the corresponding term in \eqrefPartII{WnMin} by \eqrefPartII{DistanceFromTcS}.
Therefore \eqrefPartII{WnMin} is verified if the minimum is taken only over $i_1\in\cS^{\sss(1)}_{\bar{\tau}}, i_2\in\cS^{\sss(2)}_{\bar{\tau}}$.
But statement (b) and \eqrefPartII{WntauBound} shows that the remaining terms must be strictly greater.
\end{proof}

\begin{proof}[Proof of \refthm{WnHnFromCollision}]
Since $R_1+R_2$ is strictly increasing, the relation $W_n=R_1(T_\coll)+R_2(T_\coll)$ is a reformulation of \refdefn{CollisionTimeDef}.

Recall the time $\bar{\tau}$ from the proof of \reflemma{WnAsMinimum}.
We showed there that the minimizer of \eqrefPartII{WnMin} must come from vertices $i_1\in\cS^{\sss(1)}_{\bar{\tau}}$, $i_2\in\cS^{\sss(2)}_{\bar{\tau}}$.
In particular, by \eqrefPartII{DistanceFromTcS},
\begin{equation}\labelPartII{WnMinExplicit}
W_n=R_1(T^{\sss\cS}(I_1))+Y^{\sss(K_n)}_{\set{I_1,I_2}}+R_2(T^{\sss\cS}(I_2))
\end{equation}
expresses $W_n$ as the weight of the path formed as the union of the optimal path from 1 to $I_1$; the edge from $I_1$ to $I_2$; and the optimal path from $I_2$ to 2.
In particular, $\pi_{1,2}$ is the same as this path.
Since $T^{\sss\cS}(I_j)<\tau$ and $\cS^{\sss(j)}_t=\SWT^{\sss(j)}_{R_j(t)}$ for $t<\bar{\tau}$, it follows that the optimal paths from $j$ to $I_j$ coincide with the unique paths in $\cS^{\sss(j)}_{\bar{\tau}}$ between these vertices.
The relation $H_n=H(I_1)+H(I_2)+1$ follows by taking lengths of these subpaths.

It remains to show that $T^{\sss\cS}(I_j)<T_\coll$.
Define $t_1=R_1^{-1}(R_1(T^{\sss\cS}(I_1))+Y^{\sss(K_n)}_{\set{I_1,I_2}})$.
Recalling \eqrefPartII{cSConstruction}, we see that $t_1$ is the time at which the edge from $I_1$ to $I_2$ is adjoined to $\cS^{\sss(1)}$, provided that $I_2$ has not already been added to $\cS$ at some point strictly before time $t_1$.
By construction, $I_2$ is added to $\cS^{\sss(2)}$, not $\cS^{\sss(1)}$, so it must be that $T^{\sss\cS}(I_2)<t_1$.
(Equality is not possible because of our assumption that the minimizers of \eqrefPartII{cSConstruction} are unique.)
Aiming for a contradiction, suppose that $T^{\sss\cS}(I_2)\geq T_\coll$.
Comparing the relation $W_n=R_1(T_\coll)+R_2(T_\coll)$ to \eqrefPartII{WnMinExplicit} gives $R_1(T^{\sss\cS}(I_2))+Y^{\sss(K_n)}_{\set{I_1,I_2}}\leq R_1(T_\coll)$, so that $t_1\leq T_\coll$.
This is a contradiction since $t_1>T^{\sss\cS}(I_2)\geq T_\coll$.
Similarly we must have $T^{\sss\cS}(I_1)< T_\coll$.
This shows that the unique paths in $\cS^{\sss(j)}_\tau$ from $j$ to $I_j$ are actually paths in $\cS^{\sss(j)}_{T_\coll}$, as claimed.
\end{proof}

\subsection{\texorpdfstring{$\cluster$ and $\cS$}{Variable speed FPP processes} as exploration processes: Proof of \refthm{CouplingFPP}}
\lbsubsect{ProofCouplingFPP}

Before proving \refthm{CouplingFPP}, we show that the discrete-time analogue of $\cluster$ is an exploration process determined by a minimal rule:

\begin{lemma}\lblemma{clusterAsExploration}
Let $v_k$ denote the $k^\th$ vertex added to $\cluster$, excluding the root vertices $\emptyset_1,\emptyset_2$, and set $\explore_k=\cluster_{T^\cluster_{v_k}}$ for $k\geq 1$, $\explore_0=\cluster_0=\set{\emptyset_1,\emptyset_2}$.
Then $\explore$ is an exploration process determined by a minimal rule.
\end{lemma}

\begin{proof}
Consider the $k^\th$ step and define
\begin{equation}\labelPartII{taunextDefn}
\tau_\rmnext^{\sss(j)}(k)=\min_{v\in\boundary\explore_{k-1}\intersect\tree^{\sss(j)}} T_v,
\end{equation}
the next birth time of a vertex in $\tree^{\sss(j)}$ in the absence of freezing, and let $v_k^{\sss(j)}$ denote the a.s.\ unique vertex attaining the minimum in \eqrefPartII{taunextDefn}.
Recalling the definition of the filtration $\F_k$, we see that $\tau_\rmnext^{\sss(j)}(k)$ and $v_k^{\sss(j)}$ are $\F_k$-measurable.

The variable $T_\fr^{\sss(j)}$ is not $\F_k$-measurable.
However, the event $\set{T_\fr^{\sss(j)}<\tau_\rmnext^{\sss(j)}(k)}$ \emph{is} $\F_k$-measurable.
To see this, define
\begin{equation}\labelPartII{FreezingTimeApproximation}
\tau_\fr^{\sss(j)}(k) = \inf\bigg\{t\geq 0\colon \sum_{v\in\explore_{k-1}\intersect \tree^{\sss(j)}} \indicator{T_v\leq t} \int_{t-T_v}^\infty \e^{-\lambda_n(1)(y-(t-T_v))} d\mu_n(y) \geq s_n\bigg\},
\end{equation}
so that $\tau_\fr^{\sss(j)}(k)$ is $\F_k$-measurable, and abbreviate $\tau_\unfr(k)=\tau_\fr^{\sss(1)}(k)\vee \tau_\fr^{\sss(2)}(k)$.
We will use $\tau_\fr^{\sss(j)}(k)$ as an approximation to $T_\fr^{\sss(j)}$ based on the information available in $\F_k$.
By analogy with \eqrefPartII{OnOff} and \eqrefPartII{OnOffInverse}, we also define
\begin{equation}\labelPartII{OnOffk}
\begin{aligned}
R_{j,k}(t)
&=
\left( t\wedge \tau_\fr^{\sss(j)}(k) \right) + \left((t-\tau_\unfr(k))\vee 0 \right),
\\
R_{j,k}^{-1}(t)
&=
\begin{cases}
t&\text{if }t\leq \tau_\fr^{\sss(j)}(k),\\
\tau_\unfr(k) -\tau_\fr^{\sss(j)}(k)+t &\text{if }t>\tau_\fr^{\sss(j)}(k).
\end{cases}
\end{aligned}
\end{equation}
We note the following:

\begin{enumerate}[(i)]
\item\lbitem{Previousexplore}
$\explore_{k-1}\intersect\tree^{\sss(j)}=\BP^{\sss(j)}_{t'}$ and $\tau_\rmnext^{\sss(j)}(k)=\min\set{T_v\colon v\in\boundary\BP^{\sss(j)}_{t'}}$, where $t'=T_{v_{k-1}}$.

\item\lbitem{SumsAgreeTotaunext}
The sum in \eqrefPartII{FreezingTimeApproximation} agrees with the sum in the definition \eqrefPartII{TfrDefn} of $T_\fr^{\sss(j)}$ whenever $t<\tau_\rmnext^{\sss(j)}(k)$.

\item\lbitem{TfrLessIfftaufrLess}
$T_\fr^{\sss(j)}<\tau_\rmnext^{\sss(j)}(k)$ if and only if $\tau_\fr^{\sss(j)}(k)<\tau_\rmnext^{\sss(j)}(k)$, and in this case $T_\fr^{\sss(j)}=\tau_\fr^{\sss(j)}(k)$.

\item\lbitem{RjRjkAgree}
For $j\in \set{1,2}$, let $I_k^{\sss(j)} \subseteq [0,\tau_\rmnext^{\sss(j)}(k)]$ be nonempty. The minimizers of $R_{j,k}^{-1}(t)$ and $R_j^{-1}(t)$ over all pairs $(t,j)$ with $t \in I_k^{\sss(j)}$ and $j \in \set{1,2}$ agree and return the same value.
\end{enumerate}

Statement~(\refitem{Previousexplore}) follows by induction and \eqrefPartII{clusterDefinition}.
For (\refitem{SumsAgreeTotaunext}), note that the sums in \eqrefPartII{TfrDefn} and \eqrefPartII{FreezingTimeApproximation} agree whenever $\BP_t^{\sss(j)}\subset\explore_{k-1}\intersect\tree^{\sss(j)}$, so using (\refitem{Previousexplore}) it suffices to show that $\BP_t^{\sss(j)}\subset\BP^{\sss(j)}_{t'}$ for $t<\tau_\rmnext^{\sss(j)}(k)$.
But for any $t'$ and any $t<\min\set{T_v\colon v\in\boundary \BP^{\sss(j)}_{t'}}$ we have $\BP^{\sss(j)}_t\subset\BP^{\sss(j)}_{t'}$ by definition, so the second part of (\refitem{Previousexplore}) completes the proof.

Statement~(\refitem{TfrLessIfftaufrLess}) follows from (\refitem{SumsAgreeTotaunext}): if one of the sums in \eqrefPartII{TfrDefn} or \eqrefPartII{FreezingTimeApproximation} exceeds $s_n$ before time $\tau_\rmnext^{\sss(j)}(k)$, then so does the other, and the first times where they do so are the same.
In particular, $\set{T_\fr^{\sss(j)} <\tau_\rmnext^{\sss(j)}(k)}$ is $\F_k$-measurable because $\tau_\fr^{\sss(j)}(k)$ and $\tau_\rmnext^{\sss(j)}(k)$ are.

To prove (\refitem{RjRjkAgree}), we distinguish three cases.
If $\tau_\fr^{\sss(j)}(k) \ge \tau_\rmnext^{\sss(j)}(k)$, then $R_{j,k}^{-1}$ and $R_j^{-1}$ reduce to the identity on $[0,\tau_\rmnext^{\sss(j)}(k)]$ by (\refitem{TfrLessIfftaufrLess}).
Hence, (\refitem{RjRjkAgree}) holds if $\tau_\fr^{\sss(j)}(k) \ge \tau_\rmnext^{\sss(j)}(k)$ for both $j \in \set{1,2}$.
If $\tau_\fr^{\sss(j)}(k) < \tau_\rmnext^{\sss(j)}(k)$ for both $j \in \set{1,2}$, then $R_{j,k}^{-1}$ and $R_j^{-1}$ agree everywhere according to (\refitem{TfrLessIfftaufrLess}).
Finally,
consider the case that $\tau_\fr^{\sss(j)}(k) \ge \tau_\rmnext^{\sss(j)}(k)$ and $\tau_\fr^{\sss(j')}(k) < \tau_\rmnext^{\sss(j')}(k)$ for $\set{j,j'}=\set{1,2}$.
Then $T_\fr^{\sss(j)} \ge \tau_\rmnext^{\sss(j)}(k)$ and, therefore, $R_{j,k}^{-1}(t)=R_{j}^{-1}(t)=t$ on $I_k^{\sss(j)}$.
Moreover, $T_\fr^{\sss(j')}=\tau_\fr^{\sss(j')}(k)$ implying $R_{j',k}^{-1}(t)=R_{j'}^{-1}(t)=t$ on $[0,\tau_\fr^{\sss(j')}]$ and for $t \in \ocinterval{\tau_\fr^{\sss(j')}(k),\tau_{\rmnext}^{\sss(j')}(k)}$, $R_{j',k}^{-1}(t)\ge\tau_\unfr(k) \ge \tau_\rmnext^{\sss(j)}(k)$ and $R_{j'}^{-1}(t)\ge T_\unfr \ge \tau_\rmnext^{\sss(j)}(k)$.
Hence, in all three cases, the functions agree on the relevant domain and we have proved (\refitem{RjRjkAgree}).

Set $(v_k^{\explore},j_k^{\explore})$ to be the pair that minimizes $R_{j,k}^{-1}(T_v)$ among all pairs $(v,j)$ with $v\in\boundary\explore_{k-1}\intersect\tree^{\sss(j)}$, $j\in\set{1,2}$.
Note that $R_{j,k}^{-1}(t)$ can be infinite (when $\tau_\fr^{\sss(j)}(k)<\infty$ and $\tau_\fr^{\sss(j')}(k)=\infty$, where $\set{j,j'}=\set{1,2}$) but in this case $R_{j',k}^{-1}(t)$ must be finite for all $t$.
Furthermore $R_{j,k}^{-1}$ is strictly increasing whenever it is finite.
Recalling that the times $T_v$ are formed from variables with continuous distributions, it follows that the minimizing pair $(v_k^\explore,j_k^\explore)$ is well-defined a.s.

Since $R_{j,k}^{-1}$ is $\F_k$-measurable, the choice of $v_k^\explore$ is determined by a minimal rule with $S_k=\boundary\explore_{k-1}$.
To complete the proof, we must show that $v_k^\explore=v_k$.

The vertex $v_k$ is the first coordinate of the pair $(v_k,j_k)$ that minimizes $T^\cluster_v=R_j^{-1}(T_v)$ over all pairs $(v,j)$ with $v\in\boundary\cluster^{\sss(j)}_{T^\cluster_{v_{k-1}}}$, $j\in \set{1,2}$.
(Once again, this minimizing pair is well-defined a.s., reflecting the fact that, a.s., $\cluster$ never adds more than one vertex at a time.)
Since $\explore_{k-1}=\cluster_{T^\cluster_{v_{k-1}}}$, the pairs $(v_k^\explore,j_k^\explore)$ and $(v_k,j_k)$ are both minimizers over the set $\set{(v,j)\colon v\in\boundary\explore_{k-1}\intersect\tree^{\sss(j)}}$.
Moreover, since $R_j^{-1}$ and $R_{j,k}^{-1}$ are both strictly increasing (when finite), both minimizations can be restricted to the set $\set{(v_k^{\sss(j)},j)\colon j=1,2}$, where $v_k^{\sss(j)}$ is the minimizer in \eqrefPartII{taunextDefn}.
Since $T_{v_k^{(j)}}=\tau_\rmnext^{\sss(j)}(k)$, statement~(\refitem{RjRjkAgree}) with $I_{j,k}=\set{\tau_\rmnext^{\sss(j)}(k)}$ implies $v_k^\explore=v_k$, as claimed.
\end{proof}

\begin{proof}[Proof of \refprop{EdgeWeightCouplingClusterFPP2Source}]
By \refthm{CouplingExpl}, the edge weights $X^{\sss(K_n)}_e$ associated via \eqrefPartII{EdgeWeightCoupling} to the exploration process in \reflemma{clusterAsExploration} are independent exponential random variables with mean 1.
Recalling \eqrefPartII{EdgesByDistrFunct}--\eqrefPartII{FYandg}, the corresponding edge weights $Y^{\sss(K_n)}_e=g(X^{\sss(K_n)}_e)$ are independent with distribution function $F_Y$.
To complete the proof it suffices to observe that $N(i)<N(i')$ if and only if $T^{\thinnedcluster}(i)<T^{\thinnedcluster}(i')$ and that $T^{\thinnedcluster}(i)$ is finite for all $i\in [n]$ since the FPP process explores every edge in $\tree$ eventually.
Hence, definitions \eqrefPartII{EdgeWeightCoupling-FPP2Source} and \eqrefPartII{EdgeWeightCoupling} agree.
\end{proof}

\begin{proof}[Proof of \refthm{CouplingFPP}]
We resume the notation from the proof of \reflemma{clusterAsExploration}.
Create the edge weights on $K_n$ according to \eqrefPartII{EdgeWeightCoupling}.
Denote by $\tau_{k'-1}$ the time where the $(k'-1)^\st$ vertex (not including the vertices $1$ and $2$) was added to $\cS$ (see \eqrefPartII{cSConstruction}).
As in the proof of \refthm{Coupling1Source}, both $\thinnedcluster$ and $\cS$ are increasing jump processes and $\pi_M(\thinnedcluster_0)=\cS_0$.
By an inductive argument we can suppose that $k, k'\in\N$ are such that $\thinnedExplore_k\neq\thinnedExplore_{k-1}$ and $\pi_M(\thinnedExplore_{k-1})=\cS_{\tau_{k'-1}}$.
The proof will be complete if we can prove that (a) the edge $e_{k'}'$ adjoined to $\cS_{\tau_{k'-1}}$ to form $\cS_{\tau_{k'}}$ is the same as the edge $e_k=\set{i_k,i_k'}$ adjoined to $\pi_M(\thinnedExplore_{k-1})$ to form $\pi_M(\thinnedExplore_k)$; and (b) $\tau_{k'}=T^\cluster_{v_k}$.

Let $i\in\cS_{\tau_{k'-1}}$, and let $j\in\set{1,2}$ be such that $i\in\cS^{\sss(j)}_{\tau_{k'-1}}$.
By the inductive hypothesis, $V(i)\in\thinnedExplore_{k-1}\intersect\tree^{\sss(j)}$ and the unique path in $\cS_{\tau_{k'-1}}$ from $i$ to $j$ is the image of the unique path in $\thinnedExplore_{k-1}$ from $V(i)$ to $\emptyset_j$ under the mapping $v\mapsto M_v$ (recall \refdefn{InducedGraph}).
According to \eqrefPartII{InternalEdgeWeights}, \eqrefPartII{EdgesByIncrFunct} and \eqrefPartII{fnFromParamEdges}, the edge weights along this path are $Y^{\sss(K_n)}_{\shortset{M_{\ancestor{m-1}{V(i)}},M_{\ancestor{m}{V(i)}}}}=g(X^{\sss(K_n)}_{\shortset{M_{\ancestor{m-1}{V(i)}},M_{\ancestor{m}{V(i)}}}})=g(\tfrac{1}{n}X_{\ancestor{m-1}{V(i)}})$ and $f_n(X_{\ancestor{m-1}{V(i)}})$, for $m=1,\dotsc,\abs{V(i)}$.
Summing gives $d_{\cS_{\tau_{k'-1}},Y^{\sss(K_n)}}(j,i)=T_{V(i)}$.

In addition, let $i'\notin \cS_{\tau_{k'-1}}$ and write $e=\set{i,i'}$.
By \eqrefPartII{BoundaryEdgeWeights}, $X_e^{\sss(K_n)}=\tfrac{1}{n}X_{V(i,i')}$, so that $Y_e^{\sss(K_n)}=f_n(X_{V(i,i')})$.
Thus the expression in the right-hand side of \eqrefPartII{cSConstruction} reduces to $R_j^{-1}\left( T_{V(i)}+f_n(X_{V(i,i')}) \right)$, i.e., $R_j^{-1}\left( T_{V(i,i')} \right)$.
The edge $e_{k'}'$ minimizes this expression over all $i\in\cS_{\tau_{k'-1}}$, $i'\notin\cS_{\tau_{k'-1}}$.
By \refprop{MinimalRuleThinning}, the edge $e_k=\set{i_k,i_k'}$ minimizes $R_{j,k}^{-1}(T_{V(i,i')})$ over all $i\in\pi_M(\thinnedExplore_{k-1})$, $i'\notin\pi_M(\thinnedExplore_{k-1})$ (with $j$ such that $V(i)\in\explore^{\sss(j)}_{k-1}$).
By the induction hypothesis, statement (\refitem{RjRjkAgree}) in the proof of  \reflemma{clusterAsExploration} and monotonicity, these two minimization problems have the same minimizers, proving (a), and return the same value, which completes the proof of (b).
\end{proof}

\subsection{Coupling and Cox processes: Proof of \refthm{CollCondProb}}\lbsubsect{CouplingCox}

In this section we discuss the interaction between the coupling of $\cluster$ and $\cS$ and the Cox processes $\cP^{\sss(K_n)}_n$ and $\cP_n$.
We explain the modification to the coupling needed in Theorems~\refPartII{t:CollCondProb} and \refPartII{t:CollisionThm}.
\refthm{CouplingFPPCollision} involves an equality in law rather than the specific conditional probability, and this will simplify the issues of coupling.

As we will see, \refthm{CollCondProb} is a special case of \refthm{CollisionThm}, which we will prove now:

\begin{proof}[Proof of \refthm{CollisionThm}]
As remarked after \refthm{CollCondProb}, we are using a different coupling in \refthm{CollCondProb} than in \refthm{CouplingFPP}.
To understand this additional complication, notice that the discrete-time analogue of $\thinnedcluster$ is \emph{not} an exploration process because the on-off processes $R_1,R_2$ are not measurable with respect to $\thinnedcluster$.
Consequently, knowledge of $\thinnedcluster$ may imply unwanted information about the edge weights in $\cluster\setminus\thinnedcluster$, including the edge weights used in the coupling \eqrefPartII{EdgeWeightCoupling} to determine the connecting edge weights $(X_e^{\sss (K_n)}, Y_e^{\sss (K_n)})_{e\notin\cS}$ that arise in \eqrefPartII{ConnectingEdgeCondition}.

To overcome this problem, we introduce an additional source of randomness.
As always, we have
two independent PWITs $(\tree^{\sss(j)},X^{\sss(j)})$, $j\in \set{1,2}$,
the marks $M_v$, $v \in \tree$,
and a family of independent exponential random variables $E_e$, $e\in E(K_{\infty})$, with mean $1$, independent of the PWITs and the marks.
In addition, from each vertex $v$ we initialise an independent PWIT with vertices $(v,w')$, edge weights $X'_{(v,w')}$ (such that $(X'_{(v,w'k)})_{k=1}^\infty$ forms a Poisson point process with rate $1$ on $(0,\infty)$) and marks $M'_{(v,w')}$ uniform on $[n]$, all independent of each other and of the original variables $X_v,M_v$.

First consider $(\cluster,\thinnedcluster,R_1,R_2)$ as normally constructed, without using the auxiliary PWITs with vertices $(v,w')$.
Fix $i,i'\in[n]$, $i\neq i'$, and suppose for definiteness that $T^{\sss \thinnedcluster}(i)< T^{\sss \thinnedcluster}(i')$.
(If instead $T^{\sss \thinnedcluster}(i')< T^{\sss \thinnedcluster}(i)$, interchange the roles of $i$ and $i'$ in the following discussion.)
According to \eqrefPartII{EdgeWeightCoupling-FPP2Source} or \eqrefPartII{EdgeWeightCoupling}, the edge weight $X^{\sss(K_n)}_{\set{i,i'}}$ is set to be $\tfrac{1}{n}X(i,i')$, where $X(i,i')$ is the first point in the Poisson point process $\cP=\sum_{w\colon \parent{w}=V(i), M_w=i'} \delta_{X_w}$ of intensity $1/n$.

Now condition on $V(i)$ and $V(i')$ belonging to different trees, say $V(i)\in\tree^{\sss(J)}, V(i')\in\tree^{\sss(J')}$ where $\set{J,J'}=\set{1,2}$.
For this to happen, the children of $V(i)$ having mark $i'$ must \emph{not} have been explored by time $T^{\sss\cluster}_{V(i')}$.
A child $w$ of $V(i)$ is explored at time $R_J^{-1}(T_w)=R_J^{-1}(T_{V(i)}+f_n(X_w))$, so we can reformulate this by saying that the Poisson point process $\cP$ must contain no point $x$ with $R_J^{-1}(T_{V(i)}+f_n(x)) \le T_{V(i')}^{\sss\cluster}$.
Using the relations $T_v=R_j(T^{\sss\cluster}_v)$ for $v\in\tree^{\sss(j)}$ and $T^{\sss\cS}(i)=T^{\sss\thinnedcluster}(i)=T^{\sss\cluster}_{V(i)}$ for $i\in[n]$ (by \refthm{CouplingFPP}), we can rewrite this as the condition that $\cP$ contains no point $x$ with $f_n(x)\leq R_J(T_{V(i')}^{\sss\cluster})-R_J(T_{V(i')}^{\sss\cluster})=R_J(T^{\sss\cS}(i'))-R_J(T^{\sss\cS}(i))=\Delta R(i,i')$.
However, the condition gives no information about points of larger value.
It follows that, conditional on $(\cluster_t,\thinnedcluster_t,R_1(t),R_2(t))_{t\leq T^{\sss\thinnedcluster}(i')}$,  $\cP$ is a Poisson point process of intensity $1/n$ on $(f_n^{-1}(\Delta R(i,i')),\infty)$.

To preserve this property when conditioning on $(\cluster_t,\thinnedcluster_t,R_1(t),R_2(t))$ for $t> T^{\sss\thinnedcluster}(i')$, we must ensure that $(\thinnedcluster,R_1,R_2)$ does not depend on the edge weights $(X_w\colon \parent{w}=V(i_1),M_w=i_2, f_n(X_w)>\Delta R(i_1,i_2)$.
We remark that there is already no dependence when freezing is absent, since the relevant vertices $w$ will be thinned in any case; the dependence is carried only through the processes $R_1,R_2$.

To achieve the required independence, replace the edge weights and marks of all vertices $w$ (and their descendants) for which $\parent{w}=V(i)$, $M_w=i'$ and $f_n(X_w)>\Delta R(i,i'),$ by the edge weights and marks of the auxiliary vertices $(V(i),w'))$ (and their descendants) for which $\parent{w'}=V(i)$, $M'_{(V(i),w')}=i'$ and $f_n(X'_{(V(i),w')})>\Delta R(i,i')$.
Modify the evolution of $\cluster,R_1,R_2$ for $t>T^{\sss\thinnedcluster}(i')$ so as to use the replacement edge weights $X'_{v,w'}$, but continue to use the original edge weights $X_w$ in the edge weight coupling \eqrefPartII{EdgeWeightCoupling}.
(Formally, modify the minimal rule from \reflemma{clusterAsExploration} so that vertices are ineligible for exploring once they have been replaced, but add to the sum in \eqrefPartII{FreezingTimeApproximation} the contribution from any replacement vertices from the auxiliary PWITs that would have been explored at time $t$.)

These replacements do not change the law of $\thinnedcluster,R_1,R_2$, or the edge weights $X^{\sss(K_n)}_e$.
The equality between $\pi_M(\thinnedcluster)$ and $\cS$ is unaffected: the replaced vertices are thinned and therefore do not affect $\thinnedcluster$.
Finally, the evolution of $\thinnedcluster$ for $t>T^{\sss\thinnedcluster}(i_2)$ now gives no additional information about the edge weights $\set{X_w\colon \parent{w}=V(i_1), M_w=i_2, f_n(X_w)>\Delta R(i_1,i_2)}$.
In particular, conditional on $\thinnedcluster,R_1,R_2$, the law of $\cP$ above is that of a Poisson point process with intensity measure $1/n$ on $(f_n^{-1}(\Delta R(i,i'),\infty))$.
Further, the Poisson point processes corresponding to different $i,i'$ will be conditionally independent.

Define
\begin{align}
\cP^{\sss(K_n)}_n = \sum_{i_1\in\cS^{\sss(1)}} & \sum_{i_2\in\cS^{\sss(2)}} \sum_{\set{j,j'}=\set{1,2}}\indicator{T^{\sss\cS}(i_j)<T^{\sss\cS}(i_{j'})}
\notag\\&
\times\sum_{\substack{w\colon \parent{w}=V(i_j),\\ M_w=i_{j'}}} \delta_{\left( (R_1+R_2)^{-1}(R_1(T^{\cS}(i_1))+f_n(X_w)+R_2(T^{\cS}(i_2))) , i_1, i_2\right)}
.
\labelPartII{PnKnOption}
\end{align}
By the discussion above, $\cP^{\sss(K_n)}_n$ is a Cox process with respect to the $\sigma$-field generated by $\thinnedcluster, \cS,R_1,R_2$ and $(M_v)_{v\in\thinnedcluster}$.
To complete the proof, we will show that the first point of $\cP^{\sss(K_n)}_n$ is $(T_\coll,I_1,I_2)$ and that the intensity measure of $\cP^{\sss(K_n)}_n$ is given by \eqrefPartII{CollisionIntensityKn}.

For each $i_1\in\cS^{\sss(1)},i_2\in\cS^{\sss(2)}$ with $T^{\sss\cS}(i_j)<T^{\sss(\cS)}(i_{j'})$, the first point in the restriction of $\cP^{\sss(K_n)}_n$ to $\cointerval{0,\infty}\times \set{i_1}\times\set{i_2}$ is given by the term in \eqrefPartII{PnKnOption} with $w=V(i_j,i_{j'})$.
Under the edge weight coupling \eqrefPartII{EdgeWeightCoupling}, this is the edge weight used to produce $X^{\sss(K_n)}_{\set{i_1,i_2}}$, so
\begin{equation}
Y^{\sss(K_n)}_{\set{i_1,i_2}}=g(X^{\sss(K_n)}_{\set{i_1,i_2}}) = g( \tfrac{1}{n} X_w)=f_n(X_w).
\end{equation}
Thus the first point in the restriction of $\cP^{\sss(K_n)}_n$ to $\cointerval{0,\infty}\times \set{i_1}\times\set{i_2}$ is the triple
\begin{equation}
\left( (R_1+R_2)^{-1}(R_1(T^{\sss\cS}(i_1))+Y^{\sss(K_n)}_{\set{i_1,i_2}}+R_2(T^{\sss\cS}(i_2))) , i_1, i_2\right).
\end{equation}
By \reflemma{WnAsMinimum} and the fact that $(R_1+R_2)^{-1}$ is strictly increasing, the first point of $\cP^{\sss(K_n)}_n$ is the triple whose first coordinate is $(R_1+R_2)^{-1}(W_n)$, i.e.\ $T_\coll$, and whose last two coordinates are the minimizers in \eqrefPartII{WnMin}, i.e.\ $I_1$ and $I_2$.

For the intensity measure, restrict again to $\cointerval{0,\infty}\times\set{i_1}\times\set{i_2}$ and project onto the first coordinate.
By the argument above, the points $X_w$ with $\parent{w}=V(i_j), M_w=i_{j'}$ form a Cox process of intensity $1/n$ on $(f_n^{-1}(\Delta R(i_1,i_2)),\infty)$.
Taking the image under the mapping $x\mapsto f_n(x)$ results in a Cox process of intensity $\frac{1}{n}\mu_n\big\vert_{(\Delta R(i_1,i_2),\infty)}$.
In \eqrefPartII{PnKnOption} we take the further mapping $y\mapsto (R_1+R_2)^{-1}(R_1(T^{\sss\cS}(i_1))+y+R_2(T^{\sss\cS}(i_2)))$, so the result is the intensity measure $m$ defined by
\begin{equation}
m([0,t])=\tfrac{1}{n}\mu_n\big\vert_{(\Delta R(i_1,i_2),\infty)} \left( \set{y\colon R_1(T^{\sss\cS}(i_1))+y+R_2(T^{\sss\cS}(i_2))\leq R_1(t)+R_2(t) } \right)
\end{equation}
Noting that this quantity is 0 if $t<T^{\sss\cS}(i_1)\vee T^{\sss\cS}(i_2)$, we see that it reduces to the right-hand side of \eqrefPartII{CollisionIntensityKn}.
\end{proof}

\begin{proof}[Proof of \refthm{CollCondProb}]
\refthm{CollCondProb} is the special case of \refthm{CollisionThm} where we compute $\P\condparenthesesreversed{\cP^{\sss(K_n)}_n([0,t]\times[n]\times[n])=0}{\thinnedcluster,\cS,R_1,R_2,(M_v)_{v\in\thinnedcluster}}$.
By the Cox process property, this conditional probability is given in terms of the corresponding intensity, which is obtained by summing \eqrefPartII{CollisionIntensityKn} over $i_1,i_2\in[n]$ and noting that $\mu_n(a,b)=f_n^{-1}(b)-f_n^{-1}(a)$.
\end{proof}

In \refthm{CouplingFPPCollision}, we are concerned only with an equality in law, not with a detailed coupling.
Consequently the proof reduces to a computation about the intensity measures $Z_n$ and $Z_n^{\sss(K_n)}$:

\begin{proof}[Proof of \refthm{CouplingFPPCollision}]
We consider the image, under the mapping $(t,v_1,v_2)\mapsto (t,M_{v_1},M_{v_2})$, of the restriction of $\cP_n$ to the set $\set{(t,v_1,v_2)\colon v_1,v_2\text{ unthinned}}$.
The result is another Cox process (with respect to the same $\sigma$-field), which we claim has the same law as $\cP^{\sss(K_n)}_n$.
To verify this, it suffices to show that the image, under the mapping $(v_1,v_2)\mapsto (M_{v_1},M_{v_2})$, of the restriction of $Z_{n,t}$ has the same law as $Z^{\sss(K_n)}_{n,t}$, jointly over all $t\geq 0$.

Restricted to unthinned vertices, the mapping $\tree\to[n]$, $v\mapsto M_v$ is invertible with inverse mapping $V(i)\mapsfrom i$.
So we will compare $Z_{n,t}^{\sss(K_n)}(\set{i_1}\times \set{i_2})$ with $Z_{n,t}(\set{V(i_1)}\times\set{V(i_2)})$.
Because of the relation $\pi_M(\thinnedcluster)=\cS$, we have $T^\cluster_{V(i)}=T^{\cS}(i)$, which implies that $\Delta R(i_1,i_2)=\Delta R_{V(i_1),V(i_2)}$.
Moreover the condition $i_1\in\cS^{\sss(1)}_t$, $i_2\in\cS^{\sss(2)}_t$ is equivalent to $V(i_1)\in\cluster^{\sss(1)}_t$, $V(i_2)\in\cluster^{\sss(2)}_t$.
Hence \eqrefPartII{CollisionIntensityKn} match \eqrefPartII{CollisionIntensityTree} under $(v_1,v_2)\mapsto (M_{v_1},M_{v_2})$ and we conclude that the conditional law of $(T_\coll^{\sss(\cP_n)},M_{V_\coll^{\sss(1)}},M_{V_\coll^{\sss(2)}})$ given $\cluster$ and $(M_v)_{v\in\tree}$ is the same law as the conditional law of $(T_\coll,I_1,I_2)$ given $(\thinnedcluster,\cS,R_1,R_2)$ and $(M_v)_{v\in\thinnedcluster}$.
The remaining equalities in law from \refthm{CouplingFPPCollision} concern various functions of these objects, and therefore follow from \refthm{WnHnFromCollision}.
\end{proof}

In the remainder of the paper, we will be concerned only with the equality in law from \refthm{CouplingFPPCollision}.
We can therefore continue to define $\cluster$ as in \eqrefPartII{clusterDefinition}, ignoring the modified construction given in the proof of \refthm{CollisionThm}, with the edge weight coupling \eqrefPartII{EdgeWeightCoupling-FPP2Source} between $\tree$ and $K_n$.

\section{Branching processes and random walks: Proof of \refthm{OneCharConv}}\lbsect{BPbyWalks}

In this section, we prove \refthm{OneCharConv} by continuing the analysis of the branching process $\BP^{\sss(1)}$ introduced in \refsect{OverPf}.
In \refsubsect{CTBPsRWsect} we identify a random walk which facilitates moment computations of the one-vertex characteristics.
\refsubsect{ConvRWThm} contains precise results about the scaling behavior of the random walk and the parameter $\lambda_n(a)$.
The results are proved in \refsubsect{ConvRWPf}.
\refsubsect{LimitingBounds} identifies the asymptotics of the first two moments of one-vertex characteristics.
Having understood these, we investigate two-vertex characteristics in \refsect{2VertexCharSec} and prove \refthm{TwoVertexConvRestrictedSum}.

\subsection{Continuous-time branching processes and random walks}
\lbsubsect{CTBPsRWsect}

Recall that $\BP^{\sss(1)}=(\BP_t^{\sss(1)})_{t\geq 0}$ denotes a CTBP with original ancestor $\emptyset_1$.
Using Ulam--Harris notation, the children of the root are the vertices $v$ with $\parent{v}=\emptyset_1$, and their birth times
$(T_v)_{\parent{v}=\emptyset_1}$ form a Poisson point process with intensity $\mu_n$.
For $v \in \tree^{\sss (1)}$, write $\BP^{\sss(v)}$ for the branching process of descendants of such a $v$, re-rooted and time-shifted to start at $t=0$.
Formally,
\begin{equation}\labelPartII{BPvDefinition}
\BP^{\sss(v)}_t=\set{w\in\tree^{\sss(1)}\colon vw\in\BP^{\sss(1)}_{T_v+t}}.
\end{equation}
In particular, $\BP^{\sss (1)}=\BP^{\sss (\emptyset_1)}$, and  $(\BP^{\sss(v)})_{\parent{v}=\emptyset_1}$ are independent of each other and of $(T_v)_{\parent{v}=\emptyset_1}$.
We may express this compactly by saying that the sum of point masses $\mathcal{Q}=\sum_{\parent{v}=\emptyset_1} \delta_{(T_v,\BP^{\sss(v)})}$ forms a Poisson point process with intensity $d\mu_n \otimes d\P(\BP^{\sss (1)}\in\cdot)$, where $\P(\BP^{\sss (1)}\in\cdot)$ is the law of the entire branching process.
Recalling the definition of the one-vertex characteristic from \eqrefPartII{1VertexCharDef}, we deduce
	\begin{equation}\labelPartII{zchiRecursive}
	\begin{split}
	z_t^\chi(a)
	&=
	\chi(t)+a\sum_{v\colon \parent{v}=\emptyset} \indicator{T_v\leq t} z_{t-T_v}^{\chi, \BP^{\sss (v)}}(a)
	\\
	&=
	\chi(t)+a\int d\mathcal{Q}(y,bp) \indicator{y\leq t} z_{t-y}^{\chi, bp}(a).
	\end{split}
	\end{equation}
Note that \eqrefPartII{zchiRecursive} holds jointly for all $a,t,\chi$.
To draw conclusions for $m_t^\chi(a)$ and $M_{t,u}^{\chi,\eta}(a,b)$, the expectation of $z_t^\chi(a)$ and $z_t^\chi(a)z_u^{\eta}(b)$, respectively, defined in \eqrefPartII{NonRescaled1VertexChar}, we 	
will use the formulas
	\begin{align}
	\labelPartII{PPPMeanCovar}
	\begin{split}
	\E\Big( \sum_{p\in \tilde{\mathcal{Q}}} f(p) \Big)
	&=
	\int f(p) \, d\tilde{\mu}(p),
	\\
	\Cov\Big( \sum_{p\in \tilde{\mathcal{Q}}} f_1(p),\sum_{p\in\tilde{\mathcal{Q}}} f_2(p) \Big)
	&=
	\int f_1(p) f_2(p) \, d\tilde{\mu}(p),
	\end{split}
	\end{align}
where $\tilde{\mathcal{Q}}$ is a Poisson point process with some intensity $\tilde{\mu}$ (and assuming the integrals exist).
Apply \eqrefPartII{PPPMeanCovar} to \eqrefPartII{zchiRecursive} with $f(y,bp)=\indicator{y\leq t} z_{t-y}^{\chi, bp}(a)$ to get
	\begin{align}
	m_t^\chi(a)
	&=
	\chi(t)+a\int d\mu_n(y)\indicator{y\leq t} \int d\P(\BP^{\sss(1)}=bp) z_{t-y}^{\chi,bp}(a)
	\notag\\
	&=
	\chi(t)+a\int_0^t d\mu_n(y) m_{t-y}^\chi(a).
	\labelPartII{mtchiaRecursion}
	\end{align}
Similarly
	\begin{align}
	M_{t,u}^{\chi,\eta}(a,b)-m_t^\chi(a) m_u^\eta(b)
	&=
	ab\int d\mu_n(y)\indicator{y\leq t}\indicator{y\leq u}
	\int d\P(\BP^{\sss(1)}=bp) z_{t-y}^{\chi,bp}(a) z_{u-y}^{\eta,bp}(b)
	\notag\\
	&=
	ab\int_0^{t\wedge u} d\mu_n(y) M_{t-y,u-y}^{\chi,\eta}(a,b)
	.
	\labelPartII{MtuchietaabRecursion}
	\end{align}
Recall from \eqrefPartII{lambdaaDefn} that $\hat{\mu}_n(\lambda)=\int\e^{-\lambda t}d\mu_n(t)$ denotes the Laplace transform of $\mu_n$ and that for $a>0$ the parameter $\lambda_n(a)>0$ is the unique solution to $a\hat{\mu}_n(\lambda_n(a))=1$.
(In general, if $\hat{\mu}_n(\lambda_0)$ is finite for some $\lambda_0\geq 0$, then the equation has a solution whenever $1/\hat{\mu}_n(\lambda_0)\leq a<1/\mu_n(\set{0})$, and this solution is unique if $\mu_n$ assigns positive mass to $(0,\infty)$.
Our assumptions imply that, for $n$ sufficiently large, $\lambda_n(a)$ exists uniquely for any $a>0$.)
Since $z_t^\chi(a)$ typically grows exponentially at rate $\lambda_n(a)$, we study the rescaled versions $\bar{m}_t^\chi(a)$, $\bar{M}_{t,u}^{\chi,\eta}(a,b)$ defined in \eqrefPartII{barredOneVertexChars} and let
\begin{equation}\labelPartII{nuaDef}
	d\nu_a(y)=  a \e^{-\lambda_n(a) y} d\mu_n(y).
\end{equation}
Then \eqrefPartII{mtchiaRecursion} becomes $\bar{m}_t^\chi(a)=\e^{-\lambda(a)t}\chi(t)+\int_0^t d\nu_a(y) \bar{m}_{t-y}^\chi(a)$.
Since $\nu_a$ is a probability measure by construction, this recursion can be solved in terms of a random walk:
	\begin{equation}\labelPartII{mtchiaRW}
	\bar{m}_t^\chi(a) = \E_a\Big( \sum_{j=0}^\infty \e^{-\lambda_n(a) (t-S_j)} \chi(t-S_j) \Big)
	\! ,
	\quad\text{where}
	\quad
	S_j=\sum_{i=1}^j \increm_i,
	\quad
	\P_a(\increm_i\in\cdot) = \nu_a(\cdot).
	\end{equation}
From \eqrefPartII{MtuchietaabRecursion}, we obtain similarly
	\begin{equation}\labelPartII{MtuchietaabRW}
	\bar{M}_{t,u}^{\chi,\eta}(a,b)
	=\E_{ab}\Big( \sum_{j=0}^\infty \e^{-S_j[\lambda_n(a)+\lambda_n(b)-\lambda_n(ab)]} 	
	\bar{m}_{t-S_j}^\chi(a) \bar{m}_{u-S_j}^\eta(b) \Big)
	\! ,
	\end{equation}
where $S_j=\sum_{i=1}^j \increm_i$ now has distribution $\P_{ab}(\increm_i\in\cdot)
= \nu_{ab}(\cdot)$.

Note that for a random variable $\increm$ with law $\nu_a$, for every $h\ge 0$ measurable,
\begin{equation}\labelPartII{integralnua}
\E_a(h(\increm))=\int h(y) a \e^{-\lambda_n(a) y} \, d\mu_n(y).
\end{equation}
Moreover, let $\nu^*_a$ denote the size-biasing of $\nu_a$, i.e.,
\begin{equation}\labelPartII{nuaSizeBiasedDef}
d\nu^*_a(y)=\frac{y \, d\nu_a(y)}{\int y \, d\nu_a(y)}
\qquad\text{so that}\qquad
\E_a(h(\increm^*))=\frac{\E_a(\increm h(\increm))}{\E_a(\increm)}
\end{equation}
for $h\geq 0$ measurable.
Here and in all of the following we assume that $\increm$ and $\increm^*$ have laws $\nu_a$ and $\nu^*_a$ respectively under $\E_a$.
Let $U$ be uniform on $[0,1]$, and let $(\increm_i)_{i \ge 1}$ be independent with law $\nu_a$, and independent of $U$ and $\increm^*$.
Besides the random walk $(S_j)_j$ from \eqrefPartII{mtchiaRW}--\eqrefPartII{MtuchietaabRW}, it is useful to study the random walk $(S_j^*)_j$ with
\begin{equation}\labelPartII{DefSjStar}
S_0^*=U\increm^* \qquad \text{and} \qquad S_j^*=S_0^*+\sum_{i=1}^j \increm_i \qquad \text{for all }j\ge 1.
\end{equation}

\subsection{Random walk convergence}
\lbsubsect{ConvRWThm}

In this section, we investigate asymptotics of the random walks
in \eqrefPartII{mtchiaRW}--\eqrefPartII{MtuchietaabRW} and \eqrefPartII{DefSjStar}.
Recall that the branching process $\BP^{\sss(1)}$ is derived from the intensity measure $\mu_n$, where
for $h \colon \cointerval{0,\infty} \to \cointerval{0,\infty}$ measurable,
\begin{equation}\labelPartII{munFromfn}
\int h(y) d\mu_n(y) = \int_0^\infty h(f_n(x)) dx.
\end{equation}
In particular, all the quantities $z,m,M,\nu_a,\increm_i,\P_a$ and $\E_a$ as well as the random walks $(S_j)_j$ and $(S_j^*)_j$ depend implicitly on $n$.
We will give sufficient conditions for the sequence of walks $(S_j)_j$, $(S^*_j)_j$ to have scaling limits; in all cases the scaling limit is the Gamma process.
This is the content of \refthm{ConvRW} below.

As a motivation, we first look at the key example $Y_e^{\sss(K_n)}\overset{d}{=}E^{s_n}$:
\begin{example}\lbexample{ex:Esn}
Let $Y_e^{\sss(K_n)}\overset{d}{=}E^{s_n}$.
Then
\begin{equation}
\begin{aligned}
&\FY(x)=1-\e^{-x^{1/s_n}}\qquad \text{and} \qquad \FY^{-1}(x)=(-\log(1-x))^{s_n},\\
&f_n(x)= (x/n)^{s_n}=f_n(1) x^{s_n} \quad \text{and} \quad f_n^{-1}(x)=n x^{1/s_n}.
\end{aligned}
\end{equation}
One easily checks that for all $a>0, \beta>0$,
\begin{equation}
\labelPartII{lambdanAndY_exp}
\begin{aligned}
&f_n(1) \lambda_n(a^{1/s_n})=a \Gamma(1+1/s_n)^{s_n}, \quad s_n \lambda_n(a^{1/s_n}) \E_{a^{1/s_n}}(\increm)=1,\\
&s_n \lambda_n(a^{1/s_n})^2 \E_{a^{1/s_n}}(\increm^2)=1+1/s_n, \quad a \left( \hat{\mu}_n \bigl( \lambda_n(a^{1/s_n})\beta\bigr) \right)^{s_n} = 1/\beta.
\end{aligned}
\end{equation}
Notice that $\Gamma(1+1/s_n)^{s_n} \to \e^{-\gamma}$ for $n \to \infty$.
\end{example}

\refthm{ConvRW} will show that in general the same identities hold asymptotically under our conditions on $f_n$.
In fact, we will prove \refthm{ConvRW} under weaker assumptions on $f_n$:

\begin{cond}\lbcond{fnWeakBounds}
There exists $\epsilon_0>0$ and a sequence $(\delta_n)_{n \in \N} \in \ocinterval{0,1}^{\N}$ such that $s_n f_n(1-\delta_n)/f_n(1)=o(1)$, $f_n(x^{1/s_n})\geq f_n(1)x^{\epsilon_0}$ for $x\geq 1$ and $f_n(x^{1/s_n})\leq f_n(1)x^{\epsilon_0}$ for $(1-\delta_n)^{s_n}\leq x\leq 1$.
\end{cond}

Conditions~\refPartII{cond:LowerBoundfn} and \refPartII{cond:boundfnExtended} together imply \refcond{fnWeakBounds}: we may set $\delta_n=\deltaCondition$, with $\epsilon_0$ chosen as for Conditions~\refPartII{cond:LowerBoundfn} and \refPartII{cond:boundfnExtended}, and replacing $(x,x')$ in \reflemma{ExtendedImpliesWeak} by $(1,x^{1/s_n})$ or $(x^{1/s_n},1)$ verifies the inequalities in \refcond{fnWeakBounds}.

\begin{theorem}[Random walk convergence]
\lbthm{ConvRW}
Suppose that Conditions~\refPartII{cond:scalingfn} and \refPartII{cond:fnWeakBounds} hold.
Then for any $a\in(0,\infty)$:
	\begin{enumerate}
	\item
	The parameter $\lambda_n(a)$ exists for $n$ sufficiently large, and, for all $\beta>0$,
		\begin{gather}
		\labelPartII{lambdanAsymp}
		\lim_{n\to\infty} f_n(1) \lambda_n(a^{1/s_n}) = a \e^{-\gamma} ,
		\\
		\labelPartII{EYAsymp}
		\lim_{n\to\infty} s_n\lambda_n(a^{1/s_n})\E_{a^{1/s_n}}(\increm) = 1,
		\\
		\labelPartII{EY2Asymp}
		\lim_{n\to\infty} s_n\lambda_n(a^{1/s_n})^2\E_{a^{1/s_n}}(\increm^2) = 1,
		\\
		\labelPartII{munhatAsymp}
		\lim_{n\to\infty} a \left( \hat{\mu}_n \bigl( \lambda_n(a^{1/s_n})\beta \bigr) \right)^{s_n}
	= 1/\beta,
	\end{gather}
where $\gamma$ is Euler's constant.

	\item
	\lbitem{SsnyToGamma}
	Under $\E_{a^{1/s_n}}$, the process $(\lambda_n(a^{1/s_n}) S_{\floor{s_n t}})_{t\geq 0}$
	converges in distribution (in the Skorohod topology on compact subsets) to a Gamma
	process $(\Gamma_t)_{t\geq 0}$, i.e., the L\'evy process such that $\Gamma_t$ has the Gamma$(t,1)$ distribution.

	\item\lbitem{Y*ToExp}
	Under $\E_{a^{1/s_n}}$, the variable $\lambda_n(a^{1/s_n}) \increm^*$ converges in distribution to an
	exponential random variable $E$ with mean $1$, and the process $(\lambda_n(a^{1/s_n}) S^*_{\floor{s_n t}})_{t\geq 0}$
	converges in distribution to the sum $(U E+ \Gamma_t)_{t\geq 0}$ where $U$ is Uniform on $[0,1]$
	and $U,E,(\Gamma_t)_{t\geq 0}$ are independent.
\end{enumerate}
Moreover, given a compact subset $A\subset(0,\infty)$, all of the convergences occur uniformly for $a\in A$ and, for \eqrefPartII{munhatAsymp}, for $\beta\in A$.
\end{theorem}

\refthm{ConvRW} will be proved in \refsubsect{ConvRWPf}.
We stress that the proof or \refthm{ConvRW} uses only Conditions~\refPartII{cond:scalingfn} and \refPartII{cond:fnWeakBounds} and the relevant definitions but no other results stated so far.

\subsection{Convergence of random walks: Proof of \refthm{ConvRW}}\lbsubsect{ConvRWPf}

For notational convenience, we make the abbreviations
\begin{equation}\labelPartII{tildeConvention}
\tilde{x}=x^{1/s_n}, \qquad \tilde{a}=a^{1/s_n}, \qquad \tilde{b}=b^{1/s_n},\text{ etc.,}
\end{equation}
which we will use extensively in this section and in Sections~\refPartII{s:2VertexCharSec} and \refPartII{s:FrozenGeometry}.

\begin{proof}[Proof of \refthm{ConvRW}]
We begin by proving
\begin{equation}\labelPartII{munhatUpTo1oversn}
\hat{\mu}_n\left(\frac{a\e^{-\gamma}}{f_n(1)}\right)
= 1 - \frac{\log a}{s_n} + o(1/s_n).
\end{equation}
Recalling \eqrefPartII{munFromfn}, we have
\begin{equation}\labelPartII{munhattau}
\hat{\mu}_n(\lambda)=\int_0^\infty \e^{-\lambda f_n(\tilde{x})}d\tilde{x}.
\end{equation}
Write $\tilde{f}_n(\tilde{x})=f_n(\tilde{x})\indicator{\tilde{x}\geq 1-\delta_n}$.
Then
\begin{equation}
	\hat{\mu}_n(\lambda)=\int_0^{\infty} \e^{-\lambda \tilde{f}_n(\tilde{x})} \, d\tilde{x}
	- \int_0^{1-\delta_n} 	\big(1-\e^{-\lambda f_n(\tilde{x})}\big) \, d\tilde{x}
\end{equation}
and take $\lambda=a\e^{-\gamma}/f_n(1)$ to estimate $\int_0^{1-\delta_n} (1-\e^{-a\e^{-\gamma}f_n(\tilde{x})/f_n(1)})\,d\tilde{x} = O(f_n(1-\delta_n)/f_n(1))=o(1/s_n)$ by \refcond{fnWeakBounds}.
Hence, for the purposes of proving \eqrefPartII{munhatUpTo1oversn}, it is no loss of generality to assume that $f_n(x^{1/s_n})\leq f_n(1)x^{\epsilonCondition}$ for all $x\leq 1$.

Inspired by \refexample{AllExamples}~\refitem{EsnExample}, where
$f_n(\tilde{x})=f_n(1)\tilde{x}^{s_n}$,  we compute
\begin{equation}\labelPartII{PowerOfExpIntegral}
\int_0^\infty \e^{-\lambda f_n(1)\tilde{x}^{s_n}} d\tilde{x}
= \int_0^\infty \frac{1}{s_n} x^{1/s_n - 1} \e^{-\lambda f_n(1) x} dx
= \left( \frac{\Gamma(1+1/s_n)^{s_n}}{\lambda f_n(1)} \right)^{1/s_n}.
\end{equation}
In particular, setting $\lambda=a\Gamma(1+1/s_n)^{s_n}/f_n(1)$ gives $\int_0^\infty \exp\left( -a\Gamma(1+1/s_n)^{s_n}\tilde{x}^{s_n} \right) d\tilde{x}=a^{-1/s_n}$, which is $1-(\log a)/s_n+o(1/s_n)$.
Subtracting this from \eqrefPartII{munhattau}, we can therefore prove \eqrefPartII{munhatUpTo1oversn} if we show that
\begin{equation}
s_n\int_0^\infty \left( \e^{-a \e^{-\gamma} f_n(\tilde{x})/f_n(1)} - \e^{-a \Gamma(1+1/s_n)^{s_n} \tilde{x}^{s_n}} \right)d\tilde{x}
\to 0,
\end{equation}
or equivalently, by the substitution $\tilde{x}=x^{1/s_n}$, if we show that
\begin{equation}\labelPartII{DifferenceOfIntegrals}
\int_0^\infty x^{1/s_n - 1} \left( \e^{-a \e^{-\gamma} f_n(x^{1/s_n})/f_n(1)} - \e^{-a \Gamma(1+1/s_n)^{s_n} x} \right) dx \to 0.
\end{equation}
 Note that $\Gamma(1+1/s)^s\to\e^{-\gamma}$ as $s\to\infty$.
Together with \refcond{scalingfn}, this implies that the integrand in \eqrefPartII{DifferenceOfIntegrals} converges pointwise to $0$.
 For $x\leq 1$, $f_n(x^{1/s_n}) \le f_n(1)x^{\epsilonCondition}$ means that the integrand is bounded by $O(x^{\epsilonCondition-1}+1)$.
 For $x\geq 1$, \refcond{fnWeakBounds} implies that the integrand is bounded by $\e^{-\delta x^{\epsilonCondition}}$ for some $\delta>0$.
Dominated convergence therefore completes the proof of \eqrefPartII{munhatUpTo1oversn}.
It is easy to see that the proof of \eqrefPartII{DifferenceOfIntegrals}, and hence \eqrefPartII{munhatUpTo1oversn}, holds uniformly in $a\in A$, where $A \subseteq (0,\infty)$ is a fixed but arbitrary compact set.

To conclude \eqrefPartII{lambdanAsymp} from \eqrefPartII{munhatUpTo1oversn}, we use the monotonicity of $\mu_n$ and $\lambda_n$: given $\delta>0$, we have $\hat{\mu}_n\bigl(a\e^{-\gamma}/f_n(1)\bigr)^{s_n}\leq a^{-1}+\delta$ for all sufficiently large $n$, uniformly in $a\in A$.
Replacing $a$ by $a'=(1/a - \delta)^{-1}$ shows that $\lambda_n(\tilde{a})\leq a' \e^{-\gamma}/f_n(1)$ for all $n$ large enough, uniformly in $a$.
A lower bound holds similarly; take $\delta\to 0$ to conclude \eqrefPartII{lambdanAsymp}.

The proof of \eqrefPartII{EYAsymp} is similar to that of \eqrefPartII{lambdanAsymp}.
Using \eqrefPartII{integralnua}, we compute
\begin{align}
s_n \lambda_n(\tilde{a})\E_{\tilde{a}}(\increm)
&=
s_n \lambda_n(\tilde{a})\int y\tilde{a}  \e^{-\lambda_n(\tilde{a}) y} d\mu_n(y)
\notag\\
&=
\lambda_n(\tilde{a})f_n(1)\tilde{a} \int_0^\infty x^{1/s_n - 1} \frac{f_n(\tilde{x})}{f_n(1)} \e^{-\lambda_n(\tilde{a}) f_n(\tilde{x})} dx
.
\labelPartII{EYintegral}
\end{align}
By \eqrefPartII{fnx1oversn} and \eqrefPartII{lambdanAsymp}, the integrand in \eqrefPartII{EYintegral} converges pointwise to $\e^{-a\e^{-\gamma}x}$, and satisfies a similar dominated convergence estimates as \eqrefPartII{DifferenceOfIntegrals} by \refcond{fnWeakBounds}.
Hence \eqrefPartII{EYintegral} converges to $a\e^{-\gamma}\int_0^\infty \e^{-a \e^{-\gamma} x} dx = 1$ as claimed.
The proof of \eqrefPartII{EY2Asymp} is similar.

To prove \eqrefPartII{munhatAsymp}, let $b_n$ be defined by $\lambda_n(\tilde{b}_n)=\beta\lambda_n(\tilde{a})$.
 By \eqrefPartII{lambdanAsymp} and monotonicity, it follows that $b_n\to \beta a$ (for if $\limsup_{n \to \infty} b_n\geq (1+\epsilon)\beta a$ then $\limsup_{n \to \infty} f_n(1)\lambda_n(\tilde{b}_n)\geq (1+\epsilon)\beta a\e^{-\gamma}=(1+\epsilon)\lim_{n \to \infty} f_n(1)\beta \lambda_n(\tilde{a})$, a contradiction, and similarly if $\liminf_{n \to \infty} b_n\leq (1-\epsilon)\beta a$).
 But $\hat{\mu}_n(\lambda_n(\tilde{b}_n))^{s_n}=b_n^{-1}$, giving the result.

Since $\lambda_n(\tilde{a}) S_{\floor{s_n t}}$ is non-decreasing and right-continuous, and has i.i.d.\ increments if $s_n t$ is restricted to integer values, it suffices to show that its limiting distribution is $\Gamma(t,1)$ for a fixed $t$, where $\Gamma(t,1)$ denotes a standard Gamma variable with parameter $t$.
For this, we note that its Laplace transform is
\begin{equation}
\E_{\tilde{a}}\left( \e^{-\tau \lambda_n(\tilde{a})S_{\floor{s_n t}}} \right)
=
\bigl(\tilde{a}\hat{\mu}_n\bigl( \lambda_n(\tilde{a})(1+\tau)\bigr)\bigr)^{\floor{s_n t}}
.
\end{equation}
Since $s_n\to\infty$, \eqrefPartII{munhatAsymp} yields that the right-hand side tends to $(1+\tau)^{-t}$.
This is the Laplace transform of a $\Gamma(t,1)$ variable, and thus completes the proof of \refitem{SsnyToGamma}.

For the proof
of part \refitem{Y*ToExp} define $b_n$ by $\lambda_n(\tilde{b}_n)=(1+\tau)\lambda_n(\tilde{a})$ for a given $\tau\ge 0$.
Then \eqrefPartII{nuaSizeBiasedDef} and \eqrefPartII{integralnua} yield
\begin{equation}
\E_{\tilde{a}}(\e^{-\tau \lambda_n(\tilde{a}) \increm^*})
=
\frac{\E_{\tilde{a}}(\increm\e^{-\tau\lambda_n(\tilde{a})\increm})}{\E_{\tilde{a}}(\increm)}
=
\frac{\tilde{a} \int  y \e^{-\lambda_n(\tilde{a})y(1+\tau)}d\mu_n(y)}{\E_{\tilde{a}}(\increm)}
=
\frac{\tilde{a}\E_{\tilde{b}_n}(\increm)}{\tilde{b}_n \E_{\tilde{a}}(\increm)}
.
\end{equation}
By the same argument as in the proof of \eqrefPartII{munhatAsymp}, $b_n \to (1+\tau)a$.
Combining with \eqrefPartII{EYAsymp}, we conclude that $\E_{\tilde{a}}(\e^{-\tau \lambda_n(\tilde{a}) \increm^*})\to(1+\tau)^{-1}$ and $\lambda_n(\tilde{a})\increm^*$ converges to an exponential variable with mean $1$.
So the rest of part \refitem{Y*ToExp} follows from part \refitem{SsnyToGamma}.

The remaining uniformity claims follow from the uniformity in \eqrefPartII{lambdanAsymp}.
The uniformity statements in parts \refitem{SsnyToGamma} and \refitem{Y*ToExp} follow from the observation that the Radon-Nikodym derivatives $d\P_{\tilde{a}}\left( (\lambda_n(\tilde{a})S_{\floor{s_n t}})_{0\leq t\leq K}\in\cdot \right)/d\P_{\tilde{a}'}\left( (\lambda_n(\tilde{a}')S_{\floor{s_n t}})_{0\leq t\leq K}\in\cdot \right)$ (and similarly for $\increm^*$ and $S^*$) are tight for $K<\infty$, uniformly over $a,a'\in A$.
\end{proof}

\begin{lemma}\lblemma{lambdaaDerivative}
For all $a \in (0,\infty)$, $\lambda_n'(a)= 1/(a\E_a(\increm))$.
\end{lemma}

\begin{proof}
Denote $\hat{\mu}_n'(\lambda)=\frac{d}{d\lambda} \hat{\mu}_n(\lambda)$ and  $\hat{\mu}_n''(\lambda)=\frac{d^2}{d\lambda^2} \hat{\mu}_n(\lambda)$.
Using the definitions of $\hat{\mu}_n$ and $\nu_a$, an elementary computation shows that $-a\hat{\mu}'_n(\lambda_n(a))=\E_a(\increm)$.
Moreover, by \eqrefPartII{lambdaaDefn}, $a\hat{\mu}_n(\lambda_n(a))=1$ and the claim follows.
\end{proof}

\begin{proof}[Proof of \reflemma{lambdanAsymp}]
By \eqrefPartII{lambdanAsymp}, $\lambda_n f_n(1) \to \e^{-\gamma}$ and \reflemma{lambdaaDerivative} gives
\begin{equation}\labelPartII{phinIdentity}
\phi_n=\frac{\lambda_n'(1)}{\lambda_n(1)}= \frac{1}{\lambda_n(1) \E_1(\increm)}.
\end{equation}
Now \eqrefPartII{EYAsymp} implies that $\phi_n/s_n \to 1$, as required.
\end{proof}
\begin{coro}\lbcoro{lambdanTaylor}
Uniformly for $a$ in a compact subset of $(0,\infty)$,
\begin{equation}
\lambda_n(a^{1/s_n}) = \lambda_n(1)\left( 1 +\frac{\phi_n}{s_n} \left(a-1\right)+ o(a-1)^2 \right).
\end{equation}
\end{coro}
\begin{proof}
By the same arguments as in the proof of \reflemma{lambdaaDerivative}, the function $F(a)=\lambda_n(a^{1/s_n})$ satisfies $-\tilde{a}\hat{\mu}'_n(F(a))=\E_{\tilde{a}}(\increm)$ and $\tilde{a}\hat{\mu}''_n(F(a))=\E_{\tilde{a}}(\increm^2)$.
By \eqrefPartII{lambdaaDefn},  $a\hat{\mu}_n(F(a))^{s_n}=1$ and we deduce that $s_n \E_{\tilde{a}}(\increm) F'(a)=1/a$ and
	\begin{equation}\labelPartII{lambdanDoublePrime}
	s_n \E_{\tilde{a}}(\increm) F''(a)
	= \frac{s_n \E_{\tilde{a}}(\increm^2)}{\bigl(a s_n \E_{\tilde{a}}(\increm)\bigr)^2} - \frac{1+\frac{1}{s_n}}{a^2}.
	\end{equation}
The right-hand side of \eqrefPartII{lambdanDoublePrime} converges uniformly to $0$ by \eqrefPartII{EYAsymp}--\eqrefPartII{EY2Asymp}.
Applying \eqrefPartII{EYAsymp} again and noting from \eqrefPartII{lambdanAsymp} that $\lambda_n(\tilde{a})/\lambda_n(1)$ is bounded, it follows that $F''(a)/\lambda_n(1)$ converges uniformly to $0$.
A Taylor expansion of $F$ around the point $1$ and \eqrefPartII{phinIdentity} complete the proof.
\end{proof}

The following lemma is a consequence of \refthm{ConvRW} and will be used in the proof of the second moment statements of \refthm{OneCharConv}:
\begin{lemma}\lblemma{CovarianceSums}
Assume the hypotheses of \refthm{ConvRW}, and let $A\subset(0,2)$ be compact.
For any measurable, bounded function $h \ge 0$, set
\begin{equation}
\CovSumsGen(h)=\E_{a^{1/s_n}b^{1/s_n}}\Big( \sum_{j=0}^\infty \e^{-\left[\lambda_n(a^{1/s_n})+\lambda_n(b^{1/s_n})-\lambda_n(a^{1/s_n}b^{1/s_n})\right]S_j} h(S_j) \Big).
\end{equation}
There are constants $K<\infty$ and $n_0 \in \N$ independent of $h$ such that $\CovSumsGen(h)\leq Ks_n \norm{h}_\infty$ for all $n\ge n_0$ and $a,b\in A$.
Moreover, for any $\epsilon>0$ there are constants $K'<\infty$, $n_0' \in \N$ independent of $h$ such that for all $a,b\in A$, $n\ge n_0'$,
\begin{equation}\labelPartII{CovarianceSumsConclusion}
-\epsilon \norm{h}_\infty+\frac{\inf\set{h(y)\colon \lambda_n(1)y\leq K'}}{\log(1/a+1/b)}
\leq \frac{\CovSumsGen(h)}{s_n} \leq  \epsilon \norm{h}_\infty+\frac{\sup\set{h(y)\colon \lambda_n(1)y\leq K'}}{\log(1/a+1/b)}
.
\end{equation}
\end{lemma}

Note that $\log(1/a+1/b)$ is positive and bounded away from $0$ by our assumption on $A$.

\begin{proof}[Proof of \reflemma{CovarianceSums}]
We rewrite $\CovSumsGen(h)$ in integral form, bound $h$ by its maximum, and use \eqrefPartII{integralnua} and the definition of $\hat{\mu}_n$, to obtain
\begin{align}
\CovSumsGen(h)&=s_n \int_0^\infty \E_{\tilde{a}\tilde{b}}\left( \e^{-[\lambda_n(\tilde{a})+\lambda_n(\tilde{b})-\lambda_n(\tilde{a}\tilde{b})]S_{\floor{s_nt}}} h(S_{\floor{s_n t}}) \right) dt
\labelPartII{OneCharCovInt}\\
&\le
s_n \norm{h}_\infty \int_0^\infty \left[\tilde{a}\tilde{b}\hat{\mu}_n \! \left(\lambda_n(\tilde{a})+\lambda_n(\tilde{b}) \right)\right]^{\floor{s_n t}} dt
.
\labelPartII{OneCharCovBound}
\end{align}
By \eqrefPartII{lambdanAsymp} and \eqrefPartII{munhatAsymp}, we deduce $(\lambda_n(\tilde{a})+\lambda_n(\tilde{b}))/\lambda_n(1) \to (a+b)$ and
\begin{equation}\labelPartII{munhatOfSum}
\big[\tilde{a}\tilde{b}\hat{\mu}_n(\lambda_n(\tilde{a})+\lambda_n(\tilde{b}))\big]^{s_n}\to ab/(a+b)
\end{equation}
Since $\log((a+b)/ab)=\log(1/a+1/b)$ is positive and uniformly bounded away from $0$, the integral in \eqrefPartII{OneCharCovBound} is uniformly bounded for sufficiently large $n$ and we conclude that there is a constant $K<\infty$ with $\CovSumsGen(h) \leq K s_n \norm{h}_\infty$.

For \eqrefPartII{CovarianceSumsConclusion}, let $\epsilon>0$ be given.
Since $A\subset (0,2)$ is compact, \eqrefPartII{lambdanAsymp} implies that there exists $\delta>0$ and $n_0' \in \N$ such that $\lambda_n(\tilde{a})+\lambda_n(\tilde{b})-\lambda_n(\tilde{a}\tilde{b}) \ge \delta f_n(1)$ for all $n\ge n_0'$ and $a,b\in A$.
Using again \eqrefPartII{munhatOfSum}, we may take $t_0$ and $n_0'$ sufficiently large that $\int_{t_0}^\infty [\tilde{a}\tilde{b} \hat{\mu}_n(\lambda_n(\tilde{a})+\lambda_n(\tilde{b}))]^{\floor{s_n t}} dt \leq \tfrac{1}{3}\epsilon$ and $\abs{\int_0^{t_0} [\tilde{a}\tilde{b}\hat{\mu}_n(\lambda_n(\tilde{a})+\lambda_n(\tilde{b}))]^{\floor{s_n t}} dt - 1/\log(1/a+1/b)}\leq\tfrac{1}{3}\epsilon$ for all $n\geq n_0'$, $a,b\in A$.
Furthermore, \refthm{ConvRW}\refitem{SsnyToGamma} implies that the family of laws $\P_{\tilde{a}\tilde{b}}(\lambda_n(1)S_{\floor{s_n t}}\in\cdot)$, $t\leq t_0$, $a,b\in A$, is tight.
Hence we may take $K'$ large enough that $t_0\e^{-\delta K'}\leq 1$ and $\P_{\tilde{a}\tilde{b}}(\lambda_n(1)S_{\floor{s_n t}} > K')\leq \tfrac{1}{3}\epsilon$, uniformly for $t\leq t_0$.
We conclude from \eqrefPartII{OneCharCovInt} that
\begin{align}
&\int_0^{t_0} \left[ \inf\set{h(y)\colon \lambda_n(1)y\leq K'}\E_{\tilde{a}\tilde{b}}\left( \e^{-[\lambda_n(\tilde{a})+\lambda_n(\tilde{b})-\lambda_n(\tilde{a}\tilde{b})]S_{\floor{s_n t}}}  \right) - \tfrac{1}{3}\epsilon \e^{-\delta K'} \norm{h}_\infty \right] dt
\notag\\
&\quad\leq
\frac{\CovSumsGen(h)}{s_n} \leq \int_0^{t_0} \left[ \sup\set{h(y)\colon \lambda_n(1)y\leq K'}\E_{\tilde{a}\tilde{b}}\left( \e^{-[\lambda_n(\tilde{a})+\lambda_n(\tilde{b})-\lambda_n(\tilde{a}\tilde{b})]S_{\floor{s_n t}}}  \right) + \tfrac{1}{3}\epsilon\e^{-\delta K'} \norm{h}_\infty \right] dt
\notag\\
&\quad\qquad\qquad+
\int_{t_0}^\infty \norm{h}_\infty \E_{\tilde{a}\tilde{b}}\left( \e^{-(\lambda_n(\tilde{a})+\lambda_n(\tilde{b})-\lambda_n(\tilde{a}\tilde{b}))S_{\floor{s_n t}}} \right) dt
\labelPartII{OneCharCovConvBound}
\end{align}
for $n\geq n_0'$.
Using again $\E_{\tilde{a}\tilde{b}}(\e^{-S_{\floor{s_n t}}(\lambda_n(\tilde{a})+\lambda_n(\tilde{b})-\lambda_n(\tilde{a}\tilde{b}))})=[\tilde{a}\tilde{b}\hat{\mu}_n(\lambda_n(\tilde{a})+\lambda_n(\tilde{b}))]^{\floor{s_n t}}$, we see that the hypotheses on $t_0$ and $n_0'$ imply \eqrefPartII{CovarianceSumsConclusion}.
\end{proof}

\subsection{Means of one-vertex characteristics: Proof of \refthm{OneCharConv}}\lbsubsect{LimitingBounds}
In this section, we prove \refthm{OneCharConv}.
Further, we set the stage for the proofs of Theorems \refPartII{t:TwoVertexConvRestrictedSum} and \refPartII{t:TwoVertexRemainder} in \refsect{2VertexCharSec}.

Recall from \eqrefPartII{mtchiaRW} that
\begin{equation*}
	\bar{m}_t^\chi(a) = \E_a\Big( \sum_{j=0}^\infty \e^{-\lambda_n(a)(t-S_j)}\chi(t-S_j) \Big).
\end{equation*}
Thus, $\bar{m}_t^\chi(a)$ can be understood as the expected integral of $\e^{-\lambda_n(a)t}\chi(t)$ against the counting measure on the random set $\set{t-S_j\colon j\in \N_0}$.
When $t\to\infty$, this measure approaches its stationary equivalent, which is the counting measure on the point-stationary set $\set{t-S^*_j\colon j\in\N_0}$ (see \cite{Thor00}).
Since the expected value of this stationary measure is a multiple of the Lebesgue measure, $\bar{m}_t^\chi(a)$ will approach (as $t\to \infty$) the same multiple of $\int_0^\infty \e^{-\lambda_n(a)t} \chi(t) dt$.
In the following we make this more precise.
We begin with general estimates that apply to any CTBP with any intensity measure $\mu$, and we will write simply $\lambda(a)$ for the parameter defined by the analogue of \eqrefPartII{lambdaaDefn}.
Similarly, all other notation introduced for $\BP^{\sss(1)}$ will be used for a general CTBP.

\begin{prop}\lbprop{PointStationaryProp}
Let $(S_j^*)_j$ be the random walk defined in \eqrefPartII{DefSjStar}.
Let $\chi$ be a non-negative characteristic.
Then, for all $a,t> 0$,
\begin{equation}\labelPartII{mtchiaStatRW}
\E_a(\bar{m}^\chi_{t-U\increm^*}(a))=\E_a\Big( \sum_{j=0}^\infty \e^{-\lambda(a)(t-S^*_j)}\chi(t-S^*_j) \Big)=\frac{\int_0^t \e^{-\lambda(a)u}
 \chi(u)\, du}{\E_a(\increm)}.
\end{equation}
\end{prop}
\begin{proof}
The first equality is \eqrefPartII{mtchiaRW}; the second follows because the set $\set{t-S^*_j\colon j \in \N_0}$ is point-stationary in the sense of \cite{Thor00}.
Alternatively, the equality may be verified by taking Laplace transforms with respect to $t$.
\end{proof}

In \eqrefPartII{mtchiaRW} and \eqrefPartII{mtchiaStatRW}, we may collapse the tail of the sum into a single value of $\bar{m}_u^\chi(a)$.
Namely, if $J$ is a stopping time for $(S_j)_j$ or $(S^*_j)_j$, respectively, then by the strong Markov property
\begin{align}
\labelPartII{barmWithStopping}
\begin{split}
\bar{m}_t^\chi(a)&=\E_a\Big( \bar{m}_{t-S_J}^\chi(a) +\sum_{j=0}^{J-1} \e^{-\lambda(a)(t-S_j)} \chi(t-S_j) \Big),
\\
\frac{\int_0^t \e^{-\lambda(a)u} \chi(u)\, du}{\E_a(\increm)} &=\E_a\Big( \bar{m}_{t-S^*_J}^\chi(a) +\sum_{j=0}^{J-1} \e^{-\lambda(a)(t-S^*_j)} \chi(t-S^*_j) \Big).
\end{split}
\end{align}
The following lemmas provide bounds on $\bar{m}_t^\chi(a)$ when $m_t^\chi(a)$ is non-decreasing in $t$:
\begin{lemma}\lblemma{CharsByCoupling}
Suppose $\chi$ is a non-negative characteristic such that $m_t^\chi(a)$ is non-decreasing in $t$ (in particular, this holds if $\chi$ is non-decreasing).
Let $(S_j)_j$ be as in \eqrefPartII{mtchiaRW} and $(S^*_j)_j$ as in \eqrefPartII{mtchiaStatRW}, and suppose that $(S_j)_j$ and $(S_j^*)_j$ are independent.
Let $\eps>0$ and set $J=\inf\set{j\colon \abs{S_j-S^*_j}\leq\epsilon}$.
Then, for $a,t>0$,
\begin{multline}
\e^{-2\lambda(a)\epsilon} \frac{\int_0^{t-\epsilon} \e^{-\lambda(a)u} \chi(u)\, du}{\E_a(\increm)} -\E_a\Big( \sum_{j=0}^{J-1} \e^{-\lambda(a)(t-\epsilon-S^*_j)} \chi(t-\epsilon-S^*_j) \Big)
\\
\leq
\bar{m}_t^\chi(a)
\leq
\e^{2\lambda(a)\epsilon} \frac{\int_0^{t+\epsilon} \e^{-\lambda(a)u} \chi(u)\, du}{\E_a(\increm)} +\E_a\Big( \sum_{j=0}^{J-1} \e^{-\lambda(a)(t-S_j)} \chi(t-S_j) \Big).
\end{multline}
\end{lemma}
\begin{proof}
The hypotheses imply $t-\epsilon-S^*_J \leq t-S_J \leq t+\epsilon-S^*_J$ and therefore $\e^{-2\lambda(a)\epsilon} \bar{m}_{t-\epsilon-S^*_J}^\chi(a)\leq \bar{m}_{t-S_J}^\chi(a)\leq \e^{2\lambda(a)\epsilon} \bar{m}_{t+\epsilon-S^*_J}^\chi(a)$.
Combining with \eqrefPartII{barmWithStopping} gives the result.
\end{proof}
\begin{lemma}\lblemma{mUniformBound}
Suppose $\chi$ is a non-negative characteristic such that $m_t^\chi(a)$ is non-decreasing in $t$.
Then, for all $a,t> 0$ and $K>0$,
\begin{equation}
\bar{m}_t^\chi(a) \leq \frac{\e^K}{\E_a\left( \e^{\lambda(a)S^*_0} \indicator{\lambda(a) S^*_0 \leq K} \right)} \frac{\int_0^\infty \e^{-\lambda(a)u} \chi(u) \, du}{\E_a(\increm)}.
\end{equation}
\end{lemma}
\begin{proof}
On $\set{\lambda(a) S^*_0 \leq K}$ we have $\bar{m}_{t+K/\lambda(a)-S^*_0}^\chi(a)\geq \e^{-K} \e^{\lambda(a) S^*_0} \bar{m}_t^\chi(a)$.
Apply \eqrefPartII{mtchiaStatRW} and replace the limit of integration by $\infty$ to obtain the result.
\end{proof}
\begin{lemma}\lblemma{SquaredSumBound}
Let $\chi$ be a non-negative, non-decreasing characteristic such that $\int_0^\infty \e^{-\lambda(a)u} \chi(u) \, du <\infty$, and fix $a,K>0$.
Then, for all $t>0$, $\sum_{j=0}^\infty \e^{-\lambda(a)(t-S_j)} \chi(t-S_j)$ is square-integrable under $\E_a$ and, abbreviating $C_{a,K}=\e^K/\E_a(\e^{\lambda(a)S^*_0}\indicator{\lambda(a) S^*_0\leq K})$,
\begin{equation}\labelPartII{SquaredSumBound}
\E_a \biggl( \Bigl( \sum_{j=0}^\infty \e^{-\lambda(a)(t-S_j)} \chi(t-S_j) \Bigr)^2 \biggr)
\leq
C_{a,K} \frac{\int_0^\infty \e^{-2\lambda(a)u} \chi(u)^2 du}{\E_a(\increm)}+2C_{a,K}^2\frac{(\int_0^\infty \e^{-\lambda(a)u}\chi(u) du)^2}{\E_a(\increm)^2}.
\end{equation}
The same bound holds with $(S_j)_j$ replaced by $(S^*_j)_j$.
\end{lemma}
\begin{proof}
Since $\chi$ is non-decreasing, $\int_0^\infty \e^{-\lambda(a)u} \chi(u) du <\infty$ implies that $\e^{-\lambda(a)u} \chi(u)$ must be bounded.
Hence $\int_0^\infty \e^{-2\lambda(a)u} \chi(u)^2 du <\infty$ also.
Applying \reflemma{mUniformBound} to $\chi$ and $\chi^2$, we deduce
\begin{align}
\E_a &\left( \Big(\sum_{j=0}^\infty \e^{-\lambda(a)(t-S_j)} \chi(t-S_j)\Big)^2 \right)\notag
\\
&=
\E_a\bigg( \sum_{j=0}^\infty \e^{-2\lambda(a)(t-S_j)} \chi(t-S_j)^2\bigg)  + 2\E_a\bigg( \sum_{j=0}^\infty \e^{-\lambda(a)(t-S_j)} \chi(t-S_j) \sum_{k=j+1}^\infty \e^{-\lambda(a)(t-S_k)} \chi(t-S_k)\bigg)
\notag\\
&=
\bar{m}_t^{\chi^2}(a) + 2\E_a\bigg(\sum_{j=0}^\infty \e^{-\lambda(a)(t-S_j)} \chi(t-S_j) \bar{m}_{t-S_{j+1}}^\chi(a)\bigg)
\notag\\
&\leq
C_{a,K} \frac{\int_0^\infty \e^{-\lambda(a)u} \chi(u)^2 du}{\E_a(\increm)}+2C_{a,K}\frac{\int_0^\infty \e^{-\lambda(a)u} \chi(u)du}{\E_a(\increm)}\E_a\Big(\sum_{j=0}^\infty \e^{-\lambda(a)(t-S_j)}\chi(t-S_j)\Big).
\end{align}
Another application of \reflemma{mUniformBound} gives \eqrefPartII{SquaredSumBound}.
Finally replacing $(S_j)_j$ by $(S^*_j)_j$ is equivalent to replacing $t$ by $t-U\increm^*$.
Since the upper bound in \eqrefPartII{SquaredSumBound} does not depend on $t$, the result follows.
\end{proof}

We now specialise to the offspring distribution $\mu_n$ and apply the convergence results of \refthm{ConvRW} to prove \refthm{OneCharConv}:

\begin{proof}[Proof of \refthm{OneCharConv}]
By \reflemma{mUniformBound}, for all $a,t>0$,
\begin{equation}
\bar{m}^\chi_t(\tilde{a})\leq s_n \frac{\e^1}{\P_{\tilde{a}}(\lambda_n(\tilde{a})S_0^*\leq 1)}\frac{\int_0^\infty \lambda_n(\tilde{a}) \e^{-\lambda_n(\tilde{a}) u} \norm{\chi}_\infty du}{s_n \lambda_n(\tilde{a})\E_{\tilde{a}}(\increm)},
\end{equation}
and, by \refthm{ConvRW}, $\P_{\tilde{a}}(\lambda_n(\tilde{a})S_0^*\leq 1) \to \P(UE \le 1)$ and $s_n \lambda_n(\tilde{a}) \E_{\tilde{a}}(\increm) \to 1$ uniformly in $a \in A$.
Hence, the uniform bound for $\bar{m}_t^\chi(\tilde{a})$ follows.
By the same reasoning, \reflemma{SquaredSumBound} yields a constant $C<\infty$ such that
\begin{equation}\labelPartII{SquaredSumsn2Bound}
\E_{\tilde{a}}\biggl(\Bigl(\sum_{j=0}^\infty \e^{-\lambda_n(\tilde{a})(t-S_j)} \chi(t-S_j)\Bigr)^2\biggr)\leq C s_n^2 \norm{\chi}_\infty^2
,
\end{equation}
an estimate that will be needed shortly.

For \eqrefPartII{OneCharConvFormula}, fix $\epsilon>0$.
Apply \reflemma{CharsByCoupling} with $\epsilon$ and $a$ replaced by $\tilde{\epsilon}=\lambda_n(\tilde{a})^{-1}\epsilon$ and $\tilde{a}$, with the stopping time $J_n=\inf\set{j\colon \abs{S_j - S^*_j}\leq \tilde{\epsilon}} =\inf\set{j\colon \lambda_n(\tilde{a})\abs{S_j-S^*_j}\leq \epsilon}$.
By \eqrefPartII{lambdanAsymp}, we may choose $K$ large enough that $\int_{t-\tilde{\epsilon}}^\infty \lambda_n(\tilde{a})\e^{-\lambda_n(\tilde{a})z} \chi(z) dz \leq \norm{\chi}_\infty \epsilon$ whenever $\lambda_n(1)t\geq K$.
By \eqrefPartII{EYAsymp}, it follows that the first terms in the upper and lower bounds of \reflemma{CharsByCoupling} are $s_n\int_0^\infty \lambda_n(\tilde{a})\e^{-\lambda_n(\tilde{a})z} \chi(z)dz + O(\epsilon)s_n \norm{\chi}_\infty$.
Therefore it is enough to show that the error term $\E_{\tilde{a}}(\sum_{j=0}^{J_n-1} \e^{-\lambda_n(\tilde{a})(t-\tilde{\epsilon}-S^*_j)} \chi(t-\tilde{\epsilon}-S^*_j))$ is also $O(\epsilon)s_n \norm{\chi}_\infty$ for $\lambda_n(1)t$ sufficiently large, uniformly in $a\in A$ (the same proof will work for the term with $S_j$).

To prove this, observe that the variables $(J_n/s_n, \lambda_n(\tilde{a}) S^*_{J_n})_{n\in\N,a\in A}$ are tight.
Indeed, the rescaled processes $(\lambda_n(\tilde{a})S_{\floor{s_n t}})_{t\ge 0},( \lambda_n(\tilde{a})S^*_{\floor{s_n t}})_{t\ge 0}$ converge by \refthm{ConvRW} to independent Gamma processes (with different initial conditions).
These limiting processes approach to within $\epsilon/2$ at some random but finite time, and tightness follows.

Thus we may choose $C'$ large enough that the event $\mathcal{A}=\set{J_n\leq C' s_n}\union\set{\lambda_n(\tilde{a})S^*_{J_n}\leq C'}$ has $\P_{\tilde{a}}(\mathcal{A}^c)\leq \epsilon^2$.
Using the Cauchy-Schwarz inequality and \eqrefPartII{SquaredSumsn2Bound},
\begin{align}
&\E_{\tilde{a}}\Big( \indicatorofset{\mathcal{A}^c} \sum_{j=0}^{J_n-1} \e^{-\lambda_n(\tilde{a})(t-\tilde{\epsilon}-S^*_j)} \chi(t-\tilde{\epsilon}-S^*_j) \Big)
\notag\\&\quad
\leq
\left[ \P_{\tilde{a}}(\mathcal{A}^c) \E_{\tilde{a}} \! \left( \biggl( \sum_{j=0}^\infty \e^{-\lambda_n(\tilde{a})(t-\tilde{\epsilon}-S^*_j)} \chi(t-\tilde{\epsilon}-S^*_j) \biggr)^2 \right) \right]^{1/2}
\leq
\sqrt{C} \epsilon s_n \norm{\chi}_\infty,
\end{align}
whereas
\begin{align}
\E_{\tilde{a}}\Big( \indicatorofset{\mathcal{A}} \sum_{j=0}^{J_n-1} \e^{-\lambda_n(\tilde{a})(t-\tilde{\epsilon}-S^*_j)} \chi(t-\tilde{\epsilon}-S^*_j) \Big)
&\leq
C' s_n \norm{\chi}_\infty \e^{-\lambda_n(\tilde{a})(t-\tilde{\epsilon})+C'}.
\end{align}
By \eqrefPartII{lambdanAsymp}, the right-hand side is at most $\epsilon s_n\norm{\chi}_\infty$, uniformly over $a\in A$, if $\lambda_n(1) t \ge K$ with $K$ sufficiently large.
This completes the proof of \eqrefPartII{OneCharConvFormula}.

We turn to the estimates for $\bar{M}_{t,u}^{\chi,\eta}(\tilde{a},\tilde{b})$.
In view of \eqrefPartII{MtuchietaabRW}, apply \reflemma{CovarianceSums} to $h(y)=\bar{m}_{t-y}^\chi(\tilde{a})\bar{m}_{u-y}^\eta(\tilde{b})$.
By the first part of the current proof, $\norm{h}_{\infty}=O(s_n^2)\norm{\chi}_\infty \norm{\eta}_\infty$  and for any $\eps>0$ we can make the infimum and supremum in \eqrefPartII{CovarianceSumsConclusion} differ from $s_n^2 \textstyle{\int_0^\infty} \e^{-z}\chi\bigl( z/\lambda_n(\tilde{a}) \bigr) dz \cdot \textstyle{\int_0^\infty} \e^{-w}\eta\bigl( w/\lambda_n(\tilde{b}) \bigr) dw$ by at most $\epsilon s_n^2$, by taking $\lambda_n(1)[t\wedge u]$ large enough.
\end{proof}

\section{The two-vertex characteristic}\lbsect{2VertexCharSec}

In view of Theorems~\refPartII{t:CollCondProb} and \refPartII{t:CouplingFPPCollision}, we wish to consider generation-weighted two-vertex characteristics of the form
\begin{equation}
z_{\vec{t}}^\chi(\vec{a}) = \sum_{v_1\in\BP_{t_1}^{(1)}} \sum_{v_2\in\BP_{t_2}^{(2)}} a_1^{|v_1|} a_2^{|v_2|} \chi(t_1-T_{v_1},t_2-T_{v_2})
\end{equation}
for $\chi(\vec{t})=\chi_n(\vec{t})=\mu_n(\abs{t_1-t_2},t_1+t_2)$ defined in \eqrefPartII{TwoVertexChar}.
As discussed in \refsubsect{TwoVerChar}, we split $\chi_n$ into $\chi_n^{\sss(K)}$ and $\chi_n-\chi_n^{\sss(K)}$ for some $K\in (0,\infty)$.

Regarding $t_1$ as fixed,
	\begin{equation}\labelPartII{DiffOfIncr}
	s_n \chi_n^{\sss(K)}(t_1,t_2)=
	s_n \mu_n^{\sss(K)}(t_1-t_2,t_1+t_2)
	- \indicator{t_2\geq t_1} s_n \mu_n^{\sss(K)}(t_1-t_2,t_2-t_1)
	\end{equation}
expresses the one-vertex characteristic $s_n\chi_n^{\sss(K)}(t_1,\cdot)$ as the difference of two uniformly bounded, non-negative, non-decreasing functions.

We extend our abbreviations from \eqrefPartII{tildeConvention} to vectors and write
\begin{equation}\labelPartII{tildeVectorValued}
\vec{\tilde{a}}=\vec{a}^{1/s_n} =(\tilde{a}_1,\tilde{a}_2) =(a_1^{1/s_n}, a_2^{1/s_n}) \qquad \text{etc.}
\end{equation}

\subsection{Truncated two-vertex characteristic: Proof of \refthm{TwoVertexConvRestrictedSum}}\lbsubsect{TwoVertexConvRestrictedPf}

In this section, we prove \refthm{TwoVertexConvRestrictedSum}.
For any two-vertex characteristic $\chi$, note that $z_{\vec{t}}^\chi(\vec{a})$ can be written in terms of two one-vertex characteristics as follows:
\begin{equation}\labelPartII{IteratedChar}
z_{\vec{t}}^\chi(\vec{a})=z_{t_1}^{\rho', \BP^{(1)}}(a_1), \qquad\text{where}\qquad \rho'(t'_1)=\sum_{v_2\in\BP_{t_2}^{(2)}} a_2^{|v_2|} \chi(t'_1,t_2-T_{v_2})=z_{t_2}^{\chi(t'_1,\cdot), \BP^{(2)}}(a_2).
\end{equation}
Similarly, we may evaluate the two-vertex mean $\bar{m}_{\vec{t}}^{\chi_n^{\sss(K)}}(\vec{a})$ via two one-vertex means:
\begin{equation}\labelPartII{IteratedCharMeans}
\bar{m}_{\vec{t}}^{\chi_n^{\sss(K)}}(\vec{\tilde{a}}) = \bar{m}_{t_1}^{\rho_{t_2,\tilde{a}_2}}(\tilde{a}_1),
\qquad\text{where}\qquad
\rho_{t_2,\tilde{a}_2}(t'_1) = \bar{m}_{t_2}^{\chi_n^{\sss(K)}(t'_1,\cdot)}(\tilde{a}_2).
\end{equation}
For this reason we first give estimates for the one-vertex characteristic $\chi_n^{\sss(K)}(t_1,\cdot)$ uniformly in $t_1$.
In their statements, we rely on the function $\zeta$ given in \eqrefPartII{zetaDef} and on the following definition:
	\begin{equation}\labelPartII{IzDefn}
	I(z)=\int_0^\infty \left( \e^{- \abs{y-z}} - \e^{-(y+z)} \right) \frac{dy}{y}.
	\end{equation}

\begin{prop}\lbprop{HalfTwoVertexConvRestricted}
Assume the hypotheses of \refthm{ConvRW}.
For every $\epsilon>0$ and for every compact subset $A\subset(0,2)$, there is a constant $K_0<\infty$ such that for every $K\ge K_0$ there exist constants $K'<\infty$ and $n_0 \in \N$ such that for $n \ge n_0$, $a_1,a_2,b_1,b_2\in A$ and $t'_1\geq 0$,
\begin{equation}
\begin{aligned}
\abs{\bar{m}_{t_2}^{\chi_n^{\sss(K)}(t'_1,\cdot)}(\tilde{a}_2)-I\bigl( \lambda_n(\tilde{a}_2)t'_1 \bigr)}
&\leq \epsilon,
\\
\abs{s_n^{-3}\bar{M}_{t_1,u_1}^{\rho_{t_2,\tilde{a}_2},\rho_{u_2,\tilde{b}_2}}(\tilde{a}_1,\tilde{b}_1) - \frac{\zeta(a_2/a_1)\zeta(b_2/b_1)}{\log(1/a_1+1/b_1)}}
&\leq \epsilon,
\end{aligned}
\qquad\text{if} \;
\lambda_n(1)[t_1\wedge t_2 \wedge u_1\wedge u_2] \geq K'.
\end{equation}
Moreover, for every $K<\infty$, there are constants $K''<\infty$ and $n_0' \in \N$ such that
\begin{equation}
\bar{m}_{t_2}^{\chi_n^{\sss(K)}(t_1,\cdot)}(\tilde{a}_2)\leq K'', \qquad \bar{M}_{t_1,u_1}^{\rho_{t_2,\tilde{a}_2},\rho_{u_2,\tilde{b}_2}}(\tilde{a}_1,\tilde{b}_1) \leq K'' s_n^3
\end{equation}
for all $n\ge n_0'$, $t_1,t_2,u_1,u_2\geq 0$ and $a_1,a_2,b_1,b_2\in A$.
\end{prop}

Note that $\bar{m}_{t_2}^{\chi_n^{\sss(K)}(t_1,\cdot)}(\tilde{a}_2)$ is asymptotically constant, instead of growing like $s_n$ as in \refthm{OneCharConv}.
This reflects the fact that $\chi_n^{\sss(K)}$ itself is typically of scale $1/s_n$.

\begin{proof}[Proof of \refprop{HalfTwoVertexConvRestricted}]
The integrand in \eqrefPartII{IzDefn} can be bounded by $2\e^{-y+1}$ if $z\leq 1$ and by $\indicator{y\leq 1}2\e^{-z+1}+\indicator{y>1}\e^{-\abs{y-z}}/y$ if $z\geq 1$.
It follows that we may choose $K_0<\infty$ sufficiently large that $\abs{\int_{a\e^{-K-\gamma}}^{a\e^{K-\gamma}}(\e^{- \abs{y-z}} - \e^{-(y+z)}) dy/y - I(z)}<\tfrac{1}{3}\epsilon$, for all $z\geq 0$, $a\in A$, $K \ge K_0$.
Here, $\gamma$ denotes again Euler's constant.

Applying \refthm{OneCharConv} to each of the uniformly bounded, non-negative, non-decreasing functions in \eqrefPartII{DiffOfIncr}, we conclude that for every $K<\infty$, $\bar{m}_{t_2}^{\chi_n^{\sss(K)}(t'_1,\cdot)}(\tilde{a}_2)$ is uniformly bounded and is within $\tfrac{1}{3}\epsilon$ of $\int_0^\infty s_n\lambda_n(\tilde{a}_2)\e^{-\lambda_n(\tilde{a}_2)t}\chi_n^{\sss(K)}(t'_1,t) \, dt$ if $\lambda_n(1)t_2$ is sufficiently large, uniformly over $a_2\in A$.
Use Fubini's Theorem and \eqrefPartII{munFromfn}, write $z=\lambda_n(\tilde{a}_2)t'_1$ and substitute $\tilde{x}=x^{1/s_n}$ to compute
\begin{align}
&\int_0^\infty s_n\lambda_n(\tilde{a}_2)\e^{-\lambda_n(\tilde{a}_2)t}\chi_n^{\sss(K)}(t'_1,t) \, dt
=s_n \int d\mu_n^{\sss(K)}(y) \int_0^\infty \indicator{\abs{y-t'_1}\leq t\leq y+t'_1} \lambda_n(\tilde{a}_2)\e^{-\lambda_n(\tilde{a}_2)t}  \, dt
\notag\\
&\quad=
s_n \int_{1-K/s_n}^{1+K/s_n} \left( \e^{-\lambda_n(\tilde{a}_2)\abs{f_n(\tilde{x})-t_1'}}-\e^{-\lambda_n(\tilde{a}_2)(f_n(\tilde{x})+t_1')} \right) d\tilde{x}
\notag\\
&\quad=
\int_{(1-K/s_n)^{s_n}}^{(1+K/s_n)^{s_n}} \left( \e^{-\shortabs{\lambda_n(\tilde{a}_2)f_n(\tilde{x})-z}}-\e^{-(\lambda_n(\tilde{a}_2)f_n(\tilde{x})+z)} \right) x^{1/s_n-1} dx.
\labelPartII{IntegralchiKn}
\end{align}
Monotonicity, \refcond{scalingfn} and \eqrefPartII{lambdanAsymp} imply that $\lambda_n(\tilde{a}_2)f_n(\tilde{x})\to a_2 \e^{-\gamma}x$ as $n\to\infty$, uniformly over $x\in[(1-K/s_n)^{s_n},(1+K/s_n)^{s_n}]$ and $a_2\in A$.
Hence the integral in \eqrefPartII{IntegralchiKn} is within $\tfrac{1}{3}\epsilon$ of $\int_{\e^{-K}}^{\e^K} (\e^{-\shortabs{a_2\e^{-\gamma}x-z}}-\e^{-(a_2\e^{-\gamma}x+z)})dx/x$ for $n$ sufficiently large, uniformly in $z\geq 0$ and $a_2\in A$.
Substituting $y=a_2\e^{-\gamma}x$ and using $K\ge K_0$, we obtain the desired statements for $\bar{m}_{t_2}^{\chi_n^{\sss(K)}(t'_1,\cdot)}(\tilde{a}_2)$.

For $\bar{M}_{t_1,u_1}^{\rho_{t_2,\tilde{a}_2},\rho_{u_2,\tilde{b}_2}}(\tilde{a}_1,\tilde{b}_1)$, the statements for $\bar{m}_{t_2}^{\chi_n^{\sss(K)}(t'_1,\cdot)}(\tilde{a}_2)$ can be interpreted to say that the characteristics $\rho_{t_2,\tilde{a}_2}(\cdot),\rho_{u_2,\tilde{b}_2}(\cdot)$ are uniformly bounded and lie within $\tfrac{1}{2}\epsilon$ of the characteristics $I\!\left( \lambda_n(\tilde{a}_2) \, \cdot \right)$, $I\!\bigl( \lambda_n(\tilde{b}_2) \, \cdot \bigr)$ if $\lambda_n(1)t_2,\lambda_n(1)u_2$ are sufficiently large.
It is readily verified that $I(z)$ can be written as the difference of two bounded, non-negative, non-decreasing functions.
We may therefore apply \refthm{OneCharConv} to these characteristics.
A calculation shows that
\begin{equation}\labelPartII{zetaFromI}
\int_0^\infty \e^{-z}I(rz)dz=\zeta(r),
\qquad r=\lambda_n(\tilde{a}_2)/\lambda_n(\tilde{a}_1),
\end{equation}
where $\zeta$ is defined in \eqrefPartII{zetaDef}.
By \eqrefPartII{lambdanAsymp}, we have $r\to a_2/a_1$ uniformly over $a_1,a_2 \in A$; since $\zeta$ is continuous, this completes the proof.
\end{proof}
\begin{proof}[Proof of \refthm{TwoVertexConvRestrictedSum}]
Interchanging the roles of $t_1$ and $t_2$ in \eqrefPartII{DiffOfIncr} and using \refprop{HalfTwoVertexConvRestricted}, we can write $\rho_{t_2, \tilde{a}_2}$ in \eqrefPartII{IteratedCharMeans} as the difference of two bounded, non-negative, non-decreasing functions and \refthm{OneCharConv} yields that $\frac{1}{s_n}\bar{m}_{\vec{t}}^{\chi_n^{\sss (K)}}(\vec{\tilde{a}})$ is bounded.
To show \eqrefPartII{CharMain1stMomConv}, \refprop{HalfTwoVertexConvRestricted} allows us to replace $\rho_{t_2, \tilde{a}_2}$ in \eqrefPartII{IteratedCharMeans} by $I(\lambda_n(\tilde{a}_2\cdot))$, making an error of at most $\eps s_n$.
Since $I$ can be written as the difference of two bounded, non-negative, non-decreasing functions, \refthm{OneCharConv}, \eqrefPartII{zetaFromI} and the fact that $\zeta(r)\to \zeta(a_2/a_1)$ uniformly, yield the claim.

For $\bar{M}_{\vec{t},\vec{u}}^{\chi_n^{\sss(K)},\chi_n^{\sss(K)}}(\vec{\tilde{a}},\vec{\tilde{b}})$, use \eqrefPartII{IteratedChar} to obtain (similarly to \eqrefPartII{MtuchietaabRW})
\begin{align}
\bar{M}_{\vec{t},\vec{u}}^{\chi,\eta}(\vec{a},\vec{b})
=
\E_{a_1 b_1, a_2 b_2} \Biggl( \sum_{j_1=0}^\infty\sum_{j_2=0}^\infty &\e^{-[\lambda_n(a_1)+\lambda_n(b_1)-\lambda_n(a_1 b_1)] S_{j_1}^{\sss(1)}} \e^{-[\lambda_n(a_2)+\lambda_n(b_2)-\lambda_n(a_2 b_2)] S_{j_2}^{\sss(2)}} \Biggr.
\\
&\qquad
\Biggl.
\cdot\, \bar{m}_{t_1-S_{j_1}^{\sss(1)} \! , \, t_2-S_{j_2}^{\sss(2)}}^\chi(\vec{a}) \bar{m}_{u_1-S_{j_1}^{\sss(1)} \! , \, u_2-S_{j_2}^{\sss(2)}}^\eta(\vec{b})
\!\Biggr)
\! ,
\notag
\end{align}
where now $(S_{j}^{\sss (1)})_j$ and $(S_{j}^{\sss (2)})_j$ are independent random walks and $(S_j^{\sss (i)})_j$ has step distribution $\nu_{a_i b_i}$, $i=1,2$.
Applying \reflemma{CovarianceSums} twice and using the results from the first part of the proof, we obtain the desired conclusions.
\end{proof}

\subsection{The effect of truncation: Proof of \refthm{TwoVertexRemainder}}\lbsubsect{TwoVertexRemainderPf}
In this section, we control the effect of truncation and prove \refthm{TwoVertexRemainder},
by showing that the remainder $\chi_n-\chi_n^{\sss(K)}$ has a negligible first moment.

We will write $\chi_n=\int d\mu_n(y) \Psi_y$, where
\begin{equation}
\Psi_y(t_1,t_2)=\indicator{\abs{t_1-t_2}\leq y, t_1+t_2\geq y}.
\end{equation}
The same is true for $\chi_n^{\sss(K)}$ and $\mu_n^{\sss(K)}$, so that, by \eqrefPartII{munFromfn} and the substitution $\tilde{x}=x^{1/s_n}$,
\begin{align}
s_n^{-1}\bar{m}^{\chi_n-\chi_n^{(K)}}_{\vec{t}}(\vec{1})
&=
\int_0^{1-\deltaCondition} s_n^{-1} \bar{m}^{\Psi_{f_n(\tilde{x})}}_{\vec{t}}(\vec{1}) d\tilde{x} + \left( \int_{(1-\deltaCondition)^{s_n}}^{(1-K/s_n)^{s_n}} + \int_{(1+K/s_n)^{s_n}}^\infty \right) s_n^{-2}\bar{m}^{\Psi_{f_n(\tilde{x})}}_{\vec{t}}(\vec{1}) x^{1/s_n - 1} dx
\notag\\&
=
I_0+I_1+I_2
.
\labelPartII{mbarchiFromPsi}
\end{align}
We must therefore show that $I_0,I_1,I_2$ are uniformly bounded and can be made small by making $K$ and $\lambda_n(1)[t_1\land t_2]$ large.
To this end, we will bound the two-vertex mean $\bar{m}^{\Psi_y}_{\vec{t}}(\vec{1})$ in terms of one-vertex means.
Abbreviate
\begin{equation}
\eta^{\sss(q)}(t)=\indicator{0\leq t\leq q}.
\end{equation}
\begin{lemma}\lblemma{PsiFrometa}
For any $y\in (0,\infty)$,
\begin{align}
\bar{m}^{\Psi_y}_{\vec{t}}(\vec{1}) \leq \frac{1}{y}
&\bigg[
\int_0^\infty \e^{-2\lambda_n(1)r}\bar{m}_{t_1-r}^{\eta^{(2y)}}(1)\bar{m}_{t_2-r}^{\eta^{(2y)}}(1) dr
\\
&+
\int_0^y \e^{-\lambda_n(1)(y-r)} \left( \bar{m}_{t_1-y+r}^{\eta^{(2r)}}(1)\bar{m}_{t_2}^{\eta^{(r)}}(1) + \bar{m}_{t_2-y+r}^{\eta^{(2r)}}(1)\bar{m}_{t_1}^{\eta^{(r)}}(1) \right) \! dr
\bigg].
\notag
\end{align}
\end{lemma}
\begin{proof}
Note that
\begin{equation}
\Psi_y\leq \frac{1}{y}\Big[\int_0^\infty \indicatorofset{[r,r+2y]^2}dr+\int_0^y(\indicatorofset{[y-r,y+r]\times [0,r]}+\indicatorofset{[0,r]\times [y-r,y+r]}) dr\Big]
\end{equation}
since, for any $\vec{t}$ for which $\Psi_y(\vec{t})>0$, the measure of the sets of parameter values $r$ for which $\vec{t}$ belongs to the relevant rectangles is at least $y$ in total.
Then the identities $\bar{m}^{\indicatorofset{[a,b]\times[c,d]}}_{t_1,t_2}(\vec{a})=\bar{m}_{t_1}^{\indicatorofset{[a,b]}}(a_1)\bar{m}_{t_2}^{\indicatorofset{[c,d]}}(a_2)$ and
\begin{equation}\labelPartII{mbarInterval}
\bar{m}^{\indicatorofset{[c,d]}}_t(a) = \e^{-\lambda_n(a)c} \bar{m}_{t-c}^{\eta^{(d-c)}}(a)
\end{equation}
complete the proof.
\end{proof}

Using \reflemma{PsiFrometa}, it will suffice to bound the one-vertex means $\bar{m}^{\eta^{(q)}}_t(1)$.
We will use different bounds depending on the relative sizes of $q$, $t$ and $f_n(1)$, as in the following lemma:

\begin{lemma}\lblemma{mbaretaBounds}
There is a constant $C<\infty$ such that, for $n$ sufficiently large,
\begin{subequations}
\begin{align}
\bar{m}_t^{\eta^{(q)}}(1) &\leq Cs_n && \text{for all }t,q\geq 0,
\labelPartII{mbaretaUniform}
\\
\bar{m}_t^{\eta^{(q)}}(1) &\leq Cs_n\frac{q+s_n f_n(1-\deltaCondition)}{f_n(1)} && \text{if }t\geq f_n(1-1/s_n), q\leq\tfrac{1}{2}t,
\labelPartII{mbaretaDensityLarget}
\\
\bar{m}_t^{\eta^{(q)}}(1) &\leq C\frac{q+s_n f_n(1-\deltaCondition)}{s_n t(1-f_n^{-1}(t))^2} && \text{if } t<f_n(1), q\leq\tfrac{1}{2}t,
\labelPartII{mbaretaDensitySmallt}
\\
\bar{m}_t^{\eta^{(2q)}}(1) &\leq \frac{2}{1-f_n^{-1}(q)}, && \text{if }q<f_n(1), t \ge 0.
\labelPartII{mbaretaInverse}
\end{align}
\end{subequations}
\end{lemma}
\begin{proof}
\refthm{OneCharConv} and $\norm{\eta^{(q)}}_\infty=1$ imply \eqrefPartII{mbaretaUniform}.
For \eqrefPartII{mbaretaInverse}, use the representation \eqrefPartII{mtchiaRW} and note that, starting from the first index $J$ for which $\eta^{(2q)}(t-S_J)\neq 0$ (if one exists), the total number of indices $j$ for which $\eta^{(2q)}(t-S_j)\neq 0$ is stochastically bounded by the waiting time (starting from $J$) until the second step where $Y_j>q$.
Then $\P_1(\increm_j\leq q)\leq f_n^{-1}(q)$ proves \eqrefPartII{mbaretaInverse}.

For \eqrefPartII{mbaretaDensityLarget}--\eqrefPartII{mbaretaDensitySmallt}, we employ a size-biasing argument on the jump sizes $\increm_i$.
For $i\leq j$, write $S'_{j,i}=\sum_{1\leq k\leq j, k\neq i} \increm_k$.
We can therefore rewrite \eqrefPartII{mtchiaRW} (noting that the term $j=0$ vanishes) as
\begin{equation}
\bar{m}_t^{\eta^{(q)}}(1)
=
\sum_{j=1}^\infty \sum_{i=1}^j \E_1\left( \e^{-\lambda_n(1)(t-S'_{j,i})} \E_1\condparentheses{\e^{\lambda_n(1)\increm_i} \frac{\increm_i}{S'_{j,i}+\increm_i} \indicator{t-q-S'_{j,i}\leq \increm_i \leq t-S'_{j,i}}}{S'_{j,i}} \right).
\end{equation}
We split according to whether $\increm_i>f_n(1-\deltaCondition)$ or $\increm_i\leq f_n(1-\deltaCondition)$.
For any measurable function $h\ge 0$, \eqrefPartII{integralnua} and
\reflemma{munDensityBound} imply
\begin{equation}
\E_1\left( \e^{\lambda_n(1)\increm_i} \increm_i h(\increm_i)\indicator{\increm_i> f_n(1-\deltaCondition)} \right) \leq \frac{1}{\epsilonCondition s_n} \int h(y) f_n^{-1}(y) dy.
\end{equation}
On the other hand, $\E_1\left( \e^{\lambda_n(1)\increm_i}\increm_i h(\increm_i)\indicator{\increm_i\leq f_n(1-\deltaCondition)} \right) \leq \max\set{yh(y)\colon y\leq f_n(1-\deltaCondition)}$ by \eqrefPartII{integralnua}.
Consequently, writing $x^+:=x\vee 0$,
\begin{align}
\bar{m}_t^{\eta^{(q)}}(1)
&\leq
\sum_{j=1}^\infty \sum_{i=1}^j \E_1\left( \e^{-\lambda_n(1)(t-S'_{j,i})} \indicator{t-S'_{j,i}\geq 0} \left[ \frac{1}{\epsilon_0 s_n} \int_{(t-q-S'_{j,i})^+}^{t-S'_{j,i}} \frac{f_n^{-1}(y)}{S'_{j,i}+y}dy + \frac{f_n(1-\deltaCondition)}{t-q}  \right] \right)
\notag\\
&
\leq
\sum_{j=1}^\infty \sum_{i=1}^j \E_1\left( \e^{-\lambda_n(1)(t-S'_{j,i})} \indicator{t-S'_{j,i}\geq 0} \left[ \frac{f_n^{-1}(t-S'_{j,i})}{\epsilonCondition s_n} \log\left( \frac{t}{t-q} \right) + \frac{f_n(1-\deltaCondition)}{t-q}  \right] \right)
\notag\\&
\leq
\frac{2}{t}\sum_{j=1}^\infty \sum_{i=1}^j \E_1\left( \e^{-\lambda_n(1)(t-S'_{j,i})} \indicator{t-S'_{j,i}\geq 0} \left[ \frac{qf_n^{-1}(t-S'_{j,i})}{\epsilonCondition s_n}  + f_n(1-\deltaCondition) \right] \right)
\labelPartII{metaqDoubleSum}
\end{align}
since  $-\log(1-x) \le 2x$ for $x \in [0,1/2]$ and $q\leq\tfrac{1}{2}t$.
To obtain \eqrefPartII{mbaretaDensitySmallt}, we note that $\P_1(S'_{j,i}\leq t)\leq f_n^{-1}(t)^{j-1}$, as in the proof of \eqrefPartII{mbaretaInverse}, and therefore
\begin{equation}
\bar{m}_t^{\eta^{(q)}}(1)
\leq
C'\frac{qf_n^{-1}(t)+s_nf_n(1-\deltaCondition)}{s_n t} \sum_{j=1}^\infty jf_n^{-1}(t)^{j-1},
\end{equation}
and $f_n^{-1}(t) < f_n^{-1}(f_n(1))=1$ completes the proof of \eqrefPartII{mbaretaDensitySmallt}.

Finally, to prove \eqrefPartII{mbaretaDensityLarget} we now reverse the argument that led to \eqrefPartII{metaqDoubleSum} by reintroducing a term $\increm_i$.
By \eqrefPartII{munFromfn},
\begin{equation}
\int y_i\indicator{y_i\leq f_n(1)} d\mu_n(y_i)=\int_0^1 f_n(\tilde{x})d\tilde{x}\geq \int_{1-1/s_n}^1 f_n(1-1/s_n) d\tilde{x}=\frac{f_n(1-1/s_n)}{s_n},
\end{equation}
so that $\E_1(\e^{\lambda_n(1)\increm_i}\increm_i \indicator{\increm_i\leq f_n(1)} /f_n(1-1/s_n)) \geq 1/s_n$.
Abbreviate $\rho(u)=\indicator{u\geq 0}[q f_n^{-1}(u)/\epsilonCondition+s_n f_n(1-\deltaCondition)]$.
Note that $\rho$ is increasing and that on $\set{\increm_i\leq f_n(1)}$ we have $t-S'_{j,i}\leq t-S_j+f_n(1)$.
Continuing from \eqrefPartII{metaqDoubleSum}, we estimate
\begin{align}
\bar{m}_t^{\eta^{(q)}}(1)
&\leq
\frac{2}{t}\sum_{j=1}^\infty \sum_{i=1}^j \E_1\left( \e^{-\lambda_n(1)(t-S'_{j,i})} \rho(t-S'_{j,i}) \E_1\left( \e^{\lambda_n(1)\increm_i} \frac{\increm_i}{f_n(1-1/s_n)} \indicator{\increm_i\leq f_n(1)} \right) \right)
\notag\\&
\leq
\frac{2}{t}\sum_{j=1}^\infty \sum_{i=1}^j \E_1\left( \e^{-\lambda_n(1)(t-S_j)} \rho(t+f_n(1)-S_j) \frac{\increm_i}{f_n(1-1/s_n)} \right)
\notag\\&
=
\frac{2\e^{\lambda_n(1)f_n(1)}}{f_n(1-1/s_n)} \sum_{j=0}^\infty \E_1\left( \e^{-\lambda_n(1)(t+f_n(1)-S_j)} \rho(t+f_n(1)-S_j)\frac{S_j}{t} \right)
\notag\\&
\leq
\frac{2 \e^{\lambda_n(1)f_n(1)}}{f_n(1-1/s_n)} \frac{t+f_n(1)}{t} \bar{m}_{t+f_n(1)}^\rho(1)
,\labelPartII{eq:estimCasec}
\end{align}
where we used in the last inequality that $\rho(t+f_n(1)-S_j)=0$ if $S_j > t+f_n(1)$.
As in the proof of \refthm{OneCharConv}, we use \reflemma{mUniformBound}, \refthm{ConvRW} and the definition of $\rho$ to obtain
\begin{align}
\bar{m}_{t+f_n(1)}^\rho(1) &\le O(1) \int_0^{\infty} s_n \lambda_n(1) \e^{-\lambda_n(1)u} \rho(u) \, du
\notag\\
&= O(s_n) \Big[ \frac{q}{\epsilonCondition} \int_0^{\infty} \lambda_n(1)\e^{-\lambda_n(1) u} f_n^{-1}(u) \, du +s_n f_n(1-\deltaCondition)\Big].
\end{align}
By \refcond{fnWeakBounds}, we have $f_n^{-1}(u)\leq (u/f_n(1))^{1/\epsilonCondition s_n}$ for $u\geq f_n(1)$.
Changing variables and using that $\epsilonCondition s_n\ge 1$ for large $n$, we obtain
\begin{equation}
\int_0^{\infty} \lambda_n(1)\e^{-\lambda_n(1) u} f_n^{-1}(u) \, du \le 1+ \int_0^{\infty} \lambda_n(1)f_n(1) \e^{-\lambda_n(1) f_n(1) u} u \, du= 1+\frac{1}{\lambda_n(1) f_n(1)}=O(1)
\end{equation}
according to \eqrefPartII{lambdanAsymp}.
Hence $\bar{m}_{t+f_n(1)}^\rho(1)=O(s_n)(q+s_n f_n(1-\deltaCondition))$.
The other factors in \eqrefPartII{eq:estimCasec} are $O(1/f_n(1))$ because of \eqrefPartII{lambdanAsymp}, \refcond{scalingfn}, and the assumption $t\geq f_n(1-1/s_n)$.
This completes the proof of \eqrefPartII{mbaretaDensityLarget}.
\end{proof}

To make use of the bounds \eqrefPartII{mbaretaDensitySmallt}--\eqrefPartII{mbaretaInverse}, we note the following consequence of \refcond{LowerBoundfn}:
\begin{lemma}\lblemma{fnInverseLog}
Suppose without loss of generality that the constant $\deltaCondition$ from \refcond{LowerBoundfn} satisfies $\deltaCondition<1$.
Then there exists $C<\infty$ such that, uniformly over $f_n(1-\deltaCondition)\leq u\leq f_n(1)$,
\begin{equation}
f_n^{-1}(u)\leq 1-\frac{1}{Cs_n}\log(f_n(1)/u).
\end{equation}
\end{lemma}
\begin{proof}
It suffices to show that $x\leq 1-(Cs_n)^{-1} \log(f_n(1)/f_n(x))$ for $1-\deltaCondition\leq x\leq 1$, i.e., that $\log(f_n(1)/f_n(x))\leq Cs_n(1-x)$.
But \refcond{LowerBoundfn} implies that $\log(f_n(1)/f_n(x))\leq \epsilonCondition^{-1}s_n \log(1/x)$, as in the proof of \reflemma{ExtendedImpliesWeak}, so a Taylor expansion gives the result.
\end{proof}
\begin{proof}[Proof of \refthm{TwoVertexRemainder}]
We will show that each of the terms $I_0,I_1,I_2$ in \eqrefPartII{mbarchiFromPsi} is uniformly bounded, and furthermore can be made arbitrarily small by taking $K$ large enough (for $I_1$ and $I_2$) and $\lambda_n(1)[t_1\wedge t_2]$ large enough (for $I_0$).
We begin with the term $I_2$ (i.e., $x\geq (1+K/s_n)^{s_n}$).
\reflemma{PsiFrometa}, \eqrefPartII{mbaretaUniform} and \eqrefPartII{lambdanAsymp} give
\begin{align}
&\bar{m}^{\Psi_{f_n(\tilde{x})}}_{\vec{t}}(\vec{1})
\leq \frac{O(s_n^2)}{f_n(\tilde{x})}
\left[ \int_0^\infty \e^{-2\lambda_n(1)r} dr +2\int_0^{f_n(\tilde{x})} \e^{-\lambda_n(1)(f_n(\tilde{x})-r)}dr \right]
=\frac{O(s_n^2)f_n(1)}{f_n(\tilde{x})}.
\end{align}
By \reflemma{ExtendedImpliesWeak}, $f_n(\tilde{x})\geq f_n(1)x^{\epsilonCondition}$ for $x\geq 1$, so that $I_2\leq O(1)\int_{(1+K/s_n)^{s_n}}^{\infty} x^{1/s_n-\epsilonCondition-1}\, dx$.
Since $\int_1^\infty x^{-\epsilonCondition-1}\,dx<\infty$, it follows that $I_2$ is uniformly bounded and can be made arbitrarily small by taking $K$, and hence $(1+K/s_n)^{s_n}$, large enough, uniformly over $t_1,t_2$.

For $I_1$, we again start by estimating $\bar{m}^{\Psi_{y}}_{\vec{t}}(\vec{1})$ where $y=f_n(\tilde{x})$ with $\tilde{x} \in [1-\deltaCondition,1-K/s_n]$.
Suppose for definiteness, and without loss of generality by symmetry, that $t_1\leq t_2$.
Split the first integral from \reflemma{PsiFrometa} into the intervals $[0,t_1-f_n(1-1/s_n)]$, $[t_1-f_n(1-1/s_n),t_1-4y]$ and $[t_1-4y,t_1]$ (noting that the integrand vanishes for $r>t_1$) and denote the summand by $\theta_n^{\sss 11}(y)$, $\theta_n^{\sss 12}(y)$ and $\theta_n^{\sss 13}(y)$.
The second summand in \reflemma{PsiFrometa} is called $\theta_n^{\sss 14}(y)$.
The corresponding parts of $I_1$ are denoted by $I_{11}, \ldots, I_{14}$.

We first estimate $\theta_n^{\sss 13}(y)$ and $\theta_n^{\sss 14}(y)$.
Since $f_n(1-\delta_0)\le y \le f_n(1-K/s_n) < f_n(1)$, \eqrefPartII{mbaretaInverse} and \reflemma{fnInverseLog} give
\begin{align}
\theta_n^{\sss 13}(y)+\theta_n^{\sss 14}(y) &\le \frac{1}{y}
\Big[\int_{t_1-4y}^{t_1} \Big(\frac{2}{1-f_n^{-1}(y)}\Big)^2 \, dr +
\int_0^y 2 \frac{2}{1-f_n^{-1}(r)} \frac{2}{1-f_n^{-1}(r/2)}\, dr
\Big]\notag
\\
&\le \frac{O(1)}{(1-f_n^{-1}(y))^2} \le \frac{O(s_n^2)}{\log(f_n(1)/y)^2}.
\end{align}
According to \reflemma{ExtendedImpliesWeak}, $y=f_n(\tilde{x})\leq f_n(1)x^{\epsilonCondition}$ for all $1-\deltaCondition \le \tilde{x} \le 1$.
Substitute $x=\e^{-u}$ to obtain
\begin{align}
I_{13}+I_{14} &\le O(1) \int_{(1-\deltaCondition)^{s_n}}^{(1-K/s_n)^{s_n}} \frac{1}{(\log(1/x^{\epsilonCondition}))^2} x^{1/s_n-1} \, dx\le O(1) \int_{K}^{\infty} \frac{1}{u^2} \, du.
\end{align}
Hence $I_{13}+I_{14}$ is uniformly bounded and can be made arbitrarily small by taking $K$ large.

For $\theta_n^{\sss 11}(y)$, $r \in [0,t_1-f_n(1-1/s_n)]$ implies $t_2-r \ge t_1-r \ge f_n(1-1/s_n)$ and since $2y \le 2f_n(1-K/s_n) \le \frac{1}{2} f_n(1-1/s_n)$ for large $n$ by \refcond{scalingfn}, we can apply first \eqrefPartII{mbaretaDensityLarget} and then \eqrefPartII{lambdanAsymp} to obtain
\begin{equation}
\theta_n^{\sss 11}(y) \le \frac{1}{y} \int_0^{t_1-f_n(1-1/s_n)} \e^{-2\lambda_n(1) r} O(s_n^2)\Big( \frac{2y +s_nf_n(1-\deltaCondition)}{f_n(1)}\Big)^2 \, dr= O(s_n^2) \frac{(y +s_nf_n(1-\deltaCondition))^2}{yf_n(1)}.
\end{equation}
Using that $(a+b)^2 \le 2(a^2 + b^2)$ for all $a,b \in \R$ and that $f_n(\tilde{x})\leq f_n(1)x^{\epsilonCondition}$ and, for $x\in (1-\delta_0,1-K/s_n)$, $f_n(\tilde{x})\geq f_n(1-\delta_0)$, we obtain
\begin{align}
I_{11} & \le \int_{(1-\deltaCondition)^{s_n}}^{(1-K/s_n)^{s_n}} O(1) \Big(\frac{f_n(\tilde{x})}{f_n(1)} + \frac{s_n^2 f_n(1-\deltaCondition)^2}{f_n(\tilde{x})f_n(1)} \Big) x^{1/s_n-1} \, dx
\notag\\
& \le O(1) \bigg[\int_{(1-\deltaCondition)^{s_n}}^{(1-K/s_n)^{s_n}} x^{\epsilonCondition+1/s_n-1} \, dx + s_n^2 \frac{f_n(1-\deltaCondition)}{f_n(1)} \int_{(1-\deltaCondition)^{s_n}}^{(1-K/s_n)^{s_n}} x^{1/s_n-1} \, dx\bigg].
\end{align}
The first summand is bounded and can be made arbitrarily small by choosing $K$ large.
The second summand is arbitrarily small for large $n$ uniformly in $K$ since $f_n(1-\deltaCondition)/f_n(1) \le (1-\deltaCondition)^{\epsilonCondition s_n} = o(s_n^{-3})$ according to \reflemma{ExtendedImpliesWeak}.

For $\theta_n^{\sss 12}(y)$ we substitute $u=t_1-r$ to obtain
\begin{equation}\labelPartII{eq:thetan12y}
\theta_n^{\sss 12}(y) = \frac{1}{y} \int_{4y}^{f_n(1-1/s_n)} \e^{-2\lambda_n(1) (t_1-u)} \bar{m}_u^{\eta^{\sss (2y)}}(1) \bar{m}_{t_2-t_1+u}^{\eta^{\sss (2y)}}(1) \, du.
\end{equation}
We consider the two cases $t_2-t_1 \ge f_n(1)/2$ and $0 \le t_2-t_1 <f_n(1)/2$ separately.
First $t_2 -t_1 \ge f_n(1)/2$.
Then $t_2-t_1+u \ge f_n(1)/2+4 f_n(1-\deltaCondition) \ge f_n(1-1/s_n)$ for sufficiently large $n$.
Hence \eqrefPartII{mbaretaDensityLarget}, \eqrefPartII{mbaretaDensitySmallt} and \reflemma{fnInverseLog} yield
\begin{align}
\theta_n^{\sss 12}(y)& \le  \frac{1}{y} \int_{4y}^{f_n(1-1/s_n)}O(1) \frac{2y +s_nf_n(1-\deltaCondition)}{s_n u (1-f_n^{-1}(u))^2} s_n \frac{2y +s_nf_n(1-\deltaCondition)}{f_n(1)} \, du
\notag\\
& \le
O(s_n^2) \frac{(y +s_nf_n(1-\deltaCondition))^2}{yf_n(1)} \int_{4y}^{f_n(1-1/s_n)}\frac{1}{u \log(f_n(1)/u)^2} \, du,
\end{align}
where the integral is of order one.
Hence, in the case $t_2-t_1 \ge f_n(1)/2$, we have the same bound for $\theta_n^{\sss 12}(y)$ as for $\theta_n^{\sss 11}(y)$.

Now let $t_2-t_1 <f_n(1)/2$, and abbreviate $u'=t_2-t_1+u$.
Recalling that $f_n(1-1/s_n)/f_n(1)\to \e^{-1}$, we have $u' \le f_n(1)/2 +f_n(1-1/s_n) < \tfrac{9}{10}f_n(1)$ for large $n$, uniformly over $u\leq f_n(1-1/s_n)$.
We first apply \eqrefPartII{mbaretaDensitySmallt} to both factors in \eqrefPartII{eq:thetan12y} and then use \reflemma{fnInverseLog} to obtain
\begin{align}
\theta_n^{\sss 12}(y)& \le  \frac{1}{y} \int_{4y}^{f_n(1-1/s_n)}O(1) \frac{2y +s_nf_n(1-\deltaCondition)}{s_n u (1-f_n^{-1}(u))^2} \frac{2y +s_nf_n(1-\deltaCondition)}{s_n u' (1-f_n^{-1}(u'))^2} \, du\notag \\
&\le
O(1) \frac{(y+s_nf_n(1-\deltaCondition))^2}{ys_n^2} \int_{4y}^{f_n(1-1/s_n)} \frac{s_n^2}{u \log(f_n(1)/u)^2} \frac{s_n^2}{u' \log(f_n(1)/u')^2} \, du.\labelPartII{eq:theta12nyClose}
\end{align}
In \eqrefPartII{eq:theta12nyClose}, we have $u'\geq u$, and we are particularly concerned with the case $u'=u$.
The function $u'\mapsto (u'/f_n(1)) (\log (f_n(1)/u'))^2$ is not monotone over $u' \in[u,\tfrac{9}{10}f_n(1)]$, but it is increasing over $(0,\tfrac{1}{10}f_n(1))$ and bounded from zero and infinity over $[\tfrac{1}{10}f_n(1),\tfrac{9}{10}f_n(1)]$.
We may therefore find a constant $c>0$ such that $u'(\log(f_n(1)/u'))^2 \geq c u(\log (f_n(1)/u))^2$ whenever $u\leq u'\leq \tfrac{9}{10}f_n(1)$.
The bound \eqrefPartII{eq:theta12nyClose} may therefore be simplified to
\begin{equation}
\theta_n^{\sss 12}(y)
\leq
O(s_n^2) \frac{(y+s_nf_n(1-\deltaCondition))^2}{y} \int_{4y}^{f_n(1-1/s_n)} \frac{1}{u^2 \log(f_n(1)/u)^4} \, du,
\end{equation}
and an integration by parts shows that
\begin{equation}\labelPartII{eq:theta12nylog4}
\theta_n^{\sss 12}(y)
\leq
O(s_n^2) \frac{(y+s_nf_n(1-\deltaCondition))^2}{y^2 (\log (f_n(1)/y))^4}.
\end{equation}
Inserting the bound \eqrefPartII{eq:theta12nylog4} into $I_{12}$ and using again $(a+b)^2 \le 2(a^2+b^2)$, we conclude that
\begin{align}
I_{12}
&\le
\int_{(1-\deltaCondition)^{s_n}}^{(1-K/s_n)^{s_n}} O(1)
\frac{f_n(\tilde{x})^2+s_n^2 f_n(1-\deltaCondition)^2}{f_n(\tilde{x})^2 (\log(f_n(1)/f_n(\tilde{x})))^4 } x^{1/s_n-1} \, dx.\labelPartII{eq:I12closeTestimate}
\end{align}
Recall from \reflemma{ExtendedImpliesWeak} the bounds $f_n(\tilde{x})\leq f_n(1)x^{\epsilonCondition}$ and $f_n(1-\deltaCondition)\leq f_n(\tilde{x})((1-\deltaCondition)/\tilde{x})^{\epsilonCondition s_n}$.
We split the integral for the second summand in \eqrefPartII{eq:I12closeTestimate} at $(1-\deltaCondition/2)^{s_n}$.
For $x \ge (1-\deltaCondition/2)^{s_n}$, $f_n(1-\deltaCondition)\leq f_n(\tilde{x})(1-\delta')^{\epsilonCondition s_n}$ for some $\delta' \in (0,1)$.
For $x \le  (1-\deltaCondition/2)^{s_n}$, $\log(f_n(1)/f_n(\tilde{x})) \ge \epsilonCondition \log(1/x) \ge c s_n$ for some $c>0$.
Hence
\begin{align}
I_{12} &\le O(1) \Big(1+ s_n^2(1-\delta')^{\epsilonCondition s_n}\Big) \int_{(1-\deltaCondition)^{s_n}}^{(1-K/s_n)^{s_n}} \frac{1}{x(\log (1/x))^4}\, dx
\notag \\
&\qquad+O(1) \int_{(1-\deltaCondition)^{s_n}}^{(1-\deltaCondition/2)^{s_n}}
\frac{s_n^2 f_n(1-\deltaCondition)^2}{f_n(\tilde{x})^2 s_n^4} x^{1/s_n-1} \, dx
\notag \\
&\le O(1) \int_0^{\e^{-K}} \frac{1}{x(\log (1/x))^4}\, dx +
O(s_n^{-2}) \int_{(1-\deltaCondition)^{s_n}}^{(1-\deltaCondition/2)^{s_n}}(1-\deltaCondition)^{2 \epsilonCondition s_n}
x^{-1-2 \epsilonCondition} \, dx.\labelPartII{eq:I12closeTestimate2}
\end{align}
The first summand in \eqrefPartII{eq:I12closeTestimate2} is bounded and can be made arbitrarily small by choosing $K$ large, while the second summand is arbitraily small for large $n$ uniformly in $K$.
We have now handled all four contributions to $I_1$.

Finally, for $I_0$, let $y\leq f_n(1-\deltaCondition)$ and note that $\Psi_y(t'_1,\cdot)=\indicatorofset{[|t'_1-y|,t'_1+y]}\leq \indicatorofset{[(t'_1-y)^+,t'_1+y]}$, where $(t_1'-y)^+:=(t_1'-y) \vee 0$.
From \eqrefPartII{mbarInterval} and \eqrefPartII{mbaretaInverse} it follows that
\begin{align}
\bar{m}^{\Psi_y(t'_1,\cdot)}_{t_2}(1)&\le \indicator{t_2\geq t'_1-f_n(1-\deltaCondition)} e^{-\lambda_n(1) (t_1'-y)^+} \bar{m}_{t_2-(t_1'-y)^+}^{\eta^{(t_1'+y-(t_1'-y)^{+})}}(1) \notag \\
&\leq \indicator{t_2\geq t'_1-f_n(1-\deltaCondition)} \frac{2}{1-f_n^{-1}(2f_n(1-\deltaCondition))} =O(1) \indicatorofset{[0,t_2+f_n(1-\deltaCondition)]}(t_1'),\labelPartII{mPsiVerySmally}
\end{align}
and therefore $\bar{m}^{\Psi_y}_{\vec{t}}(\vec{1})\leq O(s_n)$ by \eqrefPartII{IteratedChar}--\eqrefPartII{IteratedCharMeans} and \eqrefPartII{mbaretaUniform}.
We conclude that $I_0$ is uniformly bounded.
To show the smallness, we sharpen the bound \eqrefPartII{mPsiVerySmally} when $t'_1$ is far from $t_2$ by using \eqrefPartII{mbaretaDensityLarget} instead of \eqrefPartII{mbaretaInverse}, to obtain for  $t_2\geq t'_1+f_n(1-\deltaCondition)+f_n(1)$,
\begin{equation}\labelPartII{mPsiVerySmallyDensity}
\bar{m}^{\Psi_y(t'_1,\cdot)}_{t_2}(1)\leq\bar{m}_{t_2-(t_1'-y)^+}^{\eta^{(t_1'+y-(t_1'-y)^{+})}}(1)\le
O(s_n) \frac{2y+s_n f_n(1-\deltaCondition)}{f_n(1)}= O(s_n^2)\frac{f_n(1-\deltaCondition)}{f_n(1)}.
\end{equation}
Combining \eqrefPartII{mPsiVerySmally}--\eqrefPartII{mPsiVerySmallyDensity}, we have $\bar{m}^{\Psi_y(t'_1,\cdot)}_{t_2}(1) \leq O(s_n^2)\frac{f_n(1-\deltaCondition)}{f_n(1)}+O(1)\indicator{t_2-2f_n(1)\leq t'_1\leq t_2+f_n(1)}$.
Applying \eqrefPartII{IteratedChar}--\eqrefPartII{IteratedCharMeans}, \eqrefPartII{mbarInterval}, \eqrefPartII{mbaretaUniform} and \eqrefPartII{lambdanAsymp} we conclude that
\begin{equation}
\bar{m}^{\Psi_y}_{\vec{t}}(\vec{1})\leq O(s_n^3)\frac{f_n(1-\deltaCondition)}{f_n(1)} + O(s_n) \e^{-\lambda_n(1)t_2}
\end{equation}
and consequently $I_0$ may be made small by taking $t_2$ large.
\end{proof}

\section{First points of Cox processes: Proof of \refthm{FirstPointCoxAsymp}}\lbsect{FirstPointsCoxSec}

Let $\cX$ denote a topological space equipped with its Borel $\sigma$-field and let $(\cP_n)_{n\ge 1}$ be a sequence of Cox processes on $\R\times\cX$ with random intensity measures $(Z_n)_{n\ge 1}$.
That is, there exist $\sigma$-fields $\F_n$ such that $Z_n$ is $\F_n$-measurable and, conditional on $\F_n$, $\cP_n$ is a Poisson point process with (random) intensity $Z_n$.
For instance, \refthm{CouplingFPPCollision} expresses the first passage distance and hopcount in terms of the first point of a Cox process.
In this section, we determine sufficient conditions to identify the limiting distribution of the first points of $\cP_n$ based on the intensity measure at fixed times $t$.

We will write $\cP_{n,t}$ for the measure defined by $\cP_{n,t}(\cdot)=\cP_n(\ocinterval{-\infty,t}\times\cdot)$, and given a partition $t_0<\dotsb<t_N$ we abbreviate $\Delta\cP_{n,i}=\cP_{n,t_i}-\cP_{n,t_{i-1}}$; similarly for $Z_{n,t}, \Delta Z_{n,i}$.
Write $\abs{\mu}$ for the total mass of a measure $\mu$.

Define
\begin{equation}
T_{n,k}=\inf\set{t\colon \abs{\cP_{n,t}}\geq k}
\end{equation}
and let $A_{n,k}$ be the event that $T_{n,j}\notin\set{\pm\infty}$ and $\abs{\cP_{n,T_{n,j}}}=j$, for $j=1,\dotsc,k$.
That is, $A_{n,k}$ is the event that the points of $\supp \cP_n$ with the $k$ smallest $t$-values are uniquely defined.
On $A_{n,k}$, let $X_{n,k}$ denote the unique point for which $\cP_n(\set{T_{n,k}}\times\set{X_{n,k}})=1$, and otherwise set $X_{n,k}=\cemetery$, an isolated cemetery point.

We will impose the following conditions on the intensity measures $(Z_n)_n$, expressed in terms of a probability measure $Q$ on $\cX$ and a family $\cH$ of measurable functions $h\colon\cX\to\R$.

\begin{cond}[Regularity of Cox process intensities]{\hfill}
\lbcond{ZnBasic}
\begin{enumerate}
\item\lbitem{PPPPointwiseConv}
For any $t\in \R$ and for any $h\in\cH$,
\begin{equation}
\int_{\cX} h \, dZ_{n,t} - \abs{Z_{n,t}} \int_{\cX} h \, dQ  \convp 0.
\end{equation}

\item\lbitem{PPPTightIntensityLeft}
For each $\epsilon>0$, there exists $\underline{t}\in \R$ such that
\begin{equation}
\liminf_{n\to\infty}\P\left( \abs{Z_{n,\underline{t}}}<\epsilon \right) \geq 1-\epsilon.
\end{equation}

\item\lbitem{PPPTightIntensityRight}
For each $\epsilon>0$, there exists $\overline{t}\in \R$ such that
\begin{equation}
\liminf_{n\to\infty}\P\left( \abs{Z_{n,\overline{t}}}>1/\epsilon \right) \geq 1-\epsilon.
\end{equation}

\item\lbitem{PPPQuadVarBound}
For each $\epsilon>0$ and each $\underline{t}<\overline{t}$, there exists a partition $t_0=\underline{t}<t_1<\dotsb<t_N=\overline{t}$ of $[\underline{t},\overline{t}]$ such that
\begin{equation}
\liminf_{n\to\infty} \P\left( \sum_{i=1}^N \abs{ \Delta Z_{n,i}}^2 \leq \epsilon \right) \geq 1-\epsilon.
\end{equation}
\end{enumerate}
\end{cond}

We make the convention that any function $h$ on $\cX$ is extended to $\cX \cup \set{\cemetery}$ by $h(\cemetery)=0$.
\begin{prop}\lbprop{FirstPointsPPPbasic}
Suppose that \refcond{ZnBasic} holds for a probability measure $Q$ on $\cX$ and a family $\cH$ of bounded measurable functions $h\colon\cX\to\R$.
Then, for each fixed $k\in\N$, $\P(A_{n,k})\to 1$, the collection $\set{(T_{n,j})_{j=1}^k \colon n\in\N}$ of random vectors is tight, and
\begin{equation}\labelPartII{FirstPointPPPDiffOfExp}
\condE{\prod_{j=1}^k g_j(T_{n,j}) h_j(X_{n,j})}{\F_n} - \condE{\prod_{j=1}^k g_j(T_{n,j})}{\F_n} \prod_{j=1}^k \int_{\cX} h_j \, dQ  \convp 0
\end{equation}
for all bounded continuous functions $g_1,\dotsc,g_k\colon\R\to\R$ and all $h_1,\dotsc,h_k\in\cH$.
\end{prop}
\begin{theorem}\lbthm{FirstPointsPPPConvergence}
Suppose \refcond{ZnBasic} holds when either
\begin{enumerate}[(i)]
\item $\cH$ is the family of all bounded continuous functions on $\cX$;
\item $\cX=\R^d$ and $\cH$ is the family of functions $h(\vec{x})=\e^{i \vec{\xi}\cdot \vec{x}}$ for $\vec{\xi}\in\R^d$; or
\item $\cX=\cointerval{0,\infty}^d$ and $\cH$ is the family of functions $h(\vec{x})=\e^{-\vec{\xi}\cdot \vec{x}}$ for $\vec{\xi}\in\cointerval{0,\infty}^d$.
\end{enumerate}
Then
\begin{enumerate}
\item\lbitem{PointsIndependentLimit}
the random sequence $(X_{n,j})_{j=1}^\infty$ converges in distribution (with respect to the product topology on $(\cX\union\set{\cemetery})^\N$) to a random sequence $(X_j)_{j=1}^\infty$, where the $X_j$ are independent with law $Q$;
\item\lbitem{PointsAsympIndep}
the sequence $(X_{n,j})_{j=1}^\infty$ is asymptotically independent of $\F_n$;
\item
the collection $\set{(T_{n,j})_{j=1}^k \colon n\in\N}$ of random vectors is tight; and
\item\lbitem{PointsIndependentSubseqLimit}
if $(T_j,X_j)_{j=1}^\infty$ is any subsequential limit in distribution of $(T_{n,j},X_{n,j})_{j=1}^\infty$, then $(T_j)_{j=1}^\infty$ and $(X_j)_{j=1}^\infty$ are independent.
\end{enumerate}
\end{theorem}
\begin{proof}[Proof of \refthm{FirstPointsPPPConvergence} assuming \refprop{FirstPointsPPPbasic}]
Because of the product topology, it suffices to consider finite sequences $(T_{n,j},X_{n,j})_{j=1}^k$ for a fixed $k\in\N$.
Applying \eqrefPartII{FirstPointPPPDiffOfExp} with $g_j(t)=1$ gives the convergence of $(X_{n,j})_{j=1}^k$.
The independence of $(T_j)_{j=1}^k$ and $(X_j)_{j=1}^k$ follows from the product form of \eqrefPartII{FirstPointPPPDiffOfExp}, and the asymptotic independence of $X_j$ from $\F_n$ follows because of the conditional expectations in \eqrefPartII{FirstPointPPPDiffOfExp}.
\end{proof}

We first prove the following lemma.
Given $t_0<\dotsb<t_N$, write $B_{n,k}$ for the event that there exist (random) integers $1\leq I_1<\dotsb<I_k\leq N$ with $\abs{\Delta\cP_{n,I_j}}=1$ for $j=1,\dotsc,k$ and $\abs{\big.\smash{\cP_{n,t_{I_k}}}}=k$.
(That is, $B_{n,k}$ is the event that each of the first $k$ points of $\cP_n$ is the unique point in some interval $\ocinterval{t_{i-1},t_i}$.
In particular $B_{n,k}\subset A_{n,k}$.)

\begin{lemma}\lblemma{BnkLikely}
Assume Conditions~\refPartII{cond:ZnBasic}~\refitem{PPPTightIntensityLeft}--\refitem{PPPQuadVarBound}.
Then, given $\epsilon>0$ and $k\in \N$, there exists $[\underline{t},\overline{t}]$ and a partition $\underline{t}=t_0<\dotsb<t_N=\overline{t}$ of $[\underline{t},\overline{t}]$ such that $\liminf_{n\to\infty} \P(B_{n,k})\geq 1-\epsilon$.
In particular, $\P(A_{n,k})\to 1$.
\end{lemma}
\begin{proof}
Given a partition $\underline{t}=t_0<\dotsb<t_N=\overline{t}$, the complement $B_{n,k}^c$ is the event that $\cP_n$ contains a point in $\ocinterval{-\infty,\underline{t}}$, fewer than $k$ points in $\ocinterval{-\infty,\overline{t}}$, or more than one point in some interval $\ocinterval{t_{i-1},t_i}$.
By Conditions~\refPartII{cond:ZnBasic}~\refitem{PPPTightIntensityLeft}--\refitem{PPPTightIntensityRight}, we may choose $\underline{t},\overline{t}$ such that the first two events each have probability at most $\epsilon/3$ for $n$ large.
Since $\condP{\abs{\Delta\cP_{n,i}}\geq 2 \,}{Z_n} = 1-\e^{-\abs{\Delta Z_{n,i}}}(1+\abs{\Delta Z_{n,i}}) \leq \abs{\Delta Z_{n,i}}^2$, \refcond{ZnBasic}~\refitem{PPPQuadVarBound} gives a partition of $[\underline{t},\overline{t}]$ such that the third event also has probability at most $\epsilon/3$ for $n$ large.
\end{proof}
\begin{proof}[Proof of \refprop{FirstPointsPPPbasic}]
Fix any $\epsilon>0$ and bounded, continuous functions $g_1,\ldots,g_k$.
Choose $t_0<\dotsb<t_N$ as in \reflemma{BnkLikely}.
By taking a refinement, we may assume that $\abs{g_j(t)-g_j(t_i)}\leq \epsilon$ for each $t\in\ocinterval{t_{i-1},t_i}$ and each $i,j$.
Define $\psi(t)=t_i$ if $t_{i-1}< t \leq t_i$ and $\psi(t)=t_N$ otherwise, and set $\tilde{g}_j=g_j\circ \psi$.
Partitioning according to the integers $I_j$,
\begin{align}
&\indicatorofset{B_{n,k}}\prod_{j=1}^k \tilde{g}_j(T_{n,j})h_j(X_{n,j})=
\sum_{\vec{i}} \indicatorofset{B_{n,k}}\indicator{\vec{I}=\vec{i}} \prod_{j=1}^k g_j(t_{i_j}) \int_{\cX} h_j \, \Delta\cP_{n,i_j}
,
\end{align}
where the sum is over $\vec{i} \in \N^k$ with $1\leq i_1<\dotsb<i_k\leq N$, and we write $\vec{I}=(I_1, \ldots, I_k)$.
Observe that a Poisson point process $\cP$ with intensity $\mu$ satisfies $\E(\indicator{\abs{\cP}=1} \int h \, d\cP)=\e^{-\abs{\mu}}\int h \, d\mu$.
Consequently,
\begin{align}
&\condE{\indicatorofset{B_{n,k}}\prod_{j=1}^k \tilde{g}_j(T_{n,j})h_j(X_{n,j})}{\F_n}
=
\sum_{\vec{i}} \e^{-\abs{Z_{n,t_{i_k}}}} \prod_{j=1}^k g_j(t_{i_j}) \int_\cX h_j \, \Delta Z_{n,i_j}.
\labelPartII{EGivenZOnBnk}
\end{align}
Apply \eqrefPartII{EGivenZOnBnk} twice, with the original $h_j$'s and with the constant functions $\tilde{h}_j(x)=\int h_j\, dQ$, to get
\begin{align}
&\condE{\indicatorofset{B_{n,k}} \left[ \prod_{j=1}^k \tilde{g}_j(T_{n,j}) h_j(X_{n,j}) - \prod_{j=1}^k \tilde{g}_j(T_{n,j}) \int_\cX h_j \, dQ \right] }{\F_n}
\notag
\\
&\qquad=
\sum_{\vec{i}} \e^{-\abs{Z_{n,t_{i_k}}}} \left[ \prod_{j=1}^k g_j(t_{i_j}) \int_\cX h_j \, \Delta Z_{n,i_j} - \prod_{j=1}^k g_j(t_{i_j}) \abs{\Delta Z_{n,i_j}} \int_\cX h_j \, dQ \right].
\labelPartII{CondEOnBnk}
\end{align}
The right-hand side of \eqrefPartII{CondEOnBnk} is bounded (since $g_j$, $h_j$, and $\shortabs{\Delta Z_{n,i_j}}\e^{-\shortabs{\Delta Z_{n,i_j}}}$ are bounded) and, by \refcond{ZnBasic}~\refitem{PPPPointwiseConv}, converges to $0$ in probability, and hence also in expectation.
By the choice of the partition, $\abs{\tilde{g}_j(T_{n,j})-g_j(T_{n,j})}\leq\epsilon$ on $B_{n,k}$ and $\limsup_{n\to\infty}\P(B_{n,k}^c)\leq \epsilon$.
Now let $(F_n)_n$ be a uniformly bounded sequence of $\R$-valued random variables such that $F_n$ is $\F_n$-measurable.
Since all the functions involved are bounded, there exists $C<\infty$ such that
\begin{align}
&\limsup_{n\to\infty}
\abs{\E\left(F_n\prod_{j=1}^k g_j(T_{n,j}) h_j(X_{n,j}) \right) - \E\left(F_n\prod_{j=1}^k g_j(T_{n,j}) \right) \prod_{j=1}^k \int_\cX h_j \, dQ}
\le C\epsilon,
\end{align}
which completes the proof.
\end{proof}

When $\cX=\R^d$, another natural family is $\cH=\set{h(\vec{x})=\e^{\vec{\xi} \cdot \vec{x}}\colon \vec{\xi}\in\R^d}$.
However, these functions are \emph{not} bounded, so it is necessary to modify the argument of \refprop{FirstPointsPPPbasic} and \refthm{FirstPointsPPPConvergence}.
Recall from \eqrefPartII{MomentGeneratingNotation} that we write $\hat{R}$ for the moment generating function of a measure $R$ on $\R^d$.

\begin{prop}\lbprop{FirstPointPPPLaplace}
Let $\cX=\R^d$.
Suppose \refcond{ZnBasic} holds when $\cH$ is the family of functions $h(x)=\e^{\vec{\xi} \cdot \vec{x}}$ for $\vec{\xi}\in\R^d$, $\shortabs{\vec{\xi}}\leq\delta$, where $\delta>0$ and $\hat{Q}(\vec{\xi})<\infty$ for all $\shortabs{\vec{\xi}}\leq \delta$.	
Then the conclusions \refitem{PointsIndependentLimit}--\refitem{PointsIndependentSubseqLimit} of \refthm{FirstPointsPPPConvergence} hold.
\end{prop}
\begin{proof}
Fix any $\epsilon>0$, $k\in \N$, $g_1,\ldots,g_k$ bounded, continuous functions, and choose $t_0<\dotsb<t_N$ as in \reflemma{BnkLikely}.
By taking a refinement, we may assume that $t_i-t_{i-1}\leq \epsilon$.
Let $C_n$ be the event that $\hat{Z}_{n,t_i}(\vec{\xi}_0) \leq \abs{Z_{n,t_i}} \hat{Q}(\vec{\xi}_0) + \epsilon$ for each $i=1,\dotsc,N$ and for each $\vec{\xi}_0\in\shortset{\delta/\sqrt{d}, -\delta/\sqrt{d}}^d$.
By \refcond{ZnBasic}~\refitem{PPPPointwiseConv}, $\P(C_n)\to 1$.
Let $X_j$ be independent random variables with law $Q$, and define $\tilde{X}_{n,j}=X_{n,j}$ on $B_{n,k}\intersect C_n$ and $\tilde{X}_{n,j}=X_j$ otherwise.
Recall the notations $\psi(t),\tilde{g}_j(t)$ from the proof of \refprop{FirstPointsPPPbasic} and set $\tilde{T}_{n,j}=\psi(T_{n,j})$.
Set $h_j(\vec{x})=\e^{\vec{\xi}_j\cdot \vec{x}}$ for $\shortnorm{\vec{\xi}_j}_\infty\leq\delta/\sqrt{d}$.
By the argument of the previous proof, this time using that the $\tilde{X}_{n,j}$ have law $Q$ on $(B_{n,k}\intersect C_n)^c$, we find
\begin{align}
&\condE{\prod_{j=1}^k g_j(\tilde{T}_{n,j})h_j(\tilde{X}_{n,j}) - \prod_{j=1}^k g_j(\tilde{T}_{n,j}) \hat{Q}(\vec{\xi}_j)}{\F_n}
\notag
\\
&\quad=
\indicatorofset{C_n} \sum_{\vec{i}} \e^{-\abs{Z_{n,t_{i_k}}}} \left[ \prod_{j=1}^k g_j(t_{i_j}) \widehat{\Delta Z}_{n,i_j}(\vec{\xi}_j) - \prod_{j=1}^k g_j(t_{i_j})\abs{\Delta Z_{n,i_j}} \hat{Q}(\vec{\xi}_j) \right].
\labelPartII{ModifiedXnConv}
\end{align}
By \refcond{ZnBasic}~\refitem{PPPPointwiseConv}, the right-hand side of \eqrefPartII{ModifiedXnConv} converges to $0$ in probability.
Moreover, by the bound $\e^{\vec{\xi}_j\cdot \vec{x}}\leq \sum_{\vec{\xi}_0\in\shortset{\pm \delta/\sqrt{d}}^d} \e^{\vec{\xi}_0\cdot \vec{x}}$ and the choice of $C_n$, it is bounded as well.
Hence we may repeat the argument from the proof of \refthm{FirstPointsPPPConvergence} to find that $(\tilde{T}_{n,j}, \tilde{X}_{n,j})_j$ satisfy the desired conclusions.
But by construction, $\liminf_{n\to\infty} \P(X_{n,j} = \tilde{X}_{n,j}, |T_{n,j}-\tilde{T}_{n,j}|\leq \epsilon) \geq 1-\epsilon$.
Since $\epsilon>0$ was arbitrary, it follows that $X_{n,j}$ and $T_{n,j}$ themselves have the same convergence properties.
\end{proof}

\begin{proof}[Proof of \refthm{FirstPointCoxAsymp}]
For any $\epsilon>0$ we may define
\begin{align}
\underline{t}(\epsilon)
&=\max\set{t\in\Z,t<-1/\epsilon \colon q(t)<\epsilon^2},
\\
\overline{t}(\epsilon)
&=\min\set{t\in\Z,t>1/\epsilon \colon \liminf_{n\to\infty} \P(\shortabs{Z^*_{n,t}}>1/\epsilon)>1-\epsilon},
\\
K_1(\epsilon)
&=\min\set{k_1\in\N\colon \liminf_{n\to\infty} \P\left( \condE{\bigabs{Z^{\prime\sss(K)}_{n,\overline{t}(\epsilon)}}}{\F'_n}\leq 2q(\overline{t}(\epsilon)) \right)\geq 1-\epsilon \text{ for all } K\geq k_1}
\notag\\&\quad
\vee\min\set{k_1\in\N\colon \liminf_{n\to\infty} \P\left( \condE{\bigabs{Z^{\prime\prime\sss(K)}_{n,\overline{t}(\epsilon)}}}{\F'_n}\leq \epsilon^2 \right) \geq 1-\epsilon \text{ for all } K\geq k_1},
\end{align}
by $\lim_{t\to-\infty}q(t)=0$, assumption \refitem{PPPTightIntensityRightCoxTheorem}, \eqrefPartII{FirstMomentAgainstxi} and \eqrefPartII{SplittingAssumptions}, respectively.
Given any $\epsilon>0$, let $\epsilon'\in(0,\epsilon)$ be arbitrary.
By construction, $\overline{t}(\epsilon')\leq \overline{t}(\epsilon)$, so taking $K=K_1(\epsilon')$ shows that $\condE{\bigabs{Z^*_{n,\overline{t}(\epsilon)}}}{\F'_n}$ is uniformly bounded apart from an event of probability at most $2\epsilon'+o(1)$.
Since $\epsilon'$ was arbitrary, it follows that $\bigabs{Z^*_{n,\overline{t}(\epsilon)}}$ is tight as $n\to\infty$, so we may define
\begin{equation}\labelPartII{Z*Tightness}
z_0(\epsilon)=\min\set{z\in\N\colon \limsup_{n\to\infty} \P(\shortabs{Z^*_{n,\overline{t}(\epsilon)}}\leq z)\geq 1-\epsilon}.
\end{equation}

For $\epsilon>0,t,u\in\R,\vec{\xi}\in\R^2$, let $K_0(\epsilon,t,u,\vec{\xi})$ denote the smallest integer exceeding $K_1(\epsilon)$ such that \eqrefPartII{FirstMomentAgainstxi}--\eqrefPartII{SplittingAssumptions} hold with probability at least $1-\epsilon$ for $K\geq K_0(\epsilon)$ and $n\geq n_0(K,\epsilon,t,u,\vec{\xi})$.
Let $K_0(\epsilon)$ and $n_0(\epsilon)$ denote the maxima of $K_0(\epsilon,t,u,\vec{\xi})$ and $n_0(K_0(\epsilon),\epsilon,t,u,\vec{\xi})$, respectively, over all numbers $t,u\in[\underline{t}(\epsilon),\overline{t}(\epsilon)]$ and $\xi_1,\xi_2\in[-1/\epsilon,1/\epsilon]$ that are dyadic rationals of the form $i2^{-\ell}$ ($i\in\Z$, $\ell\in\N$) with $\ell < 1/\epsilon$.

The hypotheses imply that $K_0(\epsilon)$ and $n_0(\epsilon)$ are finite for each $\epsilon>0$.
Moreover, by construction, $\overline{t}(\epsilon)\to\infty$, $\underline{t}(\epsilon)\to-\infty$ as $\epsilon\decreasesto 0$.
Therefore, by letting $\epsilon=\epsilon_n$ decrease to 0 sufficiently slowly as $n\to\infty$ and setting $K=K_0(\epsilon_n)$, we can assume that $Z^*_n=Z'_n+Z''_n$ where
\begin{align}
\condE{ \hat{Z}'_{n,t}(\vec{\xi}) }{\F'_n}
&\convp
q(t) \hat{Q}(\vec{\xi}),
\labelPartII{Convp1Againstxi}
\\
\condE{ \Big( \frac{\hat{Z}'_{n,t}(\vec{\xi})}{q(t)\hat{Q}(\vec{\xi})}-\frac{\abs{Z'_{n,u}}}{q(u)} \Big)^2 }{ \F'_n }
&\convp
0,
\labelPartII{Convp2Againstxi}
\\
\condE{ \abs{Z''_{n,t}} }{\F'_n}
&\convp 0,
\labelPartII{Z''to0}
\end{align}
whenever $t,u,\xi_1,\xi_2$ are dyadic rationals.

Let $\cP'_n,\cP''_n$ denote the Cox processes with intensity measures $Z'_n,Z''_n$, respectively.
For any fixed $t\in\R$, \eqrefPartII{Z''to0} implies that the first point of $\cP''_n$ does not occur by time $t$ with high probability.
By \refprop{FirstPointPPPLaplace}, it therefore suffices to show that $Z'_n$ satisfies \refcond{ZnBasic} for the family of functions $h(\vec{x})=\e^{\vec{\xi}\cdot \vec{x}}$, $\vec{\xi}\in\R^2$.

Applying Chebyshev's inequality to \eqrefPartII{Convp2Againstxi},
\begin{equation}\labelPartII{ConvAgainstxiRatio}
\hat{Z}'_{n,t}(\vec{\xi})-\frac{q(t)}{q(u)}\hat{Q}(\vec{\xi})\abs{Z'_{n,u}} \convp 0
\end{equation}
(a continuity argument extends the convergence from dyadic rationals to all $t,u,\vec{\xi}$).
Taking $t=u$ verifies \refcond{ZnBasic}~\refitem{PPPPointwiseConv}.
For \refcond{ZnBasic}~\refitem{PPPTightIntensityLeft}, fix $\epsilon>0$ and choose $\underline{t}\in\R$ such that $q(\underline{t})<\tfrac{1}{4}\epsilon^2$.
Then \eqrefPartII{Convp1Againstxi} implies that $\condE{\abs{Z'_{n,\underline{t}}}}{\F_n}\leq \tfrac{1}{2}\epsilon^2$ \whp, and Markov's inequality implies that $\condP{\abs{Z'_{n,\underline{t}}}\geq\epsilon}{\F_n} \leq \tfrac{1}{2}\epsilon$ {\whpdot}
\refcond{ZnBasic}~\refitem{PPPTightIntensityRight} follows from \eqrefPartII{Z''to0} and assumption~\refitem{PPPTightIntensityRightCoxTheorem} in \refthm{FirstPointCoxAsymp}.

Finally, let $\epsilon>0$ and a compact interval $[\underline{t},\overline{t}]$ be given.
Expanding the interval if necessary, we may assume that $\underline{t},\overline{t}$ are dyadic rationals, and decreasing $\epsilon$ if necessary we may assume $\overline{t}\leq\overline{t}(\epsilon)$.
Since $q$ is continuous and non-decreasing, we may choose a partition $t_0<\dotsb<t_N$ of $[\underline{t},\overline{t}]$ consisting of dyadic rationals such that
\begin{equation}
\sum_{i=1}^N \left( \frac{q(t_i)}{q(t_{i-1})}-1 \right)^2 \leq \frac{\epsilon}{4z_0(\epsilon)^2}.
\end{equation}
(it is enough to choose the partition finely enough that $\max_i (q(t_i)-q(t_{i-1})) \leq q(\underline{t})/4z_0(\epsilon)^2 q(\overline{t})$), and bound
\begin{align}
\sum_{i=1}^N \abs{\Delta Z'_{n,i}}^2
&= \sum_{i=1}^N \left( \bigabs{Z'_{n,t_i}}-\frac{q(t_i)}{q(t_{i-1})} \bigabs{Z'_{n,t_{i-1}}} + \Bigl( \frac{q(t_i)}{q(t_{i-1})} - 1 \Bigr) \bigabs{Z'_{n,t_{i-1}}} \right)^2
\notag\\&
\leq \sum_{i=1}^N 2 \left( \bigabs{Z'_{n,t_i}}-\frac{q(t_i)}{q(t_{i-1})} \bigabs{Z'_{n,t_{i-1}}} \right)^2 + 2 \left( \frac{q(t_i)}{q(t_{i-1})} - 1 \right)^2 \bigabs{Z'_{n,t_{i-1}}}^2
\notag\\&
\leq \frac{\epsilon}{2z_0(\epsilon)^2} \abs{Z^*_{n,\overline{t}(\epsilon)}}^2 + 2 \sum_{i=1}^N \left( \bigabs{Z'_{n,t_i}}-\frac{q(t_i)}{q(t_{i-1})} \bigabs{Z'_{n,t_{i-1}}} \right)^2
.
\end{align}
The latter sum is $\sop(1)$ by \eqrefPartII{ConvAgainstxiRatio} with $\vec{\xi}=\vec{0}$, and the remaining term is at most $\epsilon/2$ on the event $\bigset{\shortabs{Z^*_{n,\overline{t}(\epsilon)}}\leq z_0(\epsilon)}$.
This event has probability at least $1-\epsilon-o(1)$ by \eqrefPartII{Z*Tightness}, which completes the proof.
\end{proof}

\section{Moment estimates and the cluster after unfreezing}\lbsect{FrozenGeometry}

In this section we study $\cluster_t$ for $t\geq T_\unfr$, when the cluster resumes its CTBP behaviour.
We will use moment methods to prove \reflemma{AtCollision} and \refthm{PnStarSatisfiesConditions}, completing the proof of our results.

We introduce the frozen intensity measures

\begin{equation}\labelPartII{FrozenIntensity}
d\mu_{n,\fr}^{\sss(j)}(y) = \sum_{v\in\cluster_\fr^{\sss(j)}} \indicator{y\geq 0} \mu_n\bigl(T_\fr^{\sss(j)}-T_v + dy\bigr).
\end{equation}
Recall that the notation $\mu(t_0+dy)$ denotes the translation of the measure $\mu$ by $t_0$; thus \eqrefPartII{FrozenIntensity} means that, for a test function $h\geq 0$,
\begin{equation}\labelPartII{FrozenIntensityTestFunction}
\int h(y) d\mu_{n,\fr}^{\sss(j)}(y) = \sum_{v\in\cluster_\fr^{\sss(j)}} \int_{T_\fr^{\sss(j)}-T_v}^\infty h\left( y-(T_\fr^{\sss(j)}-T_v) \right) d\mu_n(y).
\end{equation}
\begin{lemma}\lblemma{ExpectedUnfrozenChildren}
Almost surely, for $j=1,2$,
\begin{equation}
s_n\leq
\int \e^{-\lambda_n(1) y} d\mu_{n,\fr}^{\sss(j)}(y)
\leq s_n+1.
\end{equation}
\end{lemma}

\reflemma{ExpectedUnfrozenChildren} is a restatement of \refother{\reflemmaPartI{ExpectedUnfrozenChildren}} and is proved in that paper.

For future reference, we now state a lemma, to be used in \refsubsect{SecondMomentEstimates}, showing that most of the mass of the frozen intensity measures $\mu_{n,\fr}^{\sss(j)}$ comes from small times.
This result generalizes \refother{\reflemmaPartI{FrozenVerticesAreYoung}}, which restricted to the case $a=1$, although the proof is the same.

\begin{lemma}\lblemma{FrozenVerticesAreYoung}
Let $\delta,\delta'>0$ and $a_0>0$ be given.
Then there exists $K<\infty$ and $n_0 \in \N$ such that, for all $a\geq a_0$ and $n\ge n_0$,
\begin{equation}
\P\left(
\int \e^{-\lambda_n(a^{1/s_n})y} \indicator{\lambda_n(1)y\geq K} d\mu_{n,\fr}^{\sss(j)}(y)
>\delta s_n
\right)\leq \delta'.
\end{equation}
\end{lemma}

\begin{proof}
Let $\epsilon=a_0 \e^{-\gamma}/2$, where $\gamma$ denotes Euler's constant.
Using the definition of $\mu_{n,\fr}^{\sss(j)}$ from \eqrefPartII{FrozenIntensity}, the monotonicity of $\lambda_n(\cdot)$ and \eqrefPartII{lambdanAsymp}, we obtain $n_0 \in \N$ such that for all $K<\infty$, $a\ge a_0$ and $n\ge n_0$,
\begin{equation}\labelPartII{MassUpperTailFrozenIntensity}
\int \e^{-\lambda_n(\tilde{a})y} \indicator{\lambda_n(1)y\geq K} d\mu_{n,\fr}^{\sss(j)}(y) \le
\sum_{v \in \cluster_\fr^{\sss(j)}} \int \e^{-\epsilon y/f_n(1)} \indicator{y\ge Kf_n(1)} \mu_n(T_\fr^{\sss(j)}-T_v+dy).
\end{equation}
According to \reflemma{BoundOnContribution}, for any $\epsilon'>0$ we can choose some $K<\infty$ such that, after possibly increasing $n_0$, the right-hand side of \eqrefPartII{MassUpperTailFrozenIntensity} is bounded from above by $\abs{\cluster_\fr^{\sss(j)}} \epsilon'/s_n$. Since $\abs{\cluster_\fr^{\sss(j)}}=O_\P(s_n^2)$ by \refthm{FrozenCluster}~\refitem{FrozenVolume}, the proof is complete.
\end{proof}

\subsection{A first moment estimate: Proof of \reflemma{AtCollision} \refitem{VolumeAtCollision}}\lbsubsect{VolumeCollisionProof}

In this section we show how to express $\cluster_t\setminus\cluster_\fr$, $t\geq T_\unfr$, as a suitable union of branching processes.
This representation leads to a simple proof of \reflemma{AtCollision}~\refitem{VolumeAtCollision}.
We will also use it in \refsubsect{SecondMomentEstimates} to prove \refthm{PnStarSatisfiesConditions}.

Consider the immediate children $v\in\boundary\cluster_\fr$ of individuals in the frozen cluster $\cluster_\fr$.
Then, for $t'\geq 0$,
\begin{equation}\labelPartII{UnfrozenUnion}
\cluster_{T_\unfr+t'}\setminus\cluster_\fr = \bigunion_{v\in\boundary\cluster_\fr\colon T_v^\cluster\leq T_\unfr+t'} \set{vw\colon w\in\BP^{\sss(v)}_{t'+T_\unfr-T_v^\cluster}},
\end{equation}
where $\BP^{\sss(v)}$ denotes the branching process of descendants of $v$, re-rooted and time-shifted as in \eqrefPartII{BPvDefinition}.
Furthermore, conditional on $\cluster_\fr$, the children $v\in\boundary\cluster_\fr$ appear according to a Cox process.
Formally, the point measures
\begin{equation}\labelPartII{UnfrozenChildrenCox}
\cP_{n,\unfr}^{\sss (j)}=\sum_{v\in\boundary\cluster_\fr^{(j)}} \delta_{(T_v^\cluster-T_\unfr,\BP^{(v)})}
\end{equation}
form Cox processes with intensities $d\mu_{n,\fr}^{\sss(j)} \otimes d\P(\BP^{\sss(1)}\in\cdot)$, $j=1,2$, where the frozen intensity measures $\mu_{n,\fr}^{\sss (j)}$ were introduced in \eqrefPartII{FrozenIntensity}.

\begin{proof}[Proof of \reflemma{AtCollision}~\refitem{VolumeAtCollision}]
By \refthm{FrozenCluster}~\refitem{FrozenVolume}, the volume $\abs{\cluster_\fr}$ of the frozen cluster is $O_\P(s_n^2)$, and this is $o_\P(\sqrt{n s_n})$ since $n/s_n^3\to\infty$.
It therefore suffices to show that $\abs{\cluster_{\overline{t}}\setminus\cluster_\fr}=O_\P(\sqrt{n s_n})$ when $\overline{t}=T_\unfr+\lambda_n(1)^{-1}(\tfrac{1}{2}\log(n/s_n^3)+K)$.

Abbreviate $t'=\overline{t}-T_\unfr=\lambda_n(1)^{-1}(\tfrac{1}{2}\log(n/s_n^3)+K)$.
By \eqrefPartII{UnfrozenUnion}--\eqrefPartII{UnfrozenChildrenCox},
\begin{align}
\abs{\cluster_{\overline{t}}\setminus\cluster_\fr}
&=
\sum_{v\in\boundary\cluster_\fr\colon T_v^\cluster\leq T_\unfr+t'} z^{1,\BP^{(v)}}_{t'+T_\unfr-T_v^\cluster}(1)
=
\sum_{j=1}^2 \int \indicator{t\leq t'} z^{1,bp}_{t'-t}(1) d\cP_{n,\unfr}^{\sss(j)}(t,bp)
,
\end{align}
so that, by \refthm{OneCharConv} and \reflemma{ExpectedUnfrozenChildren} there is $K'<\infty$ such that for sufficiently large $n$,
\begin{align}
\condE{\big.\abs{\cluster_{\overline{t}}\setminus\cluster_\fr}}{\cluster_\fr}
&=
\sum_{j=1}^2 \int \indicator{t\leq t'} \e^{\lambda_n(1)(t'-t)} \bar{m}_{t'-t}^1(1) d\mu_{n,\fr}^{\sss(j)}(t)
\notag\\&
\leq
\sqrt{\frac{n}{s_n^3}}\e^K \sum_{j=1}^2 \int \indicator{t\leq t'} K' s_n \e^{-\lambda_n(1)t} d\mu_{n,\fr}^{\sss(j)}(t)
\leq K' \e^K \sqrt{\frac{n}{s_n^3}}s_n(s_n+1).
\end{align}
Markov's inequality completes the proof.
\end{proof}

\subsection{Second moment estimates: Proof of \refthm{PnStarSatisfiesConditions}}\lbsubsect{SecondMomentEstimates}

In this section we prove that $\cP_n^*$ satisfies the assumptions of \refthm{FirstPointCoxAsymp}.
Namely, we will split $\mu_n=\mu_n^{\sss(K)}+(\mu_n-\mu_n^{\sss(K)})$ into the truncated measure and a remainder, as in \refsect{2VertexCharSec}.
This induces a splitting of the intensity measure into $Z'_n+Z''_n$, and the hypothesis \refitem{MomentAssumptionsCoxTheorem} will be verified using the estimates for the two-vertex characteristics $\chi_n^{\sss(K)}$ and $\chi_n-\chi_n^{\sss(K)}$ in Theorems~\refPartII{t:TwoVertexConvRestrictedSum} and \refPartII{t:TwoVertexRemainder}.
The remaining hypothesis \refitem{PPPTightIntensityRightCoxTheorem} will be proved using a separate argument.

Throughout the proof, the times $t$ and $t^*$ are related as in \eqrefPartII{tt*}, and we recall from \eqrefPartII{MomentGeneratingNotation} that, for a measure $Q$ on $\R^d$, we write $\hat{Q}$ for its moment generating function.

\begin{proof}[Proof of \refthm{PnStarSatisfiesConditions}]
Since $\cP_n^*$ is the image of a Cox process under a mapping that is measurable with respect to $\cluster_\fr$, it is immediate that it itself is a Cox process, and its intensity measure is
\begin{equation}
Z_{n,t^*}^*
=
\sum_{v_1\in\cluster_t^{\sss(1)}\setminus\cluster_\fr^{\sss(1)}} \sum_{v_2\in\cluster_t^{\sss(2)}\setminus\cluster_\fr^{\sss(2)}} \tfrac{1}{n} \mu_n\bigl(\Delta R_{v_1,v_2},R_1(t)-R_1(T_{v_1}^\cluster)+R_2(t)-R_2(T_{v_2}^\cluster) \bigr)
\delta_{(\abs{v_1}^*,\abs{v_2}^*)}
.
\end{equation}
In this sum, $T_{v_j}^\cluster \geq T_\unfr$, so that $R_{j'}(t)-R_{j'}(T_{v_j}^\cluster)=t-T_{v_j}^\cluster$ whenever $t\geq T_{v_j}^\cluster$, $j,j'\in\set{1,2}$, and, recalling \eqrefPartII{DeltaRcluster}, $\Delta R_{v_1,v_2}=\abs{T_{v_1}^\cluster-T_{v_2}^\cluster}$.

We begin by expressing $\hat{Z}^*_{n,t^*}(\vec{\xi})$ as a sum of two-vertex characteristics.
\ch{As in \eqrefPartII{UnfrozenUnion}, } any vertex $v'\in\cluster_t\setminus\cluster_\fr$ is descended from a unique vertex $v=p^\unfr(v')\in\boundary\cluster_\fr$ and can therefore be written as $v'=vw$ for some $w\in\BP^{\sss(v)}_{t-T_v^\cluster}$.
Hence $\abs{v'}-\abs{p^\unfr(v')}=\abs{vw}-\abs{v}=\abs{w}$.
Thus
\begin{align}
\hat{Z}^*_{n,t^*}(\vec{\xi})&
=
\sum_{v_1\in\boundary\cluster_\fr^{\sss(1)}} \sum_{w_1\in\BP^{\sss(v_1)}_{t-T_{v_1}^\cluster}} \sum_{v_2\in\boundary\cluster_\fr^{\sss(2)}} \sum_{w_2\in\BP^{\sss(v_2)}_{t-T_{v_2}^\cluster}} \exp\left( \frac{\xi_1 \abs{w_1}+\xi_2\abs{w_2}}{s_n\sqrt{\log(n/s_n^3)}} - (\xi_1+\xi_2)\frac{\phi_n}{2s_n}\sqrt{\log(n/s_n^3)}\right)
\notag\\&\qquad
\cdot
\tfrac{1}{n}\mu_n(\abs{T_{v_1 w_1}^\cluster-T_{v_2 w_2}^\cluster}, t-T_{v_1 w_1}^\cluster+t-T_{v_2 w_2}^\cluster)
.
\labelPartII{xiAgainstZ*4sums}
\end{align}
We note that $T_{v_j w_j}^\cluster=T_{v_j}^\cluster+T^{\sss(v_j)}_{w_j}$, where $T^{\sss(v_j)}_{w_j}$ denotes the birth time of $w_j$ in the branching process $\BP^{\sss(v_j)}$ defined in \eqrefPartII{BPvDefinition}.
 (Note that $T_{v_j}^{\cluster}\geq T_\unfr$, so that freezing plays no role after $T_{v_j}^\cluster$ and we need not consider $T_{w_j}^\cluster$.)  It follows that
\begin{equation}
\mu_n(\abs{T_{v_1 w_1}^\cluster-T_{v_2 w_2}^\cluster}, t-T_{v_1 w_1}^\cluster+t-T_{v_2 w_2}^\cluster)
=
\chi_n(t-T_{v_1}^\cluster-T^{\sss(v_1)}_{w_1}, t-T_{v_2}^\cluster-T^{\sss(v_2)}_{w_2}),
\end{equation}
where $\chi_n$ is the two-vertex characteristic from \eqrefPartII{TwoVertexChar}.
Recalling the notation from \eqrefPartII{2VertexCharDef},
\begin{align}
\hat{Z}^*_{n,t^*}(\vec{\xi})
=
\frac{1}{n}\exp\left( -(\xi_1+\xi_2)\frac{\phi_n}{2s_n}\sqrt{\log(n/s_n^3)} \right)
\sum_{v_1\in\boundary\cluster_\fr^{\sss(1)}} \sum_{v_2\in\boundary\cluster_\fr^{\sss(2)}} z^{\chi_n,\BP^{(v_1)},\BP^{(v_2)}}_{t-T_{v_1}^\cluster \! , \, t-T_{v_2}^\cluster}(\vec{\tilde{a}})
,
\labelPartII{xiAgainstZ*2sums}
\end{align}
where $\vec{\tilde{a}}=(a_1^{1/s_n},a_2^{1/s_n})$ as in \eqrefPartII{tildeVectorValued} and
\begin{equation}
a_j=\exp \left( \frac{\xi_j}{\sqrt{\log (n/s_n^3)}} \right) \! .
\end{equation}
Note that $a_1,a_2$ depend implicitly on $n$ and $a_j\to 1$ \ch{as $n\to\infty$}.

As in \refsubsect{VolumeCollisionProof}, we express the sums over $\boundary\cluster_\fr^{\sss(1)}$, $\boundary\cluster_\fr^{\sss(2)}$ in \eqrefPartII{xiAgainstZ*2sums} in terms of the point measures $\cP_{n,\unfr}^{\sss(j)}=\sum_{v\in\boundary\cluster_\fr^{\sss(j)}} \delta_{(T_v^\cluster-T_\unfr,\BP^{\sss(v)})}$ from \eqrefPartII{UnfrozenChildrenCox}:
\begin{align}
\hat{Z}^*_{n,t^*}(\vec{\xi})
&=
\frac{1}{n}\exp\left( -(\xi_1+\xi_2)\frac{\phi_n}{2s_n}\sqrt{\log(n/s_n^3)} \right)
\int d\cP_{n,\unfr}^{\sss(1)}(t_1,bp^{\sss(1)})
\notag\\&\qquad
\times \int d\cP_{n,\unfr}^{\sss(2)}(t_2,bp^{\sss(2)}) z^{\chi_n, bp^{(1)} \! , \, bp^{(2)}}_{t-T_\unfr-t_1, t-T_\unfr-t_2}(\vec{\tilde{a}})
.
\end{align}
Recalling the notation from \eqrefPartII{barredTwoVertexChars} and \eqrefPartII{tildeConvention},
\begin{align}
&\hat{Z}^*_{n,t^*}(\vec{\xi})
=
\frac{1}{n}\exp\left( -(\xi_1+\xi_2)\frac{\phi_n}{2s_n}\sqrt{\log(n/s_n^3)}+ (t-T_\unfr)(\lambda_n(\tilde{a}_1)+\lambda_n(\tilde{a}_2)) \right)
\notag\\&\qquad
\times
\int d\cP_{n,\unfr}^{\sss(1)}(t_1,bp^{\sss(1)}) \int d\cP_{n,\unfr}^{\sss(2)}(t_2,bp^{\sss(2)}) \e^{-\lambda_n(\tilde{a}_1)t_1-\lambda_n(\tilde{a}_2)t_2} \bar{z}^{\chi_n, bp^{(1)}, bp^{(2)}}_{t-T_\unfr-t_1 \! , \, t-T_\unfr-t_2}(\vec{\tilde{a}})
\notag\\&\quad
=
\exp\left( \frac{\lambda_n(\tilde{a}_1)+\lambda_n(\tilde{a}_2)}{\lambda_n(1)}\left( t^*+\tfrac{1}{2}\log(n/s_n^3) \right)  -(\xi_1+\xi_2)\frac{\phi_n}{2s_n}\sqrt{\log(n/s_n^3)} - \log(n/s_n^3) \right)
\notag\\&\qquad
\times
\frac{1}{s_n^3} \int d\cP_{n,\unfr}^{\sss(1)}(t_1,bp^{\sss(1)}) \int d\cP_{n,\unfr}^{\sss(2)}(t_2,bp^{\sss(2)}) \e^{-\lambda_n(\tilde{a}_1)t_1-\lambda_n(\tilde{a}_2)t_2} \bar{z}^{\chi_n, bp^{(1)}, bp^{(2)}}_{t-T_\unfr-t_1 \! , \, t-T_\unfr-t_2}(\vec{\tilde{a}})
.
\labelPartII{xiAgainstZ*Long}
\end{align}
For the non-random factors on the right-hand side of \eqrefPartII{xiAgainstZ*Long}, we use the asymptotics from \refcoro{lambdanTaylor}:
\begin{align}
\frac{\lambda_n(\tilde{a}_j)}{\lambda_n(1)}
&=
1+\frac{\phi_n}{s_n}\left(\exp\left( \xi_j/\sqrt{\log(n/s_n^3)} \right)-1\right)+o\left( \exp\left( \xi_j/\sqrt{\log(n/s_n^3)} \right)-1 \right)^2
\notag\\
&=
1+\frac{\phi_n}{s_n}\left(\xi_j/\sqrt{\log(n/s_n^3)} + \frac{1}{2}\xi_j^2/\log(n/s_n^3)\right) + o(1/\log(n/s_n^3))
.
\end{align}
Combining with the asymptotics $\lambda_n(\tilde{a}_j)/\lambda_n(1)=1+o(1)$ and $s_n/\phi_n=1+o(1)$ (see \eqrefPartII{lambdanAsymp} and \reflemma{lambdanAsymp}), we conclude that, for any fixed $t^* \in \R$,
\begin{equation}\labelPartII{ExponentAsymp}
\frac{\lambda_n(\tilde{a}_j)}{\lambda_n(1)}\bigl( t^*+\tfrac{1}{2}\log(n/s_n^3) \bigr) -\xi_j\frac{\phi_n}{2s_n}\sqrt{\log(n/s_n^3)}-\frac{1}{2}\log(n/s_n^3)
=t^*+\tfrac{1}{4}\xi_j^2 +o(1)
.
\end{equation}
(When $\xi_j=0$, $a_j=1$, the term $o(1)$ in \eqrefPartII{ExponentAsymp} is absent.)  Combining \eqrefPartII{xiAgainstZ*Long} and \eqrefPartII{ExponentAsymp},
\begin{align}
\hat{Z}^*_{n,t^*}(\vec{\xi})
&=
\frac{\e^{2t^*+\frac{1}{4}\norm{\vec{\xi}}^2+o(1)}}{s_n^3}
\int d\cP_{n,\unfr}^{\sss (1)}(t_1,bp^{\sss (1)}) \int d\cP_{n,\unfr}^{\sss (2)}(t_2,bp^{\sss (2)})
\notag \\
&\qquad
\cdot \e^{-\lambda_n(\tilde{a}_1)t_1} \e^{-\lambda_n(\tilde{a}_2)t_2} \bar{z}_{t-T_\unfr-t_1 \! , \, t-T_\unfr-t_2}^{\chi_n,bp^{(1)},bp^{(2)}}(\vec{\tilde{a}})
.
\labelPartII{xiAgainstZ*}
\end{align}
This is the desired representation in terms of two-vertex characteristics.

Given $K<\infty$, define
\begin{equation}
Z'_{n,t^*}
=
\sum_{v_1\in\cluster_t^{\sss(1)}\setminus\cluster_\fr^{\sss(1)}} \sum_{v_2\in\cluster_t^{\sss(2)}\setminus\cluster_\fr^{\sss(2)}} \tfrac{1}{n} \mu_n^{\sss(K)}\left(\abs{T_{v_1}^\cluster-T_{v_2}^\cluster},t-T_{v_1}^\cluster+t-T_{v_2}^\cluster \right)
\delta_{(\abs{v_1}^*,\abs{v_2}^*)},
\end{equation}
similar to \refdefn{PnStar}, and set $Z''_{n,t^*}=Z_{n,t^*}^*-Z'_{n,t^*}$.
(We suppress the dependence of $Z'_n,Z''_n$ on $K$.)  Clearly, \eqrefPartII{xiAgainstZ*} remains true when $Z_{n,t^*}^*$ and $\chi_n$ are replaced by $Z'_{n,t^*}$ and $\chi_n^{\sss(K)}$, or by $Z''_{n,t^*}$ and $\chi_n-\chi_n^{\sss(K)}$, respectively.
We will use \eqrefPartII{xiAgainstZ*} to control first and second moments of $Z'_{n,t^*}$, for arbitrary but fixed values $t^*,\xi_1,\xi_2\in\R$, and first moments of $Z''_{n,t^*}$ with $\xi_1=\xi_2=0$.

\paragraph{Verification of \eqrefPartII{FirstMomentAgainstxi}:}
Writing $\lambda_n(\vec{\tilde{a}})=(\lambda_n(\tilde{a}_1),\lambda_n(\tilde{a}_2))$, we obtain from \eqrefPartII{xiAgainstZ*},
\begin{align}
&\condE{\hat{Z}'_{n,t^*}(\vec{\xi})}{\cluster_\fr}
=
\frac{\e^{2t^*+\frac{1}{4}\norm{\vec{\xi}}^2+o(1)}}{s_n^2}\iint \e^{-\lambda_n(\vec{\tilde{a}})\cdot \vec{t}} \frac{\bar{m}^{\chi_n^{\sss(K)}}_{t-T_\unfr-\vec{t}}(\vec{\tilde{a}})}{s_n} d\mu_{n,\fr}^{\sss(1)}(t_1) d\mu_{n,\fr}^{\sss(2)}(t_2)
.
\labelPartII{EgivenCfr}
\end{align}
Since $\lambda_n(1)(t-T_\unfr)=\tfrac{1}{2}\log(n/s_n^3)+t^*\to\infty$ and $a_j\to 1$, \refthm{TwoVertexConvRestrictedSum} and \refcoro{lambdanTaylor} imply that for $\eps \in (0,1)$ and $K<\infty$ sufficiently large, there exists $K'<\infty$ and $n_0 \in \N$ such that the integrand of \eqrefPartII{EgivenCfr} lies between $\e^{-\lambda_n(1)(t_1+t_2)-\eps}(1-\epsilon)$ and $\e^{-\lambda_n(1)(t_1+t_2)+\eps}(1+\epsilon)$ for $\lambda_n(1)[t_1 \vee t_2]\leq K'$ and $n\ge n_0$, and there are constants $K''<\infty, n_0'\in \N$ such that it is bounded by $K'' \e^{-\lambda_n(\vec{\tilde{a}})\cdot \vec{t}}$ for all $t_1,t_2\ge 0$, $n\ge n_0'$.
Using Lemmas~\refPartII{l:ExpectedUnfrozenChildren} and \refPartII{l:FrozenVerticesAreYoung}, it easily follows that
\begin{equation}
\e^{2t^*+\frac{1}{4}\norm{\vec{\xi}}^2}(1-2\epsilon)
\leq
\condE{\hat{Z}'_{n,t^*}(\vec{\xi})}{\cluster_\fr}
\leq
\e^{2t^*+\frac{1}{4}\norm{\vec{\xi}}^2}(1+2\epsilon)
\end{equation}
{\whpdot}  Since $\e^{\frac{1}{4}\norm{\vec{\xi}}^2}=\hat{Q}(\vec{\xi})$, where $Q$ is the law of two independent $N(0,\tfrac{1}{2})$ variables, we have therefore verified the first moment condition \eqrefPartII{FirstMomentAgainstxi} of \refthm{FirstPointCoxAsymp} with $q(t^*)=\e^{2t^*}$.

\paragraph{Verification of \eqrefPartII{SplittingAssumptions}:}
For $Z''_{n,t^*}$, set $\vec{\xi}=\vec{0}:=(0,0)$.
In \eqrefPartII{EgivenCfr}, we can replace $Z'_{n,t^*}$ and $\chi_n^{\sss(K)}$ by $Z''_{n,t^*}$ and $\chi_n-\chi_n^{\sss(K)}$, and \refthm{TwoVertexRemainder} implies that, for any $\eps \in (0,1)$ and large $K$, there is $K'<\infty$ such that the integrand of the resulting equation is at most $\epsilon \e^{-\lambda_n(1)(t_1+t_2)}$, uniformly for $\lambda_n(1)[t_1 \vee t_2]\leq K'$ for large $n$, and there is $K''<\infty$ such that it is bounded by $K''\e^{-\lambda_n(1)(t_1+t_2)}$ for all $t_1,t_2\ge 0$, $n$ large.
This verifies the first moment condition \eqrefPartII{SplittingAssumptions}, as in the previous case.

\paragraph{Verification of \eqrefPartII{SecondMomentAgainstxi}:}
The second moment estimates for $Z'_{n,t^*}$, though somewhat more complicated, are similar in spirit.
Note the importance of freezing, which is not so far apparent from the first moment calculations only: the freezing times $T_\fr^{\sss(j)}$ and the rescaled times $t\approx T_\unfr+\tfrac{1}{2}\lambda_n(1)^{-1}\log(n/s_n^3)$ have exactly the scaling needed so that \emph{both} the first and second moments of $Z^*_{n,t^*}$ will have order $1$, even though $\bar{m}_{\vec{t}}^{\chi_n^{\sss(K)}}(\vec{\tilde{a}})\approx s_n$, $\bar{M}^{\chi_n^{\sss(K)},\chi_n^{\sss(K)}}_{\vec{t},\vec{u}}(\vec{\tilde{a}},\vec{\tilde{b}})\approx s_n^4$ by \refthm{TwoVertexConvRestrictedSum}.

Equation \eqrefPartII{xiAgainstZ*} for $Z_{n,t^*}'$ expresses $\hat{Z}'_{n,t^*}(\vec{\xi})$ in terms of a double integral with respect to a pair of Cox processes, whose intensity is measurable with respect to $\cluster_\fr$.
An elementary calculation (applying \eqrefPartII{PPPMeanCovar} twice) shows that, for a pair $\cP^{\sss(1)},\cP^{\sss(2)}$ of Poisson point processes with intensities $\nu_1,\nu_2$,
\begin{align}
&\E\left( \iint f(x_1,x_2) d\cP^{\sss(1)}(x_1)d\cP^{\sss(2)}(x_2) \iint g(y_1,y_2) d\cP^{\sss(1)}(y_1)d\cP^{\sss(2)}(y_2) \right)
\notag\\
&\quad=
\iint f(x_1,x_2)g(x_1,x_2) d\nu_1(x_1)d\nu_2(x_2) + \iiint f(x_1,x_2) g(x_1,y_2) d\nu_1(x_1) d\nu_2(x_2) d\nu_2(y_2)
\notag\\
&\qquad+
\iiint f(x_1,x_2) g(y_1,x_2) d\nu_1(x_1) d\nu_1(y_1) d\nu_2(x_2)
\notag\\
&\qquad+
\iint f(x_1,x_2)d\nu_1(x_1)d\nu_2(x_2)\iint g(y_1,y_2)d\nu_1(y_1)d\nu_2(y_2).
\labelPartII{SecondMomentDoublePPP}
\end{align}
We apply \eqrefPartII{SecondMomentDoublePPP} to \eqrefPartII{xiAgainstZ*}.
In the notation of \eqrefPartII{SecondMomentDoublePPP}, we have $x_j=(t_j,bp^{\sss(j)})$, and $f(x_1,x_2)=g(x_1,x_2)=\e^{-\lambda_n(\tilde{a}_1)t_1} \e^{-\lambda_n(\tilde{a}_2)t_2} \bar{z}_{t-T_\unfr-t_1 \! , \, t-T_\unfr-t_2}^{\chi^{(K)}_n,bp^{(1)},bp^{(2)}}(\vec{\tilde{a}})$.
Integration against the intensity measure $d\nu_j=d\mu_{n,\fr}^{\sss(j)} \otimes d\P(\BP^{\sss(j)}\in\cdot)$ is equivalent to a branching process expectation together with an integration over $t_j$, and we will therefore obtain first or second moments of one- or two-vertex characteristics from the various terms in \eqrefPartII{SecondMomentDoublePPP}.
Namely, for $\vec{\xi}, \vec{\zeta} \in \R^2$ and writing $b_j=\exp(\zeta_j/\sqrt{\log(n/s_n^3)})$, we obtain
\begin{align}
&s_n^6 \e^{-2t^*-\frac{1}{4}\norm{\vec{\xi}}^2-2u^*-\frac{1}{4}\norm{\vec{\zeta}}^2-o(1)}\condE{ \hat{Z}'_{n,t^*}(\vec{\xi}) \hat{Z}'_{n,u^*}(\vec{\zeta})  }{\cluster_\fr}
\notag\\
&\quad=
\iint \e^{-(\lambda_n(\tilde{a}_1)+\lambda_n(\tilde{b}_1))t_1} \e^{-(\lambda_n(\tilde{a}_2)+\lambda_n(\tilde{b}_2))t_2} \bar{M}_{t-T_\unfr-\vec{t},u-T_\unfr-\vec{t}}^{\chi_n^{\sss(K)},\chi_n^{\sss(K)}}\left(\vec{\tilde{a}},\vec{\tilde{b}}\right) d\mu_{n,\fr}^{\sss(1)}(t_1) d\mu_{n,\fr}^{\sss(2)}(t_2)
\notag\\
&\qquad
\begin{aligned}
+ \iiint
&\e^{-(\lambda_n(\tilde{a}_1)+\lambda_n(\tilde{b}_1))t_1}\e^{-\lambda_n(\tilde{a}_2)t_2}\e^{-\lambda_n(\tilde{b}_2)u_2}
\bar{M}_{t-T_\unfr-t_1,u-T_\unfr-t_1}^{\rho_{t-T_\unfr-t_2,\tilde{a}_2},\rho_{u-T_\unfr-u_2,\tilde{b}_2}}(\tilde{a}_1,\tilde{b}_1)
\\&\quad
\times d\mu_{n,\fr}^{\sss(1)}(t_1) d\mu_{n,\fr}^{\sss(2)}(t_2) d\mu_{n,\fr}^{\sss(2)}(u_2)
\\
+\iiint
&\e^{-(\lambda_n(\tilde{a}_2)+\lambda_n(\tilde{b}_2))t_2}\e^{-\lambda_n(\tilde{a}_1)t_1}\e^{-\lambda_n(\tilde{b}_1)u_1}
\bar{M}_{t-T_\unfr-t_2,u-T_\unfr-t_2}^{\rho_{t-T_\unfr-t_1,\tilde{a}_1},\rho_{u-T_\unfr-u_1,\tilde{b}_1}}(\tilde{a}_2,\tilde{b}_2)
\\&\quad
\times d\mu_{n,\fr}^{\sss(1)}(t_2) d\mu_{n,\fr}^{\sss(2)}(t_1) d\mu_{n,\fr}^{\sss(2)}(u_1)
\end{aligned}
\notag\\
&\qquad+
\iint \e^{-\lambda_n(\tilde{a}_1)t_1} \e^{-\lambda_n(\tilde{a}_2)t_2} \bar{m}_{t-T_\unfr-\vec{t}}^{\chi_n^{\sss(K)}}(\vec{\tilde{a}}) d\mu_{n,\fr}^{\sss(1)}(t_1) d\mu_{n,\fr}^{\sss(2)}(t_2)
\notag\\
&\qquad\qquad\qquad \times
\iint \e^{-\lambda_n(\tilde{b}_1)u_1} \e^{-\lambda_n(\tilde{b}_2)u_2} \bar{m}_{u-T_\unfr-\vec{u}}^{\chi_n^{\sss(K)}}(\vec{\tilde{b}}) d\mu_{n,\fr}^{\sss(1)}(u_1) d\mu_{n,\fr}^{\sss(2)}(u_2)
,
\labelPartII{SecondMomentZ*}
\end{align}
where $\rho_{t_2,\tilde{a}_2}(t'_1)=\bar{m}_{t_2}^{\chi_n^{\sss(K)}(t'_1,\cdot)}(\tilde{a}_2)$ is the characteristic from \eqrefPartII{IteratedCharMeans}.
Abbreviate
\begin{equation}
D=\e^{-2t^*-\frac{1}{4}\norm{\vec{\xi}}^2}\hat{Z}'_{n,t^*}(\vec{\xi})-\e^{-2u^*}\abs{Z'_{n,u^*}}=\e^{-2t^*-\frac{1}{4}\norm{\vec{\xi}}^2}\hat{Z}'_{n,t^*}(\vec{\xi})-\e^{-2u^*}\hat{Z}'_{n,u^*}(\vec{0}).
\end{equation}
Then
\begin{align}
\condE{D^2}{\cluster_\fr}
&=
\e^{-4t^*-\frac{2}{4}\norm{\vec{\xi}}^2}\condE{\hat{Z}_{n,t^*}'(\vec{\xi}) \hat{Z}_{n,t^*}'(\vec{\xi})}{\cluster_\fr}-2\e^{-2t^*-\frac{1}{4}\norm{\vec{\xi}}^2-2u^*}\condE{\hat{Z}_{n,t^*}'(\vec{\xi}) \hat{Z}_{n,u^*}'(\vec{0})}{\cluster_\fr}
\notag\\
& \qquad \qquad
+\e^{-4u^*}\condE{\hat{Z}_{n,u^*}'(\vec{0}) \hat{Z}_{n,u^*}'(\vec{0})}{\cluster_\fr}
.
\labelPartII{SecondMomentOfDifference}
\end{align}
Applying \eqrefPartII{SecondMomentZ*} to each of the three terms in \eqrefPartII{SecondMomentOfDifference} gives twelve summands.
From the first term in the right-hand side of \eqrefPartII{SecondMomentZ*}, we obtain
\begin{align}
\frac{\e^{o(1)}}{s_n^2}\iint
s_n^{-4}&\left[
\e^{-2\lambda_n(\tilde{a}_1)t_1} \e^{-2\lambda_n(\tilde{a}_2)t_2} \bar{M}_{t-T_\unfr-\vec{t},t-T_\unfr-\vec{t}}^{\chi_n^{\sss(K)},\chi_n^{\sss(K)}}\left(\vec{\tilde{a}},\vec{\tilde{a}}\right)
\right.
\notag\\
&
- 2\e^{-(\lambda_n(\tilde{a}_1)+\lambda_n(1))t_1} \e^{-(\lambda_n(\tilde{a}_2)+\lambda_n(1))t_2} \bar{M}_{t-T_\unfr-\vec{t},u-T_\unfr-\vec{t}}^{\chi_n^{\sss(K)},\chi_n^{\sss(K)}}\left(\vec{\tilde{a}},\vec{1}\right)
\notag\\
&\left.
+ \, \e^{-2\lambda_n(1)t_1} \e^{-2\lambda_n(1)t_2} \bar{M}_{u-T_\unfr-\vec{t},u-T_\unfr-\vec{t}}^{\chi_n^{\sss(K)},\chi_n^{\sss(K)}}\left(\vec{1},\vec{1}\right)
\right]
d\mu_{n,\fr}^{\sss(1)}(t_1) d\mu_{n,\fr}^{\sss(2)}(t_2)
.
\labelPartII{FirstTermSecondMomentOfDifference}
\end{align}
As in \eqrefPartII{EgivenCfr}, \refthm{TwoVertexConvRestrictedSum} and \refcoro{lambdanTaylor} imply that the integrand in \eqrefPartII{FirstTermSecondMomentOfDifference} is at most $(4\epsilon+o(1))\e^{-2\lambda_n(1)t_1+o(1)} \e^{-2\lambda_n(1)t_2+o(1)}$ in absolute value, uniformly for $\lambda_n(1)[t_1 \vee t_2]\leq K'$, and is otherwise bounded by $4K''\e^{-2\lambda_n(1\wedge \tilde{a}_1)t_1}\e^{-2\lambda_n(1\wedge\tilde{a}_2)t_2}$.
Using Lemmas~\refPartII{l:ExpectedUnfrozenChildren} and \refPartII{l:FrozenVerticesAreYoung}, it easily follows that the quantity in \eqrefPartII{FirstTermSecondMomentOfDifference} is at most $5\epsilon$ in absolute value, {\whpdot}
 From the second term in the right-hand side of \eqrefPartII{SecondMomentZ*}, we obtain similarly
\begin{align}
\labelPartII{SecondTermSecondMomentOfDifference}
\frac{\e^{o(1)}}{s_n^3}\iiint
s_n^{-3}&\left[
\e^{-2\lambda_n(\tilde{a}_1)t_1} \e^{-2\lambda_n(\tilde{a}_2)(t_2+u_2)} \bar{M}_{t-T_\unfr-t_1,t-T_\unfr-t_1}^{\rho_{t-T_\unfr-t_2,\tilde{a}_2},\rho_{t-T_\unfr-u_2,\tilde{a}_2}}\left(\tilde{a}_1,\tilde{a}_1\right)
\right.
\\
&
- 2\e^{-(\lambda_n(\tilde{a}_1)+\lambda_n(1))t_1}\e^{-\lambda_n(\tilde{a}_2)t_2}\e^{-\lambda_n(1)u_2} \bar{M}_{t-T_\unfr-t_1,u-T_\unfr-t_1}^{\rho_{t-T_\unfr-t_2,\tilde{a}_2},\rho_{u-T_\unfr-u_2,1}}(\tilde{a}_1,1)
\notag\\
&\left.
+ \, \e^{-2\lambda_n(1)t_1} \e^{-2\lambda_n(1)(t_2+u_2)} \bar{M}_{u-T_\unfr-t_1,u-T_\unfr-t_1}^{\rho_{u-T_\unfr-t_2,1},\rho_{u-T_\unfr-u_2,1}}\left(1,1\right)
\right]
d\mu_{n,\fr}^{\sss(1)}(t_1) d\mu_{n,\fr}^{\sss(2)}(t_2) d\mu_{n,\fr}^{\sss(2)}(u_2)
.
\notag
\end{align}
Arguing from \refprop{HalfTwoVertexConvRestricted} instead of \refthm{TwoVertexConvRestrictedSum}, the quantity in \eqrefPartII{SecondTermSecondMomentOfDifference} is again at most $5\epsilon$ in absolute value, {\whpdot}
The third and fourth terms from \eqrefPartII{SecondMomentZ*} are analogous.
This verifies the second moment condition \eqrefPartII{SecondMomentAgainstxi}.

\paragraph{Verification of \eqrefPartII{ManyParticles}:}
Finally we verify condition \refitem{PPPTightIntensityRightCoxTheorem} of \refthm{FirstPointCoxAsymp}.
For $R\in (0,\infty)$, let $v_{j,R}$ be the first vertex in $\cluster^{\sss(j)}$ born after time $T_\unfr$ with $\abs{\BP_{f_n(1)}^{\sss(v)}}\ge Rs_n^2$.
It was proved in \refother{\reflemmaPartI{RluckyBornSoon}} that
\begin{equation}\labelPartII{RluckyBornSoon}
T_{v_{j,R}}^\cluster =T_\unfr +\Op(f_n(1)).
\end{equation}
Let $\eps>0$ and choose $n_0 \in \N$ and $C' <\infty$ such that $\lambda_n(1)(f_n(1)+f_n(1+1/s_n)) \le C'$ for all $n\ge n_0$.
Choose $R \in (0,\infty)$ such that $\frac{16}{(\log 2)^4}\e^{4C'} (1+2R)/R^2 \le \eps$.
After possibly increasing $n_0$, \eqrefPartII{RluckyBornSoon} and \eqrefPartII{lambdanAsymp} yield a constant $C'' \in (0,\infty)$ such that with probability at least $1-\eps$, $\lambda_n(1)(T_{v_{j,R}}^\cluster - T_\unfr)\le C''$.
Denote this event by $\mathcal{A}$ and assume for the remainder of the proof that $\mathcal{A}$ holds.

Set $\cL_R^{\sss(j)}$ to be the collection of descendants $w$ of $v_{j,R}$ such that $T_w^\cluster-T_{v_{j,R}}^\cluster\leq f_n(1)$ (thus $\abs{\cL_R^{\sss(j)}}\geq Rs_n^2$ by definition).
We will repeat the previous arguments used to analyze $Z_{n,t^*}$, with $\cL_R^{\sss(j)}$ playing the role of $\cluster_\fr^{\sss(j)}$.
Instead of $\boundary\cL_R^{\sss(j)}$, we consider the subset
\begin{equation}
\cU_R^{\sss(j)}=\set{v\in\boundary\cL_R^{\sss(j)}\colon T_{\parent{v}}^\cluster+f_n(1)<T_v^\cluster<T_{\parent{v}}^\cluster+f_n(1+1/s_n)}
\end{equation}
(the set of immediate children of $\cL_R^{\sss(j)}$ born to a parent of age between $f_n(1)$ and $f_n(1+1/s_n)$).
Since $v_{j,R}$ is born after $T_\unfr$, it follows immediately that $\cluster_\fr^{\sss(j)} \cap \cU_R^{\sss(j)}=\emptyset$.
Moreover the arrival time $T_v^\cluster$ of any $v\in\cU_R^{\sss(j)}$ satisfies
\begin{equation}\labelPartII{muntildeSupport}
T_{v_{j,R}}^{\cluster}+ f_n(1) <T_{\parent{v}}^\cluster+f_n(1) <T_v^\cluster <T_{\parent{v}}^\cluster+f_n(1+1/s_n)\le T_{v_{j,R}}^{\cluster}+ f_n(1) + f_n(1+1/s_n).
\end{equation}

Let $\tilde{\cP}^{\sss(j)}_{n,R}=\sum_{v\in\cU_R^{\sss(j)}} \delta_{(T_v^\cluster-T_\unfr, \BP^{\sss(v)})}$ (cf.\ \eqrefPartII{UnfrozenChildrenCox}).
The definitions of $v_{j,R}$ and $\cL_R^{\sss(j)}$ depend only on descendants born to parents of age at most $f_n(1)$, whereas $\cU_R^{\sss(j)}$ consists of descendants born to parents of age greater than $f_n(1)$.
It follows that $\tilde{\cP}^{\sss(j)}_{n,R}$ is a Cox process conditional on $\cL_R^{\sss(j)}$ with intensity measure $d\tilde{\mu}_{n,R}^{\sss(j)}\otimes d\P(\BP^{\sss(1)}\in\cdot)$, where $\tilde{\mu}_{n,R}$ is the measure such that, for all measurable, nonnegative functions $h$ on $\R$,
\begin{equation}
\int h(y) \, d\tilde{\mu}_{n,R}^{\sss(j)}(y)= \sum_{w \in \cL_R^{\sss(j)}} \int_{f_n(1)}^{f_n(1+1/s_n)} h(y-(T_\fr^{\sss(j)}-T_w)) \, d\mu_n(y).
\end{equation}
(These are the analogues of $\cP_{n,\unfr}^{\sss(j)}$ and $\mu_{n,\fr}^{\sss(j)}$, see \eqrefPartII{FrozenIntensityTestFunction} and \eqrefPartII{UnfrozenChildrenCox}.)
In particular, the total mass of $\tilde{\mu}_{n,R}^{\sss(j)}$ is at least $Rs_n$.
Using \eqrefPartII{muntildeSupport} we find that its support is in $[T_{v_{j,R}}^\cluster-T_\unfr+f_n(1), T_{v_{j,R}}^\cluster-T_\unfr+f_n(1)+f_n(1+1/s_n)]$ and, uniformly over $t_j$ in that set and $t^*\ge 0$,
\begin{align}
\lambda_n(1)(T_{v_{j,R}}^\cluster - T_\unfr)  \le \lambda_n(1) t_j \le  \lambda_n(1)(T_{v_{j,R}}^\cluster - T_\unfr) +C',
\labelPartII{tjBounds} \\
\lambda_n(1) (t-T_\unfr-t_j) \ge \frac{1}{2}\log(n/s_n^3)-C''-C'.
\labelPartII{TimeArgumentLarge}
\end{align}
In particular, the fact that the right-hand side of \eqrefPartII{TimeArgumentLarge} tends to infinity implies that, possibly after increasing $n_0$, $\cU_R^{\sss(j)} \subset \cluster_t^{\sss(j)}$ whenever $n\geq n_0$ and $t^*\geq 0$.
Furthermore, \refthm{TwoVertexConvRestrictedSum} and \refprop{HalfTwoVertexConvRestricted} yield a constant $K<\infty$ such that
\begin{align}
&\bar{m}^{\chi_n^{(K)}}_{t-T_\unfr-t_1,t-T_\unfr-t_2}(\vec{1}) \geq s_n/2,
\labelPartII{full1stMoment}\\
& \bar{M}_{t-T_\unfr-\vec{t},u-T_\unfr-\vec{t}}^{\chi_n^{\sss(K)},\chi_n^{\sss(K)}}\left(\vec{1},\vec{1}\right) \le \frac{2}{(\log 2)^2} s_n^4,
\labelPartII{full2ndMomentChar}\\
&\bar{M}_{t-T_\unfr-t_i,u-T_\unfr-t_i}^{\rho_{t-T_\unfr-t_j,1},\rho_{u-T_\unfr-u_j,1}}(1,1)\le  \frac{2}{\log 2}s_n^3,
\labelPartII{partial2ndMomentChar}
\end{align}
for $\set{i,j}=\set{1,2}$, all $t_1,t_2$ in the support of $\tilde{\mu}_{n,R}^{\sss(j)}$ and $n\ge n_0$.
Using this $K$ to truncate, let $\tilde{Z}'_{n,t^*,R}$ denote the restriction of $Z'_{n,t^*}$ to pairs $(v_1,v_2)$ for which each $v_j$ is a descendant of some vertex of $\cU_R^{\sss(j)}$.
As in the argument leading to \eqrefPartII{xiAgainstZ*}, we conclude that
\begin{align}
\abs{Z_{n,t^*}^*} \ge \bigabs{\tilde{Z}'_{n,t^*,R}}
=\frac{\e^{2t^*}}{s_n^3} \int d\tilde{\cP}^{\sss(1)}_{n,R}(t_1,bp^{\sss(1)}) \int d\tilde{\cP}^{\sss(2)}_{n,R}(t_2,bp^{\sss(2)}) \e^{-\lambda_n(1)(t_1+t_2)} \bar{z}^{\chi_n^{(K)},bp^{(1)},bp^{(2)}}_{t-T_\unfr-t_1,t-T_\unfr-t_2}(\vec{1}).
\end{align}
Hence, on $\mathcal{A}$, we may use \eqrefPartII{tjBounds} and \eqrefPartII{full1stMoment} to obtain
\begin{align}
&\condE{\bigabs{\tilde{Z}'_{n,t^*,R}}}{\cL^{\sss(1)}_R, \cL^{\sss(2)}_R}
\labelPartII{tildeZmean}\\&\quad
\geq
\frac{\e^{2t^*}}{s_n^3} \int d\tilde{\mu}_{n,R}^{\sss(1)}(t_1) \int d\tilde{\mu}_{n,R}^{\sss(2)}(t_2) \exp\left(-\lambda_n(1)(T_{v_{1,R}}^\cluster-T_\unfr+T_{v_{2,R}}^\cluster-T_\unfr)-2C'\right)
s_n/2
\notag\\& \quad
=
\frac{1}{2}\exp\left(2t^*-\lambda_n(1)(T_{v_{1,R}}^\cluster-T_\unfr+T_{v_{2,R}}^\cluster-T_\unfr)-2C'\right)\frac{\abs{\tilde{\mu}_{n,R}^{\sss(1)}}}{s_n} \frac{\abs{\tilde{\mu}_{n,R}^{\sss(2)}}}{s_n}
\labelPartII{tildeZmeanbound}\\
& \quad
\ge  \frac{R^2}{2} \exp\left(2t^*- 2 (C''+C')\right).
\labelPartII{tildeZmeanRoughBound}
\end{align}
Similarly, using \eqrefPartII{SecondMomentDoublePPP} to compute the variance of $\bigabs{\tilde{Z}'_{n,t^*,R}}$ as in \eqrefPartII{SecondMomentZ*}, and employing \eqrefPartII{tjBounds} and \eqrefPartII{full2ndMomentChar}--\eqrefPartII{partial2ndMomentChar} for the estimation, we find that, on $\mathcal{A}$,
\begin{align}
\Var\condparentheses{ \bigabs{\tilde{Z}'_{n,t^*,R}} }{\cL^{\sss(1)}_R, \cL^{\sss(2)}_R}
&\leq
\frac{2}{(\log 2)^2}\exp\left(4t^*-2\lambda_n(1)(T_{v_{1,R}}^\cluster - T_\unfr+T_{v_{2,R}}^\cluster - T_\unfr) \right)
\notag\\
&\quad
\times\left(
\frac{\abs{\tilde{\mu}_{n,R}^{\sss(1)}} \abs{\tilde{\mu}_{n,R}^{\sss(2)}} s_n^4 }{s_n^6}
+\,
\frac{\abs{\tilde{\mu}_{n,R}^{\sss(1)}} \abs{\tilde{\mu}_{n,R}^{\sss(2)}}^2 s_n^3
+
\abs{\tilde{\mu}_{n,R}^{\sss(1)}}^2 \abs{\tilde{\mu}_{n,R}^{\sss(2)}} s_n^3 }{s_n^6}
\right)
.
\labelPartII{tildeZvariance}
\end{align}
Abbreviate the conditional mean from \eqrefPartII{tildeZmean} as $m$.
Chebyshev's inequality, \eqrefPartII{tildeZmeanbound} and \eqrefPartII{tildeZvariance} give
\begin{align}
\condP{ \bigabs{\tilde{Z}'_{n,t^*,R}} \geq \tfrac{1}{2} m}{\cL^{\sss(1)}_R, \cL^{\sss(2)}_R}
&\geq 1-\frac{16}{(\log 2)^4}\e^{4C'}
\left( \frac{s_n^2}{\abs{\tilde{\mu}_{n,R}^{\sss(1)}} \abs{\tilde{\mu}_{n,R}^{\sss(2)}}}
+ \frac{s_n}{\abs{\tilde{\mu}_{n,R}^{\sss(1)}}} + \frac{s_n}{\abs{\tilde{\mu}_{n,R}^{\sss(2)}}} \right)
\notag \\
&\ge 1-\frac{16}{(\log 2)^4}\e^{4C'}\left( \frac{1}{R^2} + 2\frac{1}{R} \right) \ge 1-\eps,
\labelPartII{tildeZbound}
\end{align}
on $\mathcal{A}$.
For given $C<\infty$, we use \eqrefPartII{tildeZmeanRoughBound} to choose $t^*$ sufficiently large that $\frac{1}{2}m \ge C$.
The claim now follows from $\abs{Z_{n,t^*}^*} \ge \bigabs{\tilde{Z}_{n,t^*,R}'}$ and \eqrefPartII{tildeZbound}.
\end{proof}

\subsection{No collisions from the frozen cluster: Proof of \reflemma{AtCollision}~\refitem{NoFrozenCollision}}\lbsubsect{FrozenCollisions}

In this section we prove \reflemma{AtCollision}~\refitem{NoFrozenCollision}, which will show that \whp\ the collision edge neither starts nor ends in the frozen cluster.

\begin{defn}\lbdefn{Lucky}
Let $\epsilon_1$ denote the constant from \reflemma{ModerateAgeContribution} for $K=1$.
Call a vertex $v \in\tree^{\sss (j)}$ \emph{lucky} if $\bigabs{\BP^{\sss(v)}_{f_n(1)}} \geq s_n^2/\epsilon_1$, and set
\begin{equation}\labelPartII{LuckyTimeDefn}
T^{\sss(j)}_\lucky=\inf\set{T_v\colon v\in\tree^{\sss(j)}\setminus\set{\emptyset_j}\text{ is lucky and }T_v> T_{\parent{v}}+f_n(1)},
\end{equation}
the first time that a lucky vertex is born to a parent of age greater than $f_n(1)$.
\end{defn}

In view of \refdefn{Freezing} and \reflemma{ModerateAgeContribution}, we have
\begin{equation}\labelPartII{FreezingFromLucky}
v\in\tree^{\sss(j)}\text{ is lucky} \quad \implies \quad T_\fr^{\sss(j)}\leq T_v+f_n(1).
\end{equation}
In other words, a lucky vertex has enough descendants in time $f_n(1)$ that the integral in the definition \eqrefPartII{TfrDefn} of the freezing time must be at least $s_n$.

It has been proved in \refother{\reflemmaPartI{SumfninverseExponential}} that the distribution of
\begin{equation}\labelPartII{Sumfninverse}
\sum_{v\in\BP^{\sss(j)}_{T^{\sss(j)}_\lucky}} \left( f_n^{-1}\left( T^{\sss(j)}_\lucky-T_v \right)-1 \right)^+
\qquad \text{is exponential with rate }
\qquad\P(v\text{ is lucky})
\end{equation}
and that there is a constant $\delta>0$ so that for sufficiently large $n$,
\begin{equation}\labelPartII{LuckyProb}
\P(v\text{ is lucky}) \ge \delta/s_n.
\end{equation}

Now we are in the position to prove \reflemma{AtCollision}~\refitem{NoFrozenCollision}:
\begin{proof}[Proof of \reflemma{AtCollision}~\refitem{NoFrozenCollision}]
It suffices to show that the Cox intensity $Z_{n, \overline{t}}$ satisfies
\begin{equation}
\sum_{v_1 \in \cluster_{\overline{t}}^{\sss(1)}} \sum_{v_2 \in \cluster_{\overline{t}}^{\sss(2)}}
\indicator{\set{v_1,v_2} \intersect \cluster_{\fr} \neq \emptyset} Z_{n, \overline{t}}(\set{v_1} \times\set{v_2}) = o_{\P}(1),
\end{equation}
where we recall that $Z_{n,t}(\set{v_1} \times \set{v_2})= \tfrac{1}{n} \mu_n\bigl(\Delta R_{v_1,v_2}, R_1(t)-R_1(T_{v_1}^\cluster)+R_2(t)-R_2(T_{v_2}^\cluster)\bigr)$.

We begin with the contribution to $Z_{n,\overline{t}}$ arising from the restriction of $\mu_n$ to $(f_n(1),\infty)$.
Note that, by construction, $R_j(\overline{t})-R_j(T_\fr^{\sss(j)})=\lambda_n(1)^{-1}(\tfrac{1}{2}\log(n/s_n^3)+K)$ and $\left( R_j(T_\fr^{\sss(j)})-R_j(T_{v_j}^\cluster) \right)^+ = \left( T_\fr^{\sss(j)}-T_{v_j}^\cluster \right)^+$.
Hence
\begin{align}
& \tfrac{1}{n}\mu_n\big\vert_{(f_n(1),\infty)}\left( \Delta R_{v_1,v_2}, R_1(\overline{t})-R_1(T_{v_1}^\cluster)+R_2(\overline{t})-R_2(T_{v_2}^\cluster) \right)
\notag\\&\quad
\leq
\tfrac{1}{n} \mu_n\left( f_n(1), \frac{\log(n/s_n^3)+2K}{\lambda_n(1)}+R_1(T_\fr^{\sss(1)})-R_1(T_{v_1}^\cluster)+R_2(T_\fr^{\sss(2)})-R_2(T_{v_2}^\cluster) \right)
\notag\\&\quad
\leq
\tfrac{1}{n}\mu_n\left( f_n(1),3\frac{\log(n/s_n^3)+2K}{\lambda_n(1)} \right) + \sum_{j=1}^2 \tfrac{1}{n}\mu_n\left( f_n(1), f_n(1)\vee 3\left( T_\fr^{\sss(j)}-T_{v_j}^\cluster \right)^+ \right)
.
\labelPartII{3WaySum}
\end{align}
\reflemma{munDensityBounded} and \eqrefPartII{lambdanAsymp} imply that the first term in \eqrefPartII{3WaySum} is $O(1) (\log(n/s_n^3)+K)/(ns_n)$.
In the second term, the $j^\th$ summand is zero if $v_j \not\in \cluster_\fr^{\sss(j)}$.
For $v_j \in \cluster_\fr^{\sss(j)}$, we consider separately the intervals
\begin{equation}
\begin{split}
I_1&=\left(f_n(1), f_n(1) \vee \left( T_{\mathrm{lucky}}^{\sss(j)}-T_{v_j}^\cluster \right)^+\right),\\
I_2&=\left( f_n(1)\vee \left( T_{\mathrm{lucky}}^{\sss(j)}-T_{v_j}^\cluster \right)^+, f_n(1)\vee \left( T_\fr^{\sss(j)}-T_{v_j}^\cluster \right)^+ \right), \\
I_3&= \left( f_n(1)\vee \left( T_\fr^{\sss(j)}-T_{v_j}^\cluster \right)^+, f_n(1)\vee 3\left( T_\fr^{\sss(j)}-T_{v_j}^\cluster \right)^+ \right),
\end{split}
\end{equation}
where $T_{\mathrm{lucky}}^{\sss(j)}$ was defined in \refdefn{Lucky}.
The definition of $\mu_n$ gives $\mu_n(I_1)=\bigl( f_n^{-1}(T_{\mathrm{lucky}}^{\sss(j)}-T_{v_j}^\cluster)-1 \bigr)^+$.
For $I_2$, note from \eqrefPartII{FreezingFromLucky} that $I_2$ is a subinterval of $(f_n(1),\infty)$ of length at most $f_n(1)$, so $\mu_n(I_2)\leq O(1/s_n)$ by \reflemma{munDensityBounded}.
For $I_3$, note from \reflemma{ExtendedImpliesWeak} that $f_n(3^{1/\epsilonCondition s_n} m)\geq 3f_n(m)$ for any $m\geq 1$.
It follows that $f_n^{-1}(3y)\leq 3^{1/\epsilonCondition s_n} f_n^{-1}(y)=(1+O(1/s_n))f_n^{-1}(y)$ uniformly over $y\geq f_n(1)$, so that $\mu_n(I_3)\leq O(1/s_n)f_n^{-1}(T_\fr^{\sss(j)})$.
We conclude that
\begin{align}
& \tfrac{1}{n}\mu_n\big\vert_{(f_n(1),\infty)}\left( \Delta R_{v_1,v_2}, R_1(\overline{t})-R_1(T_{v_1}^\cluster)+R_2(\overline{t})-R_2(T_{v_2}^\cluster) \right)
\notag\\&
\leq
O(1) \left(\frac{\log(n/s_n^3)}{n s_n} + \sum_{j=1}^2 \indicator{v_j \in \cluster_\fr^{\sss(j)}}\left( \tfrac{1}{n}\left( f_n^{-1}\left( T_{\mathrm{lucky}}^{\sss(j)}-T_{v_j}^\cluster \right) -1 \right)^+ + \frac{1+f_n^{-1}(T_\fr^{\sss(j)})}{n s_n} \right)\right)
.\labelPartII{intensityLargeYValues}
\end{align}
By \refthm{TfrScaling}, $f_n^{-1}(T_\fr^{\sss(j)})=O_\P(1)$.
Sum \eqrefPartII{intensityLargeYValues} over $v_1 \in \cluster_{\overline{t}}^{\sss(1)}, v_2 \in \cluster_{\overline{t}}^{\sss(2)}$ with $\set{v_1,v_2} \cap \cluster_{\fr} \neq \emptyset$ and use $T_{v_j}^\cluster=T_{v_j}$ for $v_j \in \cluster_\fr^{\sss(j)}$ to obtain
\begin{equation}\labelPartII{ContributionBound}
O(1) \sum_{\set{j,j'}=\set{1,2}} \abs{\cluster^{\sss(\smash{j'})}_{\overline{t}}} \Big(\frac{\log(n/s_n^3)+O_\P(1)}{n s_n}\abs{\cluster^{\sss(j)}_\fr} + \frac{1}{n} \sum_{v_j\in\cluster^{\sss(j)}_\fr} \left( f_n^{-1}\left( T_{\mathrm{lucky}}^{\sss(j)}-T_{v_j} \right) -1 \right)^+ \Big) .
\end{equation}
By \reflemma{AtCollision}~\refitem{VolumeAtCollision}, $\abs{\cluster^{\sss(\smash{j'})}_{\overline{t}}}=O_\P(\sqrt{n s_n})$, and
by \refthm{FrozenCluster}~\refitem{FrozenVolume},  $\abs{\cluster^{\sss(j)}_\fr}=O_\P(s_n^2)$.
In the sum over $v_j\in\cluster^{\sss(j)}_\fr$, only terms with $v_j\in\cluster^{\sss(j)}_{T_{\mathrm{lucky}}^{\sss(j)}}$ can contribute, so \eqrefPartII{FreezingFromLucky}--\eqrefPartII{LuckyProb} imply that the inner sum in \eqrefPartII{ContributionBound} is $O_\P(s_n)$.
Hence \eqrefPartII{ContributionBound} is $O_\P(\log(n/s_n^3)/\sqrt{n/s_n^3})$, which is $o_\P(1)$ since $n/s_n^3\to\infty$.

We now turn to the contribution to $Z_{n,\overline{t}}$ arising from the restriction of $\mu_n$ to $[0,f_n(1)]$ and split the sum into three groups of vertex pairs.
Let $(J,J')$ denote the random ordering of $\set{1,2}$ for which $T_\fr^{\sss(J)} < T_\fr^{\sss(J')}$.
The first group of vertex pairs are those with $v_J \in \cluster_{\overline{t}}^{\sss(J)}$ and $v_{J'} \in \cluster_{T_\fr^{\sss (J)}}^{\sss(J')}$.
That is, the vertex in the slower-growing cluster is born before the faster cluster freezes.
We show that the number $\bigabs{\cluster_{T_\fr^{\sss (J)}}^{\sss(J')}}$ of such choices for $v_{J'}$ is $O_\P(1)$.
By \refthm{TfrScaling}, $f_n^{-1}(T_\fr^{\sss(j)}) \convp M^{\sss(j)}$ for $j\in \set{1,2}$ and $M^{\sss(1)} \neq M^{\sss(2)}$ a.s.
Hence $T_\fr^{\sss(J)} < f_n(M^{\sss(J')})$ whp.
A vertex $v'\in\cluster^{\sss(J')}_\fr$ with $T_{v'}^\cluster<f_n(M^{\sss(J')})$ must be connected to the root $\emptyset_{J'}$ by edges all of PWIT edge weight less than $M^{\sss(J')}$.
The number of such vertices is finite and independent of $n$, and is in particular $O_\P(1)$.
Since the measure $\mu_n\big\vert_{[0,f_n(1)]}$ has total mass $1$ by construction, the total contribution to $Z_{n,\overline{t}}$ of this group of vertex pairs is at most $\tfrac{1}{n}\abs{\cluster_{\overline{t}}^{\sss(J)}}O_\P(1)$, which is $o_\P(1)$ by \reflemma{AtCollision}~\refitem{VolumeAtCollision}.

The second group are pairs $(v_1,v_2)$ with $T_{v_{J'}}^\cluster\geq T_\fr^{\sss(J)}\geq T_{v_J}^\cluster$.
For these pairs, by \eqrefPartII{DeltaRcluster},
\begin{equation}\labelPartII{DeltaRBound}
\Delta R_{v_1,v_2} = R_J(T_{v_{J'}}^\cluster)-R_J(T_{v_J}^\cluster)\geq R_J(T_\fr^{\sss(J)})-R_J(T_{v_J}^\cluster)=T_\fr^{\sss(J)}-T_{v_J}^\cluster.
\end{equation}
We can therefore bound the contribution to $Z_{n,\overline{t}}$ in terms of the contribution to $\mu_{n,\fr}^{\sss(J)}$:
\begin{align}
&\frac{1}{n}\sum_{v_1,v_2\colon T_{v_{J'}}^\cluster\geq T_\fr^{\sss(J)}\geq T_{v_J}^\cluster} \mu_n\big\vert_{[0,f_n(1)]}\left( \Delta R_{v_1,v_2}, R_1(\overline{t})-R_1(T_{v_1}^\cluster)+R_2(\overline{t})-R_2(T_{v_2}^\cluster) \right)
\notag\\&\quad
\leq
\tfrac{1}{n}\abs{\cluster^{\sss(\smash{J'})}_{\overline{t}}}  \sum_{v_J\in\cluster^{\sss(J)}_\fr} \mu_n\left( f_n(1)\wedge \left( T_\fr^{\sss(J)}-T_{v_J}^\cluster \right), f_n(1) \right)
\notag\\&\quad
\leq
\tfrac{1}{n} O_\P(\sqrt{n s_n}) \sum_{v_J\in\cluster^{\sss(J)}_\fr} \e^{\lambda_n(1)f_n(1)} \int_0^\infty \e^{-\lambda_n(1)y} \mu_n\left( T_\fr^{\sss(J)}-T_{v_J}^\cluster +dy \right)
\leq
O_\P(\sqrt{s_n^3/n})
\labelPartII{Group2Bound}
\end{align}
by \reflemma{AtCollision}~\refitem{VolumeAtCollision}, \reflemma{ExpectedUnfrozenChildren} and \eqrefPartII{lambdanAsymp}.
The last group of vertex pairs are those with $T_{v_J}^\cluster \ge T_\fr^{\sss(J')} \ge T_{v_{J'}}^\cluster$.
These pairs satisfy $\Delta R_{v_1,v_2} \geq T_\fr^{\sss(J')}-T_{v_{J'}}^\cluster$ instead of \eqrefPartII{DeltaRBound}, and their contribution can therefore be handled as in \eqrefPartII{Group2Bound} with $J$ and $J'$ interchanged.
\end{proof}

\noindent
{\bf{Acknowledgements:}} A substantial part of this work has been done at Eurandom and Eindhoven University of Technology.
ME and JG are grateful to both institutions for their hospitality.\\
The work of JG was carried out in part while at Leiden University (supported in part by the European Research Council grant VARIS 267356), the Technion, and the University of Auckland (supported by the Marsden Fund, administered by the Royal Society of New Zealand) and JG thanks his hosts at those institutions for their hospitality.
The work of RvdH is supported by the Netherlands
Organisation for Scientific Research (NWO) through VICI grant 639.033.806
and the Gravitation {\sc Networks} grant 024.002.003.
\vskip1cm

\labelPartII{notation}
\paragraph{A short guide to notation:}
\begin{itemize}
\item $\SWT^{\sss(j)}_t$ is SWT from vertex $j\in\set{1,2}$.
\item$\cS^{\sss(j)}_t$ is the SWT from vertex $j\in \set{1,2}$ such that $\cS^{\sss(1)}$ and $\cS^{\sss(2)}$ cannot merge and with an appropriate freezing procedure.
\item $\cS_t=\cS_t^{\sss(1)}\cup \cS_t^{\sss(2)}$
\item $\BP^{\sss(j)}$ is branching process copy number $j$ where $j\in \set{1,2}$, without freezing.
\item $\cluster^{\sss(j)}$ is branching process copy number $j$ where $j\in \set{1,2}$, with freezing.
\item $\cluster_t$ is the union of 2 CTBPs with the appropriate freezing of one of them.
\item $\thinnedcluster_t$ is the union of 2 CTBPs with the appropriate freezing of one of them, and the resulting thinning.
Thus, $\thinnedcluster_t$ has the same law as the frozen $\cS_t$.
\item $f_n$ is the function with $Y_e^{\sss(K_n)}\overset{d}{=} f_n(nE)$, where $E$ is exponential with mean $1$.
\item $\mu_n$ is the image of the Lebesgue measure on $(0,\infty)$ under $f_n$
\item $\lambda_n(a)$ is the exponential growth rate of the CTBP, cf.
\eqrefPartII{lambdaaDefn}.
\item $z_t^{\chi}(a)$ and $z_{\vec{t}}^{\chi}(\vec{a})$ are the generation-weighted vertex characteristics from one and two vertices, respectively, cf.\eqrefPartII{1VertexCharDef} and \eqrefPartII{2VertexCharDef}.
\item $\bar{m}_t^{\chi}(a)$ and $\bar{m}_{\vec{t}}^{\chi}(\vec{a})$ are the expected, rescaled vertex characteristics, cf.\ \eqrefPartII{barredOneVertexChars} and \eqrefPartII{barredTwoVertexChars}, respectively
\end{itemize}

{\small \bibliographystyle{plain}
\bibliography{FPP}}

\end{document}